\documentclass[a4paper,10pt]{article}
\usepackage[utf8x]{inputenc}

\usepackage{amssymb,a4wide}
\usepackage{amsmath,bm}
\usepackage{amsthm}
\usepackage[english]{babel}
\usepackage{lmodern}
\usepackage{microtype}
\usepackage[mathscr]{euscript}
\usepackage[colorlinks=true]{hyperref}
\hypersetup{
           breaklinks=true,   
           colorlinks=true,   
           pdfusetitle=true,  
        }
\usepackage[usenames,dvipsnames]{xcolor}
\usepackage{todonotes}
\usepackage{appendix}
\usepackage{cleveref}
\usepackage{bm}\usepackage{cancel}
\usepackage{upgreek}
\usepackage[normalem]{ulem}

\numberwithin{equation}{section}

\theoremstyle{plain}
\newtheorem{definition}{Definition}[section]

\newtheorem{theorem}[definition]{Theorem}
\newtheorem{proposition}[definition]{Proposition}
\newtheorem{lemma}[definition]{Lemma}
\newtheorem{corollary}[definition]{Corollary}

\theoremstyle{definition}
\newtheorem{remark}[definition]{Remark}
\newtheorem{example}[definition]{Example}

\newcommand\red[1]{{\color{black} {#1}}}

\newcommand{\R}{\mathbf R}
\newcommand{\bS}{{\mathbf S}}
\newcommand{\Q}{\mathbf Q}
\newcommand{\N}{\mathbf N}

\renewcommand{\d}{\mathrm{d}}
\newcommand{\n}{\mathsf{n}}
\newcommand{\Z}{\mathbf{Z}}

\newcommand{\eps}{\varepsilon}

\usepackage[xcolor]{changebar}
\cbcolor{blue}

\newcommand{\X}{\mathsf{X}}

\newcommand{\cone}{\mathcal{C}}

\newcommand{\s}{{\sf s}}
\newcommand{\h}{{h}}

\usepackage[inline]{enumitem}
\newcommand{\enumlabelformat}{\roman}

\newlength{\thelabelsep}
\newcounter{inlineenum}
\renewcommand{\theinlineenum}{\enumlabelformat{inlineenum}}

\let\epsilon\varepsilon
\let\phi\varphi

\newcommand{\push}{*}
\newcommand{\KE}{\mathscr{A}}
\newcommand{\met}{\mathsf{d}}
\newcommand{\meas}{\mathfrak{m}}

\newcommand{\spt}{\supp}
\newcommand{\p}{\partial}
\newcommand{\M}{{\rm M}}

\newcommand{\Prob}{\mathscr{P}}

\DeclareMathOperator{\supp}{spt}

\newcommand{\scrS}{\mathscr{S}}

\newcommand{\bdpi}{\boldsymbol{\pi}}
\newcommand{\eval}{\mathsf{e}}
\newcommand{\CC}{\mathrm{LCC}}
\newcommand{\CD}{\mathsf{CD}}
\newcommand{\RCD}{\mathsf{RCD}}
\newcommand{\MCP}{\mathsf{MCP}}
\newcommand{\restr}{\mathrm{restr}}
\newcommand{\mres}{\mathbin{\vrule height 1.6ex depth 0pt width
0.13ex\vrule height 0.13ex depth 0pt width 1.3ex}}
\newcommand{\rmd}{\mathrm{d}}

\newcommand{\TCD}{\mathsf{TCD}}
\newcommand{\TMCP}{\mathsf{TMCP}}
\newcommand{\LTCD}{\mathsf{LTCD}}

\DeclareMathOperator*{\esssup}{ess\,sup}

\newcommand{\Leb}{\mathscr{L}}
\newcommand{\TGeo}{\mathrm{TGeo}}
\newcommand{\CGeo}{\mathrm{CGeo}}
\newcommand{\AGeo}{\mathrm{AGeo}}

\newcommand{\vg}{\mathrm{vol}_g}

\newcommand{\mm}{\meas}
\renewcommand{\div}{{\rm div}}
\newcommand{\pert}{{\rm Pert}}
   
\newcommand{\limi}{\liminf}
\newcommand{\lims}{\limsup}
\newcommand{\sfd}{{\sf d}}
\newcommand{\sfD}{\mathsf{D}}
\newcommand{\Pert}{{\rm Pert}}
\newcommand{\dom}{\mathrm{Dom}}
\newcommand{\ppi}{\bdpi}
\newcommand{\e}{\eval}
\newcommand{\nchi}{{\raise.3ex\hbox{$\chi$}}}
\newcommand{\pem}{{\Prob_{\sf em}(\M)}}

\usepackage[runin]{abstract}

\makeatletter
\let\@fnsymbol\@arabic
\makeatother

\newcommand{\PP}{{p}}
\newcommand{\QQ}{{q}}

\newcommand{\fff}{initial }

\newcommand{\sym}{{\mathrm{sym}}}

\newcommand{\be}{{\mathrm{be}}}
\newcommand{\emr}{{\mathrm{em}}}
\newcommand{\intm}{{\sf Int}}
\newcommand{\IKN}{\intm^{K,N}}
\allowdisplaybreaks

\def\Xint#1{\mathchoice
{\XXint\displaystyle\textstyle{#1}}%
{\XXint\textstyle\scriptstyle{#1}}%
{\XXint\scriptstyle\scriptscriptstyle{#1}}%
{\XXint\scriptscriptstyle\scriptscriptstyle{#1}}%
\!\int}
\def\XXint#1#2#3{{\setbox0=\hbox{$#1{#2#3}{\int}$ }
\vcenter{\hbox{$#2#3$ }}\kern-.6\wd0}}

\def\dashint{\Xint-}

\begin{document}
\title{
A nonlinear 
d'Alembert comparison theorem\\ and  causal 
differential calculus on 
\\ metric measure spacetimes
}
\author{Tobias Beran\thanks{Faculty of Mathematics, University of Vienna, Oskar-Morgenstern-Platz 1, 1090 Vienna, Austria. \texttt{tobias.beran@univie.ac.at}},\ \ Mathias Braun\thanks{Institute of Mathematics, EPFL, 1015 Lausanne, Switzerland. \texttt{mathias.braun@epfl.ch}},\ \ Matteo Calisti\thanks{Faculty of Mathematics, University of Vienna, Oskar-Morgenstern-Platz 1, 1090 Vienna, Austria. \texttt{matteo.calisti@univie.ac.at}},\\Nicola Gigli\thanks{SISSA, Via Bonomea 265, 34136 Trieste, Italy.
\texttt{ngigli@sissa.it}},\ \ Robert J. McCann\thanks{Department of Mathematics, University of Toronto, 40 St. George Street Room 6290, Toronto, Ontario M5S 2E4, Canada. \texttt{mccann@math.toronto.edu}},\ \ Argam Ohanyan\thanks{Department of Mathematics, University of Toronto, 45. St. George Streem Room PG 111, Toronto, Ontario M5S 2E5, Canada. \texttt{argam.ohanyan@utoronto.ca}},\\Felix Rott\thanks{SISSA, Via Bonomea 265, 34136 Trieste, Italy.
\texttt{frott@sissa.it}},\ \ Clemens S\"amann\thanks{{Faculty of Mathematics, University of Vienna, Oskar-Morgenstern-Platz 1, 1090 Vienna, Austria. \texttt{clemens.saemann@univie.ac.at}}}}

\date{\today}


\maketitle

\setlength{\abstitleskip}{-\absparindent}
\abslabeldelim{\ \ }

\begin{abstract} We introduce a variational first-order Sobolev calculus on metric measure spacetimes. The key object is the \red{maximal} weak subslope of an arbitrary causal function, which plays the role of the (Lorentzian) modulus of its differential.  It is shown to satisfy certain chain and Leibniz rules, certify a locality property, and be compatible with its smooth analog.  In this setup, we propose a quadraticity condition termed {\em infinitesimal Minkowskianity}, which singles out genuinely Lorentzian structures  among Lorentz--Finsler spacetimes. 
Moreover, we establish a comparison theorem for a nonlinear yet elliptic $\PP$-d'Alembertian in a weak form under the timelike measure contraction property. As a particular case, this extends Eschenburg's classical estimate past the timelike cut locus.

\vskip 1em

\medskip


\end{abstract}

\clearpage

\tableofcontents

\section{Introduction}

Einstein's theory of gravity is set on a smooth Lorentzian manifold, but the singularity theorems 
of Hawking \cite{Haw66} and Penrose \cite{Pen65} show one cannot expect smoothness to continue
to hold in the far past (and future): the Big Bang and black hole singularities are unavoidable consequences of the theory.
Furthermore, the linearization of the Einstein field equations which relate the Ricci curvature of spacetime to its matter content ---
being a wave equation --- propagates singularities (in the form of gravitational waves) without smoothing them.   For these reasons and others
discussed in \cite{KS:18} and its references,
 it appears highly desirable
to relax the smooth Lorentzian framework of Einstein's theory to allow for rougher, more flexible, settings.

Developments in positive (Riemannian) signature provide a model for such relaxations. The first is the metric geometry of Alexandrov \cite{Aleksandrov51}, Gromov \cite{Gromov1999MetricStructures}, and others,  
where sectional curvature bounds are defined using triangle comparison. The second is the metric measure geometry developed following work of Fukaya \cite{Fuk87a} and Bakry--\'Emery \cite{BK85}, 
in which a measure is also  used to define Ricci bounds, as in 
\begin{itemize}
    \item the Ricci limit spaces of Cheeger--Colding \cite{CC97, CC20_1, CC20_2},  \item the curvature-dimension spaces $\CD(K,N)$ of Sturm \cite{Sturm:2006a,Sturm:2006b} and Lott--Villani \cite{LottVillani:2009}, 
which satisfy a lower bound of $K$ on the Ricci curvature and upper bound of $N$ on the dimension \red{expressed using refined versions of McCann's entropic displacement convexity \cite{McCann97,OttoVillani00,CorderoMcCannSchmuckenschlager01,vonRenesseSturm05},}  and 
\item the Riemannian $\RCD(K,N)$ spaces of Ambrosio--Gigli--Savar\'e and Gigli, which  for $N=\infty$ \cite{AGS:14b} and $N\in (1,+\infty)$ \cite{Gigli:2015} single out among curvature-dimension spaces those which are infinitesimally Hilbertian.  
\end{itemize}
It is in the last setting that several of the most powerful results of the theory have been achieved,  such as the splitting theorem of Gigli \cite{gigli2013} and its consequences \cite{Ket15, MN19}, a second-order calculus~\cite{Gig:18}, and the non-branching result of Deng \cite{Den21}. Here ideas related to calculus of variations, gradient or heat flows, Sobolev calculus and integration by parts play a central role.

Recently, in Lorentzian signature,  several analogous steps have been taken.
Building on earlier works such as \cite{Busemann:1967, KronheimerPenrose:1967, AH98, AB08, Nol04, SW96}, 
Kunzinger--S\"amann proposed a theory of Lorentzian (pre)length spaces --- based largely on a time separation function --- and timelike triangle comparison \cite{KS:18}.
(They also required a causal order,  but this can be embedded into a modified time separation function which takes both signs as in McCann \cite{McCann:2020,McC:23}.)
By adding a measure, 
Cavalletti--Mondino were able develop a  theory of timelike curvature-dimension spaces $\TCD^e_\QQ(K,N)$ which manifest nonsmooth timelike lower Ricci bounds \cite{CM:20}. These are based on entropic convexity estimates from the smooth setting, that had there been shown to be equivalent to the strong energy condition of Hawking--Penrose by McCann \cite{McCann:2020}  
and to the Einstein field equations by Mondino--Suhr \cite{MS:22}.  
A burst of developments have followed, e.g.~Braun \cite{Braun:2023Good,Braun:2023Renyi}, Braun--McCann \cite{BraunMcCann:2023}, McCann \cite{McC:23}, Ketterer~\cite{Ket24}, and Cavalletti--Mondino \cite{CM24} with Manini \cite{CavallettiManiniMondino24+}. 
Of course there is also an abundance of work beyond the scope of Ricci curvature in this setting, including, e.g., by Burtscher--García-Heveling~\cite{BG-H:2024}, Kunzinger--Steinbauer \cite{KS:2022}, Müller \cite{Mue22,Mue24}, Minguzzi--Suhr \cite{MinguzziSuhr:2022},  Beran--Harvey--Napper--Rott~\cite{BHNR23}, Beran--Ohanyan--Rott--Solis  \cite{BeranOhanyanRottSolis23}, etc. 

Just as the $\CD(0,N)$ spaces of Lott--Sturm--Villani include all Banach spaces and Finsler manifolds of dimension $n \le N$ \cite{Oht09}, the $\TCD^e_\QQ(0,N)$ spaces of Cavalletti--Mondino
include all Finsler spacetimes of dimension $n \le N$~\cite{BO:23}. One of the goals of the present manuscript is to develop a nonsmooth criterion which distinguishes those 
timelike curvature-dimension spaces whose time separation function behaves as if it came from a signed inner-product;  we call such spaces {\em infinitesimally Minkowskian.}

\begin{definition}[Lorentzian $\TCD$ spaces] We will say a metric measure spacetime satisfies the \emph{Lorentzian (entropic)  timelike curvature-dimension condition}, briefly $\smash{\LTCD^e_q(K,N)}$, if it is an infinitesimally Minkowskian $\smash{\TCD_q^e(K,N)}$ metric measure spacetime;
{(c.f.  \eqref{eq:definfmink} below).} 
\end{definition}


This terminology is in analogy with the Riemannian curvature dimension condition of \cite{AGS:14b, Gigli:2015}. 
Just as infinitesimal Hilbertianity \cite{Gigli:2015} is essential for the metric splitting theorem of Gigli \cite{gigli2013}, so we expect infinitesimal Minkowskianity
to play an equally crucial role in obtaining a nonsmooth analog of the Eschenburg, Galloway and Newman splitting theorems from Lorentzian geometry 
\cite{Eschenburg:1988, Galloway:89, Newman:90}. 
Having this goal in mind, we also pursue the development of 
a nonsmooth Lorentzian analog of a negative-homogeneity d'Alembert operator,
which --- like the positive-homogeneity Riemannian $p$-Laplace operator and in sharp contrast to the the usual 1-homogeneous wave operator --- is  elliptic, albeit non-uniformly, 
on non-decreasing (viz.~causal) functions.
Indeed, we establish that for $0\ne \PP<1$ the $\PP$-d'Alembert operator satisfies a comparison inequality,  which estimates its value on Lorentz distance functions by a bound in terms of $K$ and $N$ that becomes sharp in constant curvature spaces. This is an important step toward a nonsmooth metric splitting theorem
under timelike Ricci rather than sectional curvature \cite{BeranOhanyanRottSolis23} bounds.
To do so requires us to develop a nonsmooth Sobolev calculus for (future-directed) causal 
curves and functions, which may be of independent interest.
One of the challenges we face is that such 
functions --- due to the monotonicity enforced by the arrow of time --- form a convex Sobolev cone,  rather than a Sobolev space.
In separate work inspired by the present one, five of us \cite{QuintetEllipticsplitting} show our $\PP$-d'Alembert comparison bound leads to a new proof of the Lorentzian splitting theorems due to Eschenburg, Galloway and Newman in the smooth setting. Using this $p$-d'Alembertian approach to the Lorentzian splitting theorems, Caponio--Ohanyan--Ohta \cite{CaponioOhanyanOhto24+} are able to prove a splitting theorem for Finsler spacetimes, generalizing a previous such result due to Lu--Minguzzi--Ohta \cite{LuMinguzziOhtasplitting23}. Moreover, our method can be used to prove Lorentzian splitting theorems for metrics of low regularity, see \cite{QuintetNonsmooth25+}.

\subsection*{Setting and overview}

Let us now give an overview of the strategy, organization and results of the present paper. As detailed in \S\ref{Ch:Curves}, we work in a setting slightly different from that of previous authors; we do not require any form of global hyperbolicity and replace it with some form of order completeness. This allows our framework to accommodate infinite dimensions, possibly facilitating its amenability to an eventual passage to a quantum theory of gravity, in whatever form that may take.
We start by introducing the concept of metric spacetime, namely a set $\M$ equipped with a (signed) time separation function $\ell\colon \M\times \M \to \{-\infty\} \cup [0,+\infty]$ which (is zero on the diagonal and) satisfies the reverse triangle inequality.  We use $\ell$ to define the transitive relations $\{ \ell \ge 0\}$ and $\{ \ell>0\}$;
these relations are called {\em causal} and {\em chronological}, and denoted by $\le$ and $\ll$ respectively. To develop any sort of calculus we need to have a topology $\uptau$ on $\M$: rather than fixing a predetermined one, we shall leave the choice of topology independent from that of time separation, asking only for minimal relations. We shall typically assume that $\uptau$ is Polish and that the set $\{\ell\geq 0\}$ is closed in $\M^2$ and call the resulting structure $(\M,\uptau,\ell)$ a Polish metric spacetime. Then, inspired by the Dedekind rather than Cauchy's strategy for constructing the real numbers, 
we say that  $(\M,\uptau ,\ell)$ is {\em forward} if every non-decreasing
sequence $(y_i)_{i \in \Z} \subset \M$ which is order-bounded $x \le y_i \le y_{i+1} \le z$ admits a future limit $y_{+\infty} = \lim_{i \to +\infty} y_i$.  
Similarly $(\M,\uptau,\ell)$ is {\em backward} if every such sequence admits a past limit
$y_{-\infty} = \lim_{i \to -\infty} y_i$. In principle, one might decide to work with spacetimes that are both forward and backward,  but we eschew this assumption as we are skeptical about its stability (see also the discussion in \cite{Gigli24+}). Our most important existence results (such as the limit curve \Cref{prop-cad-comp}  or the lifting \Cref{Th:Lifting}) are stated on forward spacetimes, typically under the additional relation between $\uptau$ and $\ell$ that the topology is {\em locally causally convex}, i.e.\ every point has a neighbourhood basis made of causally convex sets. All these results  could equally well be formulated in {backward} spacetimes. We show {forward} metric spacetimes lie somewhere between causally simple and globally hyperbolic metric spacetimes on the nonsmooth causal ladder (\Cref{R:causal ladder}). 

A curve $\gamma\colon [0,1]\to M$ is called {\em causal} if $s\le t$ implies $\gamma_s  \le \gamma_t$; dually,  a function $f\colon \M \to [-\infty,+\infty]$ 
is called {\em causal} if $x \le y$ implies $f(x) \le f(y)$.   Since our space $M$ may not be a manifold, we need to develop structures to compensate for the absence
of a tangent and cotangent space.  We rely on causal curves for the former purpose,  and causal functions  for the latter.  
The nonsmoothness of our spaces prevents us from defining tangent or cotangent vectors pointwise;  instead --- inspired by \cite{AGS:14a,Gigli:2015} 
--- we represent causal tangent fields using
probability measures $\bdpi$ on causal curves,  and (exact, future-directed) cotangent fields using causal functions $f$ or more heuristically, their differentials.  The operation of horizontally differentiating
$f$ along $\bdpi$,  which corresponds to a global average of the initial derivative of $f$ along the vector field represented by $\bdpi$, 
defines a positively bilinear pairing of these objects. 
These proxies for causal tangent and cotangent fields  are the subjects of \S\ref{Ch:Curves} and \S\ref{Ch:differential}, respectively.  We shall also need to set up a convex-analytic Legendre transform to identify tangent fields with cotangent fields and vice versa;  
this operation depends on $\ell$ and a conjugate pair of exponents $\PP^{-1}+\QQ^{-1}=1$ with $0\ne \QQ<1$ and 
is homogeneous but --- as in \cite{Gigli:2015} and in contrast to both smooth Lorentzian and nonsmooth infinitesimally Hilbertian geometry --- 
is nonlinear and often multivalued. These challenges --- and the fact that causal functions form a convex cone rather than a vector space  --- are addressed in \S\ref{ch:horizontal and vertical derivatives}, where we also compute the ``vertical'' (i.e.~outer variational) derivative of the $\PP$-Dirichlet energy at $f$,  and relate it to the tangent field proxies $\bdpi$ corresponding to the nonlinear Legendre-transform of  the cotangent field proxied by $f$.  Inspired by analogous results concerning the nonlinear Laplacian which arises in the Finslerian \cite{OS:09}, Hamiltonian \cite{Ohta:2014}, and 
nonsmooth  \cite{Gigli:2015}  settings, in \S\ref{Section: Effects of curvature assumption} we {define}
and  estimate relative to constant curvature models the $\PP$-d'Alembert operator $\square_\PP f$ in terms of the curvature and dimension parameters $(K,N)$ of the forward metric measure spacetime.

\begin{remark}[
Beyond global hyperbolicity and $q<0$]\label{R:q<0,forward}
Although we allow $0\ne q<1$ of both signs,
in the Lorentzian optimal transport literature to date \cite{EM:17, Suh:18, McCann:2020, MS:22, CM:20, CavallettiMondino22, Braun:2023Renyi, Braun:2023Good, BraunMcCann:2023}, only the cost function $\ell^q/q$ with $q \in (0,1)$ (sometimes with $q=1$ included) has been considered. However, since the function $u_q(z):=z^q/q$ (c.f.\ \Cref{Sub:InfConv}) has the same monotonicity and concavity properties for $q \in (0,1)$ as for $q < 0$, the majority of the fundamental results established in the above literature can be generalized to the case $q < 0$. Moreover, global hyperbolicity is almost universally assumed in the references cited above, but it turns out that for most purposes it suffices that the given spacetime be forward (c.f.\ \Cref{D:spacetimes}). Any result we mention in this work for $0 \neq q < 1$ on forward spacetimes while citing previous works belongs to the category of results that can be generalized in this manner; it is mainly in \S\ref{Section: Effects of curvature assumption} that such results are required and each of these few is flagged by a reference to Remark \ref{R:q<0,forward}.
The detailed treatment of these generalizations to the setting $q < 0$ (and even more general Lorentz costs) and to forward spacetimes is the content of an ongoing work in progress~\cite{CalistiOhanyanSalamo}.\hfill$\blacksquare$
\end{remark}

We now lay out the contents of each section in further detail,  with an interlude between sections \S\ref{ss:Curves} and \S\ref{ss:differential} describing more of the smooth motivation for the ensuing constructions. Before doing so, let us note that inspired by the current work, a very satisfactory and more constructive approach to the existence of the $\PP$-d'Alembert operator has been developed by Braun \cite{braun+}. This enables him not only to establish many of its fundamental properties for the first time,  but also to give an exact representation formula covering most cases of interest
using one-dimensional localization and especially \cite{CM:20,CM24,BraunMcCann:2023} in $\QQ$-essentially timelike non-branching spaces \cite{Braun:2023Renyi}.

\subsection[Order-completeness; tangent notions for worldlines,  curves of measures, and measures over worldlines]{Order-completeness; tangent notions for worldlines, curves of\\ measures, and measures over worldlines}
\label{ss:Curves}

\red{Much like Cauchy-completeness is central in classical metric geometry to obtain many of the key existence results in that setting, so is the order-completeness encoded in the definition of forward (or backward) spacetimes in our setting for similar purposes. The first instance where this naturally occurs is in the class of curves we consider: in a forward spacetime any causal curve bounded from above and defined on a dense subset of $[0,1]$ including zero can uniquely be extended to a left continuous causal curve on $[0,1]$ and if the underlying spacetime is Polish, then so is the space of left-continuous causal curves equipped with a natural $L^1$-distance. Under suitable compactness assumptions on the set of curves or the underlying spacetime, this completeness of the space of curves improves to compactness (\Cref{prop-cad-comp}), a result that we interpret as a limit curve theorem in our setting (c.f.~\cite{Min08}) and is reminiscent of Helly’s selection theorem.}

A leitmotiv of this work will be reliance on causal monotonicity 
 in Lorentzian signature instead of the 
 Lipschitz or absolute continuity 
 typically required in Riemannian signature.
 After laying out the details of our setting,
 the remainder of \S\ref{Ch:Curves} is largely devoted to defining the speed of a causal curve.
 Heuristically,  one would expect this to be given pointwise by a limit such as
 \begin{equation}\label{naive causal speed}
|\dot \gamma_t |=  \lim_{h \downarrow 0} \frac{\ell(\gamma_t,\gamma_{t+h})}{h}.
 \end{equation}
Instead,  we take this limit in the distributional sense
and call the resulting measure the {\em causal speed} $|\dot{\bm\gamma}|$ of $\gamma$. Moreover, we show
$|\dot{\bm\gamma}|$ coincides on the interval $(0,1)$ with
the maximum non-negative Radon measure $\nu$  satisfying $\nu((s,t)) \le \ell(\gamma_s,\gamma_t)$
for all $0< s< t<1$.  The latter is shown to exist and to be unique by differentiation of the time separation function restricted to the curve in \S\ref{ss:causal speed};
its absolutely continuous part agrees with the pointwise limit \eqref{naive causal speed} for almost all $t \in [0,1]$.
When $\ell(\gamma_s,\gamma_t) = G(t)-G(s)$ for some non-decreasing function $G$, then $\nu =G'$ coincides with the distributional derivative  $G'$ of $G$.

Associated to the causal speed are various power-law actions (with power $0 \neq \QQ<1$) whose maximizers are by definition {\em causal geodesics.} This point of view has the virtue of being inherited by
curves in the space $\Prob(\M)$ of probability measures on $\M$: using optimal transport to maximize the cost function $\ell(x,y)^\QQ$ over couplings with prescribed marginals $\mu_0$ and $\mu_1$
defines a (signed) time separation function 
$\ell_\QQ(\mu_0,\mu_1)$ 
on the space  $\Prob(\M)$ of Borel probability measures. In \Cref{eq:probmellq} we show that if $(\M,\uptau,\ell)$ is a Polish (forward) metric spacetime, then so is $(\Prob(\M),\ell_q)$ equipped with the narrow topology.  A relevant result of the chapter, and the main one of \Cref{Sub:Lifting} is the  Lorentzified version of a lifting theorem of Lisini \cite{Lisini:2007}, see also \cite{AGS:08}, that
asserts that any suitable causal curve $\mu\colon [0,1]\to \Prob(\M)$  of measures can be lifted to a probability measure $\bdpi$ on causal curves --- called a {\em plan}
(or ``dynamical plan'') ---
such that  the expectation of the $\QQ$-action of the individual curves $\gamma$ under $\bdpi$ agrees with the
$\QQ$-action of the path $\mu$ in $(\Prob(\M),\ell_\QQ)$.  In particular, 
{causal geodesics for $\ell_\QQ$ lift to plans $\bdpi$ concentrated on
 causal geodesic} curves $\gamma$ for $\ell$.  We  use the plan  $\bdpi$ as a proxy for a tangent to the curve of measures:  whereas the causal speed represents only the magnitude of the tangent,  the plan also encodes its direction.
 
The final section of the chapter discusses various non-branching notions. A notable one that arises in this setting is that of `timelike non-branching at 0 (resp.\ 1)'. A spacetime has this property if any two timelike geodesics that agree on the whole $(0,1)$ also agree at 0 (resp.\ at 1). Of course, in positive signature this always occurs by the continuity of geodesics, but in our framework this cannot be excluded, and in some circumstances it is necessary to ask for this property, at least in a suitable $\mm$-a.e.\ version, see \Cref{D:$p$-regularity}.

\subsection{Smooth paradigm: 
Lagrangian-Hamiltonian duality}

In \S\ref{Ch:differential} we explore a cotangent calculus in duality with the tangent concepts from \S\ref{Ch:Curves}.
This must be developed using positive linear rather than linear operations: the need to rely on causality of our functions and curves means their tangent spaces should be represented by convex cones rather than vector spaces.  We therefore resort to concepts from convex analysis.  Let us recall them
in the familiar setting of an $n$-dimensional time-oriented Lorentzian manifold $M$ equipped with a smooth globally hyperbolic metric tensor $g$ of signature $(+,-,\cdots,-)$. A vector $v \in T_xM$ is called called {\em causal} if $g(v,v) \ge 0$ and {\em timelike} if the previous inequality holds strictly. Moreover the analogous terminology applies also to covectors, curves and functions: a curve is said to have property Q if and only if its tangent vectors all have property Q; a smooth function  $f \in C^1(\M)$ is said to have property Q if and only if $\d f(x)$ has property Q for all $x \in \M$. The set of timelike vectors has two connected components, called the {\em future} and {\em past}; the vectors in the closure of the future component are called {\em future-directed}. The time-orientability of $M$ means the future can (and has) been chosen to vary continuously with $x \in \M$.  A causal co-vector is called {\em future-directed} if its action on one (and hence all) timelike future vectors is non-negative; it is called past-directed otherwise. Fix dual exponents $0\ne \QQ<1$ and $\PP^{-1}+ \QQ^{-1}=1$. On the tangent space $T_xM$ to $M$ at $x$ we define the Lagrangian
 \begin{equation*}
 L(v) := 
 \begin{cases}
 \displaystyle\frac1\QQ\, g(v,v)^{\QQ/2} & \mbox{if $v$ is future-directed causal}
 \\ -\infty & {\rm else},
 \end{cases}
 \end{equation*}
 with the convention that $0^\QQ= +\infty$ if $\QQ<0$.
Remarkably, $L$ is a concave function \cite{McCann:2020, Minguzzi2019}; its strict concavity in the timelike future was fundamental for McCann \cite{McCann:2020} and Mondino--Suhr \cite{MS:22}.
We define the $\QQ$-action of a smooth curve $\gamma\colon [0,1] \to  M$ by integrating $L$ over the curve.
We define the time separation $\ell(x,y)$ between two events $x,y \in \M$ by declaring $\tfrac1q \ell(x,y)^\QQ$  to be the supremum of the $\QQ$-actions of curves joining 
$\gamma_0=x$ to $\gamma_1=y$, where it is intended --- and consistent --- that $\ell(x,y)=-\infty$ if no  future-directed causal curve from $x$ to $y$ exists (as in \cite{McCann:2020}); this definition is independent of $0\ne \QQ<1$. Global hyperbolicity is well-known to ensure this supremum is attained unless no future-directed causal curve joins $x$ to $y$; the maximizers have constant speed and, when timelike, satisfy the geodesic equation.
  We define the concave Hamiltonian dual to $L$ on the cotangent space by
 \begin{align*}
 H(\omega) :&= \inf_{v \in T_x M} \big[\omega(v) - L(v)\big]
= \begin{cases}
 \displaystyle\frac1\PP\, g^*(\omega,\omega)^{\PP/2} & \mbox{if $\omega$ is future-directed causal}
 \\ -\infty & {\rm else},
 \end{cases}
 \end{align*}
 where $g^*$ is the cotangent metric induced by $g$. 
 Obviously, $H$ is also concave and together with $L$ satisfies the Fenchel--Young inequality
 \begin{equation}\label{new Fenchel-Young}
 L(v) + H(\omega) \le \omega(v);
 \end{equation}
moreover, equality holds if and only if $\omega = \mathrm{D}L(v)$ --- or equivalently, if and only if $v= \mathrm{D}H(\omega)$. Thus the $v$-derivative $\mathrm{D}L$ of $L$ and its inverse map $\mathrm{D}H$ set up a nonlinear correspondence between (future-directed timelike) tangent and cotangent vectors at $x$. 
Define
\begin{equation}\label{dual norms}
\|v\|:= \lim_{\QQ\uparrow 1} L(v), \quad 
\|\omega\|_* := \lim_{\PP \uparrow 1} H(\omega).
\end{equation}
For $v$ and $\omega$ both future-directed, this allows us to re-express \eqref{new Fenchel-Young} in a form
 \begin{equation}\label{poor Fenchel-Young}
 \frac1\QQ\,\|v\|^\QQ + \frac1\PP\, \|\omega\|^\PP_* \le \omega(v)
 \end{equation}
better suited to the nonsmooth setting:  it requires only the Lorentzian magnitudes of $v$ and $\omega$ as well as their bilinear pairing $\omega(v)$ to be known, but not the directions of $v$ and $\omega$. Since equality holds if and only if $\omega=\mathrm{D}L(v)$,  \eqref{poor Fenchel-Young}
 tests whether or not $v$ and $\omega$ are identified under the nonlinear duality $\mathrm{D}L$.
In the nonsmooth setting of \S\ref{ch:horizontal and vertical derivatives}, 
our Sobolev calculus will use an integrated version of \eqref{poor Fenchel-Young} to establish the analogous nonlinear identification of 
tangent and cotangent field pairs. 
The limits defined by \eqref{dual norms} give examples of objects we refer to as {\em hyperbolic norms}; their properties are further explored in
Appendix \ref{ss:hyperbolic norms} and \cite{Gigli24+}.

Before returning to the nonsmooth setting,  let us also point
out that for causal functions $f \in C^\infty(\M)$ it is the convex functional
\begin{equation}\label{energy}
\mathscr{E}(f) := -\int_M H(\rmd f) \d \vg
\end{equation}
which plays the role of the $\PP$-Dirichlet energy.  Convexity of $\mathscr{E}$ helps to determine when the (vertical) directional derivative
\begin{align}
\label{variation of energy}
\begin{split}
\lim_{t \to 0} \frac{\mathscr{E}(f+t \phi)-\mathscr{E}(f)}{t} 
&= - \int_M \rmd\phi\big(\mathrm{D}H(\rmd f)\big) \d\vg  =: \int_M\,\phi\, \frac{\delta \mathscr{E}}{\delta f}  \d\vg 
\end{split}
\end{align}
exists. It makes the variational derivative 
\[
\frac{\delta \mathscr{E}}{\delta f} = \div_g\big(\big\Vert \rmd f\big\Vert_*^{\PP-2}\, \nabla f\big) =: -\square_\PP f
\]
a non-uniformly elliptic operator that we call the {\em $\PP$-d'Alembertian.} (This is analogous to how uniform convexity of the usual $2$-Dirichlet energy makes the Riemannian Laplacian elliptic.) Notice however a conceptual issue:  
for the two-sided derivative \eqref{variation of energy} to exist,  it is necessary that $f+t\phi$ be causal (i.e. have $\d f + t \d \phi$ future-directed) for sufficiently small $t$ {\em of both signs}.
Since the set of causal functions is a only a convex cone, outside of the smooth setting we will frequently need to be satisfied with a one-sided derivative and an inequality in place of \eqref{variation of energy}.
Nevertheless, this construction allows us to prove a weak nonsmooth $\PP$-d'Alembert comparison theorem in \S\ref{Section: Effects of curvature assumption}. In particular, the latter establishes
the $2$-d'Alembert comparison inequality
\begin{equation}\label{2-d'Alembert comparison}
\square_2 g^o \le \frac{\dim M-1}{|g^o|}
\end{equation}
for the {\em Lorentz distance function} $g^o := -\ell(\cdot,o)$ on a spacetime with non-negative timelike Ricci curvature; notice that $\square_\PP f$ is independent of $\PP$ when $\|\rmd f\|_*=1$. Although 
\eqref{2-d'Alembert comparison} was proved in the region 
where $g^o$ is smooth by \cite{Eschenburg:1988},
our extension across the cut locus seems new;
it can be viewed as an analog of Calabi's result for the Riemannian Laplacian \cite{Calabi:1958}, even though both our formulations --- in the weak as opposed to viscosity sense --- and 
our methods of proof are utterly different.
For a more direct proof of this extension inspired by the present work yet using more classical techniques, see \cite{QuintetEllipticsplitting}. We return now to the nonsmooth setting with one caveat: in the nonsmooth setting,
the notions of causal and timelike will tacitly be taken to mean future-directed unless otherwise mentioned.

\subsection[The Sobolev cone of cotangent fields; infinitesimal Minkowskianity]{The Sobolev cone of cotangent fields; infinitesimal\\Minkowskianity}
\label{ss:differential}

Having established nonsmooth tangent notions  in \S\ref{Ch:Curves}, 
we turn to the problem of developing nonsmooth cotangent notions in \S\ref{Ch:differential}. 
Although the consideration of linear functions on Minkowski space makes it 
tempting to define the slope at $y\in \M$ of a causal function $f\colon \M\to \bar{\R}$ using (backward and forward) 
limits like 


 \begin{equation}\label{naive slope}
 |\partial^{-} f|(y) \doteq 
 \liminf_{x \uparrow y} \frac{f(y)-f(x)}{\ell(x,y)} 
\quad \mbox{or} \quad
|\partial^{+} f|(y)| \doteq
\liminf_{z \downarrow y} \frac{f(z) - f(y)}{\ell(y,z)} 
 \end{equation}
 restricted to $x \ll y \ll z$,  in nonsmooth spaces it is not obvious when these limits should agree.
 It is even less obvious whether such a pointwise definition will cooperate with the kind of integration by parts formulas needed to 
give meaning in a weak sense to the nonsmooth analog of the $\PP$-d'Alembertian $\square_\PP f$.  We focus instead to a different strategy inspired by
developments in positive signature \cite{KMcM:98,Cheeger99,Shanmugalingam00,AGS:14a,Gigli:2015}, as well as by
the above-mentioned Lagrangian-Hamiltonian duality.

We require our forward metric spacetime $(\M,\ell)$ to be equipped with a Radon reference measure $\meas$ to express integration by parts.  
The triple $(\M,\ell,\meas)$ is called a metric measure spacetime.    
Relative to $\meas$,
the plans $\bdpi$ on causal curves which appear in integration by parts formulas
should not concentrate too much mass at any particular point in time and spacetime; 
quantifying a version of this condition leads to the notion of a {\em test plan}, namely a plan $\bdpi$ admitting $C>0$ such that for every $t\in[0,1]$, the push-forward $(\eval_t)_\push\bdpi$ under the evaluation map $\eval_t(\gamma) :=\gamma_t$ is $\meas$-absolutely continuous with Radon--Nikodým density bounded from above by $C$. \red{Remarkably, this Lorentzian notion is even simpler than its Riemannian antecedent \cite{AGS:14a}, which in turn was inspired by the regular Lagrangian flows of Ambrosio--DiPerna--Lions \cite{DiPernaLions89,Ambrosio04} (that, in relation with lower Ricci curvature bounds, turned out to be connected to Cheeger--Colding’s segment inequality \cite[Theorem 2.11]{CheegerColding96}).}

Instead of \eqref{naive slope}, the slope $\vert\rmd f\vert\colon \M\to [0,+\infty]$ that we seek to define 
should heuristically be characterized as the largest function $G$  
satisfying
\begin{align}\label{curvewise}
f(\gamma_1) - f(\gamma_0) \geq \int_0^1 G(\gamma_t)\,\vert\dot\gamma_t\vert\d t
\end{align}
for all causal curves $\gamma$.  However, it is not clear that such a function exists;  moreover,  for rough spaces and discontinuous functions,  
the values of $f$ may be ambiguous at individual points or even along individual curves.  Thus we hope instead that \eqref{curvewise} holds along most curves.
Lacking a preferred reference measure to define a notion of `almost every' curve \cite{Shanmugalingam00,AGS:14a},  
we instead require \eqref{curvewise} to hold when integrated against every test plan, 
and deal with equivalence classes of functions $f$ that agree $\meas$-a.e. For this integration to make sense,
we restrict our attention to $\meas$-measurable causal functions; {if a causal function is not known to be measurable we may therefore call it {\em rough causal}.}
In \S\ref{ss:causal functions} we show that every rough causal function is continuous outside a countable union of achronal sets,  
hence rough causality of $f$ already implies $\meas$-measurability in many spaces of interest, {such as those having synthetic lower Ricci curvature bounds as in \Cref{cor:fmeastmcp}.}

 There are two equivalent ways  of defining a ``candidate'' for $\vert \rmd f\vert$ rigorously. 
 Either pathwise (as in \Cref{Le:Distr der}) or more robustly, in integrated form:
\begin{definition}[Weak subslope] A Borel  function $G\colon \M \to [0,+\infty]$ is a \emph{weak subslope} of $f$ if for all test plans $\bdpi$,
\begin{align}\label{weak subslope}
\int \big[f(\gamma_1) - f(\gamma_0)\big]\d\bdpi(\gamma)\geq \iint_0^1 G(\gamma_t)\,\vert\dot\gamma_t\vert \d t\d\bdpi(\gamma).
\end{align}
\end{definition}
\noindent
Under mild assumptions on $f$ and $\ell$,  both limits \eqref{naive slope} become weak subslopes (\Cref{prop:subslopes2}). Weak subslopes are stable under limits (\Cref{Pr:closed}). Moreover, the set of weak subslopes for $f$ is stable under pointwise maxima (see \eqref{eq:maxsubslopes}). 
Combining these two properties allows us to deduce the existence and uniqueness of a maximal $G$ satisfying \eqref{weak subslope} 
which we call a {\em maximal weak subslope} 
and denote by $|\rmd f|$. This represents a Lorentzian analog of the notions of minimal weak upper gradient introduced and studied in positive signature by 
Koskela--MacManus, Cheeger, Shanmugalingam, and Ambrosio--Gigli--Savar\'e \cite{KMcM:98,Cheeger99,Shanmugalingam00,AGS:14a}.
In the smooth setting we show this notion coincides with $\| \rmd f\|_*$ pointwise {a.e.} in  \Cref{thm:diffcaus}.

Inspired by the notion of infinitesimal Hilbertianity introduced by Gigli \cite{Gigli:2015}, 
we now formulate our definition of infinitesimal Minkowskianity.  It is based on requiring the usual parallelogram law to hold for $f+g$ and $f$.

\begin{definition}[Infinitesimal Minkowskianity]\label{Def:Inf Minkow} We call $(\M,\ell,\meas)$ \emph{infinitesimally Minkowskian} if for every two causal functions $f,g\colon \M\to\bar{\R}$, we have
\begin{equation}
\label{eq:definfmink}
2\, \vert\rmd(f+g)\vert^2 +2\,\vert\rmd f\vert^2  
= \vert \rmd g\vert^2 + \vert\rmd(2f + g)\vert^2\qquad\meas\textnormal{-a.e.}
\end{equation}
\end{definition}

Coupled with the a.e.~concavity  we shall eventually prove for $|\rmd f|$, Theorem \ref{T:VD=HD} shows this parallelogram law is akin to asserting that $|\rmd f|(x)$ can be polarized to produce an indefinite inner product with Lorentzian signature for 
$\meas$-a.e.~$x\in \M$ (see also Lemma~\ref{L:parallelogram law and polarization}). 
\red{Simple variants of the results established in \Cref{ss:compatibility} --- whose proofs we omit for simplicity --- show that a smooth Finsler spacetime is infinitesimally Minkowskian if and only if it is a Lorentzian spacetime.} A related preceding result by Braun--Ohta~\cite{BO:23} shows Finsler spacetimes with timelike sectional curvature bounds are already Lorentzian manifolds, hence infinitesimally Minkowskian. 
Another open problem already expressed in  \cite{CavallettiMondino22}:
for metric measure spaces, infinitesimal Hilbertianity
permits curvature-dimension bounds to be reformulated
equivalently in terms of a Bochner-type inequality \cite{AGS:14b,ErbarKuwadaSturm:2015};
it would be interesting to know whether infinitesimal
Minkowskianity permits an analogous reformulation of timelike curvature-dimension conditions for metric measure spacetimes. \red{In the smooth setting,  one direction of this equivalence can be found in \cite{Ohta:2014,QuintetEllipticsplitting}. In the nonsmooth case, irrespective of infinitesimal Minkowskianity, the timelike curvature-dimension condition is connected to a special Bochner-type inequality for Lorentz distance functions (whose simplifying advantage is that their ``Hessian'' vanishes identically in the tangential direction) by Braun \cite{braun+} based on the localization formalism of  Cavalletti and Mondino \cite{CavallettiMondino20}.}

The remainder of \S\ref{Ch:differential} is devoted to establishing calculus rules in the necessary generality for our applications, including concavity a.e.\ of the maximal weak subslope (\Cref{Pr:Linear comb}), locality properties  in \Cref{Th:Locality}, chain rule  (\Cref{Th:Chain rule II}), and Leibniz rule (\Cref{Cor:Leibniz}). 
As a consequence, although the maximal weak subslope defines only the magnitude $\vert \rmd f\vert$ of the slope and not its direction, we expect our approach paves the way to defining a genuine differential $\rmd f$ for a causal function following the strategies from positive signature by Sauvageot \cite{Sau89, Sau90}, Weaver \cite{Wea00}, 
and Gigli \cite{Gig:18}.

\subsection{Nonsmooth Lagrangian--Hamiltonian duality}

In \S\ref{ch:horizontal and vertical derivatives} we set up the negative-homogeneity correspondence between tangent and cotangent objects induced by
a pair of exponents $\PP^{-1}+\QQ^{-1}=1$ with $0\ne \PP<1$.  Inspired by the smooth Fenchel--Young inequality \eqref{poor Fenchel-Young},  one expects causal curves $\gamma$ and causal functions $f$ to 
satisfy 
\begin{equation}\label{curve Fenchel-Young}
\rmd f(\dot\gamma_0) \geq \frac{1}{\QQ}\,\vert\dot\gamma_0\vert^\QQ + \frac{1}{\PP}\,\vert\rmd f\vert^\PP,
\end{equation}
provided we can assign an appropriate meaning to the positively bilinear pairing on the left which represents the horizontal (inner) derivative of $f$ along the curve $\gamma$. 
In the nonsmooth setting,  we establish this inequality not for individual curves,  but only after integrating against an arbitrary test plan $\bdpi$ over curves 
(\Cref{Pr:Lower bound}). Much as in positive signature \cite{AGS:14b,Gigli:2015}, the positively bilinear pairing which appears on the left of \eqref{curve Fenchel-Young} will be replaced by the so-called ``horizontal derivative'' to obtain 
\begin{align*}
&
\liminf_{t\downarrow 0} \int \frac{f(\gamma_t) - f(\gamma_0)}{t}\d\bdpi(\gamma)
\geq \frac{1}{\PP}\int \vert\rmd f\vert^\PP(\gamma_0)\d\bdpi(\gamma) + \liminf_{t\downarrow 0} \frac{1}{t\QQ}\iint_0^t \vert\dot\gamma_r\vert^\QQ\d r\d\bdpi(\gamma).
\end{align*}
When the opposite inequality also holds, with $\limsup$ replacing $\liminf$ on the left,  we say $\bdpi$ {\em represents the \fff $\PP$-gradient of $f$}
(just as the Hamiltonian gradient $\smash{\dot \gamma_0 = \mathrm{D}H(\rmd f) = \|\rmd f\|_*^{\PP-2}\, \nabla f}$ from the smooth case attains equality in  \eqref{curve Fenchel-Young}). 
Heuristically, we may think of $\bdpi$ as a tangent field also denoted by $|\rmd f|^{\PP-2}\, \nabla f\, \mu_0$ waiting to act linearly on the causal 
proxy $\phi$ for the cotangent field $\rmd \phi$, where $\mu_0  = (\eval_0)_* \bdpi$. 
This follows the strategy used by  
Gigli \cite{Gigli:2015} in positive signature metric measure geometry, which in turn was inspired by De Giorgi's theory of metric gradient flows \cite{DeGiorgi92}.
 Although the plan $\bdpi$ representing the \fff $\PP$-gradient of $f$
is never unique, for a given $f$ the set of such $\bdpi$ is  convex.  Continuing the analogy with $\dot \gamma_0=\mathrm{D}H(\rmd f)$ that has been rigorously developed in
positive signature \cite{AGS:14a},   in the very special case $(\eval_0)_\push \bdpi =\meas$ we may view $\bdpi$ as inducing  a subgradient of the nonsmoothly adapted convex $\PP$-energy 
\begin{equation}\label{nonsmooth energy}
\mathscr{E}(f) := -\frac1\PP \int_M |\rmd f|^\PP \d\meas
\end{equation}
(via horizontal differentiation, which also leads to a weak definition for the divergence).
To pursue this analogy further is beyond the scope of this paper.

This sets up the desired negative-homogeneity correspondence between the tangent objects $\bdpi$ and cotangent objects $\rmd f$. 

The remainder of \S\ref{ch:horizontal and vertical derivatives} is devoted to developing and using more calculus rules to
nonsmoothly compare the two sides of \eqref{variation of energy}, with \eqref{nonsmooth energy} replacing \eqref{energy} in the ``vertical'' derivative on the left,  while the right side is replaced by horizontal differentiation of $\phi$ along $\bdpi$. \Cref{thm:horver} then shows that an analog of identity \eqref{variation of energy} continues to hold as an inequality for single-sided 
derivatives, viz.
\begin{align}\label{VD le HD}
\begin{split}
\lim_{\varepsilon\downarrow 0}  
 \int\frac{\vert\rmd (f+\varepsilon\phi)\vert^\PP(\gamma_0)
 -\vert\rmd f\vert^\PP(\gamma_0)
 }{\PP\varepsilon} \d\bdpi(\gamma)  \le \limi_{t\downarrow0}\int\frac{\phi(\gamma_t)-\phi(\gamma_0)}{t}\d\bdpi(\gamma).
\end{split}
\end{align}
By \emph{linearity} of the right-hand side, when $f+  \varepsilon\phi$ is causal for small $\varepsilon$ of {\em both} signs the complementary inequality
\begin{align*}
\lim_{\varepsilon\uparrow 0}  
 \int\frac{\vert\rmd (f+\varepsilon\phi)\vert^\PP(\gamma_0)
 -\vert\rmd f\vert^\PP(\gamma_0)
 }{\PP\varepsilon}  \d\bdpi(\gamma) \ge \lims_{t\downarrow0}\int\frac{\phi(\gamma_t)-\phi(\gamma_0)}{t}\d\bdpi(\gamma)
\end{align*}
follows. In this case, if one can show the two-sided vertical derivatives on the respective left-hand sides exist and are equal 
--- as we establish when two-sided perturbations are causal and the spacetime is infinitesimally Minkowskian, cf.~\Cref{T:VD=HD} --- 
then \eqref{VD le HD} improves to an equality, meaning the horizontal ($t \downarrow 0$) derivative also exists and coincides with the vertical ($\epsilon\to 0$) derivative. 
Some settings in which this improvement becomes possible are identified in Appendix~\ref{ss:null distance}. Here, one of our contributions is to identify (for a given function $f$) a specific and more tangible class of test functions $\varphi$ satisfying the previous causality obstructions on $f+\varepsilon\varphi$; cf.~\Cref{Pr:1.lip}.  

Note the asymmetric roles played by $f$ and $\phi$ in the estimates above: 
whereas causality implies $f$ is analogous to a non-decreasing function on the real line, being a difference of causal functions,  $\varepsilon\phi = (f+\varepsilon\phi) - f$ is analogous to a function of locally bounded variation on the real line; this analogy proves quite fruitful.
When both $f$ \red{and $\phi$ are  causal, if}  $|\rmd f|$ and $|\rmd \phi|$ have strictly positive real values
$\meas$-a.e., equation \ref{eq:dgnfscal} shows
$\rmd^+f(\nabla \phi)=\rmd^+\phi(\nabla f)$ $\meas$-a.e.~for infinitesimally Minkowskian forward spacetimes. This expression becomes a positively bilinear indefinite quadratic form which can be extended bilinearly to the vector space generated by causal functions. \red{Effectively, in some sense we recover a bilinear metric tensor $\meas$-a.e.~from $\ell$ in this case.}

\subsection{Timelike Ricci curvature bounds and applications}

 The last part of our paper is devoted to applications of our first-order Sobolev calculus (which require additional hypotheses related to continuity of $\ell_+$ and of timelike geodesics).
 More significantly, they require a proxy for timelike lower Ricci curvature bound $K$ and upper dimension bound $N$.  The synthetic version  of this which we adopt is a hybrid combination $\smash{\TMCP^h_+(K,N)}$ of the timelike measure contraction property $\TMCP_+(K,N)$ of Braun~\cite{Braun:2023Renyi} and the entropic variant 
 $\smash{\TMCP^e_+(K,N)}$ of Cavalletti--Mondino \cite{CM:20} upon which it is based.  Although $\TMCP_+(K,N)$ is a priori \emph{stronger} than $\TMCP^e_+(K,N)$, they are conjecturally equivalent under appropriate non-branching hypotheses (as already known for their  positive signature counterparts \cite{cavalletti-sturm,cavalletti-milman2021}), while the reduced variant $\smash{\TMCP_+^*(K,N)}$ of $\TMCP_+(K,N)$ is equivalent to $\TMCP^e(K,N)$ under such circumstances  \cite{Braun:2023Renyi}.  The two conditions $\smash{\TMCP_+(K,N)}$ and $\smash{\TMCP_+^*(K,N)}$ are defined by requiring 
 different $(K,N)$-dependent
 convexity estimates for the $N$-Renyi entropy functional
 \begin{equation}\label{N-Renyi}
    \scrS_N(\mu) := -\int_M \rho^{(N-1)/N}\d\meas,\qquad\text{where}\quad\mu=\rho\mm+\mu^\perp\quad\text{with}\quad \mu^\perp\perp\mm,
\end{equation}
 along suitable $\ell_\QQ$-geodesics $(\mu_t)$ joining an arbitrary probability measure $\mu_0$ to a Dirac mass in its timelike future; when $K=0$, both {follow from} ordinary convexity of the assignment $t  \mapsto \scrS_N(\mu_t)$.   The Cavalletti--Mondino condition instead requires a growth bound for the 
Boltzmann-Shannon entropy
\begin{align}\label{BS}
    \scrS_\infty(\mu) :=
    \begin{cases}
    \int_M \rho \log \rho \d\meas &\text{ if }\quad \mu=\rho\,\meas \quad \text{and} \quad \mm[\spt\mu]<\infty,
    \\ +\infty & \text{otherwise} 
    \end{cases}
\end{align}
along the same curves 
which {follows from} ordinary convexity in the case $N=0$.
Our hybrid condition $\TMCP^\h_+(K,N)$ requires suitable $\ell_q$-geodesics to satisfy both estimates.  We need it only to gain compactness;  whenever the compactness is provided by globally hyperbolicity, $\TMCP_+(K,N)$ would suffice.

 We note that qualitatively, all our results in \Cref{Section: Effects of curvature assumption} will hold irrespective of whether we assume $\smash{\TMCP^\h_+(K,N)}$ or $\smash{\TMCP^{\h,*}_+(K,N)}$ (or 
 $\smash{\TMCP_+(K,N)}$ or $\smash{\TMCP^{*}_+(K,N)}$ when globally hyperbolicity holds). The benefit of the unstarred conditions is that they provide quantitative consequences in \emph{sharp} form even in the absence of any non-branching hypothesis. Notably, an interesting setting where this comparison theory thus applies are vector spaces with hyperbolic norms, cf.~\Cref{ss:hyperbolic norms}. As our \Cref{Pr:TCD COND} below establishes, Euclidean space endowed with \emph{any} appropriate hyperbolic norm has non-negative synthetic timelike Ricci curvature, yet these structures may admit many branching geodesics in general.

 For compatibility with the test plans of our Sobolev calculus, the entropic displacement convexity which defines $\TMCP^\h_+(K,N)$ must be satisfied by geodesics of measures having density bounds.  The construction of these ``good'' geodesics 
 in \Cref{Le:Existence test plans} and \Cref{Le:Existence test plans2}
 involves new subtleties beyond those from the globally hyperbolic case \cite{Braun:2023Good}.  To enforce initial and final conditions (and one-sided continuity) the concept of non-branching at 0 or 1 we already alluded to comes into play; it relaxes concepts from
 \cite{KS:18, McC:23}.  To handle Dirac endpoints with forward-completeness instead of compactness, we exploit the narrow coercivity of 
 $\scrS_\infty$ on finite volume subsets of $\M$ and the growth bounds from $\TMCP^e_+(K,N)$.

 The first result that we show is a nonsmooth converse to the Hawking--King--McCarthy--Malament theorem \cite{HKMcC:76,Malament77}.  
 Working in the smooth context,  with the time separation function
 built out of the metric tensor, they showed that a bijection which preserves the class of continuous timelike curves is in fact a smooth conformal isometry. If it also preserves the time separation then it preserves the metric tensor.  In the nonsmooth context,  we have instead used the time separation to define a maximal weak subslope $|\rmd f|\colon \M \to [0,+\infty]$ for causal functions $\smash{f\colon \M\to \bar{\R}}$.  We begin by showing that $(\M,\ell,\meas)$ satisfying $\TMCP^\h_+(K,N)$ implies the \emph{Sobolev-to-steepness} property (\Cref{Th:SobtoSteep}), inspired by Gigli's Sobolev-to-Lipschitz property in positive signature~\cite{gigli2013}. This means the condition $|\rmd f| \ge 1$ $\meas$-a.e.~implies a ``reverse Lipschitz'' condition we call {\em $1$-steepness}, viz.
 \begin{equation*}
 f(y) - f(x) \ge \ell(x,y)
 \end{equation*}
 for all $x,y \in \supp\meas$.
 This implies $\ell$ can be recovered on the support of $\meas$  from the assignment $f \mapsto |\rmd f|$.  It also implies that if $T\colon\supp \meas_1 \to \supp\meas_2$ is a surjection between two such spaces that preserves measure, causality, and maximal weak subslopes, then it also preserves $\ell$, i.e.\ is a global isometry (\Cref{Th:HawkingKingMcCarthy}). Much like in positive signature, this result makes use of Sobolev notions viable for deriving precise metric conclusions.

Our second application is set in the same class of spaces.  We  connect our first-order Sobolev calculus with the optimal transport of appropriate Borel probability measures 
$\mu_0,\mu_1 \in \Prob(\M)$ by showing a metric \cite{AGS:14a,Gigli:2015} Brenier--McCann \cite{Brenier:1991,McCann95,McCann:01,McCann:2020} theorem.  Under appropriate hypotheses~\cite{CM:20}, Cavalletti--Mondino have shown a coupling of $\mu_0$ and $\mu_1$ is $\ell_\QQ$-optimal if and only if it is supported in the $\ell^\QQ/\QQ$-superdifferential $\partial_{\ell^\QQ/\QQ}f$ of an $\ell^\QQ/\QQ$-concave function $f$.
In  \Cref{Th:Brenier}, we show that if a test plan has marginals $(\eval_0,\eval_1)_* \bdpi$ supported in $\partial_{\ell^\QQ/\QQ}f$, then $\bdpi$ represents the \fff $\PP$-gradient of $f$. Moreover, $\bdpi$-a.e.\ $\gamma$ satisfies the \emph{equality} 
\begin{equation}\label{metric Brenier}
\ell(\gamma_0,\gamma_1) = \vert \rmd f \vert^{\PP-1}(\gamma_0).
\end{equation}
Thus $\mu_0$-a.e., our maximal weak subslope captures the distance transported into the future. This also  opens the door for a natural notion of exponentiation on nonsmooth spacetimes with synthetic curvature bounds akin to \cite{giglirajalasturm}.

This result plays a role in our  
third application,  which is a \emph{sharp} $\PP$-d'Alembert comparison theorem for $\TMCP^\h_+(K,N)$ spaces (Theorems \ref{Cor:Time I} and \ref{Th:Dalem comp}).  The theorem has the cleanest form in the case $K=0$. 
For $\ell^\QQ/\QQ$-concave functions $f$ with $\partial_{\ell^\QQ/\QQ} f \subset \{\ell>0\}$,  it states 
the following inequality in a weak form:
\begin{equation}\label{q-Dalem comp}
\square_\PP f \le N.
\end{equation}
Applying \eqref{q-Dalem comp} to the function $f^o = -\ell(\cdot, o)^\QQ/\QQ$, the chain rule yields a weak expression of the bound
\begin{equation*}
\square_\PP (-\ell(\cdot,o)) \le \frac{N-1}{\ell(\cdot,o)}
\end{equation*}
on $I^-(o)$ for the $\PP$-d'Alembertian of the Lorentz distance to $o \in \supp\meas$. On $I^+(o)$, a similar logic yields a weak \red{reformulation} of
$$
\square_\PP \ell(o,\cdot) \ge -\frac{N-1}{\ell(o,\cdot)}
$$
\red{under the modified hypothesis $\TMCP^\h_-(0,N)$, in which transport involves a Dirac endpoint in the past rather than the future.} 
Along the way, some key inequalities have to be ``reversed''; this is discussed separately in \Cref{Sub:Modiback}.

Based on the analogous result in positive signature \cite{Gigli:2015}, 
the initial idea of our proof is to consider a mass distribution $\mu_0$ given in terms of an appropriate non-negative test function $\varphi$ supported in $I^-(o)$, and to transport $\mu_0$ to $\mu_1 := \delta_o$. The corresponding $\ell^q/q$-concave function is $f=f^o$ on $I^-(o)$. (To simplify the presentation in this introduction, 
we neglect technicalities pertaining to causality, integrability, etc., of appropriate linear combinations of $\varphi$ and $f$ for now.) More precisely, since the entropy functional \eqref{N-Renyi} is given by the convex integrand $s_N(r) = -r^{(N-1)/N}$, 
we will choose $\mu_0$ to be $\meas$-absolutely continuous with density $\rho_0 = (c_0\,\varphi)^{N/(N-1)}$, where $c_0$ is a normalization constant. \red{Theorems}~\ref{Le:Existence test plans} and  \ref{Th:Brenier} plus \Cref{Le:Existence test plans2} establish the existence of a plan $\bdpi$ representing the \fff $\PP$-gradient of $\red f$ such that the assignment $t\mapsto \scrS_N(\mu_t)$ {satisfies an appropriate} $\ell_\QQ$-geodesic convexity estimate, 
where $\mu_t := (\eval_t)_\push\bdpi$. This convexity entails
\begin{align}\label{above-tangent}
\limsup_{t\downarrow 0} \frac{\scrS_N(\mu_t)-\scrS_N(\mu_0)}{t}
& \le \scrS_N(\mu_1)-\scrS_N(\mu_0)
 = c_0 \int_M \varphi \d\meas.
\end{align}
On the other hand, ordinary convexity of $s_N$ gives the complementary bound 
relating \eqref{above-tangent} to a horizontal derivative, which in turn can be bounded by a vertical derivative as in \eqref{VD le HD}:
\begin{align*}
\liminf_{t\downarrow 0}\frac{\scrS_N(\mu_t) -  \scrS_N(\mu_0)}{t} &\ge \liminf_{t\downarrow 0}\int \frac{s_N'\circ \rho_0(\gamma_t) - s_N'\circ\rho_0(\gamma_0)}{t}\d\bdpi(\gamma) \\&\geq \liminf_{\varepsilon \downarrow 0}
\int_M\frac{\vert\rmd (\red f+\varepsilon\, s_N' \circ \rho_0)\vert^\PP
 -\vert\rmd \red f\vert^\PP
 }{\PP\varepsilon} \d \mu_0
\\&= \int_M \rmd^+(s_N'\circ \rho_0)(\nabla \red f)\,\vert\rmd \red f\vert^{\PP-2}\d\mu_0
\\&= \frac{c_0}N \int_M \rmd^+\varphi(\nabla f)\,\vert\rmd 
f\vert^{\PP-2}\d\meas,
\end{align*}
where the last identity follows from a suitable chain rule and the specific choice of $\rho_0$. Canceling $c_0$ gives the desired weak formulation of  \eqref{q-Dalem comp}, viz.
$$
\int_M \rmd^+\varphi(\nabla f)\,\vert\rmd f\vert^{\PP-2}\d\meas \le N\int_M \varphi \d\meas.
$$
For non-zero $K$,  the initial ``chord above tangent'' bound \eqref{above-tangent} is complicated by the presence of additional  distortion coefficients
depending on $\ell(\gamma_0,\gamma_1)$, which are related to  $|\rmd f|^{\PP-1}(\gamma_0)$ by our metric 
 Brenier--McCann theorem \eqref{metric Brenier}.

 In  \Cref{S:defining square}  we discuss how the $\PP$-d'Alembert comparison estimate can be coupled with infinitesimal Minkowskianity (or related notions) to develop a notion of `distributional' $p$-d'Alembertian and start studying its properties. A basic one is that the Riesz--Daniell theorems allow us to represent such an operator as a Radon measure, suitably defined.

 In the final  \Cref{Sub:Modiback} we comment on how our results are modified if we work on $\TMCP^\h_-(K,N)$ spaces rather than on $\TMCP^\h_+(K,N)$ ones (while keeping forward completeness).

\subsection[Two-sided perturbations, hyperbolic norms,  compatibility with smooth notions]{Two-sided perturbations, hyperbolic norms, and compatibility with\\ smooth notions}

We include three appendices. The first establishes compatibility between many of the nonsmooth notions introduced here and their  analogs in smooth spacetimes.  The second  uses Kunzinger--Steinbauer's nonsmooth adaptation~\cite{KS:2022} 
of the Sormani--Vega null distance \cite{SV:16} to give conditions under which a causal function $f$ will admit two-sided linear perturbations, both of which remain causal.  The third introduces hyperbolic norms as a generalization of Lorentz--Finsler norms, and
explores weak lower Ricci bounds in this setting as well as consequences of the parallelogram law.

\color{black}

\section{Curves in the nonsmooth Lorentzian setting}\label{Ch:Curves}


\subsection{Infinity conventions} \label{Sub:InfConv} 

Throughout this paper, we employ the conventions 
\begin{equation}
\label{eq:conventions1}
\begin{split}
+\infty - (+\infty)& := +\infty,\\
-\infty - (-\infty) &:= +\infty,\\
\pm \infty + z &:= \pm \infty \quad \text{for every } z \in [-\infty,+\infty],\\
0 \cdot (\pm \infty)& := 0. 
\end{split}
\end{equation}
The latter is standard in measure theory. The first three will enter into play, for instance, when presenting the reverse triangle inequality for the time separation function \eqref{eq: reverse triangle inequality} or when discussing calculus with causal functions as in \eqref{eq:defwd}, as we do not want to exclude from our analysis cases in which $\ell$ or $f$ take the values $\pm\infty$. We accept that our conventions imply the non-commutativity $-\infty+(+\infty)=-\infty<+\infty=+\infty+(-\infty)$. In connection with the reverse triangle inequality, it is worth to point out that for  $a,b,c\in\bar\R$ the above yields
\begin{equation}
\label{eq:cambiolato}
\min\{a+b,b+a\}\leq c\qquad\Leftrightarrow\qquad a\leq c-b\quad\text{and}\quad b\leq c-a,
\end{equation}
as can be easily checked.


For $0 \ne p\leq 1$, we also define the concave non-decreasing function $u_p:\bar\R\to\bar\R$ by setting $u_p(z):=\tfrac{z^p}p$ for $z>0$ and then extending this by monotonicity and concavity, i.e.\ we put
\begin{equation}
\label{eq:up}
\begin{array}{rll}
u_p(z)&=-\infty\qquad &\text{for $z\in[-\infty,0)$ and any $p\leq 1$},\\
u_p(0)&=-\infty\quad\text{(resp.\ 0)}\qquad &\text{for  $p< 0$ (resp.\ $p\in(0,1]$)},\\
u_p(+\infty)&=\phantom{-}0\quad\text{(resp.\ $+\infty$)}\qquad &\text{for  $p< 0$ (resp.\ $p\in{(}0,1]$)}.
\end{array}
\end{equation}

\subsection[Notions of spacetimes]{Notions of spacetimes}
We shall work in a setting which generalizes Lorentz(--Finsler) spacetimes and their basic causality theory  to possibly nonsmooth spaces. Such an abstract approach was pioneered by Kunzinger--Sämann through their Lorentzian pre-length spaces \cite{KS:18} (inspired by early contributions of Busemann \cite{Busemann:1967} and Kronheimer--Penrose \cite{KronheimerPenrose:1967}). An alternative approach where many topological properties are implied by basic axioms on the time separation function was proposed by Minguzzi--Suhr \cite{MinguzziSuhr:2022}. Sequella of  \cite{CM:20,McC:23,BraunMcCann:2023,BykovMinguzziSuhr:2024+} extend \cite{KS:18,MinguzziSuhr:2022}. (We point out \cite{McC:23,BraunMcCann:2023} use the term ``metric spacetime'' for related yet slightly different structures than we do.) An approach more tied to a distance in the classical metric sense (the so-called null distance) is given in Sormani--Vega \cite{SV:16}. 

Let $\M$  be a set and  $\ell\colon \M\times \M\rightarrow \{-\infty\}\cup [0,+\infty]$.  We  say that  $\ell$ is  a \emph{time separation} if
\begin{subequations}
\begin{align}
\label{eq:ellxx}
&\ell(x,x)=0,\\
 \label{eq: reverse triangle inequality}
&\min\{\ell(x,y)+\ell(y,z), \ell(y,z)+\ell(x,y)\} \leq \ell(x,z).
  \end{align}
\end{subequations}
hold for any  $x,y,z\in \M$. We define the transitive relations \emph{chronology} $\ll$ and \emph{causality} $\leq$ on $\M$ as:
\[
\ll\, := \ell^{-1}((0,+\infty])=\{\ell> 0\},\qquad\qquad \leq\, := \ell^{-1}([0,+\infty]) = \{\ell\ge 0\}.
\]
If $\ell$ is a time separation, so is its {\em time-reversal}   $\ell^*(y,x) := \ell(x,y)$.

We shall denote by $\ell_+:= \max\{\ell,0\} = \ell \vee 0$ the positive part of $\ell$. For $x\in \M$, the sets
 \begin{alignat*}{3}
  I^+(x)&:= \{y\in \M: x\ll y\},& \qquad\qquad  && I^-(x) &:=\{y\in \M: y \ll x\},\\
  J^+(x)&:= \{y\in \M: x\leq y\},\, &&& J^-(x) &:=\{y\in \M: y\leq x\},
 \end{alignat*}
are the \emph{chronological future and past} of $x$ and the \emph{causal future and past} of $x$, respectively. 
Analogously, for $X,Y\subset \M$ we define 
\[
J^\pm(X):=\bigcup_{x\in X} J^\pm(x),\qquad\qquad J(X,Y):=J^+(X)\cap J^-(Y).
\]
If $X$ and $Y$ are both compact we refer to sets of the form $J(X,Y)$ as
\emph{causal emeralds}, or simply emeralds. If $X = \{x\}$ and $Y = \{y\}$, we will refer to these sets as \emph{diamonds} instead of emeralds and write $J(x,y)$ for short. We shall also deal with \emph{chronological diamonds}, i.e.\ sets $I(x,y)$ of the form $I^+(x)\cap I^-(y)$.
A set $X\subset\M$ is called \emph{causally convex}  if and only if it contains the causal diamond $J(x,y)\subset X$ for all $x,y\in X$.

\begin{definition}[Notions of spacetimes] 
\label{D:spacetimes}
We shall say that:
\begin{itemize}
\item[i)]  $({\rm M},\ell)$ is a metric spacetime if $\ell$ is a time separation on  $\M$  so that $\leq$ is a partial order, i.e.\ $x\leq y$ and $y\leq x$ implies $x=y$.
\item[ii)] $(\M,\uptau,\ell)$  is a Polish metric spacetime if $\uptau$ is a Polish topology (i.e.\ induced by a complete and separable distance) on the metric spacetime $(\M,\ell)$ so that $\{\ell\geq 0\}$ is closed in $\M^2$.
\item[iii)] $(\M,\uptau,\ell)$  is a forward metric  spacetime if it is a Polish metric spacetime and every sequence $(x_n)$ such that $x_n\leq x_{n+1}\leq \bar x$ for every $n\in\N$ and some $\bar x\in \M$ admits a $\uptau$-limit.

Symmetrically, if we ask for existence of limits of decreasing sequences bounded from below we get the notion of backward metric spacetime.
\end{itemize}
Finally, a Polish spacetime is called locally causally convex if every point has a neighbourhood basis made of causally convex sets.
\end{definition}
Forward/backward spacetimes  give notions of completeness  which obviate our need for the local compactness guaranteed by global hyperbolicity.
A new compactness result in forward metric spacetimes, \Cref{prop-cad-comp} below, gives a limit curve theorem for causal curves.

\begin{remark}[The choice of the topology] 
\label{R:push-up2}
We are not going to require any further relation between $\uptau$ and $\ell$ on top of those already asked by the concepts of Polish and forward metric spacetimes. We opt for this choice both because for the moment we do not need anything more and because we are looking for notions that get inherited by the space of Borel probability measures equipped with the time separation $\ell_q$ coming from optimal transport: see \Cref{eq:probmellq} for details.

In practical applications however, and especially in relation to convergence of spacetimes, it is surely natural to look for a topology that is induced in some way by $\ell$. Also, our strongest version of the limit curve theorem \Cref{prop-cad-comp} holds for locally {causally} convex topologies and often topologies somehow induced by $\ell$ have this property. Examples of $\ell$-induced topologies are:
\begin{itemize}
\item[-] The Alexandrov topology, i.e.\ the one generated by chronological diamonds,
\item[-] The   chronological topology \cite{KS:18}, i.e.\ the one generated by sets of the form $I^-(x)$ and $I^+(x)$ as $x$ varies in $\M$,
\item[-] The smallest topologies making $\ell$ upper semicontinuous and its positive part $\ell_+$ continuous.
\end{itemize}
These are ordered from the coarsest to the finest and are all locally causally convex.
 \hfill$\blacksquare$
\end{remark}

\begin{remark}[Push-up property] 
\label{R:push-up1}
 In any metric spacetime $(\M,\ell)$, the \emph{push-up principle} holds, i.e.~if $x\ll y$ and $y\leq z$ (or if $x\leq y$ and $y\ll z$) then $x\ll z$. This follows directly  from the reverse triangle inequality. 
   \hfill$\blacksquare$
\end{remark}

We collect a couple of basic results about the relation of $\ell$ and the topology.
\begin{lemma}[Upper semicontinuous time separations attain maxima on emeralds]
\label{L:usc ell admits emerald bounds}
Let $(\M,\uptau,\ell)$ be a Polish metric  spacetime.  Then emeralds are closed. If moreover $\ell_+$ is real valued and  upper semicontinuous, then it   is bounded above on each emerald.
\end{lemma}

\begin{proof}
Let $E \subset \M$ be an emerald, say $E = J(C_0,C_1)$ for $C_0,C_1\subset M$ compact, and $(x_n)\subset E$ be a sequence converging to a limit $x$. Then there are $(y_n)\subset C_0$ and $(z_n)\subset C_1$ so that $y_n\leq x_n\leq z_n$ for every $n\in\N$. By compactness, up to passing to a non-relabelled subsequence we can assume that $y_n\to y\in C_0$ and  $z_n\to z\in C_1$. As $\{\ell\geq 0\}$ is closed, we conclude that $y\leq x\leq z$, so that $x\in E$ as well.

For the second claim, notice that  by upper semicontinuity $\ell$ attains its maximum value $L<\infty$ on $C_0 \times C_1$.
We claim $\ell \le L$ on $E^2$.  Each $x \in E$ lies in $J(x_0,x_1)$
for some $x_i \in C_i$.  Similarly $y \in E$ lies in $J(y_0,y_1)$
for some $y_i \in C_i$.  If $\ell(x,y) > -\infty$ then
$x_0 \le x \le y \le y_1$ and the reverse triangle inequality yields
$\ell(x,y) \le \ell(x_0,y_1) \le L$ as desired.
\end{proof}
Somewhat conversely, the following provides a condition ensuring  lower semicontinuity of $\smash{\ell_+(\cdot,o)}$ for a point $o\in \M$. The assumption  is tailored to the setting we will work with in \Cref{Section: Effects of curvature assumption}:  in  typical  applications, $z$ comes from the $t$-evaluation $\gamma_t$ of a rough $\ell$-geodesic $\gamma$ from $x$ to $o$ (see \Cref{Def:l-geodesics}). 
\begin{lemma}[Rough timelike geodesy yields lower semicontinuous time to a point]
\label{Le:LSC lplus} Let $(\M,\uptau,\ell)$ be a Polish metric  spacetime with $\uptau$ containing the chronological topology. Assume that a point $o\in \M$  has the following property: for every $x\ll o$ and every $t\in(0,1)$ there exists $z \in \M$ such that 
\begin{align*}
    \ell(x,z) = t\,\ell(x,o),\qquad\qquad\ell(z,o) = (1-t)\,\ell(x,o).
\end{align*}
Then the assignment $z\mapsto \ell_+(z,o)$ is lower semicontinuous.
\end{lemma}
\begin{proof} It suffices to prove that for every $c>0$ the superlevel set
$\{\ell(\cdot,o)>c\}$ is open. For every $x\in \{\ell(\cdot,o) > c\}$, by our assumption there exists $z\in \M$ such that $\ell(x,z) > 0$ and
$\ell(z,o)>c$. In particular, $x$ belongs to the chronological past of
$z$, which is open by definition. We thus proved   that $\{\ell(\cdot,o) > c\}$ is
the union of chronological pasts of all of its elements, hence it is open.
\end{proof}

Further interesting regularity properties can be proved under the (reasonable) assumption that in our spacetime any point can be approximated by points in its chronological past, i.e.\ 
\begin{equation}
\label{eq:apprpast}
x\in \overline{I^-(x)}\qquad\forall x\in \M.
\end{equation}
This holds e.g.~in metric spacetimes where every point is in the interior of a continuous timelike curve, such as in globally hyperbolic regular Lorentzian length spaces.

The notion of forward-completeness which defines forward spacetimes (Definition
\ref{D:spacetimes}) does not quite coincide with the usual notion of order completeness,  which would require existence of a supremum rather than a limit.  These two notions are linked by the following two results.

\begin{lemma}[Chronology determines causality]\label{le: chrondetcaus} Let $(\M,\uptau,\ell)$ be a Polish metric spacetime for which \eqref{eq:apprpast} holds.  Then the inclusion $I^-(x) \subset I^-(y)$ 
for some $x,y\in \M$ implies $x\leq y$.
\end{lemma}
\begin{proof}
This can partially be compared with one direction of \cite[Prop.~2.26]{APS20}.  Taking closures in the given inclusion and recalling that $\leq$ is closed we get $\overline{I^-(x)}\subset\overline{I^-(y)}\subset J^-(y)$. Then the assumption \eqref{eq:apprpast} yields $x\in J^-(y)$, which is the claim.
\end{proof}

The following proposition gives sufficient conditions for the 
limit of any $\le$-non-decreasing sequence to coincide with its $\le$-supremum.

\begin{proposition}[Limit from the past vs.\ order supremum]
Let $(\M,\uptau,\ell)$ be a Polish metric  spacetime  for which \eqref{eq:apprpast} holds and so that $\uptau$ contains the chronological topology. 
Let the limit $x = \lim_{n \to \infty} x_n$ of some non-decreasing sequence $x_n \le x_{n+1}$ in $M$ converge. Then $x_n \leq x$ for all $n \in \N$ and moreover: if $x_n \leq y \in M$ for all $n \in \N$ then $x \leq y$. 
\end{proposition}

\begin{proof} Letting $m\to\infty$ in  $x_n \leq x_m$ (which is valid for $n\leq m$) and using that $\leq$ is closed we deduce that   $x$ is an $\leq$-upper bound for the sequence. 
Let $y$ be another $\leq$-upper bound. Since $x_n \to x$, any neighbourhood $U$ of $x$ contains all $x_n$ for large enough $n$. 
In particular, this holds for $U=I^+(z)$ for some $z\ll x$ provided by \eqref{eq:apprpast}; $U$ is open by our assumption on $\uptau$. 
For such $U$ we must have $z \ll x_n$ for sufficiently large $n$ and since $y\geq x_n$ we must also have $z \ll y$ by the push-up property (Remark \ref{R:push-up1}). 
We thus proved that if $z\ll x$ then $z\ll y$, so that \Cref{le: chrondetcaus} implies $x\leq y$, as desired.
\end{proof}

We  comment on how all this relates to previously investigated settings.

\begin{remark}[Metric spacetimes versus other abstract generalizations of spacetimes]
\label{prop: metric spacetimes are LpLS} Our setup of Polish metric spacetimes is compatible with the approaches of Kunzinger--Sämann's Lorentzian pre-length spaces \cite{KS:18} and Minguzzi--Suhr's bounded Lorentzian metric spaces \cite{MinguzziSuhr:2022} whenever the respective standing hypotheses intersect. (We refer to Beran--Harvey--Napper--Rott \cite[§4.1]{BHNR23} for the relation between \cite{KS:18} and \cite{MinguzziSuhr:2022}.) Indeed, let $(\M,\uptau,\ell)$ be a Polish metric spacetime with $\uptau$ being  the chronological topology. Let $\sfd$ be a fixed  metric inducing $\uptau$. If $\ell_+$ is lower semicontinuous, then $(\M,\sfd,\ll,\leq,\ell_+)$ is a Lorentzian pre-length space. On the other hand, a Lorentzian pre-length space $(\M,\sfd,\ll,\leq,\tau)$ where the reference metric $\sfd$  is Polish and induces the chronological topology and such that the relation $\leq$ is closed becomes a Polish metric spacetime by defining $\ell$ as $-\infty$ outside $\leq$ and $\tau$ otherwise. 

Lastly, assume the pair $(\M,\tau)$ is a bounded Lorentzian metric space after Minguzzi--Suhr \cite{MinguzziSuhr:2022}. Let $\uptau$ denote the reference topology, which is Polish and in fact uniquely certifies the definition of $(\M,\tau)$ being a bounded Lorentzian metric space  \cite[Cor.~1.7, Thm.~1.10]{MinguzziSuhr:2022}. Define $\ell$ as in the previous paragraph, where $\leq$ is replaced by the (closed) extended causal relation from \cite[§5.1]{MinguzziSuhr:2022}. These observations make $(\M,\ell)$ a Polish metric spacetime.
\hfill$\blacksquare$
\end{remark}

\begin{definition}[Global hyperbolicity] A Polish metric spacetime is called \emph{globally hyperbolic} if   all of its  emeralds are compact. 
\end{definition}

Together with our ad hoc assumption of closedness of $\{l\geq 0\}$, this corresponds to the notion of global hyperbolicity for general topological ordered spaces suggested by Minguzzi \cite{minguzzitopapp}.

\begin{remark}[About global hyperbolicity] 
\label{R:global hyperbolicity} In a  globally hyperbolic metric spacetime, any bounded monotone sequence has a limit, as it lives in a (compact) emerald and the closedness of $\leq$ implies any two limits $x,x'\in \M$ satisfy $x\leq x'$ and $x'\leq x$, hence $x=x'$ as $\leq$ is a partial order. In particular, such a spacetime is both forward and backward. As we do not assume a strong relation between a metric (inducing the Polish order topology) and the causal structure (i.e.~$\sfd$-compatibility  \cite[Def.\ 3.13]{KS:18}), a globally hyperbolic metric spacetime is not a globally hyperbolic Lorentzian pre-length space in the sense of \cite[Def.\ 2.35]{KS:18}.  On the other hand if we assume $\sfd$-compatibility, causal path-connectedness (i.e.~every pair of causally related points can be joined by a continuous causal curve) and lower semicontinuity of $\ell_+$, then by \cite[Thm.\ 3.7]{Min:23} the quintuple $(\M,\sfd,\ll,\leq,\ell_+)$ is a globally hyperbolic Lorentzian pre-length space.\hfill$\blacksquare$
\end{remark}

\begin{remark}[Forward metric spacetimes on the causal ladder]
\label{R:causal ladder} A spacetime is called \emph{causally simple} if the causal relation $\le$ is both a partial ordering and closed; this definition  makes sense in our setting as well.  
By Remark \ref{R:global hyperbolicity},
any globally hyperbolic metric spacetime   is forward, while by definition  any Polish metric  spacetime is causally simple.

In our setting, these inclusions are strict. On the one hand, Anti-de Sitter space is causally simple but not forward-complete. Alternatively, consider any causal diamond $J(x,y)$ in the Minkowski plane but remove spacelike infinity $\partial I^+(x)\cap\partial I^-(y)$.  

Also, take the diamond $J((-1,0),(1,0))$ in the Minkowski plane and add the  horizontal line $\R\hat{e}_2$ to the space and extend the causal relation by $(-1,0) \leq (0,z) \leq (1,0)$ for all $z \in \R$. This space is causally simple and forward, but not globally hyperbolic (and not a manifold).\hfill$\blacksquare$
\end{remark}

\subsection{The spacetime of Borel probability measures}\label{eq:probmellq}

Let $(\M,\uptau,\ell)$ be  a Polish metric spacetime. We shall denote by $\Prob(\M)$ the collection of Borel probability measures on it and by $\pem\subset\Prob(\M)$ the subset of those with support in some emerald. Here we are interested in the study of the  \emph{Lorentz--Wasserstein time separation}  $\ell_\QQ$ (originating in Eckstein--Miller \cite{EM:17} in the range $\QQ\in (0,1]$) on $\Prob(\M)$. We start with the definition:

\begin{definition}[Lorentz--Wasserstein distance]\label{Def:Lorwass}  Let $(\M,\uptau,\ell)$ be a Polish metric spacetime in which $\ell$ is Borel. For $\mu,\nu\in\Prob(\M)$ let $\Pi_{\leq}(\mu,\nu)\subset\Prob(\M\times \M)$ be the set of \emph{causal} couplings, i.e.\ of measures $\pi \ge 0$ concentrated on $\{\ell\geq 0\}$ having marginals $\mu = (\Pr_1)_*\pi$ and $\nu=(\Pr_2)_*\pi$;
here $\Pr_1(x,y):=x$ and $\Pr_2(x,y):=y$ on $\M^2$.

For $0\neq \QQ \leq 1$, and $u_\QQ$ from \eqref{eq:up}, the \emph{$\QQ$-Lorentz--Wasserstein distance}  $\ell_\QQ(\mu,\nu)$ of $\mu,\nu\in\Prob(\M)$ is defined as $-\infty$ if $\Pi_{\leq}(\mu,\nu)$ is empty (consistently with the usual convention $\sup\emptyset := -\infty$) and otherwise as that element in $[0,+\infty]$ such that
\begin{equation}
\label{eq:defellq1}
     u_q\big( \ell_\QQ(\mu,\nu)\big) := \sup_{\pi\in \Pi_{\leq}(\mu,\nu)} \!\int_{\M\times \M}u_q\circ \ell \,\d \pi.
\end{equation}
\end{definition}
In other words, if no infinities appear, $\ell_\QQ(\mu,\nu)$ can be defined via
\begin{equation}
\label{eq:defellq2}
\ell_\QQ(\mu,\nu):=\sup_{\pi\in \Pi_{\leq}(\mu,\nu)} \Big(\int_{\M\times \M} \ell^\QQ\,\d \pi\Big)^{\frac1q}.
\end{equation}
A coupling $\pi$ of $\mu$ and $\nu$ with $\mu\preceq\nu$ is called \emph{$\ell_\QQ$-optimal} if it attains the  supremum in \eqref{eq:defellq1} (or, equivalently, in \eqref{eq:defellq2}). 

We shall prove in a moment that $\ell_q$ is a time separation on $\Prob(\M)$. For the moment, we notice that the causal relation $\preceq$ it induces is independent of $q$  and can be defined as: $\mu\preceq\nu$ if and only if $\Pi_{\leq}(\mu,\nu)$  is not empty. A standard gluing argument shows that $\preceq$ is transitive, while it is less obvious that it is antisymmetric (but see \Cref{prop:ellqst} below). For the interplay of $\preceq$ and the narrow topology we refer to \cite[Thm.~8]{EM:17} and  \cite[Thm.~B.5]{BraunMcCann:2023}.  We caution the reader that unlike the causal relation described above,  the timelike relation $\prec$ induced by $\ell_q$ can depend on $q$ when we allow $q<0$.

To study the properties of $\ell_q$ it is convenient to keep in mind some basic facts about $L^q$ `norms' for $0\neq q<1$, defined for $[0,+\infty]$-valued measurable functions as
\[
\|f\|_q:=\Big(\int f^{q}\,\d\mm\Big)^{\frac1q}\in[0,+\infty].
\]
The following is simple and well known, see also \cite{Gigli24+} for more about duality properties of $L^p$ and $L^q$ spaces for $p,q<1$:

\begin{lemma}[Basic inequalities in $L^q$, $0\neq q<1$]\label{le:lplq}
Let $(\X,\mm)$ be a $\sigma$-finite measure space, $0\neq p<1$ and $0 \neq q < 1$ satisfy $\tfrac1p+\tfrac1q=1$, and $f:\X\to[0,+\infty]$ be measurable. Then
\begin{equation}
\label{eq:duallplq}
\|f\|_q=\inf\Big\{\int fg\,\d\mm\ :\ \text{ $g:\X\to[0,+\infty]$ is measurable with $\|g\|_p=1$}\Big\}
\end{equation}
and
 \begin{equation}
\label{eq:revlq}
\begin{split}
\|f_1+f_2\|_q&\geq \|f_1\|_q+\|f_2\|_q.
\end{split}
\end{equation}
Moreover, if equality holds in \eqref{eq:revlq} and $\|f_1+f_2\|_q\in(0,+\infty)$,   then for some $\alpha\geq 0$ we have $f_1=\alpha f_2$.
\end{lemma}
\begin{proof} We observe that $u_p$ is the Legendre-Fenchel transform of the concave function $u_q$, i.e.:
\[
\begin{split}
u_q(z)+u_p(w)&\leq zw\qquad \forall z,w\geq 0,\\
u_q(z)+u_p(w)&= zw\qquad\Leftrightarrow\qquad z,w>0,\ u_q'(z)=w,\ u_p'(w)=z.
\end{split}
\]
The proof of this is trivial.  We now prove $\leq$  in  \eqref{eq:duallplq}. This is obvious   if  $\|f\|_q=0$. Assume thus ---  for the moment --- that $\|f\|_q\in(0,+\infty)$ and notice that the above yields
\begin{equation}
\label{eq:perduallp}
\tfrac1{\|f\|_q}\int  f  g\,\d\mm\geq \int u_q(\tfrac {f}{\|f\|_q})+u_p( g)\,\d\mm= \int u_q(\tfrac{ f}{\|f\|_q})\,\d\mm+\int u_p( g)\,\d\mm=\tfrac1q+\tfrac1p=1,
\end{equation}
and thus  the desired $\leq$. For the general case we find a non-decreasing sequence $(f_n)$ increasing to $f$  so that the $\|f_n\|_q$'s are finite and converging to $+\infty$ (here the assumption of $\sigma$-finiteness matters): letting $n\uparrow\infty$ in $\|f_n\|_q\leq \int f_ng\,\d\mm\leq \int fg\,\d\mm$, by the arbitrariness of $g$ we conclude that $\leq $ holds. We pass to $\geq$ and observe that it is obvious if $\|f\|_q=+\infty$. Then assume  --- up to normalization --- that $\|f\|_q=1$, notice that in this case $g:=f^{\frac qp}$ satisfies $\|g\|_p=1$ and that $\int fg\,\d\mm=\int f^{q}\,\d\mm=1$ yielding the claim. Finally, if $\|f\|_q=0$ find a decreasing sequence $(f_n)$ converging to $f$ so that the $\|f_n\|_q$'s are positive and converging to $0$ (again, here the assumption of $\sigma$-finiteness helps). By the above discussion, for each $n$ there is $g_n$ with $\|g_n\|_p=1$ so that $\|f_n\|_q=\int f_ng_n\,\d\mm$, thus in particular $f_ng_n\in L^1(\mm)$ and letting $m\to\infty$ in the bound $\int f_ng_n\,\d\mm\geq\int f_mg_n\,\d\mm$ valid for any $m\geq n$,   by dominated convergence we conclude that $\|f_n\|_q\geq \int fg_n\,\d\mm\geq \inf_g\int fg\,\d\mm$. Letting $n\uparrow\infty$ we conclude.

The second part of the statement is now an easy consequence of the first. Notice indeed that 
\begin{equation}
\label{eq:perrevtr}
\|f_1+f_2\|_q=\inf_g\int (f_1+f_2)g\,\d\mm\geq \inf_{g_1}\int f_1g_1\,\d\mm+\inf_{g_2}\int f_2g_2\,\d\mm=\|f_1\|_q+\|f_2\|_q,
\end{equation}
where $g,g_1,g_2$ are as in \eqref{eq:duallplq}. For the equality case we observe that if $\|f_1+f_2\|_q\in(0,+\infty)$ we proved above that there is $g$ as in  \eqref{eq:duallplq} so that $\|f_1+f_2\|_q=\int (f_1+f_2)g\,\d\mm$ and \eqref{eq:perduallp}  shows that the only such $g$ is given by  $g=(\frac{f_1+f_2}{\|f_1+f_2\|_q})^{\frac qp}$. If also $\|f_1\|_q,\|f_2\|_q\in(0,+\infty)$, then the same argument and the equality in \eqref{eq:duallplq} force $g=(\frac{f_1}{\|f_1\|_q})^{\frac qp}=(\frac{f_2}{\|f_2\|_q})^{\frac qp}$ and the conclusion easily follows. If instead one of $f_1,f_2$, say $f_1$, has $q$-norm equal to 0, then equality in \eqref{eq:revlq} implies that $\int f_2^q\,\d\mm=\int(f_1+f_2)^q\,\d\mm$ and since both sides are real (from $\|f_1+f_2\|_q\in(0,+\infty)$) the two integrands are in $L^1(\mm)$ and since $f_2\leq f_1+f_2$ holds $\mm$-a.e., the equality forces $f_1=0$ $\mm$-a.e., that is the conclusion. Finally, the case $\|f_1\|_q=\|f_2\|_q=0$ cannot occur in conjunction with the equality in \eqref{eq:revlq} and $\|f_1+f_2\|_q\in(0,+\infty)$.
\end{proof}

\begin{remark}
For $q\in(0,1)$ we might have $\|f_1\|_q=\|f_2\|_q=\|f_1+f_2\|_q=+\infty$ in \eqref{eq:revlq} without any rigidity. Similarly, if $q<0$ we might have  $\|f_1\|_q=\|f_2\|_q=\|f_1+f_2\|_q=0$ without  rigidity.\hfill$\blacksquare$
\end{remark}

In order to prove antisymmetry of $\preceq$, the following lemma will be useful:

\begin{lemma}[Recurrence of self-couplings]
\label{le:perunic}
Let $\X$ be a Polish space and $\pi\in\Prob(\X^2)$ have equal first and second marginals.

Then for every distance $\sfd$ inducing the Polish topology on $\X$ and $r>0$ we have: for $\pi$-a.e.\ $(x,y)$ there is a sequence $(x_n)\subset\X$ with $x_0=x$, $x_1=y$, $(x_n,x_{n+1})\in\supp\pi$ for every $n\in\N$ and
\begin{equation}
\label{eq:dpoinc}
\liminf_{n\to\infty}\sfd(x_0,x_n)\leq r.
\end{equation}
\end{lemma}
\begin{proof} The proof is based on a variant `for couplings' of the classical Poincar\'e recurrence theorem (for which we will follow the presentation from Tao \cite{TaoPoincare}). 
Let $\mu:=({\rm Pr}_1)_*\pi=({\rm Pr}_2)_*\pi$ and let $(\pi_x)$ be the disintegration of $\pi$ w.r.t.\ the projection ${\rm Pr}_1$ on the first marginal. Define $\ppi\in \Prob(\X^\N)$ as
\[
\d\ppi(x_0,x_1,\ldots)=\cdots\d\pi_{x_2}(x_3)\d\pi_{x_1}(x_2)\d\pi_{x_0}(x_1)\d\mu(x_0).
\]
In other words, $\ppi$ is the measure given  by Kolmogorov's extension theorem associated to the measures $\pi^N\in\Prob(\X^N)$, where $\d\pi^N(x_0,\ldots,x_{N-1})=\d\pi_{x_{N-2}}(x_{N-1})\cdots\d\pi_{x_0}(x_1)\d\mu(x_0).$ Notice that the construction grants 
\begin{equation}
\label{eq:ppin}
({\rm Pr}_n,{\rm Pr}_{n+1})_*\ppi=\pi\qquad\forall n\in\N
\end{equation}
and thus
\begin{equation}
\label{eq:ppiconc}
\text{ $\ppi$ is concentrated on sequences $(x_n)$ such that $(x_n,x_{n+1})\in\supp\pi$ for every $n\in\N$.}
\end{equation}
 We claim that for every $E\subset\X^2$ Borel we have 
\begin{equation}
\label{eq:recurrconcl}
\limsup_{n\to\infty}\ppi\Big(\Big\{(x_i)\in\X^\N\ :\ (x_0,x_1),(x_n,x_{n+1})\in E\Big\}\Big)\geq \pi(E)^2.
\end{equation}
Indeed, from \eqref{eq:ppin}  it follows that $\int\sum_{i< n} 1_{({\rm Pr}_i,{\rm Pr}_{i+1})^{-1}(E)}\,\d\ppi=n\pi(E)$ and the Cauchy-Schwarz inequality yields
\[
n^2\pi(E)^2\leq\int\Big(\sum_{i< n}1_{({\rm Pr}_i,{\rm Pr}_{i+1})^{-1}(E)}\Big)^2\,\d\ppi= \sum_{i,j< n}\ppi\Big(\underbrace{({\rm Pr}_i,{\rm Pr}_{i+1})^{-1}(E)\cap({\rm Pr}_j,{\rm Pr}_{j+1})^{-1}(E)}_{=({\rm Pr}_i,{\rm Pr}_{i+1},{\rm Pr}_j,{\rm Pr}_{j+1})^{-1}(E\times E)}\Big).
\]
From the definition it also follows that $({\rm Pr}_i,{\rm Pr}_{i+1},{\rm Pr}_j,{\rm Pr}_{j+1})_*\ppi= ({\rm Pr}_0,{\rm Pr}_1,{\rm Pr}_{j-i},{\rm Pr}_{j-i+1})_*\ppi$ whenever $i<j$, thus calling $L$ the $\limsup$ in \eqref{eq:recurrconcl}, it is not hard to see that
\[
\limsup_{n\to\infty}\tfrac1{n^2} \sum_{i,j< n}\ppi\Big(({\rm Pr}_i,{\rm Pr}_{i+1})^{-1}(E)\cap({\rm Pr}_j,{\rm Pr}_{j+1})^{-1}(E)\Big)\leq L,
\]
and the  claim \eqref{eq:recurrconcl} follows. We now claim that for every $E\subset \X^2$ Borel it holds that
\begin{equation}
\label{eq:altroclaim}
\text{for $\pi$-a.e.\ $(x,y)\in E$ we have }\ppi_{(x,y)}\big(\{(x_i)\ :\ (x_n,x_{n+1})\in E\text{ for infinite many $n$}\}\big)>0,
\end{equation}
where $(\ppi_{(x,y)})$ is the disintegration of $\ppi$ w.r.t.\ the projection on the first two coordinates. Indeed, if this fails for some $E$, for some Borel set $F\subset E$ with $\pi(F)>0$ we have
\[
0=\ppi_{(x,y)}\big(\{ (x_n,x_{n+1})\in E\text{ for infinite $n$'s}\}\big)\geq \ppi_{\{(x,y)}\big( (x_n,x_{n+1})\in F\text{ for infinite many $n$}\}\big)
\]
and therefore  $\ppi(\{(x_i) : (x_n,x_{n+1})\in F\text{ for infinite $n$'s}\})=0$, which contradicts \eqref{eq:recurrconcl}.

Hence \eqref{eq:altroclaim} holds and applying it  with $E:=B\times \X$, where $B\subset\X$ is a ball of diameter $<r$ and recalling \eqref{eq:ppiconc} and that $\ppi_{(x,y)}$ is --- for $\pi$-a.e.\ $(x,y)$ --- concentrated on sequences with $x_0=x$ and $x_1=y$, we see that for $\pi$-a.e.\ $(x,y)$ with $x\in B$ there is $(x_n)$ with $(x_n,x_{n+1})\in\supp\pi$ for every $n\in\N$ and $x_n\in B$ for infinite $n$'s. In particular $\liminf_n\sfd(x,x_n)\leq r$ and since $\X$ can be covered by a countable collection of such balls $B$, the conclusion follows.
\end{proof}
Notice that if the coupling $\pi$ in the statement is induced by  a map, then the claim \eqref{eq:recurrconcl} reduces to the classical recurrence theorem by picking $E:=F\times \X$ for arbitrary $F\subset\X$ Borel. Also, the measure $\ppi$ built in the statement reflects the idea that the coupling $\pi$ can be seen as the map  sending $x$ to the fiber $\pi_x$ of the disintegration of $\pi$ w.r.t.\ the projection onto the first coordinate, so that `to iterate $\pi$ means to iterate this assignment', in a sense.

\medskip

With this said, we have the   following general statement; 
c.f.~\cite{McC:23} \cite[Theorem B.2]{BraunMcCann:2023}.

\begin{proposition}[Probability measures inherit Polish structure from metric spacetime]
\label{prop:ellqst}
Let $(\M,\uptau,\ell)$ be a Polish metric spacetime so that $\ell$ is Borel and $0\neq q\leq 1$.  

Then $(\Prob(\M),\ell_q)$ equipped with the narrow topology is a Polish metric spacetime as well.
\end{proposition}

\begin{proof}
The reverse triangle inequality is known and based on a standard gluing argument together with the reverse triangle inequality \eqref{eq:revlq} in $L^q$ for $0\neq q<1$: see e.g.\  \cite[Thm.~13]{EM:17} or \cite[Prop.~2.5]{CM:20} and notice that  the  proofs  do not depend on global hyperbolicity and that, even though  are stated in the range $\QQ\in (0,1]$, remain unchanged for negative $\QQ$ (see also \Cref{le:lplq} above). We now claim that
\begin{equation}
\label{eq:diag}
\text{ for any $\mu\in\Prob(\M)$ the plan }\quad ({\rm id},{\rm id})_*\mu\quad\text{is the only element of }\quad \Pi_{\leq}(\mu,\mu).
\end{equation}
The fact that  $({\rm id},{\rm id})_*\mu\in  \Pi_{\leq}(\mu,\mu)$ is obvious. To show that it is the only element   we argue by contradiction, so we assume that $\pi\in \Pi_{\leq}(\mu,\mu)$ is not concentrated on the diagonal. {Let $r>0$ and apply  \Cref{le:perunic}  to find  a sequence $(x_n)\subset\M$ with $x_0\neq x_1$, $(x_n,x_{n+1})\in\supp\pi\subset\{\ell\geq 0\}$ (and thus $x_n\leq x_{n+1}$) for every $n\in\N$ satisfying \eqref{eq:dpoinc}. In particular there is $x^r=x_{n(r)}$ for a suitable $n(r)$ with $x_0\leq x_1\leq x^r$ and $x^r\in B_r(x_0)$. Letting $r\downarrow0$ and using the closure of $\{\ell\geq 0\}$ we deduce that $x_0\leq x_1\leq x_0$ and} conclude that $x_0\leq x_1\leq x_0$, i.e.\ that $x_0=x_1$, contradicting the assumption.

Thus \eqref{eq:diag} holds, and therefore $\ell_q(\mu,\mu)=0$ for any $\mu\in\Prob(\M)$. To conclude that $(\Prob(\M),\ell_q)$ is a metric spacetime it remains to show that $\ell_q(\mu,\nu)=\ell_q(\nu,\mu)=0$ implies $\mu=\nu$. To see this, let $\pi_1\in\Pi_\leq(\mu,\nu)$, $\pi_2\in\Pi_\leq(\nu,\mu)$ and $\pi\in\Prob(\M^3)$ a gluing of these along the common $\nu$-marginal. Since $\pi_1,\pi_2$ are causal, the construction grants that $x\leq y\leq z$ for $\pi$-a.e.\ $(x,y,z)$ and since $({\rm Pr}_1,{\rm Pr}_3)_*\pi\in\Pi_\leq(\mu,\mu)$, we deduce from \eqref{eq:diag} that for $\pi$-a.e.\ $(x,y,z)$ we have $x=z$, and thus that $x=y=z$ for $\pi$-a.e.\ $(x,y,z)$. In particular  $\mu=\nu$, as desired.  

We come to the topological aspects. The fact that the narrow topology on $\Prob(\M)$ is Polish is well known (see e.g.\ \cite[Remark 7.1.7]{AGS:08}), so to conclude we need only to show that $\{(\mu,\nu):\mu\preceq\nu\}$ is closed. This, however, is obvious: if $(\mu_n),(\nu_n)$ narrowly converge to $\mu,\nu$ respectively and $\pi_n\in\Pi_\leq(\mu_n,\nu_n)$ for every $n$, then $(\pi_n)$ is tight and any narrow limit is, by the closure of $\{\ell\geq 0\}$, an element of $\Pi_\leq(\mu,\nu)$. 
 \end{proof}
 We now want to investigate how forward completeness of $\M$ is inherited by  $\Prob(\M)$. To this aim, the following simple lemma will be useful:
 \begin{lemma}[`Limit' as Borel map]\label{le:limborel} Let $\X$ be a Polish space, $\hat\X:=\Pi_{n\in\N}\X$ with the (Polish) product topology and consider the subset $\hat\X_{\sf L}\subset\hat\X$ of sequences admitting a limit in $\X$.

Then $\hat\X_{\sf L}$ is a Borel subset of $\hat\X$ and the map  sending a sequence in $\hat\X_{\sf L}$ to its limit in $\X$ is Borel.
\end{lemma}
\begin{proof} Let $(y_k)\subset\X$ be countable and dense and, for $r>0$, consider the set $\hat\X_{r,k}\subset\hat\X$ defined as
\[
\hat\X_{r,k}:=\bigcup_{n\in\N}\big\{(x_i)\in\hat\X\ :\ x_i\in B_r(y_k)\ \forall i\geq n\big\}=\bigcup_{n\in\N}\bigcap_{i\geq n}{\rm Pr}_i^{-1}\big(B_r(y_k)\big)
\]
that is clearly Borel. Then so is $\hat \X_r:=\cup_k\hat\X_{r,k}$ and so is the map ${\sf L}_r:\hat \X_r\to\X$ sending $(x_i)$ to $y_k$, where $k\in\N$ is the least integer such that $(x_i)\in \hat\X_{r,k}$. To conclude, notice that for any  $r_j\downarrow 0$ we have $\hat\X_{\sf L}=\cap_j\hat\X_{r_j}$ and that for  $(x_i)\in \hat\X_{\sf L}$ the limit in $j$ of ${\sf L}_{r_j}((x_i))$ exists and is the limit of $(x_i)$.
\end{proof}

We then have:

\begin{proposition}[Probability measures inherit forward structure from metric spacetime]\label{prop:pmforward} Let $(\M,\uptau,\ell)$ be a forward metric spacetime so that $\ell$ is Borel and $0\neq q<1$. 

Then $(\Prob(\M),\ell_q)$ equipped with the narrow topology is a forward metric spacetime as well.
\end{proposition}
\begin{proof} By \Cref{prop:ellqst} we need only to check forward completeness, thus let  $(\mu_n)\subset\Prob(\M)$ be such that $\mu_n\preceq\mu_{n+1}\preceq\bar \mu$ for every $n\in\N$ and some $\bar\mu\in\Prob(\M)$.

Let $\hat \M:=\Pi_{n\in\N\cup\{\infty\}}\M$ be equipped with the product (hence Polish) topology and $\mathcal K\subset\Prob(\hat \M)$ be the collection of measures $\alpha$
 such that $({\rm Pr}_n)_*\alpha=\mu_n$ for any $n\in\N$ and $({\rm Pr}_\infty)_*\alpha=\bar\mu$. We claim that $\mathcal K$ is narrowly compact and since it clearly is narrowly closed (by the continuity of the projections), to see this it suffices to prove that it is tight. Thus fix $\eps>0$ and find compact sets $K_n\subset \M$, $n\in\N\cup\{\infty\}$, such that $\mu_n(K_n^c)<\frac\eps{2^n}$ for every $n\in\N$ and $\bar\mu(K_\infty^c)<\eps$. Then the set $K:=\Pi_{n\in\N\cup\{\infty\}}K_n$ is compact and it is easy to see that $\alpha(K^c)<2\eps$ for every $\alpha\in \mathcal K$, thus proving the claim.

Now for every $n\in\N$ let $\pi_n\in\Pi_\leq (\mu_n,\mu_{n+1})$ and $\tilde\pi_n\in\Pi_\leq(\mu_n,\bar\mu)$ (the existence of such plans is ensured by our assumptions). Using a gluing argument and then Kolmogorov's theorem we can find, for every $n\in\N$, a measure $\alpha_n\in \mathcal K$ such that 
\begin{equation}
\label{eq:defalphan}
\begin{split}
({\rm Pr}_i,{\rm Pr}_{i+1})_*\alpha_n&=\pi_i\qquad\forall i\in\N,\\
({\rm Pr}_n,{\rm Pr}_{\infty})_*\alpha_n&=\tilde \pi_n.
\end{split}
\end{equation}
By the compactness of $ \mathcal K$, possibly passing to  a non-relabelled subsequence we can assume that $(\alpha_n)$ narrowly converges to a limit $\alpha\in\mathcal K$. We claim that $\supp(\alpha)$ is contained in the set of sequences ${\bf x}=(\ldots,x_i,\ldots,x_\infty)\in\hat \M$ such that $x_{1}\leq\ldots\leq x_i\leq\ldots\leq x_\infty$. To see this, let ${\bf x}\in\supp(\alpha)$ and  find ${\bf x}^n\in\supp(\alpha_n)$ converging to ${\bf x}$ in $\hat \M$ (i.e.\ $x^n_i\to x_i$ as $n\to\infty$ for every $i\in \N\cup\{\infty\}$). Then for every $i\in\N$ and every $n\geq i$, by the defining properties \eqref{eq:defalphan} we have $x^n_i\leq x^n_{i+1}$ and $x_i^n\leq x^n_\infty$. Passing to the limit in these (recall that $\{\ell\geq 0\}$ is closed), we get our claim.

It follows by forward completeness that for $\alpha$-a.e.\ ${\bf x}$ there exists the limit  ${\sf lim}({\bf x})$ of $i\mapsto x_i$ and since the partially defined limit map ${\sf lim}:\hat \M\to \M$ is Borel (by  \Cref{le:limborel} above), we can define $\mu_\infty:={\sf lim}_*\alpha$. We check that $\mu_\infty$ is the desired narrow limit of $(\mu_n)$: let $\varphi\in C_b(\M)$ and notice that
\[
\lim_{n\to\infty}\int\varphi\,\d\mu_n=\lim_{n\to\infty}\int \varphi(x_n)\,\d\alpha({\bf x})=\int \varphi({\sf lim}({\bf x}))\,\d\alpha({\bf x})=\int \varphi\,\d\mu_\infty,
\]
having used the dominated convergence theorem in the second identity.
\end{proof}

For later use we also record the following simple fact:
\begin{lemma}[Upper semicontinuity of $\ell_q$ and attainment on emeralds]\label{le:USCRE}
Let $(\M,\uptau,\ell)$ be a Polish metric spacetime such that $\ell$ is upper semicontinuous, does not take the value $+\infty$ and let $0\neq q<1$.

Then for every $\mu,\nu\in\pem$ with $\mu\preceq\nu$ an $\ell_q$-optimal coupling exists. Moreover, if   $E\subset \M$ is an emerald and $(\mu_n),(\nu_n)$ sequences of measures concentrated in $E$ narrowly converging to $\mu,\nu$ respectively, then $\mu,\nu$ are also concentrated on $E$ and
\[
\limsup_{n\to\infty}\ell_q(\mu_n,\nu_n)\leq\ell_q(\mu,\nu).
\]
\end{lemma}
\begin{proof}  We start from the second claim. The closedness of $E$ established in \Cref{L:usc ell admits emerald bounds}  ensures that $\mu,\nu$ have support in $E$ and that $\sup_n\ell_q(\mu_n,\nu_n)<+\infty$. Now let  $\pi_n\in \Pi_{\leq}(\mu_n,\nu_n)$ be  such that $\int u_q\circ \ell \,\d\pi_n\geq u_q(\ell_q(\mu_n,\nu_n))-\tfrac1n$ and notice that the tightness of $(\mu_n),(\nu_n)$ yields that of $(\pi_n)$, so that passing to a non-relabelled subsequence we can assume that the sequence narrowly converges to some $\pi\in\Pi_\leq(\mu,\nu)$. Since   $\ell$ is upper semicontinuous and bounded from above on $E^2$, the narrow convergence yields
\[
\limsup_{n\to\infty}\int u_q\circ \ell \,\d\pi_n\leq \int u_q\circ \ell \,\d\pi\leq u_q(\ell_q(\mu,\nu)),
\]
and the conclusion follows. For the first claim just apply the above with $\mu_n=\mu$ and $\nu_n=\nu$ for every $n\in\N$.
\end{proof}

\subsection{Causal speed}
\label{ss:causal speed}
Let us fix a metric spacetime $(\M,\ell)$. A map $\gamma:[0,1]\to \M$ is called \emph{causal path} if it is monotone, i.e.\  $t,s\in[0,1]$ with $t\leq s$ imply $\gamma_t\leq \gamma_s$. Note that the time-orientation is inherently given by the relation $\leq$, so that all causal paths are future directed. We do not assume causal paths are continuous (at this stage $\M$ does not even have a topology, but even when topology will be added, continuity will not be imposed).

The goal of this section is to discuss  in which sense a causal path  admits a \emph{causal speed}.  Heuristically, in analogy to the  speed of absolutely continuous curves in metric geometry \cite{Ambrosio90}, its causal speed $\vert \dot\gamma_t\vert$ at an  appropriate $t\in[0,1]$ should be given as $\lim_{h\downarrow 0} \ell(\gamma_t,\gamma_{t+h})/h$. Making this idea precise requires a generalization of the well-known fact that the distributional derivative of a monotone function $\smash{F:\R \to \bar{\R}}$ is a non-negative Radon measure on the interior of its interval of finiteness.  This is also related to the  classical result of Lebesgue asserting that monotone functions are differentiable $\Leb^1$-a.e.
    
\Cref{Pr:speed} and its \Cref{Cor:speed} establish similar differentiability results 
for time separation functions
$T\colon [0,1]\times [0,1] \to \{-\infty\} \cup [0,+\infty]$ near the diagonal of the unit square.  Choosing $T(s,t) :=\ell(\gamma_s,\gamma_t)$ allows us to associate the desired causal speed and --- for any suitable choice of Lagrangian --- a global action with each causal path $\gamma$ in $\M$.   Choosing $T(s,t):=F(t)-F(s)$   recovers the distributional and a.e.\ derivatives of a monotone function $F$ as above.

Let us start by recalling a few standard facts.

\begin{definition}[Vitali cover]
    A Vitali cover of a set $E\subset\R$ is a collection $\mathscr V$  of closed non-trivial intervals such that for every $x\in E$ and every $\eps>0$ there is $I\in \mathscr V$  containing $x$ \textnormal{(}so in particular the collection is a cover\textnormal{)} with $\Leb^1(I) \leq \eps$.
\end{definition}

For a proof of the following well known theorem we refer to  Bogachev \cite[Thm.\ 5.5.1]{Bog:07a}.

\begin{theorem}[Vitali's covering lemma]
    Let $E\subset\R$ be Borel with $\Leb^1(E)<\infty$, $\eps>0$, and $\mathscr V$ a Vitali cover of $E$. Then there  are finitely many intervals $I_n\in\mathscr V$ that are disjoint and so that $\smash{\Leb^1(E\setminus \bigcup_nI_n)<\eps}$.
\end{theorem}

The following simple lemma will be useful in defining the causal speed:
\begin{lemma}[Non-oscillation criterion]
\label{le:simple}
Let $f,\omega\colon (0,1)\to(0,+\infty)$ be  such that   $\lim_{h\downarrow0}\omega(h)=0$ and  so that for every $n\in\N$, every finite sequence $(\alpha_i) \subset [0,1]$ with $\sum \alpha_i=1$, and every $h\in (0,1)$ we have
\[
\sum^n_{i=0} \alpha_if(\alpha_ih)\leq f(h)+\omega(h).
\]
Then the limit $\lim_{h\downarrow 0}f(h)$ exists.
\end{lemma}

\begin{proof} Let $h,\alpha \in (0,1]$
and $n$ be the integer part of $1/\alpha$. Our assumption yields
\[
n\alpha f(\alpha h)\leq n\alpha f(\alpha h)+(1-n\alpha)f((1-n\alpha) h)\leq f(h)+\omega(h)
\]
for every $h \in (0,1)$. Keeping $h$ fixed and letting $\alpha \downarrow 0$ (noting that $n\alpha\to 1$ in this procedure) we obtain
\begin{align*}
\limsup_{h' \downarrow0} f(h')=\limsup_{\alpha \downarrow0} f(\alpha h)\leq f(h)+\omega(h).
\end{align*}
Letting $h\downarrow0$ we conclude  $
\limsup_{h' \downarrow 0}f(h')\leq \liminf_{h\downarrow 0}f(h)$,
which is the claim.
\end{proof}

In the sequel, consider the upper left halfsquare
\begin{align*}
H := \{(s,t)\in [0,1]^2 : s\leq t\}.
\end{align*}
\begin{theorem}[Differentiation of the reverse triangle inequality]
\label{Pr:speed} Let $T\colon H\rightarrow [0,+\infty)$ be a function satisfying
\begin{equation}
\label{eq:revtr}
T(r,s)\leq T(r,t) -T(s,t)
\end{equation} 
for all $0 \le r\leq s\leq t \le 1$. Then the following hold.
\begin{enumerate}[label=\textnormal{(\roman*)}]
\item There exists a unique maximal element $\mu$ among all Borel measures $\nu$ on $\R$ which satisfy
\begin{equation}
\label{eq:permaxmu}
\nu((a,b))\leq  {T}(a,b),\qquad\forall a,b\in\R,\ a<b,
\end{equation}
where here we are extending $T$ to the whole halfplane $\{s\leq t\}\subset \R^2$ by putting
\begin{equation}
\label{eq:extT}
T(s,t):=T\big(0\vee s\wedge 1,0\vee t\wedge 1\big ).
\end{equation}
(notice that this extension still satisfies \eqref{eq:revtr}).
\item The measure $\mu$ is  the weak limit as $h\downarrow 0$ of the measures $\mu_h$ given by
\begin{align*}
\rmd\mu_h(t) := \frac{{T}(t,t+h)}{h}\d\Leb^1(t).
\end{align*} 
Moreover, writing $\smash{\mu=\rho\,\Leb^1+\mu^\perp}$ for the Lebesgue decomposition of $\mu$ with respect to $\Leb^1$, we have that $\Leb^1$-a.e.~$t \in \R$ satisfies
\begin{equation}\label{Eq:rhot ptw cvg}
    \rho(t)=\lim_{n\to +\infty} \frac{{T}(r_n,s_n)}{s_n-r_n}
\end{equation}
for all sequences $(r_n),(s_n)\subset[0,1]$ so that $r_n\uparrow t$, $s_n\downarrow t$  and $r_n<s_n$ for every $n\in\N$.
\item For every $0 \ne \QQ<1$, write  $u_q(z):=\tfrac{z^q}{q}$ (recall also the conventions in Section \ref{Sub:InfConv} for $u_q(0)$) and for a partition $P=\{0=t_0 < t_1 < \dots < t_n=1\}$ of $[0,1]$ define the `discrete $q$-action' $A_q(P)$ as
\[
A_q(P):=\sum_{i=0}^{n-1}(t_{i+1}-t_i)\,u_q\Big(\frac{T(t_i,t_{i+1})}{t_{i+1}-t_i}\Big).
\]
Then
\begin{equation}
\label{eq:monpart}
\text{$P$ finer than $Q$}\qquad\Rightarrow\qquad A_q(P)\leq A_q(Q)
\end{equation}
and
\begin{equation}
\label{eq:uqinf}
\int_0^1u_q(\rho)\,\d\Leb^1= \inf A_q(P),
\end{equation}
where the infimum is taken among all partitions of $[0,1]$.
\item For a partition $P=\{0=t_0 < t_1 < \dots < t_n=1\}$ we set $|P|:=\max|t_{i+1}-t_i|$. Then for any sequence $(P_k)$ of partitions such that $|P_k|\to0$ we have
\begin{equation}
\label{eq:limpart}
\int_0^1u_q(\rho)\,\d\Leb^1=\lim_{k\to\infty} A_q(P_k).
\end{equation}
\end{enumerate}
\end{theorem}

\begin{proof} 
(ii.a) Notice that the measures $\mu_h$ are non-negative, because $T$ is,  and that for any $[a,b]\subset\R$ we have
\begin{equation}
\label{eq:massmuh}
\begin{split}
\mu_h([a,b])&= \tfrac1h\int_a^b {T}(t,t+h)\,\d t= \tfrac1h\int_a^{a+h} \sum_{i=0}^{n}T(t+ih,t+(i+1)h)\,\d t\stackrel{\eqref{eq:revtr} }\leq T(a,b+2h),
\end{split}
\end{equation}
where here $n$ is the integer part of $\tfrac{b-a}{h}$. 
Now  the existence of the weak limit   follows the idea from the ``subpartition lemma'' of Korevaar--Schoen \cite[Lem.~1.3.1]{KorevaarSchoen93}.  Fix a non-negative function $\varphi\in C_c(\R)$ and denote its modulus of continuity by $\omega\colon [0,+\infty)\to[0,+\infty)$. Pick $h>0$ and a finite sequence $(\alpha_i)\subset [0,1]$ with $\sum \alpha_i=1$. Set $\smash{A_i:=\sum_{j\leq i}\alpha_j}$ and notice that
 \[
\begin{split}
\sum_{i=1}^n\alpha_i\int\varphi\,\d\mu_{\alpha_ih}&=\frac1h\sum_{i=1}^n\int\varphi(t)\, {T}(t,t+\alpha_ih)\,\d t\\
&=\frac1h\int\,{{\sum_{i=1}^n \varphi(t)T\big(t+A_{i-1} h,t+ A_{i} h\big)}}-\big(\varphi(t -  A_{i-1} h)-\varphi(t)\big)\, T(t,t+\alpha_ih)\,\d t\\
\text{(by  \eqref{eq:revtr})}\qquad&\leq \int \varphi\,\d\mu_h+\frac1h\sum_{i=1}^n\omega(h)\int T(t,t+\alpha_ih)\,\d t\\
&=\int \varphi\,\d\mu_h+\omega(h)\sum_{i=1}^n\alpha_i \,\mu_{\alpha_ih}(\R)\leq \int \varphi\,\d\mu_h+\omega(h)\,T(0,1),
\end{split}
\] 
having used \eqref{eq:massmuh} and \eqref{eq:extT} in the last step.  Applying \Cref{le:simple} to $\smash{f(h):=\int\varphi\,\d\mu_h}$ we infer that $L(\varphi):=\smash{\lim_{h\downarrow 0}\int\varphi\d\mu_h}$ exists. It is then clear that this limit also exists for $\varphi\in C_c(\R)$ possibly negative and that the resulting functional $L:C_c(\R)\to\R$ is non-negative and finite. It follows by the Riesz-Markov-Kakutani theorem that there is a (unique) non-negative Radon measure $\mu$ representing $L$, i.e.\ such that $\int\varphi\,\d\mu=\lim_n\int \varphi\,\d\mu_n$.

We now come to the verification of all claims for this measure $\mu$.

(i)  Let $\varepsilon > 0$ satisfy $a<b-\varepsilon$. Then   lower semicontinuity on open sets under narrow convergence of finite measures and the estimate \eqref{eq:massmuh} imply
\begin{align*}
\mu((a,b-\varepsilon)) \leq \liminf_{h\downarrow0} \mu_{h}((a,b-\varepsilon)) \leq \liminf_{h\downarrow0}  \mu_{h}((a,b-2h)) \leq {T}(a,b).
\end{align*}
Letting $\eps\downarrow0$ we see that $\mu$ satisfies \eqref{eq:permaxmu}.

To show maximality of $\mu$, assume $\nu$ is any Borel measure on $\R$ satisfying the claimed inequality.  Since formulas \eqref{eq:extT} and \eqref{eq:revtr} grant that $T(a,b)=0$ when $a<b$ are both negative or both bigger than 1, we see that $\nu$ must be be concentrated on $[0,1]$, thus for $h>0$ we can define $\nu_h:= \nu * 1_{[0,h]}/h$. The assumption  \eqref{eq:permaxmu} implies $\smash{\nu_h \leq \mu_{h}}$ for every $n$, and sending $h\downarrow0$ yields $\nu \leq \mu$.

(ii.b) We now prove  the identity \eqref{Eq:rhot ptw cvg}. Given $t\in[0,1]$ we denote by $W(t)$ the collection of closed non-trivial intervals in $\R$   that contain $t$. Define $I,S:[0,1]\to[0,+\infty]$ by
\[
I(t):=\lim_{\eps\to 0}\inf_{\substack{[r,s]\in W(t),\\ |s-r|<\eps}}\frac{{T}(r,s)}{s-r},\qquad\qquad\qquad 
S(t):=\lim_{\eps\to 0}\sup_{\substack{[r,s]\in W(t),\\ |s-r|<\eps}}\frac{{T}(r,s)}{s-r}.
\]
The limits exist because the relevant quantities are monotone in $\eps$; for the same reason, from the Lebesgue measurability of $T$ (that follows from coordinatewise monotonicity --- see \cite[Thm.\ 4]{ChabCrou09})  we get Lebesgue measurability of $I$ and $S$. We claim that 
\begin{enumerate}
\item $\Leb^1$-a.e.~$t\in [0,1]$ satisfies
\begin{equation*}
    I(t)=S(t)<\infty,
\end{equation*}
\item the density $\rho$ is equal to $I$ (and hence $S$) $\Leb^1$-a.e.\ on $[0,1]$.
\end{enumerate}
These will readily imply the claimed identity.

(ii.b.1) To prove the equality of $I$ and $S$, it is sufficient to prove that the set $\{t\in[0,1]:I(t)<S(t)\}$ is $\Leb^1$-negligible. For $\alpha<\beta$ rational let $E:=\{t\in[0,1]:I(t)<\alpha<\beta<S(t)\}$. By the arbitrariness of $\alpha$ and $\beta$,  to prove the claim it suffices to show that $\mathscr{L}^1(E)=0$.

Fix $\eps>0$ and find an open superset $U\supset E$ with $\mathscr{L}^1(U\setminus E)<\eps$ (notice that $I=S=0$ outside $[0,1]$, so $E\subset[0,1]$ has finite measure). Let $\mathscr V$ be the collection of closed nontrivial intervals $[r,s]$ contained in $U$ such that $T(r,s)<\alpha(s-r)$. By definition of $E$, the family $\mathscr V$ is a Vitali covering of $E$. We thus apply Vitali's covering lemma to find a finite disjoint family $[r_k,s_k]$ such that $\mathscr{L}^1(E\setminus\bigcup_{k}[r_k,s_k])<\eps$. We have
\[
\sum_kT(r_k,s_k)\leq \alpha\sum_k (s_k-r_k)\leq \alpha\,\Leb^1(U)\leq \alpha(\mathscr{L}^1(E)+\eps).
\]
Let now $\smash{U':=\bigcup_k(r_k,s_k)}$ and $F:=E\cap U'$ and notice that  $\mathscr{L}^1(F)>\mathscr{L}^1(E)-\eps$. Also, let $\mathscr V'$ be the collection of closed non-trivial intervals $[a,b]$ contained in $U'$ such that  $T(a,b)>\beta(b-a)$. Again, by definition of $E$ we have that $\mathscr V'$ is a Vitali covering of $F$, hence we can find a finite disjoint subfamily $[a_j,b_j]$ such that $\mathscr{L}^1(F\setminus\bigcup_j[a_j,b_j])<\eps$. Therefore,
\begin{align*}
\sum_jT(a_j,b_j)>\beta \sum_j(b_j-a_j)>\beta(\mathscr{L}^1(F)-\eps)>\beta(\mathscr{L}^1(E)-2\eps).
\end{align*} 
Since each of the intervals $[a_j,b_j]$ is included in some of the $[r_k,s_k]$, from \eqref{eq:revtr} and the fact that $T$ is non-negative, non-increasing in the first entry and non-decreasing in the second one (from \eqref{eq:revtr}) we deduce 
\[
\begin{split}
\alpha(\mathscr{L}^1(E)+\eps)\geq \sum_kT(r_k,s_k)\geq\sum_jT(a_j,b_j)>\beta\sum_j(b_j-a_j)>\beta(\mathscr{L}^1(E)-2\eps)
\end{split}
\] 
and the arbitrariness of $\eps$ and the choice $\alpha<\beta$ imply $\Leb^1(E)=0$, as desired.

Next, we claim $\{t\in[0,1]:S(t)=+\infty\}$ is $\Leb^1$-negligible. 

For $\beta$ rational, let $E$ be the set of $t$'s with  $S(t)>\beta$, and let $\eps>0$.  Then, arguing as before, we can find finite disjoint intervals $[a_j,b_j]$ such that $\mathscr{L}^1(E\setminus\bigcup_j[a_j,b_j])<\eps$ and $T(a_j,b_j)>\beta(b_j-a_j)$ for every $j$. It is not restrictive to assume that these intervals are all contained in $[0,1]$. Then adding up, by \eqref{eq:revtr} and the fact that $T$ is non-negative we get
\[
\infty>T(0,1)\geq\sum_jT(a_j,b_j)>\beta\sum_j(b_j-a_j)>\beta(\mathscr{L}^1(E)-\eps).
\]
The arbitrariness of $\eps$ then yields $\mathscr{L}^1(E)<T(0,1)/\beta$, and the claim follows.

(ii.b.2) To conclude we need to prove that $\rho$ is equal to $I$, thus to $S$, $\Leb^1$-a.e.\ on $[0,1]$. To see this recall that by classical results about differentiation of measures (cf.\ e.g.\ \cite[Thm.\ 5.8.8]{Bog:07a} and notice that it is not a coincidence that these are proved along the same lines we just used) we have $h^{-1}\,\mu([t,t+h])\to \rho(t)$ as $h\downarrow0$ for $\Leb^1$-a.e.~$t\in[0,1]$, thus by the defining property (i) of $\mu$ we deduce $\Leb^1$-a.e.~$t\in[0,1]$ satisfies
\[
\rho(t)\leq I(t).
\]
On the other hand, let $C\subset[0,1]$ be compact with $\mu^\perp(C)=0$. Then the weak convergence of $(\mu_h)$ to $\mu$ as $h\downarrow 0$ provided by the proof of (ii.a) and the fact that the set $C$ is closed give
\[
\int_C\rho \,\d\Leb^1 =\mu(C)\geq \lims_{h\downarrow 0}\mu_h(C),
\]
while Fatou's lemma yields
\[
\limi_{h\downarrow 0}\mu_h(C) \geq\int_C\limi_{h\downarrow 0}\frac{T(t,t+h)}h\, \d t=\int_C I\,\d\Leb^1\geq \int_C\rho\,\d\Leb^1.
\]
Thus the inequalities must be equalities, and therefore $\rho=I$ almost everywhere on $C$. By the arbitrariness of $C$ (and inner regularity of the Lebesgue measure) we conclude that $\rho=I$  $\Leb^1$-a.e.~on $[0,1]$, as desired.

(iii)-(iv)   Since $u_q$ is concave and  non-decreasing, for any $t<r<s$ we have
\[
(r-t)u_q\big(\tfrac{T(t,r)}{r-t}\big)+(s-r)u_q\big(\tfrac{T(r,s)}{s-r}\big)\leq (s-t)u_q\big(\tfrac{T(t,r)}{r-t}+\tfrac{T(r,s)}{s-r}\big)\leq (s-t)u_q\big(\tfrac{T(t,s)}{s-t}\big)
\]
which is (equivalent to) \eqref{eq:monpart}. Since $u_q$ is also upper semicontinuous, for  $a<b$ we have
\[
\tfrac1{b-a} \int_a^b -u_q(\rho)\,\d\Leb^1\geq   -u_q\Big(\frac{1}{b-a}\int_a^b\rho\,\d\Leb^1\Big)\geq  -u_q\Big(\frac{\mu((a,b))}{b-a}\Big)\geq -u_q\Big(\frac{T(a,b)}{b-a}\Big),
\]
having used Jensen's inequality. From this,  ``$\leq$" in our claim \eqref{eq:uqinf} easily follows. 

Let now $(P_k)$ be a sequence of partitions as in (iv), i.e.\ such that $|P_k|\to0$, say $P_k=\{0=t_{k,0}<\ldots<t_{k,n_k}=1\}$. Define the `piecewise constant'  non-negative Radon measures $\nu_k$ on $[0,1]$ as $\nu_k=\eta_k\mathcal L^1$, where $\eta_k(x)=\tfrac{ T(t_{k,i},t_{k,i+1})}{t_{k,i+1}-t_{k,i}}$ if $x\in [t_{k,i},t_{k,i+1})$. Notice that $\nu_k([0,1])=\sum_iT(t_{k,i},t_{k,i+1})\leq T(0,1)$, so passing to a non-relabelled subsequence we can assume that $(\nu_k)$ has a weak limit $\nu$. We claim that $\nu=\mu$, and since this result does not depend on the particular subsequence chosen, this proves  weak convergence of the full original sequence of measures to $\mu$.

Let $[a',b']\subset(a,b)\subset\R$ and then, for every $k\in\N$, let $[a_k,b_k]$ be the smallest interval containing $[a',b']$ with $a_k,b_k\in P_k$. Then we have
\[
\nu_k([a',b'])\leq \nu_k({[a_k,b_k]})=\sum_{i=i_1}^{i_2} T(t_{k,i},t_{k,i+1})\stackrel{\eqref{eq:revtr}}\leq T(a_k,b_k),
\]
where $i_1,i_2$ are so that $t_{k,i_1}=a_k$ and $t_{k,i_2}=b_k$. Since $|P_k|\to 0$ we eventually have $[a_k,b_k]\subset(a,b)$ and therefore
\[
\nu((a',b'))\leq\liminf_k\nu_k((a',b'))\leq \limsup_k\nu_k([a',b'])\leq T(a,b).
\]
By interior approximation of the interval $(a,b)$ we see that $\nu$ satisfies \eqref{eq:permaxmu} and thus, by what already proved, that $\nu\leq \mu$.

Conversely, let $[a,b]\subset(a',b')\subset\R$ and then, for every $k\in\N$, let $[a_k,b_k]$ be the smallest interval containing $[a,b]$ with $a_k,b_k\in P_k$. Then we have
\[
\mu([a,b])\leq\mu([a_k,b_k])=\sum_{i=i_1}^{i_2}\mu([t_{k,i},t_{k,i+1}])\stackrel{\eqref{eq:permaxmu}}\leq \sum_{i=i_1}^{i_2}T(t_{k,i},t_{k,i+1})= \nu_k([a_k,b_k]).
\]
As before, from $|P_k|\to0$ we deduce that $[a_k,b_k]$ is eventually contained in $(a',b')$ and therefore
\[
\mu([a,b])\leq\limsup_k\nu_k([a_k,b_k])\leq\limsup_k\nu_k([a',b'])\leq \nu([a',b']).
\]
By exterior approximation of the interval $[a,b]$ we conclude that $\mu([a,b])\leq\nu([a,b])$, then by interior approximation of open intervals we deduce $\mu((a,b))\leq\nu((a,b))$ for any $(a,b)\subset\R$, then $\mu(U)\leq\nu(U)$ for any open set (as these are countable unions of disjoint intervals) and finally $\mu\leq\nu$ by outer regularity. Hence $\mu=\nu$, as claimed.

The conclusion now follows standard lower semicontinuity arguments.  Define the functional $V_q$ on the space of non-negative Radon measures on $[0,1]$ as:
$$V_q(\nu) := -\int_0^1 u_q(\eta)\d\mathscr L^1,\qquad\text{where}\quad \nu=\eta\Leb^1+\nu^\perp\quad\text{with}\quad\nu^\perp\perp\Leb^1.
$$
Since $-u_q$ is lower semicontinuous and sublinear,  $V_q$ is lower semicontinuous on the space of non-negative finite measures on $\R$ equipped with the weak convergence in duality with $C_c(\R)$ (see e.g.\ \cite[Proposition 7.7]{Santambrogio15} --- notice that in such reference the convex integrand is assumed to be real-valued, while in our case we have $-u_q(0)=+\infty$ if $q<0$, still this is easily cured by approximating $-u_q$ from below with real-valued convex and lower semicontinuous functions). As we just proved that $(\nu_k)$ weakly converges to $\mu$ we deduce that
\[
-V_q(\mu)\geq \limsup_{k\to\infty}(-V_q(\nu_k)),
\]
which, expanding the definitions, is inequality $\geq$ in \eqref{eq:limpart}. Since we already established the inequality $\leq$ in \eqref{eq:uqinf}, the proof is complete.
\end{proof}

The previous \Cref{Pr:speed}  assumed finiteness of $T$ (in which case the total mass of the measure $\mu$ therein has total mass bounded from above by $T(0,1)$; in particular, $\mu$ is a Radon measure). For functions $T$ taking infinity as a value (such as the $\ell_q$-time separation between probability measures), we shall need a generalization of \Cref{Pr:speed} given by the following corollary.  Recall $S \subset \R^2$ is called \emph{non-decreasing} if $(s'-s)(t'-t) \ge 0$ for all $(s,t),(s',t') \in S$;
for example, the diagonal of $\R^2$ is non-decreasing.

\begin{corollary}[Differentiation of the extended reverse triangle inequality]
\label{Cor:speed}
Let $T\colon H \to [0,+\infty]$ satisfy
\begin{equation*}
T(r,s)\leq T(r,t) -T(s,t)
\end{equation*} 
for every $r\leq s \leq t$. Then:
\begin{itemize}
\item[i)] There is a maximal \textnormal{(}with respect to inclusion\textnormal{)} non-decreasing closed subset $S\subset H$
such that $S_\infty\subset \{T=\infty\}$ and $S_0 \subset \{T<\infty\}$, where $S_\infty$
denotes the connected component of $H \setminus S$ containing the top left corner $(0,1)$ of the square 
\textnormal{(}but is empty if no such connected component exists\textnormal{)} and $S_0 := H \setminus (S \cup S_\infty)$ is relatively open. 
\item[ii)] The intersection $S_0 \cap \partial H$ denotes a relatively open subset of the diagonal and \Cref{Pr:speed} applies separately to each of its connected components in place of the unit interval.  
\item[iii)] The sum of the measures it yields can be extended to a maximal Borel measure $\mu$ on the diagonal satisfying  the inequality
\begin{equation*}
\mu((a,b)) \leq T(a,b)
\end{equation*}
whenever $a\leq b$ by setting $\mu(A):=+\infty$ unless $A \subset S_0 \cap \partial H$.
\end{itemize}
\end{corollary}

\begin{proof} Define $f(s):=\inf(\{t\in [0,1]:T(s-t,s+t)=\infty\}\cup\{s,1-s\})$. Then set $S:=\{(s-t,s+t)\in H:t=f(s)\}$. Thus we have $S_\infty=\{(s-t,s+t)\in H:t>f(s)\}$ and $S_0=\{(s-t,s+t)\in H:t<f(s)\}$. It is clear that $S$ is maximal as a non-decreasing subset, and that $S_0\subset\{T<\infty\}$. Note that as $T$ is non-decreasing in the second coordinate and non-increasing in the first, we get that $T(s-t_0,s+t_0)\leq T(s-t,s+t)$ for $t_0<t$, and in particular $S_\infty\subset\{T=\infty\}$. 

By construction, the set $S$ contains $(0,0)$ and $(1,1)$, thus $S_0\cap\partial H$ is a subset of the diagonal.
\end{proof}

The preceding results allow us to define the \emph{causal speed} of a causal path, i.e.\ maps $\gamma:[0,1]\to\M$ such that $\gamma_t\leq\gamma_s$ whenever $t\leq s$. Indeed, for such a path, the map $T(s,t):=\ell(\gamma_s,\gamma_t)$ satisfies \eqref{eq:revtr}, thanks to the reverse triangle inequality on $\M$ and is non-negative by causality of $\gamma$. Hence \Cref{Pr:speed} and \Cref{Cor:speed} apply, thus we can give:
\begin{definition}[Causal speed of curves in $\M$]\label{Def:causal speed} Let $\gamma:[0,1]\to \M$ be a causal path with values in the metric spacetime $(\M,\ell)$.

The  maximal measure $|\dot{\bm\gamma}|:=\mu$ on $[0,1]$ 
associated to $\gamma$ via \Cref{Pr:speed} and \Cref{Cor:speed} will be called   \emph{causal speed} of $\gamma$. Its Lebesgue decomposition with respect to $\Leb^1$ is denoted by $|\dot{\bm \gamma}|=|\dot\gamma|\,\mathscr L^1+|\dot{\bm\gamma}|^\perp$ and, abusing terminology, we shall also refer to $|\dot \gamma|$ as the path's causal speed.
\end{definition}
To be explicit: maximality of $|\dot{\bm\gamma}|$ and its absolutely continuous density $|\dot \gamma|$ means that
\begin{equation}
\label{eq:boundspeed}
\int_t^s|\dot\gamma_r|\,\d r\leq |\dot{\bm\gamma}|([t,s])\leq\ell(\gamma_t,\gamma_s)\qquad\forall t,s\in[0,1],\ t<s,
\end{equation}
and that these are the maximal measure and $L^1$-function bounded above by $\ell(\gamma_t,\gamma_s)$ in this sense. Also, property \eqref{Eq:rhot ptw cvg} implies, in particular, that
\begin{equation}
\label{eq:speedlimit}
|\dot\gamma_t|=\lim_{h\downarrow0}\frac{\ell(\gamma_t,\gamma_{t+h})}{h}=\lim_{h\downarrow0}\frac{\ell(\gamma_{t-h},\gamma_{t})}{h}\qquad a.e.\ t.
\end{equation}

%

\subsection{Left-continuous causal paths and their topology}
In this section we deal with Polish metric spacetimes and study topology on and of causal paths.

\begin{definition}[Left-continuous causal paths in $\M$]\label{Def:LCC paths M} Let $(\M,\uptau,\ell)$ be a Polish metric spacetime. A map $\gamma\colon[0,1]\rightarrow \M$  is called a \emph{left-continuous causal path} if $\gamma$ is left-continuous (i.e.\  $\gamma_t= \lim_{s\uparrow t}\gamma_s$ for every $t\in (0,1]$) and causal. The space of all such paths is denoted $\CC([0,1];\M)$.
\end{definition}
We choose to work with this class of curves for its good stability properties, see in particular \Cref{Pr:CC Polish} and \Cref{prop-cad-comp}. Left continuity is a concept closely related to that of forward completeness, which ensures that left limits exist. For this reason, in this section we shall mostly work on forward spacetimes.

Intuitively, when a test particle's worldline is known to be in $\CC([0,1];\M)$,  then its present position provides more information about its past than its future. 

\begin{lemma}[Countability of causal curve discontinuities]
\label{L:countable discontinuities2}
Let $(\M,\uptau,\ell)$ be a forward metric spacetime, $A\subset[0,1]$ and $\gamma:A\to \M$ causal (meaning $s,t \in A$ with $s\le t$ imply $\gamma_s \le \gamma_t$; no continuity assumptions).

Then $\gamma$ has at most a countable number of  discontinuities.
\end{lemma}
\begin{proof} Fix $c>0$ and consider the set $B:=\{t\in A:\lims_{s\to t,\ s\in A}\sfd(\gamma_s,\gamma_t)> c\}$: if we prove that $B$ is countable for all $c>0$ we are done. To derive a contradiction suppose $B$ is uncountable. Then there are  $t_0,T\in B$ with $t_0<T$ such that $B\cap(t_0,T]$ is uncountable, hence there is $t_1\in B\cap (t_0,T]$ so that $B\cap(t_1,T]$ is uncountable and by recursion we can  build an increasing sequence $(t_n)\subset B$ with $B\cap(t_n,T]$ uncountable for every $n$. By  definition of $B$ 
there is also a sequence $(s_n)\subset A$ with $s_n\in(t_{n-1},t_{n+1})$ and 
\begin{equation}
\label{eq:tnsnc}
\sfd(\gamma_{s_n},\gamma_{t_n})>c\qquad\text{  for every }n>0.
\end{equation}
The two sequences of points $\gamma_{t_1},\gamma_{t_2},\gamma_{t_3},\ldots$ and $\gamma_{t_1},\gamma_{s_2},\gamma_{t_3},\gamma_{s_4},\ldots$ are both non-decreasing   and bounded from above by $\gamma_T$, thus by forward completeness they must both have a limit. Since the second contains a subsequence of the first, these limits must coincide, and thus we must have $\sfd(\gamma_{t_n},\gamma_{s_n})\to 0$. This, however, is in contradiction with \eqref{eq:tnsnc}.
\end{proof}

Forward-completeness 
will be used to deduce the following. Note the explicit exclusion of the starting point from the statement.

\begin{proposition}[Recovering a left-continuous causal path from a dense set of values]
\label{prop:distord}
 Let $(\M,\uptau,\ell)$ be a forward metric spacetime, $\mathscr{D}\subset(0,1]$ dense and  $\eta\colon\mathscr{D}\to \M$ a  $\leq$-monotone map. 
 
 Then there is a unique left-continuous causal curve $\gamma\colon (0,1)\to \M$  such that $\eta(r)\leq \gamma_t\leq\eta(s)$ for every $r,s\in \mathscr{D}$ with $r< t\leq  s$.  In particular, if $\eta$ is left-continuous, then we have $\gamma\rvert_\mathscr{D}= \eta\rvert_\mathscr{D}$.  
 
 If $\eta$ has an upper bound, then $\gamma$ has a unique extension to $(0,1]$ sharing the same properties.
 \end{proposition}
\begin{proof} 
For $t\in (0,1)$ we define 
\[
\gamma_t:=\lim_{\substack{s\uparrow t,\\ s\in \mathscr{D}}}\eta(s).
\]
Forward-completeness and the existence of $r\in(t,1)\cap \mathscr D$ (so that $\eta(r)$ is an upper bound for $\{\eta(s):s\leq t\}$), ensures that the limit exists, thus the definition is well-posed. The construction  ensures continuity from the left. Causality of $\gamma$ is clear from the monotonicity of $\eta$ and the closedness of $\leq$, thus we have the claimed existence. 

For uniqueness, let $\tilde\gamma\colon (0,1)\to M$ be another such path. Given any $t\in (0,1)$ let $s\in \mathscr{D}$ with $s<t$. As $\smash{\gamma_s\leq \eta(s) \leq \tilde{\gamma}_t}$ by assumption, letting $s\uparrow t$ reveals $\smash{\gamma_t \leq \tilde{\gamma}_t}$. Reversing the roles of $\tilde{\gamma}$ and $\gamma$ and using antisymmetry of $\leq$ gives $\tilde{\gamma}_t= \gamma_t$. The last claim is now obvious.
\end{proof}

We now topologize the space $\CC([0,1];\M)$ of left-continuous  causal paths. To this aim, let us fix a bounded distance $\sfd$ inducing the Polish topology on our metric spacetime (in asking boundedness we are not losing generality, as we can always replace a given $\sfd$ with $1\wedge\sfd$). Then we  define a distance $\sfD$ on $\CC([0,1];\M)$ via
\[
\sfD(\gamma,\eta):= \sfd(\gamma_0,\eta_0)+\int_0^1\sfd(\gamma_t,\eta_t)\d t.
\]
On $\CC([0,1];\M)$ we put the topology induced by $\sfD$. As the next result shows, this induced topology depends only on the topology $\uptau$ of $\M$ and not on the particular choice of $\sfd$ (see also \cite[Remarks 9 and 14]{McC:23} for similar considerations in more regular settings):

\begin{proposition}[Left-continuous causal paths inherit Polishness]
\label{Pr:CC Polish}
 Let $(\M,\uptau,\ell)$ be a forward metric  spacetime.  Then $(\CC([0,1];\M),\sfD)$ is a complete and separable metric space and the following are equivalent:
 \begin{itemize}
 \item[i)] $\sfD(\gamma_n,\gamma)\to 0$ as $n\to\infty$
 \item[ii)] Any sequence $n_k\uparrow\infty$ admits a subsequence $n_{k_j}\uparrow\infty$ such that $\gamma_{n_{k_j},t}\to\gamma_t$ as $j\uparrow\infty$ for $t=0$ and for a.e.\ $t\in[0,1]$.
 \end{itemize}
 If moreover the topology    is  locally causally convex, then these are also equivalent to:
 \begin{itemize}
 \item[iii)]  $\gamma_{n,t}\to \gamma_t$ as $n\to\infty$ for  $t=0$ and for any continuity point $t$ of $\gamma$,
 \item[iv)]  $\gamma_{n,t}\to \gamma_t$ as $n\to\infty$ for a  set  of $t$'s that is dense in $[0,1]$ and contains 0.
 \end{itemize}
\end{proposition}
\begin{proof} Since $\M$ is separable, so is   the space  $L^1([0,1];\M)$ of measurable maps $f\colon [0,1]\to \M$    with respect to the $L^1$-distance $\smash{\int_0^1 \met(\cdot,\cdot)\,\d t}$ (recall that $\sfd$ is bounded). Since $(\CC([0,1];\M),\sfD)$ is a  subspace of $\M\times L^1([0,1];\M)$, we see that $\CC([0,1];\M)$ is separable.

For completeness, let $(\gamma^n) \subset \CC([0,1];\M)$ be a Cauchy sequence. First, this implies the sequence $(\gamma_0^n)$ of initial points converges to some $\gamma_0\in \M$ by completeness of $\M$. Second, if necessary we pass  to a non-relabelled  subsequence such that $\smash{\sum_n \sfD(\gamma^n,\gamma^{n+1})<+\infty}$ and thus that converges for $\smash{\Leb^1}$-a.e.~$t\in[0,1]$. Call $\mathscr{D}\subset[0,1]$ the set of $t$'s for which we have pointwise convergence and $\eta_t$ the corresponding limit (with $\eta_0 = \gamma_0$). Let $\gamma\in\CC([0,1];\M)$ be associated to $\eta$ via \Cref{prop:distord} and notice that by \Cref{L:countable discontinuities2}  we have that $\gamma_t=\eta_t$ for every $t\in \mathscr D$ except at most a countable set.

We shall prove that $\sfD(\gamma^n,\gamma)\to0$ as $n\to+\infty$, the argument being standard. By construction, for a.e.\ $t$ and for $t=0$ we have $\sfd(\gamma_t,\gamma^n_t)\leq\sum_{i\geq n}\sfd(\gamma^{i+1}_t,\gamma^i_t)$, thus by monotone convergence 
\begin{align*}
\limsup_{n\to+\infty}\sfD(\gamma,\gamma^n)&=\limsup_{n\to+\infty}\Big(\,\sfd(\gamma_0,\gamma_0^n)+\int_0^1\sfd(\gamma_t,\gamma^n_t)\,\d t \Big)\leq \limsup_{n\to+\infty}\sum_{i\geq n} \sfD(\gamma^{i+1},\gamma^i)=0,
\end{align*}
having used $\sum_n \sfD(\gamma^n,\gamma^{n+1})<+\infty$ in the last step. This argument also shows that $(i)\Rightarrow(ii)$. The converse implication follows directly from the definition of $\sfD$.

We pass to the locally causally convex case.

\noindent $(i)\Rightarrow(iii)$. Since $\sfd(\gamma_{n,0},\gamma_0)\leq\sfD(\gamma_n,\gamma)$,  the convergence  $\gamma_{n,0}\to\gamma_0$ is clear. Now let $t$ be a continuity point of $\gamma$. We will show that any subsequence $(n_k)$ has a further extraction $(n_{k_j})$ such that $\gamma_{n_{k_j},t}\to\gamma_t$, as this suffices to conclude.  Thus let $n_k\uparrow\infty$ be arbitrary and notice that, as above,   there is  a subsequence $(n_{k_j})$ such that $\gamma_{n_{k_j},s}\to\gamma_s$ for a.e.\ $s\in[0,1]$. Now let $U$ be an arbitrary causally convex neighbourhood of $\gamma_t$. Then there are $t_1<t<t_2$ with $\gamma_{t_1},\gamma_{t_2}\in U$ such that $\gamma_{n_{k_j},t_i}\to\gamma_{t_i}$, $i=1,2$. Hence $\gamma_{n_{k_j},t_i}$ is eventually in $U$, $i=1,2$,  and by causal convexity and causality of the curves we get that $\gamma_{n_{k_j},t}$ is eventually in $U$. This proves that $\gamma_{n_{k_j},t}\to\gamma_t$, as desired.

\noindent $(iii)\Rightarrow(i)$. Direct consequence of Lemma \ref{L:countable discontinuities2} and the dominated convergence theorem.

\noindent $(iii)\Rightarrow(iv)$. Direct consequence of Lemma \ref{L:countable discontinuities2}.

\noindent $(iv)\Rightarrow(iii)$. Arguments similar to those in ``$(i)\Rightarrow(iii)$'' apply: let $U$ be a causally convex neighbourhood of $\gamma_t$, where $t$ is a continuity point of $\gamma$. Then for $t_1<t<t_2$ belonging to the given dense set and sufficiently close to $t$ we have $\gamma_{t_1},\gamma_{t_2}\in U$, hence eventually $\gamma_{n,t_1},\gamma_{n,t_2}\in U$ and by causal convexity eventually we get $\gamma_{n,t}\in U$.
\end{proof}

The following compactness result is reminiscent of Helly's selection theorem:

\begin{theorem}[Limit curve theorem]
\label{prop-cad-comp}
Let $(\M,\uptau,\ell)$ be a forward metric spacetime. Let $\Gamma\subset \CC([0,1];\M)$ be so that:
\begin{itemize}
\item[i)] For some   $\mathscr{D}\subset[0,1]$  Borel with $\mathcal L^1(\mathscr D)=1$ and  $0\in {\mathscr D}$   the set $\{\gamma_t:\gamma\in\Gamma\}$ is relatively compact for each  $t\in\mathscr D$,
\item[ii)] There is a compact set $K\subset\M$ such that for every $\gamma\in\Gamma$ there is $x\in K$ with $\gamma_1\leq K$.
\end{itemize}
Then $\Gamma$ is $\sfD$-relatively compact.

If moreover the topology is locally causally convex, then we can weaken $(i)$ above to
\begin{itemize}
\item[i')] For some   $\tilde{\mathscr{D}}\subset[0,1]$  dense with $0\in\tilde{\mathscr D}$   the set $\{\gamma_t:\gamma\in\Gamma\}$ is relatively compact for each  $t\in\tilde{\mathscr D}$.
\end{itemize}
\end{theorem}

\begin{proof} Since we are dealing with a distance, compactness is equivalent to sequential compactness. Thus let $(\gamma_n)\subset\Gamma$ and $(x_n)\subset K$ be corresponding points as in $(ii)$; our goal is to prove that $(\gamma_n)$ has a $\sfD$-convergent subsequence. Let ${\mathscr C}\subset\mathscr D$ be countable dense and containing 0. Then a diagonal argument and the assumption of relative compactness   grant that up to passing to a non-relabelled subsequence we can assume that $n\mapsto \gamma_{n,t}$ converges to some limit $\eta(t)$ for every $t\in { \mathscr{C}}$ and similarly that $x_n\to x$. Clearly we have $\eta(t)\leq x$ for any $t\in{\mathscr C}$, so we can apply \Cref{prop:distord} to obtain from $\eta$ a left-continuous causal path $\gamma$ on $(0,1]$  and extend it by putting $\gamma_0:=\eta(0)$ (this does not affect causality and left continuity --- here we used that $0\in {\mathscr C}$ to have a limit point $\eta(0)$).

We claim that $\sfD(\gamma_n, \gamma)\to0$. For this it suffices, by \Cref{Pr:CC Polish}  above, to show that for a.e.\ $t\in[0,1]$ and for $t=0$ we have $\gamma_{n,t}\to\gamma_t$. We are going to show that this holds for any continuity point $t\in\mathscr D$ of $\gamma$; (by \Cref{L:countable discontinuities2} this suffices). Fix such a $t$ and notice that by compactness, after passing to a non-relabelled subsequence we can assume that $\gamma_{n,t}\to p$ for some $p\in \M$. For any $r,s\in {\mathscr C}$ with $r<t<s$ passing to the limit in  $\gamma_{n,r}\leq\gamma_{n,t}\leq\gamma_{n,s}$ we deduce that $\gamma_r\leq p\leq \gamma_s$. Letting $r\uparrow t$ and $s\downarrow t$, by the continuity of $\gamma$ at $t$ we conclude that $\gamma_t\leq p\leq\gamma_t$, and so $p=\gamma_t$. Since this conclusion holds independently of the subsequence chosen, the proof in this case is complete.

In the locally causally convex case we instead argue as follows. Let, as before, $t$ be a continuity point of $\gamma$, then let $U$ be a causally convex neighbourhood of $\gamma_t$ and notice that the relations $ \gamma_{s_1}\leq \eta(s_2)\leq \gamma_{s_3} $ valid for $s_1<s_2<s_3$ together with the continuity of $\gamma$ at $g$ ensure that $\eta(s)\in U$ for $s$ sufficiently close to $t$. In particular this holds if  $s$ is also in $\tilde{\mathscr{D}}$, so for such $s$ we eventually have $\gamma^n_s\in U$ and thus by causal convexity we conclude that eventually we have $\gamma^n_t\in U$. We thus proved that $\gamma^n_t\to\gamma_t$ for any continuity point $t$, which again suffices to conclude.
\end{proof}

In what follows we shall often consider the evaluation maps $\eval_t:\CC([0,1];\M)\to \M$ defined for $t\in[0,1]$ as $\gamma\mapsto\gamma_t$ and the map $\eval:\CC([0,1];\M)\times[0,1]\to \M$ sending $(\gamma,t)$ to $\gamma_t$. We shall frequently use the following basic result without explicitly mentioning it:

\begin{proposition}[Evaluation maps are Borel]
\label{prop:evborel}
Let $(\M,\uptau,\ell)$ be a Polish metric spacetime. Then $\eval_t:\CC([0,1];\M)\to \M$ is Borel for any $t\in(0,1]$ and continuous for $t=0$. Also  $\eval:\CC([0,1];\M)\times[0,1]\to \M$ is Borel.
\end{proposition}

\begin{proof}
The continuity of $\eval_0$ is clear, thus let $t\in(0,1]$. Since $\M$ is Polish we can establish measurability of $\eval_t$ by proving that $\eval_t^{-1}(B_r(p))$ is Borel for any $p\in \M$ and $r>0$. Left continuity implies 
\[
\eval_t^{-1}(B_r(p))=\bigcup_{n\in\N}\bigcap_{m\geq n}\Big\{\gamma\in\CC([0,1];\M)\ :\ \int_{0\wedge(t-\tfrac1m)}^t\sfd(\gamma_s,p)\,\d s<\tfrac rm\Big\}.
\]
Thus  the claim follows from the fact that the set in parenthesis is $\sfD$-open for every $m$. The same argument shows that $\eval^{-1}(B_r(p))$ is the union of $\eval_0^{-1}(B_r(p))\times\{0\}$ and 
\[
\bigcup_{n\in\N}\bigcap_{m\geq n}\Big\{(\gamma,t)\in\CC([0,1];\M)\times(0,1]\ :\ \int_{0\wedge(t-\tfrac1m)}^t\sfd(\gamma_s,p)\,\d s<\tfrac rm\Big\},
\]
and the second claim follows as well.
\end{proof}


Another map we shall need to use is the `restriction map': given $s,t\in[0,1]$ with $s<t$, we denote by $\smash{\restr_s^t\colon \CC([0,1];\M)\to \CC([0,1];\M)}$  the  map
\begin{align*}
\restr_s^t(\gamma)_r := \gamma_{(1-r)s + rt}.
\end{align*}
Intuitively,  this map first restricts the input curve to $[s,t]$ and then ``stretches'' it to $[0,1]$.

From Proposition \ref{prop:evborel} it is not hard to see that $\restr_s^t$ is Borel for any $s,t\in[0,1]$ with $s<t$ (and continuous if $s=0$). Indeed, given $\eta\in\CC([0,1];\M)$ and $r>0$, we see that $\sfD(\restr_s^t(\gamma),\eta)<r$ if and only if there is $q\in[0,r]\cap\Q$ such that $\sfd(\gamma_s,\eta_0)<q$ and $\int_0^1\sfd(\gamma_{(1-v)s + vt},\eta_v)\,\d v<r-q$. Equivalently, in symbols
\[
(\restr_s^t)^{-1}(B_r(\eta))=\bigcup_{q\in[0,r]\cap\Q}\Big(\eval^{-1}_s(B_q(\eta_0))\ \bigcap\ \Big\{\gamma\ :\ \int_0^1\sfd(\gamma_{(1-v)s + vt},\eta_v)\,\d v<r-q\Big\}\Big),
\]
which shows that $(\restr_s^t)^{-1}(B_r(\eta))$ is Borel.

\subsection{Action and geodesics}

Let $(\M,\ell)$ be a metric spacetime, $\gamma:[0,1]\to \M$ a causal curve and $0 \neq \QQ \le 1$. We define the $\QQ$-action $  \KE_\QQ(\gamma)\in\bar\R$ as
\begin{equation}
\label{Eq:KE M}
    \KE_\QQ(\gamma) := 
    \begin{cases}
    {\displaystyle \int_0^1u_q(|\dot\gamma_t|)\,\d t} & {\rm if}\ 0\ne \QQ <1    \\ &
\\     |\dot{\bm \gamma}|([0,1]) & {\rm if}\ \QQ=1,
    \end{cases}
\end{equation}
where we recall $u_q(z)=\tfrac1qz^q$ for $z>0$ and from   \eqref{eq:up} that $u_q(0)$ is equal to $0$ (resp.\ $-\infty$) if $q\in(0,1)$ (resp.\ $q<0$).  In particular, for $q<0$ we have $\KE_\QQ(\gamma)\leq 0$ and $\KE_\QQ(\gamma)=-\infty$ as soon as $\vert \dot\gamma\vert = 0$ holds on an $\Leb^1$-non-negligible set.

Irrespective of the sign of $\QQ$, \Cref{Pr:speed} and Jensen's inequality imply
\begin{align}\label{eq:actioUpperBound}
    \KE_\QQ(\gamma) \leq u_q(\ell(\gamma_0,\gamma_1)) \qquad\quad (=
    \frac{1}{\QQ}\,\ell(\gamma_0,\gamma_1)^\QQ \ {\rm if}\ q \vee \ell(\gamma_0,\gamma_1) >0).
\end{align}
In analogy with the definition of (constant speed) geodesics in the classical setting of metric spaces, we propose the following definition:

\begin{definition}[Geodesics on $\M$]\label{Def:l-geodesics} Let $(\M,\ell)$ be a metric spacetime and $x,y\in \M$ be with $x\leq y$. We call  $\gamma\colon [0,1]\to \M$ a rough geodesic from $x$ to $y$ provided $x\leq \gamma_t\leq y$ for every $t\in[0,1]$ and
\begin{equation}
\label{eq:defgeo}
\ell(\gamma_t,\gamma_s)=(s-t)\ell(x,y),\qquad\forall t,s\in[0,1],\ t\leq s.
\end{equation}
If $\M$ is also Polish and the  rough geodesic $\gamma$ belongs to $\CC([0,1];\M)$ we call it a \emph{geodesic}, thus dropping the term `rough'. The collection of all geodesics on $\M$ will be denoted $\CGeo(\M)\subset\CC([0,1];\M)$, the `C' standing for `causal'.

A causal (possibly rough) geodesic is called \emph{timelike} if $\ell(\gamma_0,\gamma_1)\in(0,+\infty)$. The collection of all left continuous, timelike  geodesics is denoted  $\TGeo(\M)\subset \CGeo(\M)$. We use the term $\ell$-geodesic synonymously with geodesic, particular when more than one time separation 
is under discussion. 
\end{definition}

\begin{remark}The term {\em rough} is used analogously in \cite{BraunMcCann:2023}. Although not needed here, we reserve the label $\AGeo(M)$ for the set of {\em affinely parameterized} causal geodesics defined as the $\sfD$-closure of $\TGeo(M)$;
a related closure was instead denoted by $\CGeo(M)$ in the final version of \cite{McC:23}.
\end{remark}

Notice that in order for $\gamma$ to be a geodesic from $x$ to $y$ we are \emph{not} insisting on $\gamma_0=x$ or $\gamma_1=y$, still, for any $t\in[0,1]$ we have:
\begin{equation}
\label{eq:geo}
\ell(x,\gamma_t)=\ell(\gamma_0,\gamma_t)\quad\text{and}\quad
\ell(\gamma_t,y)=\ell(\gamma_t,\gamma_1)\quad\text{in addition to}\quad
\ell(x,y)=\ell(\gamma_0,\gamma_1).
\end{equation}
Indeed, if $\ell(x,y)=+\infty$ the claim follows easily from \eqref{eq:defgeo} and $x\leq \gamma_t\leq y$, while if $\ell(x,y)<+\infty$, from non-negativity of $\ell$ on $\le$ 
we deduce  $\ell(x,y)
\geq \ell(x,\gamma_t) + \ell(\gamma_t,y)
\geq \ell(\gamma_0,\gamma_t)+\ell(\gamma_t,\gamma_1)=\ell(x,y)$, having used \eqref{eq:defgeo} in the last step. Since $\ell(x,y)<+\infty$, the claim follows. Sufficient conditions to ensure that $\gamma_0=x$ and/or $\gamma_1=y$ will be explored in \Cref{Sub:qessreg}.

The link between geodesics and the $\QQ$-action functionals is given by the following result:

\begin{proposition}[Geodesics maximize the $\QQ$-action]\label{Pr:l-geodesics maximize} Let $(\M,\ell)$ be a metric spacetime,\\ $\gamma:~[0,1]\to~\M$ causal and $x,y\in \M$ be with $x\leq \gamma_0$, $\gamma_1\leq y$ and $\ell(\gamma_0,\gamma_1)<+\infty$. Then the following  are equivalent.
\begin{enumerate}[label=\textnormal{(\roman*)}]
    \item The curve $\gamma$ is a rough geodesic from $x$ to $y$.
    \item  We have: 
    \begin{equation}
\label{eq:geoalt}
({s-t})\,\ell(x,y) \leq \ell(\gamma_t,\gamma_s),\qquad\forall t,s\in[0,1], \ t\leq s.
\end{equation}
    \item There exists a real constant $c\geq 0$ such that $\vert\dot\gamma\vert = c$ $\mathscr{L}^1$-a.e.~and    $\KE_1(\gamma) \geq \ell(x,y).$
   \item For all $0 \ne \QQ<1$ we have
    \begin{align*}
        \KE_\QQ(\gamma) = u_\QQ(\ell(x,y)) \qquad\quad  (=\frac{1}{\QQ}\, \ell(x,y)^\QQ 
        \mbox{\rm\ \ if}\ \ q \vee \ell(\gamma_0,\gamma_1)>0).
    \end{align*}
    \item For some $0 \neq \QQ < 1$ we have $\KE_\QQ(\gamma) = u_\QQ(\ell(x,y))$.    
\end{enumerate}
\end{proposition}


\begin{proof} \ \\
\noindent{$(i) \Longrightarrow (iii)$}. By \eqref{eq:defgeo} we see that   the measure $\mu_h$ from \Cref{Pr:speed} is equal to $\ell(x,y)\, \mathscr{L}^1\mres [0,1-h]$ on $[0,1-h]$. In particular, the associated maximal measure $\mu$ is $\ell(x,y)\Leb^1$,   from which the claim follows.

\noindent{$(iii)\Longrightarrow (iv)$}. The assumption  forces $\vert\dot\gamma\vert = \ell(x,y)$ $\mathscr{L}^1$-a.e, so the claim follows from the definition of $\KE_\QQ(\gamma)$.

\noindent{$(iv) \Longrightarrow (v)$}. Obvious. 

\noindent{$ (v) \Longrightarrow (ii)$}.   Let $0 \neq \QQ < 1$ satisfy the assumption. Then
\begin{align*}
\frac{1}{\QQ}\,\ell(x,y)^\QQ = \frac{1}{\QQ}\int_0^1\vert\dot\gamma_t\vert^\QQ \d t \stackrel{{\rm (Jensen)}}\leq \frac{1}{\QQ}\,\Big[\!\int_0^1 \vert\dot\gamma_t\vert \d t\Big]^\QQ\ \stackrel{\eqref{eq:boundspeed}}\leq \ \frac{1}{\QQ}\,\ell(\gamma_0,\gamma_1)^\QQ\leq \frac{1}{\QQ}\,\ell(x,y)^\QQ.
\end{align*}
Since  by assumption we have $\ell(\gamma_0,\gamma_1)<+\infty$, we deduce that $\ell(x,y)<+\infty$ as well. If $\ell(x,y)=0$ the conclusion is trivial, thus we can assume $\ell(x,y)\in(0,+\infty)$.  In this case by strict concavity of  $z\mapsto \tfrac1qz^q=u_q(z)$, the equality in Jensen's inequality forces  $\vert \dot\gamma\vert = \ell(x,y)$ $\Leb^1$-a.e. The conclusion  \eqref{eq:geoalt} follows from \eqref{eq:boundspeed}.

\noindent{$(ii) \Longrightarrow (i)$}. By the reverse triangle inequality for $\ell$, for every $t,s\in[0,1]$ with $t\leq s$  we have
\begin{align*}
    \ell(x,y) \geq     \ell(\gamma_0,\gamma_1) &\geq \ell(\gamma_0,\gamma_t) + \ell(\gamma_t,\gamma_s) + \ell(\gamma_s,\gamma_1)\geq \big[t + (s-t) + (1-s)\big]\,\ell(x,y)= \ell(x,y).
\end{align*}
Since $\ell(\gamma_0,\gamma_1)<+\infty$ by assumption, this forces  $\ell(x,y)<+\infty$ and thus equality throughout.
\end{proof}

\begin{remark}[Relation to customary lengths]  \Cref{Def:l-geodesics} is inspired by the length functional $L_\ell$ induced by $\ell$ \cite[Def.~2.24]{KS:18} and its maximizers. We shortly comment on the connections between maximizers of $L_\ell$ and our action functionals, to relate our approach to others used in the literature. To adapt our setting to \cite{KS:18}, we shall assume $\ell_+$ is lower semicontinuous.

On the space $\CC([0,1];\M)$ we define
\begin{align*}
L_\ell(\gamma) := \inf\!\Big\lbrace\!\sum_{i=0}^{n-1} \ell(\gamma_{t_i}, \gamma_{t_{i+1}}) : n\in\N,\, 0 = t_0 < t_1 < \dots < t_n = 1\Big\rbrace.
\end{align*}
By \eqref{eq:boundspeed}  and the definition of $L_\ell$, we have
\begin{align}\label{Eq:Inequ ALl}
\KE_1(\gamma)\leq L_\ell(\gamma) \leq \ell(\gamma_0,\gamma_1).
\end{align}
A geodesic clearly attains the second inequality, and every $\gamma\in \CC([0,1];\M)$ with $\KE_1(\gamma) = \ell(\gamma_0,\gamma_1)$ maximizes $L_\ell$ by \eqref{Eq:Inequ ALl}. On the other hand, if $\gamma$ is a geodesic, \Cref{Pr:l-geodesics maximize} implies $\KE_1(\gamma) = \ell(\gamma_0,\gamma_1)$, hence equality holds throughout \eqref{Eq:Inequ ALl}. In order to relate general  maximizers of $L_\ell$ (which do not need to be affinely  parametrized) to those of $\KE_1$, further assumptions are needed. Suppose $\ell$ is upper semicontinuous and finite on $\le$, and moreover we suppose that $\ell_+$ is lower semicontinuous. Then   $\ell_+$ is continuous. Hence, if $\gamma\in\CC([0,1];\M)$ is a \emph{timelike}
maximizer of $L_\ell$, by \cite[Cor.~3.35]{KS:18} it has a strictly increasing  reparametrization $\sigma$ which is a geodesic. As noted above, this implies $\KE_1(\sigma) = \ell(\sigma_0,\sigma_1) = \ell(\gamma_0,\gamma_1)$. On the other hand, as $\KE_1$ is easily seen to be invariant under strictly increasing reparametrizations, equality holds in \eqref{Eq:Inequ ALl}.\hfill$\blacksquare$ 
\end{remark}

\begin{remark}[`Snowflakes']
A standard construction in positive signature in metric geometry is that of `snowflaking' a metric. Say that $(\M,\sfd)$ is a metric space. Then for any $\alpha\in(0,1)$ the function $\sfd^\alpha$ is still a distance, it is equivalent to the original one and has the property that any non-constant curve has infinite $\sfd^\alpha$-length.

A similar construction is in place for spacetimes (albeit in this case `snowflake' does not convey the appropriate geometric intuition): given a metric spacetime $(\M,\ell)$ and $\alpha>1$, the function $\ell^\alpha$ (where $\ell^\alpha(x,y):=-\infty$ is intended if $\ell(x,y)=-\infty$) is still a time separation inducing the same chronological and causal relations. If $\gamma$ is a causal curve and $t\in[0,1]$ such that $\lim_{h\downarrow 0}\frac{\ell(\gamma_t,\gamma_{t+h})}{h}$ exists and is finite, then clearly  $\lim_{h\downarrow 0}\frac{\ell(\gamma_t,\gamma_{t+h})^\alpha}{h}\to0$. It follows that the $q$-energy, $0\neq q<1$, of any causal curve w.r.t.\ $\ell^\alpha$ is zero.
\hfill$\blacksquare$ \end{remark}

We come to topological properties of geodesics:
\begin{proposition}[Causal and null geodesics form closed sets]
\label{C:CGeo is closed} Let $(\M,\uptau,\ell)$ be a forward spacetime such that    $\ell_+$ is continuous. 

Then the set of causal geodesics    from \Cref{Def:l-geodesics} is  $\sfD$-closed and that of   timelike geodesics is Borel.
\end{proposition}

\begin{proof} 
Let $(\gamma^n) \subset \CGeo(\M)$ be a sequence $\sfD$-converging to $\gamma\in \CC([0,1]; \M)$.  By \eqref{eq:defgeo} we see that
\[
    \ell(\gamma_s^n,\gamma_t^n) = (t-s)\,\ell(\gamma_0^n,\gamma_1^n),\qquad\forall s,t\in[0,1],\ s\leq t.
\]
By  \Cref{Pr:CC Polish} we see that for the Borel set  $\mathscr D\subset[0,1]$ containing 0 from the statement, passing to the limit in the above (notice that $1\in\mathscr D$ may not hold) we get
\[
\ell(\gamma_s,\gamma_t)=\tfrac{t-s}{t'}\ell(\gamma_0,\gamma_{t'})\qquad\forall s,t,t'\in\mathscr D,\ s\leq t\leq t'.
\]
Then by the left continuity of $\gamma$ as $t' \to 1$ and the continuity of $\ell_+$ again we easily conclude that $\gamma$ is a causal geodesic. The argument also shows that the collections of causal geodesics $\gamma$ satisfying $\ell(\gamma_0,\gamma_1)=+\infty$ (and similarly $\ell(\gamma_0,\gamma_1)=0$) are both closed as well, so the second claim follows.
\end{proof}

The $q$-action \eqref{Eq:KE M} inherits
upper semicontinuity from $\ell$ via identity \eqref{eq:uqinf}:

\begin{lemma}[Upper semicontinuity of the $\QQ$-action]\label{prop-p-act-usc}
Let $(\M,\uptau,\ell)$ be a metric spacetime with $\ell$ upper semicontinuous and let $0\neq \QQ<1$. Assume $(\gamma_n)\subset \CC([0,1];\M)$ $\sfD$-converges to $\gamma$ and that either $\sup_n \ell(\gamma_n(0),\gamma_n(1)) < +\infty$ or $\gamma_n(1)\to\gamma(1)$. Then
    \begin{equation*}
        \limsup_{n\to+\infty}\, \KE_\QQ(\gamma_n) \leq \KE_\QQ(\gamma).
    \end{equation*}
\end{lemma}
\begin{proof} We can pass to a non-relabelled subsequence realizing the $\limsup$ and then --- by \Cref{Pr:CC Polish} --- to a further extraction so that $\gamma_{n,t}\to\gamma_t$ for every $t\in \mathscr D$, where  $ \mathscr D\subset[0,1]$ is a Borel set of full measure containing 0. We can thus find   partitions $P_k=\{0=t_{k,0},\ldots,t_{k,n_k}=1\}$ of $[0,1]$ contained in $\mathscr D\cup\{1\}$ with $|P_k|\to 0$ (notation from item $(iv)$ of Theorem \ref{Pr:speed}).  Notice that since $u_q$ is non-decreasing, upper semicontinuity of $\ell$ implies that of $u_q\circ\ell$, thus if we are in the case  $\gamma_{n,1}\to\gamma_1$ we have
\[
\begin{split}
\sum_{i}(t_{k,i+1}-t_{k,i})\, u_q\big(\tfrac{\ell(\gamma_{t_{k,i}},\gamma_{t_{k,{i+1}}})}{t_{k,i+1}-t_{k,i}}\big)&\geq \limsup_{n\to\infty}\sum_{i}(t_{k,i+1}-t_{k,i})\, u_q\big(\tfrac{\ell(\gamma_{n,t_{k,i}},\gamma_{n,t_{k,{i+1}}})}{t_{k,i+1}-t_{k,i}}\big)\\
\text{(by \eqref{eq:uqinf})}\qquad&\geq         \limsup_{n\to+\infty}\, \KE_\QQ(\gamma_n),
\end{split}
\] 
so that letting $k\to\infty$ and recalling item $(iv)$ of Theorem \ref{Pr:speed} we are done. 

If we don't know whether $\gamma_{n,1}\to\gamma_1$, still the same argument shows that 
\[
\limsup_{n\to+\infty}\, \KE_\QQ(\gamma_n,[0,T]) \leq \KE_\QQ(\gamma,[0,T])\qquad\forall T\in\mathscr D,
\]
where $ \KE_\QQ(\gamma,[0,T]):=\int_0^Tu_q(|\dot\gamma_t|)\,\d t$ is the $q$-action on the interval $[0,T]$.  It is easy to see that $\lim_{T\uparrow 1}\KE_\QQ(\gamma,[0,T])=\KE_\QQ(\gamma,[0,1])$, so the conclusion will follow if we show that
\[
\lim_{T\uparrow1}\sup_n\KE_\QQ(\gamma_n,[T,1]) \, \le\, 0.
\]
To see this, let  $\sup_n\ell(\gamma_n(0),\gamma_n(1))<L <\infty$ and notice that Jensen's inequality applied to the concave upper semicontinuous function $u_q$ yields
\[
\KE_\QQ(\gamma_n,[T,1]) 
\leq (1-T)u_q\Big((1-T)^{-1}\int_T^1|\dot\gamma_{n,t}|\,\d t\Big)
\leq (1-T)u_q((1-T)^{-1}L)=\frac{L^\QQ}{\QQ}(1-T)^{1-q},
\]
where the second inequality follows from monotonicity of $u_q$ and the bound \eqref{eq:boundspeed}. The conclusion now follows from the fact that $q<1$.
\end{proof}


We conclude the section with a couple of simple results about the structure of (rough) $\ell_q$-geodesics. In particular these will show $\ell_q$-geodesics to/from Dirac masses do not depend on~$q$. Note that the concepts of geodesic defined in \Cref{Def:l-geodesics} apply equally well to the metric spacetime $(\Prob(\M),\ell_q)$ equipped with its narrow topology.

\begin{proposition}[On the structure of $\ell_q$-geodesics]\label{prop:ellqgeo} Let $(\M,\uptau,\ell)$ be a Polish metric spacetime in which $\ell$ is upper semicontinuous and does not take the value $+\infty$  and let $0\neq q<1$. Fix $\nu_0,\nu_1\in\pem$ with $\ell_q(\nu_0,\nu_1)\in(0,+\infty)$  and $(\mu_t)$ 
a rough $\ell_q$-geodesic from $\nu_0$ to $\nu_1$.

Fix  $t,s\in[0,1]$ with $t\leq s$ and  $\pi^0,\pi^t,\pi,\pi^s,\pi^1\in\Prob(\M^2)$ $\ell_q$-optimal for $(\nu_0,\mu_0)$, $(\mu_0,\mu_t)$, $(\mu_t,\mu_s)$, $(\mu_s,\mu_1)$ and $(\mu_1,\nu_1)$ respectively and a gluing  $\hat\pi\in\Prob(\M^6)$ of these along their common marginals. Then the coupling $({\rm Pr}_1,{\rm Pr}_6)_*\hat\pi$ is $\ell_q$-optimal for $(\nu_0,\nu_1)$ and
\begin{equation}
\label{eq:tuttiinfila}
\ell(x_0,y_t)=\ell(y_0,y_t),\quad
\ell(y_t,x_1)=\ell(y_t,y_1),\quad
\ell(x_0,x_1)=\ell(y_0,y_1),\quad
\ell(y_t,y_s)=(s-t)\ell(x_0,x_1)
\end{equation}
holds  for $\hat\pi$-a.e.\ $(x_0,y_0,y_t,y_s,y_1,x_1)$.
\end{proposition}

\begin{proof} Note $\ell_q(\nu_0,\nu_1)>0$ forces $\hat \pi[\{\ell>0\}]=1$
when $\QQ<0$.
For $\hat\pi$ as in the statement the inequality 
\begin{equation}
\label{eq:6variabili}
\ell(x_0,x_1)\geq\ell(x_0,y_0)+\ell(y_0,y_t)+\ell(y_t,y_s)+\ell(y_s,y_1)+\ell(y_1,x_1)
\end{equation}
holds for $\hat\pi$-a.e.\ $(x_0,y_0,y_t,y_s,y_1,x_1)$ and therefore
\[
\begin{split}
\ell_q(\nu_0,\nu_1)&\geq \Big({ \int\ell^q(x_0,x_1)\,\d\hat\pi}\Big)^{\frac 1q}
\\
\text{(by \eqref{eq:6variabili})}\qquad
&\geq\Big(\int\big(\ell(x_0,y_0)+\ell(y_0,y_t)+\ell(y_t,y_s)+\ell(y_s,y_1)+\ell(y_1,x_1)\big)^q\,\d\hat\pi\Big)^{\frac 1q}
\\&\stackrel*\geq
\ell_q(\nu_0,\mu_0)+\ell_q(\mu_0,\mu_t)+\ell_q(\mu_t,\mu_s)+\ell_q(\mu_s,\mu_1)+\ell_q(\mu_1,\nu_1)
\stackrel{\eqref{eq:geo}}=\ell_q(\nu_0,\nu_1),
\end{split}
\]
where
in the starred inequality we used  \eqref{eq:revlq} and optimality of the given plans. It follows that all the inequalities are equalities, and also inspecting the equality case in \eqref{eq:revlq} the conclusion quickly follows.
\end{proof}

From the above it follows that $\ell_q$-geodesics to/from Dirac masses do not depend on $q$:


\begin{proposition}[On $\ell_q$-geodesics with Dirac endpoints]\label{prop:geodindq}
 Let $(\M,\uptau,\ell)$ be a Polish metric spacetime  in which $\ell$ is upper semicontinuous and does not take the value $+\infty$. Also, let $\nu\in\pem$ and $\bar x\in \M$ be such that $\log(\ell(\cdot,\bar x))\in L^\infty(\nu)$ and let $(\mu_t)_{t \in [0,1]} \subset \Prob(\M)$.
 
 Then $(\mu_t)$ is a rough $\ell_q$-geodesic from $\nu$ to $\delta_{\bar x}$ for some $0\neq q<1$ if and only if it is so for every   $0\neq q<1$. In this case, for $t,s\in[0,1]$, $t\leq s$ and $\pi\in\Pi_\leq(\mu_t,\mu_s)$, the $\ell_q$-optimality of $\pi$ is also independent of $q$.

Similarly statements hold for rough $\ell_q$-geodesics from Dirac masses to other measures.
\end{proposition}

\begin{proof}
Suppose that $(\mu_t)$ is a rough $\ell_q$-geodesic from $\nu$ to $\delta_{\bar x}$ for a given  $0\neq q<1$, fix $t,s\in[0,1]$ with $t\leq s$ and let     $\pi^0,\pi,\pi^1\in\Prob(\M^2)$ be  $\ell_q$-optimal for $(\nu,\mu_t)$, $(\mu_t,\mu_s)$ and $(\mu_s,\delta_{\bar x})$ respectively. Let $\hat\pi\in\Prob(\M^4)$ be a gluing of these along the common marginals and notice that  \Cref{prop:ellqgeo} above (which applies because $\log(\ell(\cdot,\bar x))\in L^\infty(\nu)$  implies $\ell_q(\nu,\delta_{\bar x})\in(0,+\infty)$) says that for  $\hat\pi$-a.e.\  $(x_0,y_t,y_s,x_1)$ we have
\begin{equation}
\label{eq:allineate}
\ell(x_0,y_t)=t\ell(x_0,x_1)\quad\text{and}\quad\ell(y_t,y_s)=(s-t)\ell(x_0,x_1)\quad\text{and}\quad\ell(y_s,x_1)=(1-s)\ell(x_0,x_1)
\end{equation}
and logarithmically bounded. Now let $0\neq q'<1$ be another exponent and notice that 
\[
\begin{split}
\ell_{q'}(\nu,\delta_{\bar x})&\geq \ell_{q'}(\nu,\mu_t)+\ell_{q'}(\mu_t,\mu_s)+\ell_{q'}(\mu_s,\delta_{\bar x})\\
&\geq \Big(\int \ell^{q'}(x_0,y_t)\,\d\hat\pi\Big)^{\frac1{q'}}+ \Big(\int \ell^{q'}(y_t,y_s)\,\d\hat\pi\Big)^{\frac1{q'}}+ \Big(\int \ell^{q'}(y_s,x_1)\,\d\hat\pi\Big)^{\frac1{q'}}\\
\text{(by \eqref{eq:allineate})}\qquad&=  \Big(\int \ell^{q'}(x_0,x_1)\,\d\hat\pi\Big)^{\frac1{q'}}=\ell_{q'}(\nu,\delta_{\bar x}),
\end{split}
\]
thus again all the inequalities must be equalities, implying in particular that the plans $\pi^0,\pi,\pi^1$ are also $\ell_{q'}$-optimal. This computation and \eqref{eq:allineate} also show that $\ell_{q'}(\mu_t,\mu_s)=(s-t)\ell_{q'}(\nu,\delta_{\bar x})$, so the arbitrariness of $t,s$ imply that $(\mu_t)$ is also a rough $\ell_{q'}$-geodesic.
\end{proof}

\subsection[Lifting curves of probability measures to measures on curves]{Lifting curves of probability measures to measures on curves}\label{Sub:Lifting}

The main result of this section is  \Cref{Th:Lifting}, where we prove a ``lifting''  result for left-continuous causal paths on the space of Borel probability measures $\Prob(\M)$ on a given Polish metric spacetime $(\M,\uptau,\ell)$.

Thanks to \Cref{prop:ellqst} we know  that  $(\Prob(\M),\ell_q)$ is a Polish metric spacetime;  we shall always endow $\Prob(\M)$ with the narrow topology,
so that $\CC([0,1];\Prob(\M))$ refers to narrowly left continuous causal curves of measures. In particular, the definition of   $q$-action we gave in \eqref{Eq:KE M} also applies here, giving  the functional ${\mathscr{A}}_\QQ$ on causal curves in $\Prob(\M)$ defined as
\begin{align*}
{\mathscr{A}}_\QQ(\mu_\cdot) := \begin{cases} \displaystyle\frac{1}{\QQ}\int_0^1 \big\vert\dot\mu_t\big\vert_\QQ^\QQ\d t & \text{if } 0 \neq \QQ< 1,\\ &\\ |\dot{\bm \mu}|_1([0,1]) & \text{if } \QQ =1,
\end{cases}
\end{align*}
where  here and below $|\dot\mu_t|_q$ denotes the causal speed w.r.t.\ $\ell_q$ and for clarity we shall often write $\tfrac1q|\dot\mu_t|_q^q$ in place of $u_q(|\dot\mu_t|_q)$,
with the conventions analogous to \eqref{eq:up} implicitly understood. 
Calling a Borel probability measure $\bdpi$ on left continuous causal curves
a {\em plan} (or dynamical plan),
we have the following general relation between actions at the level of measures and at the level of curves:


\begin{lemma}[Relation between $\ell$- and $\ell_q$-causal speeds]
\label{Le:speed comparison} Let $(\M,\uptau,\ell)$ be a Polish metric spacetime, $0\neq \QQ < 1$ and    $\boldsymbol{\pi} \in \mathscr{P}(\CC([0,1]; \M))$. Put $\mu_t:=(\eval_t)_*\ppi$ for every $t\in[0,1]$. Then $(\mu_t)\in \CC([0,1];\Prob(\M))$ and 
\begin{equation}
\label{eq:intub}
{\mathscr{A}}_\QQ(\mu_\cdot) \geq \frac{1}{\QQ}\iint_0^1 \vert\dot\gamma_r\vert^\QQ\d r\d\bdpi(\gamma).
\end{equation}
Moreover, the pointwise bound
\begin{equation}
\label{eq:pointspbound}
\frac{1}{\QQ}\big\vert \dot\mu_t\big\vert^\QQ_\QQ \geq \frac{1}{\QQ}\int\vert\dot\gamma_t\vert^\QQ\d\boldsymbol{\pi}(\gamma)\qquad a.e.\ t
\end{equation}
holds if either $q\in(0,1)$ or  $\smash{\int \vert\dot\gamma_\cdot\vert^\QQ\d\bdpi(\gamma) \in L^1([0,1];\Leb^1)}$.
\end{lemma}

\begin{proof} For $t,s\in[0,1]$ with $t\leq s$ notice that since $\ppi$ is concentrated on causal curves the plan $(\eval_t,\eval_s)_*\ppi$, that is admissible for $(\mu_t,\mu_s)$, is concentrated on $\{\ell\geq 0\}$, thus $t\mapsto\mu_t\in\Prob(\M)$ is causal. For left continuity let $\varphi\in C_b(\M)$ and notice that since $\ppi$ is concentrated on left continuous curves, an application of the dominated convergence theorem yields $\lim_{s\uparrow t}\int \varphi(\gamma_s)\,\d\ppi(\gamma)=\int \varphi(\gamma_t)\,\d\ppi(\gamma)$, showing narrow left continuity of $(\mu_t)$.
 
 We turn to \eqref{eq:intub}. For any $t\in (0,1)$ and any $h\in (0,1-t)$, $(\eval_t,\eval_{t+h})_\push\bdpi$ is a causal coupling of $\mu_t$ and $\mu_{t+h}$ by construction. Thus 
\begin{align*}
    \frac{\ell_\QQ(\mu_t, \mu_{t+h})^\QQ}{q} &\geq   \int \frac{\ell(\gamma_t,\gamma_{t+h})^\QQ }q\d\bdpi(\gamma) \geq \tfrac{1}{\QQ}\int \Big[\!\int_t^{t+h} \vert\dot\gamma_r\vert\d r\Big]^\QQ\d\bdpi(\gamma) \geq \tfrac{h^{\QQ-1}}{\QQ}\iint_t^{t+h} \vert\dot\gamma_r\vert^\QQ\d r\d\bdpi(\gamma).
    \end{align*}
    where the second and third inequality follow from the bound \eqref{eq:boundspeed} and Jensen's inequality respectively.
    Now \Cref{Pr:speed}(iii) yields \eqref{eq:intub}. 
    For \eqref{eq:pointspbound} we divide the above by $h^q$ and let $h\downarrow0$: the left hand side converges to $\frac{1}{\QQ}\big\vert \dot\mu_t\big\vert^\QQ_\QQ $ for a.e.\ $t$ by \Cref{Pr:speed}(ii), so we discuss  the right hand side. If $q\in(0,1)$ we notice that for $g:[0,1]\to [0,+\infty]$ Borel we have $\varliminf_{h}\tfrac1h\int_t^{t+h}g\,\d\Leb^1\geq g(t)$ for a.e.\ $t$ (because $\varliminf_{h}\tfrac1h\int_t^{t+h}g\,\d\Leb^1\geq\varliminf_{h}\tfrac1h\int_t^{t+h}n \wedge  g\,\d\Leb^1\geq n\wedge g(t)$ for every $n$ and a.e.\ $t$), thus the claim follows by the non-negativity of $\tfrac1q|\dot\gamma_t|^q$. If instead $q<0$ the conclusion follows from our integrability assumption and Lebesgue's differentiation theorem.
\end{proof}

We are interested in understanding whether, given $(\mu_t)$, we can find $\ppi$ for which equality holds in \eqref{eq:intub}. To this aim, the following compactness criterion will be useful (compare with  the limit curve \Cref{prop-cad-comp}).

\begin{lemma}[A tightness criterion]\label{le:tightcrit} Let $(\M,\uptau,\ell)$ be a forward spacetime and $\mathcal G\subset \Prob(\CC([0,1];\M))$ a collection of measures such that:
\begin{itemize}
\item[i)]  For some Borel $\mathscr D\subset[0,1]$ of full measure with $0\in{\mathscr D}$ we have: for any $t\in\mathscr D$ there is a compact set $K_t$ such that  $(\e_t)_*\ppi$ is concentrated on $K_t$ for every $\ppi\in\mathcal G$ and $t\in\mathscr D$,
\item[ii)] There is $\mathcal K\subset \Prob(\M)$ tight so that for all $\ppi\in\mathcal G$ there is $\mu\in\mathcal K$ with $(\e_1)_*\ppi\preceq\mu$. 
\end{itemize}
Then $\mathcal G$ is tight.

If moreover the topology $\uptau$ on $\M$ is locally causally convex, then we can weaken $(i)$ above to
\begin{itemize}
\item[i')]  For some   $\tilde{\mathscr D}\subset[0,1]$ dense with $0\in\tilde{\mathscr D}$  the set $\{(\e_t)_*\ppi:\ppi\in\mathcal G\}\subset\Prob(\M)$ is tight for all $t\in\tilde{ \mathscr D}$.
\end{itemize}
\end{lemma}

\begin{proof} Let $\eps>0$ and find $K\subset\M$ compact such that   $\nu(K^c)<\eps$ for every $\nu\in\mathcal K$. Consider the collection $\hat K$ of curves $\gamma$ such that $\gamma_t\in K_t$ for every $t\in\mathscr D$ and $\gamma_1\leq x$ for some $x\in K$: by \Cref{prop-cad-comp}, $\hat K$ is  compact in $\CC([0,1];\M)$ and by our assumptions  we have $\ppi(\hat K)\geq1-\eps$, proving the claim in this case.

Under the assumption $(i')$ we instead argue as follows. It is not restrictive to assume that $\tilde{\mathscr D}$ is countable. Let $\eps>0$, find $\{\eps_t>0:t\in\tilde{\mathscr D}\}$ such that $\sum_{t\in\tilde{\mathscr D}}\eps_t<\eps$ and then --- by the tightness assumption --- $H_t\subset \M$ compact so that $(\e_t)_*\ppi(H_t)>1-\eps_t$ for every $\ppi\in\mathcal G$ and $t\in\tilde{\mathscr D}$. Also, let   $H_1\subset \M$ be compact so that $\nu(H_1)>1-\eps_1$ for every $\nu\in\mathcal K$.

For  $\ppi\in\mathcal G$ let $\nu\in\mathcal K$ be so that $(\e_1)_*\ppi\preceq\nu$ and $\pi\in\Pi_\leq((\e_1)_*\ppi,\nu)$: since $\pi$ is concentrated on $\{(x,y):x\leq y\}$   we have $(\e_1)_*\ppi(J^-(H_1)^c) \, \le\, \pi(\M\times  (H_1^c))=\nu(H_1^c)<\eps$.

Let $\hat H:=\{\gamma\in G:\gamma_t\in H_t\text{ for any $t\in\tilde{\mathscr D}$ and $\gamma_1\in J^-(H_1)$}\}$ and notice that by \Cref{prop-cad-comp}  we know that $\hat H$ is relatively compact. To conclude observe that $\hat  H^c\subset\cup_{t\in\tilde{\mathscr D}}\e_t^{-1}(H_t^c)\cup \e_1^{-1}(J^-(H_1)^c)$ and thus $\ppi(\hat H^c)\leq (\e_1)_*\ppi(J^-(H_1)^c)+\sum_{t\in \tilde{\mathscr D}}(\e_t)_*\ppi(H_t^c)\leq\eps+\sum_{t\in \tilde{\mathscr D}}\eps_t<2\eps$   for every $\ppi\in\mathcal G$.
\end{proof}

We are now ready to state and prove the main result of the section.  Notice that in order to gain tightness we shall need to assume something either at the level of the topology of  $\M$ (conditions $(A),(B)$) or at the level of the specific curve of measures considered (condition $(C)$).  In relation to $(B)$, it is worth recalling that several natural topologies induced by $\ell$ are locally causally convex (see \Cref{R:push-up2}).

\begin{theorem}[Lifting paths of measures to measures on paths]\label{Th:Lifting} Let $(\M,\uptau,\ell)$ be a forward metric spacetime  such that $\ell$ is upper semicontinuous and does not take the value $+\infty$. Let $(\mu_t)\in \CC([0,1];\Prob(\M))$ satisfy $\mu_t\in\pem$ for every $t\in[0,1]$ and assume at least one of the following holds:
\begin{itemize}
\item[(A)] $\M$ is globally hyperbolic, i.e. emeralds are compact;
\item[(B)] the topology on $\M$ is locally causally convex;
\item[(C)] for some   $D\subset[0,1]$ dense  with $0\in  D$   the set   $\cup_{t\in  D}\supp\mu_t$ is relatively compact.
\end{itemize}
Let   $0\neq q<1$. Then $(\mu_t)$ is induced  by a plan $\bdpi\in \mathscr{P}(\CC([0,1];\M))$ with the property 
\begin{equation}
\label{action identity}
{\KE}_\QQ(\mu_\cdot) = \frac{1}{\QQ}\iint_0^1 \vert\dot\gamma_t\vert^\QQ\d t\d\bdpi(\gamma).
\end{equation}
Moreover, if ${\KE}_\QQ(\mu_\cdot)>-\infty$, then any such plan also satisfies  the a.e.\ pointwise identity
    \[
        \big\vert\dot\mu_t\big\vert_\QQ^\QQ = \int\vert\dot\gamma_t\vert^\QQ\d\bdpi(\gamma)\qquad a.e.\ t.
\]
\end{theorem}

\begin{proof}\ \\
\noindent{\sc Step 0: preliminary considerations.} From the bound \eqref{eq:actioUpperBound} and \Cref{L:usc ell admits emerald bounds} we have ${\KE}_\QQ(\mu_\cdot) \leq \tfrac1q\ell^q_q(\mu_0,\mu_1)<\infty$, thus if ${\KE}_\QQ(\mu_\cdot)>-\infty$ the function $t\mapsto         \big\vert\dot\mu_t\big\vert_\QQ^\QQ$ is in $L^1([0,1])$. It follows  that the second statement is a consequence of the   first one together with  \Cref{Le:speed comparison}: integrating \eqref{eq:pointspbound} yields \eqref{eq:intub} whence comparison with \eqref{action identity}  and the integrability just noted force equality in \eqref{eq:pointspbound} for a.e.\ $t$.
Here for $q<0$ we have used the identity \eqref{action identity} and the assumption ${\KE}_\QQ(\mu)>-\infty$ to ensure that $\smash{\int \vert\dot\gamma_\cdot\vert^\QQ\d\bdpi(\gamma) \in L^1([0,1];\Leb^1)}$ so that \eqref{eq:pointspbound} holds. 

\noindent{\sc Step 1: construction of the plan.} The idea, originating in an analogous standard construction in  different settings \cite{AGS:08, Lisini:2007, Villani:2009}, is to build $\bdpi$ as the narrow limit of piecewise constant plans which ``interpolate'' $(\mu_t)$ at intermediate points. 

Let $n\mapsto  P_n:=\{0=t^n_0<\ldots<t^n_n=1\}$ be a sequence  of partitions of $[0,1]$ containing $0,1$, so that $P_n\subset P_{n+1}$ for every $n\in\N$ and with $|P_n|\to 0$ (notation from \Cref{Pr:speed}): under the assumption $(C)$ we pick these so that $P_n\subset  D\cup\{1\}$ for any  $n$, while under assumptions $(A),(B)$ we pick these arbitrarily.

Use the upper semicontinuity of $\ell$ and \Cref{le:USCRE} to find an $\ell_q$-optimal plan $\pi^n_i\in\Pi_\leq(\mu_{ t^n_i},\mu_{ t^n_{i+1}} )$. Recursively glue these plans along their common marginals to produce $\pi^n\in\Prob(\M^{n+1})$, then consider the map $F^n:\M^{n+1}\to\CC([0,1];\M)$ sending $(x_0,\ldots,x_{n+1})$ to the curve that is equal to  $x_i$ on $(\tfrac{t^n_{i-1}+t^n_{i}}2,\tfrac{t^n_i+t^n_{i+1}}2]$, where $t^n_{-1}:=0$ and $t^n_{n+1}:=1$.  Notice that $F^n$ is continuous and  set $\ppi^n:=F^n_*\pi^n$.

We claim that the sequence $(\ppi^n)$ is tight and to this aim we shall make use of one of assumptions $(A),(B),(C)$.

\underline{{\sc (A)}} Let $\eps>0$ and find a compact set $K\subset\M$ with $\mu_0(K^c)+\mu_1(K^c)<\eps$. Then $E:=J(K,K)$ is compact by assumption and thus by \Cref{prop-cad-comp} (with $\mathscr D=[0,1]$ and $K_t=E$ for every $t\in[0,1]$) the set $\hat E:=\{\gamma\in\CC([0,1];\M):\gamma_t\in E,\ \forall t\in[0,1]\}=\e_0^{-1}(J^+(K))\cap \e_1^{-1}(J^-(K))$ is compact in $\CC([0,1];\M)$.  Since for every $n\in\N$ we have $(\e_0)_*\ppi^n=\mu_0$ and $(\e_1)_*\ppi^n=\mu_1$, we deduce
\[
\ppi^n(\hat E^c)\leq \ppi^n(\e_0^{-1}(K^c))+\ppi^n(\e_1^{-1}(K^c))=\mu_0(K^c)+\mu_1(K^c)<\eps\qquad\forall n\in\N
\] 
proving   the desired tightness.

\underline{{\sc (B)}}  We shall apply  \Cref{le:tightcrit}. We  have  $(\e_1)_*\ppi^n=\mu_1$ for every $n\in\N$, so assumption $(ii)$ of  \Cref{le:tightcrit} holds.  Moreover,  $\tilde{\mathscr D}:=\cup_nP_n$ is dense in $[0,1]$ and  for every $t\in\tilde{\mathscr D}$ the measure $(\e_t)_*\ppi^n$ is eventually equal to $\mu_t$ (here we used that $P_n\subset P_{n+1}$), thus assumption $(i')$ of \Cref{le:tightcrit} holds. Hence  \Cref{le:tightcrit} gives the desired tightness. 

\underline{{\sc (C)}} We shall apply again \Cref{le:tightcrit} and as before we notice that $(\e_1)_*\ppi^n=\mu_1$ for every $n\in\N$, so assumption $(ii)$ holds. Also, for every $t\in[0,1)$ and $n\in\N$ we have $(\e_t)_*\ppi^n\in\{\mu_t:t\in \cup_nP_n\}$, thus if we  call $K\subset\M$ a compact set containing $\supp\mu_t$ for any $t\in  \mathscr D\supset\cup_nP_n$, then assumption $(i)$ of \Cref{le:tightcrit} is satisfied with $\mathscr D=[0,1)$ and   $K_t:=K$ for any $t\in\mathscr D$. Thus, as before, \Cref{le:tightcrit} gives the desired tightness. 

Having proved tightness, after passing   to a non-relabelled subsequence we can therefore assume that  $(\ppi^n)$ narrowly converges to some $\ppi\in\Prob(\CC([0,1];\M))$.
 
We claim that $(\e_t)_*\ppi=\mu_t$ for every $t\in[0,1]$:  since $t\mapsto(\e_t)_*\ppi$ is narrowly left continuous (\Cref{Le:speed comparison}), by \Cref{prop:distord} applied with $\Prob(\M)$ in place of $\M$  it suffices to prove   $\mu_s\preceq(\e_t)_*\ppi\preceq\mu_t$ for all dyadics $s<t$. Thus fix such $s,t$ and notice that for any $\varphi\in C_b(\M)$ the map $\gamma\mapsto\tfrac1{t-s}\int_s^t\varphi(\gamma_r)\,\d r$ is continuous from $\CC([0,1];\M)$ to $\R$ (by the choice of $\sfD$-topology on  $\CC([0,1];\M)$ and an application of the dominated convergence theorem), thus 
\[
\tfrac1{t-s}\iint_s^t\varphi(\gamma_r)\,\d r\,\d\ppi^n(\gamma)\quad\to\quad\tfrac1{t-s}\iint_s^t\varphi(\gamma_r)\,\d r\,\d\ppi(\gamma)\qquad\text{as $n\to\infty$}.
\]
In other words, the measures $\mu^n_{s,t}:=\tfrac1{t-s}\int_s^t(\e_r)_*\ppi^n\,\d r$ (the integral  being interpreted in the weak sense) narrowly converge to $\mu_{s,t}:=\tfrac1{t-s}\int_s^t(\e_r)_*\ppi\,\d r$ and since $\mu_s\preceq \mu^n_{s,t}\preceq \mu_t$ for every $n$, by the closure of $\{\ell\geq 0\}$ we deduce that  $\mu_s\preceq \mu_{s,t}\preceq \mu_t$. Finally, to establish the only remaining claim $\mu_s\preceq(\e_t)_*\ppi\preceq\mu_t$ of this paragraph, we use left narrow continuity and $\preceq$-monotonicity of $ r\mapsto(\e_r)_*\ppi $ 
to deduce that $\mu_{r,t}$ converges narrowly and $\preceq$-monotonically
to $(\eval_t)_*\bdpi$ as $r\uparrow t$, and recall from \Cref{prop:ellqst} that the causal relation $\preceq$ on $\Prob(M)$ is narrowly closed.

\noindent{\sc Step 2: proof of the action identity.} By  \Cref{Le:speed comparison} it suffices to prove the reverse  inequality $\le$ to \eqref{eq:intub} also holds. To this aim we start imitating the notation of \Cref{Pr:speed} and define, for $\gamma\in \CC([0,1];\M)$ and a partition $P=\{0=t_0<\ldots<t_n=1\}$ of $[0,1]$, the quantity
\[
A_q(\gamma,P):=\sum_{i=0}^{n-1} (t_{i+1}-t_i)\,u_q\Big(\frac{\ell(\gamma_{t_i},\gamma_{t_{i+1}})}{t_{i+1}-t_i}\Big).
\]
Then notice  that the construction and  the properties \eqref{eq:monpart} and \eqref{eq:limpart} give
\[
{\KE}_\QQ(\mu_\cdot) = \lim_{n\to +\infty} \int A_q(\gamma,P_n)\,\d\ppi^n(\gamma)\leq\liminf_{n\to\infty}  \int u_q(\ell(\gamma_0,\gamma_1))\,\d\ppi^n(\gamma).
\]
Now fix ${\sf t}_0,{\sf t}_1\in(0,1)$ belonging to some of the $P_n$'s --- and thus eventually to all of them --- with ${\sf t}_0<{\sf t}_1$ and notice that applying the above to the curve $t\mapsto\mu_{(1-t){\sf t}_0+t{\sf t}_1}$  gives
\[
\begin{split}
\int_{{\sf t}_0}^{{\sf t}_1}u_q(|\dot\mu_t|)\,\d t&\leq \liminf_{n\to\infty }\int u_q(\ell(\gamma_{{\sf t}_0},\gamma_{{\sf t}_1}))\,\d\ppi^n(\gamma)\leq\limsup_{n\to\infty}\int\dashint_0^{{\sf t}_0}\!\!\!\dashint_{{\sf t}_1}^1u_q(\ell(\gamma_t,\gamma_s))\,\d s\,\d t\,\d\ppi^n(\gamma)
\end{split}
\]
having used also the monotonicity of $u_q$ and the causality of the $\gamma$'s in the second inequality. Arguing as above, starting from the observation that for $\varphi\in C_b(\M)$ the map $\CC([0,1];\M)\ni \gamma\mapsto \dashint_0^{{\sf t}_0}\!\!\dashint_{{\sf t}_1}^1\varphi(\gamma_t,\gamma_s)\,\d s\,\d t\in\R$ is continuous we see that the sequence $\dashint_0^{{\sf t}_0}\!\!\dashint_{{\sf t}_1}^1(\e_t,\e_s)_*\ppi^n \,\d s\,\d t$ narrowly converges to $\dashint_0^{{\sf t}_0}\!\!\dashint_{{\sf t}_1}^1(\e_t,\e_s)_*\ppi\,\d s\,\d t$. Since $\ell$ is upper semicontinuous  and  bounded from above on $E$ (\Cref{L:usc ell admits emerald bounds}), we can pass to the limit in the above and get
\[
\int_{{\sf t}_0}^{{\sf t}_1}u_q(|\dot\mu_t|)\,\d t\leq \int\dashint_0^{{\sf t}_0}\!\!\!\dashint_{{\sf t}_1}^1u_q(\ell(\gamma_t,\gamma_s))\,\d s\,\d t\,\d\ppi(\gamma)\leq \int u_q(\ell(\gamma_0,\gamma_1))\,\d\ppi(\gamma),
\]
having used again the monotonicity of $u_q$ and causality of the $\gamma$'s. Letting ${\sf t}_0\downarrow0$ and  ${\sf t}_1\uparrow1$ we obtain $\int_0^1 u_q(|\dot\mu_t|)\,\d t\leq\int u_q(\ell(\gamma_0,\gamma_1))\,\d\ppi(\gamma)$. A simple scaling argument now readily implies
\[
\int_0^1u_q(|\dot\mu_t|)\,\d t\leq \int A_q(\gamma,P)\,\d\ppi(\gamma)\qquad\text{ for every partition $P$ of $[0,1]$.}
\]
To conclude it therefore suffices to show that
\[
\iint_0^1u_q(|\dot\gamma_t|)\,\d t\,\d\ppi(\gamma)=\lim_{k\to\infty }\int A_q(\gamma,P_k),\d\ppi(\gamma),\qquad\text{for given partitions $(P_k)$  with $|P_k|\to 0$}.
\]
To see this, recall that $A_q(\gamma,P)\leq u_q(\ell(\gamma_0,\gamma_1))\leq u_q(\sup_{E^2}\ell)<\infty$ for any $\gamma\in\supp\ppi$, then use \eqref{eq:limpart} in conjunction with the monotone convergence theorem. 
\end{proof}

It is worth encoding the property established by the previous result in a definition:
\begin{definition}[Lifting causal curves of measures]
    Let $(\M,\uptau,\ell)$ be a Polish metric spacetime,  $\bdpi \in \mathscr{P}(\CC([0,1];\M))$ and $(\mu_t)\in \CC([0,1];\mathscr{P}(\M))$. We say that $\ppi$ is a lifting of   $(\mu_t)$ provided $\mu_t = (\eval_t)_*\bdpi$ for each $t\in[0,1]$ and \eqref{action identity} holds.
\end{definition}

To clarify the upcoming discussion, we recall that according to  \Cref{Def:l-geodesics},  a rough timelike $\ell_q$-geodesic  is a curve $(\mu_t)\subset\Prob(\M)$ such that 
\[
 0<\ell_q(\mu_t,\mu_s) = (s-t)\,\ell(\mu_0,\mu_1) <\infty\qquad\forall t,s\in[0,1],\ t<s.
\]
This, however, does not imply that  `the mass moves all along timelike geodesics'; more precisely, if $\ppi$ is a lifting of $(\mu_t)$ we do not necessarily have that $\ppi$ is concentrated on timelike geodesics. We thus give the following:

\begin{definition}[Strongly timelike $\ell_q$-geodesics]
Let $(\M,\uptau,\ell)$ be a Polish metric spacetime, $0\neq q<1$ and $(\mu_t)\subset\Prob(\M)$ a rough timelike $\ell_q$-geodesic. We say that $(\mu_t)$ is a rough strongly timelike $\ell_q$-geodesic provided furthermore for any $t,s\in[0,1]$ with $t<s$ that any $\ell_q$-optimal coupling of $(\mu_t,\mu_s)$ is concentrated on $\{\ell\in(0,+\infty)\}$.
\end{definition}

We then have the following:
\begin{corollary}[Lifting $\smash{\ell_\QQ}$-geodesics]\label{Cor:Lifting geos} Let $(\M,\uptau,\ell)$ be a forward metric spacetime in which $\ell$ is upper semicontinuous and does not take the value $+\infty$. Let  $0\neq q<1$ and $(\mu_t)\in \CC([0,1];\Prob(\M))$ be a causal
$\ell_\QQ$-geodesic from $\nu_0$ to $\nu_1$, with $\nu_0,\nu_1\in\pem$. Assume also that $\ell_q(\nu_0,\nu_1)\in(0,+\infty)$ and that either of $(A),(B),(C)$ of \Cref{Th:Lifting} hold.

Then $(\mu_t)$ is induced by a plan $\bdpi\in \Prob(\CC([0,1];\M))$ which is concentrated on causal geodesics and so that  $(\eval_0,\eval_1)_*\ppi$ is $\ell_\QQ$-optimal. Moreover we have
\begin{equation}
\label{eq:conclgeo}
u_q\big(\ell_q(\nu_0,\nu_1)) = \int u_q(\ell(\gamma_0,\gamma_1))\, \d\bdpi(\gamma).
\end{equation}
If $(\mu_t)$ is a strongly timelike $\ell_q$-geodesic, then every $\ppi$ as above is concentrated on timelike geodesics.

Conversely, let $\ppi\in\Prob(\CC([0,1];\M))$ and $\nu_0,\nu_1\in\pem$ be such that    $\nu_0\preceq(\e_0)_*\ppi$ and $(\e_1)_*\ppi\preceq\nu_1$. Let $\mu_t:=(\e_t)_*\ppi$ for every $t\in[0,1]$. Then $(\mu_t)$ is  an $\ell_q$-geodesic from $\nu_0$ to $\nu_1$ if and only if
\begin{equation}
\label{eq:convgeo}
\frac1q\iint_0^1|\dot\gamma_t|^q\,\d t\,\d\ppi(\gamma)\geq u_q\big(\ell_q(\nu_0,\nu_1)\big).
\end{equation}
\end{corollary}

\begin{proof} Let $\ppi$ be built from $(\mu_t)$ via \Cref{Th:Lifting}. Then
\[
\int u_q(\ell(\gamma_0,\gamma_1))\,\d\ppi(\gamma)\leq u_q\big(\ell_q(\mu_0,\mu_1)\big)\leq u_q\big(\ell_q(\nu_0,\nu_1)\big)
\]
and since $(\mu_t)$ is a $\ell_q$-geodesic from $\nu_0$ to $\nu_1$ we also have
\[
u_q\big(\ell_q(\nu_0,\nu_1)\big)=\int_0^1u_q\big(|\dot\mu_t|_q\big)\stackrel{\eqref{action identity}}=\iint_0^1 u_q(|\dot\gamma_t|)\,\d t\,\d\ppi(\gamma)\stackrel{\eqref{eq:uqinf}}\leq\int u_q(\ell(\gamma_0,\gamma_1))\,\d\ppi(\gamma).
\]
The assumption $\ell_q(\nu_0,\nu_1)\in(0,+\infty)$ implies that $u_q\big(\ell_q(\nu_0,\nu_1)\big)\in\R$ and thus the equalities in the above force the optimality of $(\e_0,\e_1)_*\ppi$ and that  $\int_0^1u_q(|\dot\gamma_t|)\,\d t=\ell(\gamma_0,\gamma_1)$ for $\ppi$-a.e.\ $\gamma$, that by \Cref{Pr:l-geodesics maximize} means that $\gamma$ is a geodesic. The claim about strong timelike $\ell_q$-geodesics is now obvious by definition and the optimality of $(\e_0,\e_1)_*\ppi$.

All this also establishes the `only if' in the  last claim. For the `if' we put $\mu_t:=(\e_t)_*\ppi$ for every $t\in[0,1]$ and notice that 
\[
\begin{split}
u_q\big(\ell_q(\nu_0,\nu_1)\big)\stackrel{\eqref{eq:convgeo}}\leq\int \mathscr A_q(\gamma)\,\d\ppi(\gamma)\stackrel{\eqref{eq:intub}}\leq\mathscr A_q(\mu_\cdot)\stackrel{\eqref{eq:actioUpperBound}}\leq u_q\big(\ell_q(\mu_0,\mu_1)\big)\leq u_q\big(\ell_q(\nu_0,\nu_1)\big),
\end{split}
\]
having used the causality relation in the last step. The claim follows.
\end{proof}
As a consequence of this discussion we can see the relation between   $\M$ being geodesic and $\Prob(\M)$ being so. Here and below we say that a metric spacetime is geodesic if for any two points $x,y$ with $x\leq y$ a  causal geodesic from $x$ to $y$ exists.


\begin{proposition}[Heredity of geodesy on metric spacetimes]
\label{P:heredity of geodesy}
Let $(\M,\uptau,\ell)$ be a Polish metric spacetime  in which $\ell$ is upper semicontinuous and $\ell_+$ is continuous and real valued. 

If $(\M,\uptau,\ell)$ is geodesic then $(\pem,\ell_q)$ is geodesic for every $0\neq q<1$. 

Conversely, assume that  $(\pem,\ell_q)$ is geodesic for some  $0\neq q<1$ and that the topology of $\M$ satisfies $(A)$ or $(B)$ of \Cref{Th:Lifting}. Then 
every $x,y\in \M$ with $\ell(x,y) \in [0,+\infty)$ are joined by a 
geodesic.
\end{proposition}

\begin{proof} Assume that $\M$ is geodesic. Notice that from \Cref{C:CGeo is closed} we easily get that  the multivalued map taking couples $(x,y)\in \{\ell\geq 0\} $  returning the set of causal geodesics from $x$ to $y$  has closed graph. Hence by standard measurable selection arguments (see e.g.\  \cite[Thm.~6.9.2]{Bog:07b})  there is a universally  measurable map ${\sf GeodSel}:\{\ell\geq 0\}\to \CGeo(\M)$  such that  ${\sf GeodSel}(x,y)$ is one such geodesic for every $(x,y)\in \{\ell\geq 0\}$. Now let $\mu,\nu\in\pem $  be with $\mu\preceq\nu$ and let $\pi$ be a $\ell_q$-optimal plan (this exists by \Cref{le:USCRE}). A direct computation shows that putting $\ppi:={\sf GeodSel}_*\pi$, the curve $t\mapsto(\e_t)_*\ppi$ is a $\ell_q$-geodesic from $\mu$ to $\nu$.

Conversely, assume that  $(\pem,\ell_q)$ is geodesic and let $x,y\in \M$ with 
$x \le y$.  If $\ell(x,y)=0$ then the lightlike curve 
$$\gamma_t := \begin{cases}x & {\rm if}\ t=1/2\\ y & {\rm if}\ t \in (1/2,1]\end{cases}$$
can be verified to be geodesic (using conventions
\eqref{eq:up} and \eqref{eq:defellq1} if $\QQ<0$).
Therefore, assume
$\ell(x,y)\in(0,\infty)$ and let $(\mu_t)$ be an $\ell_q$-geodesic from $\delta_x$ to $\delta_y$.  Since $\ell_q(\delta_x,\delta_y)=\ell(x,y)\in(0,+\infty)$, by \Cref{Cor:Lifting geos} above  $(\mu_t)$ admits a lifting $\ppi$ that is concentrated on $\CGeo(\M)$. We claim that $\ppi$-a.e.\ $\gamma$ is a geodesic from $x$ to $y$ (and thus in particular at least a geodesic exists).  To see this, notice that since $(\mu_t)$ is a  $\ell_q$-geodesic from $\delta_x$ to $\delta_y$,  we have $\delta_x\preceq \mu_0=(\e_0)_*\ppi$ and $(\e_1)_*\ppi=\mu_1\preceq\delta_y$, so that $\ppi$-a.e.\ $\gamma$ satisfies $x\leq \gamma_0$ and $\gamma_1\leq y$ and therefore $u_q(\ell(\gamma_0,\gamma_1))\leq u_q(\ell(x,y))$.  Then \eqref{eq:conclgeo} and the fact that $u_q(\ell(x,y))\in\R$ forces $u_q(\ell(\gamma_0,\gamma_1))= u_q(\ell(x,y))$ for $\ppi$-a.e.\ $\gamma$, which is the conclusion.
\end{proof}

\subsection{Non-branching notions}\label{Sub:qessreg}
In this section we shall start working with Polish metric spacetimes equipped with a reference measure. Let us isolate the relevant definition:
\begin{definition}[Metric measure spacetime]
A metric measure spacetime, or shortly a mm spacetime, is a quadruple $(\M,\uptau,\ell,\mm)$ such that $(\M,\uptau,\ell)$ is a Polish metric spacetime and $\mm$ is a non-negative and non-zero Radon measure on $\M$.

A forward (resp.\ backward) metric measure spacetime, or shortly a forward mm spacetime, is a metric measure spacetime so that  $(\M,\uptau,\ell)$ is also forward (resp.\ backward).
\end{definition}

In the classical theory of optimal transport it is well known that a non-branching condition implies strong regularity properties of geodesics of measures (see \cite{Villani:2009} for an overview on the topic). In the context of metric geometry, it has been understood in \cite{Gigli12a} that a lower Ricci bound and a non-branching assumption yield existence of optimal maps (see also \cite{giglirajalasturm} and \cite{cavalletti-mondino2017} for further results).

In non-smooth Lorentzian signature a similar picture is --- starting from \cite{CM:20} --- emerging, see also \cite{Braun:2023Renyi} for subsequent developments.  In this section we review some of the existing notions in our framework and propose the concept of `non-branching at 0/1'.

Notice that in our setting  one should exclude geodesics with either 0 or infinite length, because these are invariant by reparametrization and thus always branch. In other words, a non-branching condition only makes sense if required for timelike geodesics. Also, as in \cite{CM:20}, one might distinguish between `forward' and `backward' non-branching, as discussed in the next definition,  coming from   \cite[Def.~1.10]{CM:20}:
 \begin{definition}[Timelike non-branching conditions] Let $(\M,\uptau,\ell)$ be a Polish metric spacetime and $G$ a set of  timelike geodesics.  We say that $G$ is forward (resp.\ backward) timelike  non-branching if  any two $\gamma,\eta\in G$  that agree on $[0,t]$ (resp.\ $[t,1]$) for some $t\in(0,1)$ agree on the whole $[0,1]$.
 
 We then say that $(\M,\uptau,\ell)$ is forward (resp.\ backward) timelike  non-branching if   $\TGeo(\M)$ is so.
  \end{definition}
This concept applies equally well to the spacetime $(\Prob(\M),\ell_q)$, but in this case it is better to restrict the attention to strongly timelike $\ell_q$-geodesics. The problem is that a lifting of general timelike $\ell_q$-geodesics may give positive mass to non-timelike geodesics and reparametrizing these curves we easily see that in general circumstances $(\Prob(\M),\ell_q)$ is never non-branching.
  
 With this said, the timelike non-branching  of $\M$ is related to that of $\Prob(\M)$:
 
 \begin{proposition}[Heredity of timelike non-branchingness]
 \label{prop:MnonbrPMnonbr} Let $(\M,\uptau,\ell)$ be a forward spacetime. Let $0\neq q<1$ and assume that the collection of strongly timelike $\ell_q$-geodesics in $\pem$  is forward (resp.\ backward) timelike  non-branching. Then $(\M,\uptau,\ell)$ is forward (resp.\ backward) timelike  non-branching.
 
Conversely, assume that  $(\M,\uptau,\ell)$ is forward (resp.\ backward) timelike  non-branching,  that $\ell$ is upper semicontinuous and does not take the value $+\infty$ and that either $(A)$ or $(B)$ of \Cref{Th:Lifting} hold. Then for every $0\neq q<1$  the collection of strongly timelike $\ell_q$-geodesics in $\pem$  is forward (resp.\ backward) timelike  non-branching.
 \end{proposition}
 
 \begin{proof} We deal with the `forward' case, the `backward' one being analogous.  Let $0\neq q<1$ and suppose that   the collection of strongly timelike $\ell_q$-geodesics in $\pem$  is forward   timelike  non-branching. Notice that  for any $x,y\in \M$ we have $\ell(x,y)=\ell_q(\delta_x,\delta_y)$, hence   $t\mapsto\gamma_t\in (\M,\ell)$  is a  timelike geodesic  if and only if $t\mapsto \delta_{\gamma_t}\in(\Prob(\M),\ell_q)$ is strongly timelike. The conclusion follows.

Conversely, assume that $(\M,\uptau,\ell)$ is forward timelike non-branching  and let $(\mu^1_t),(\mu^2_t)\subset\pem$  be two  strongly timelike  $\ell_q$-geodesics that agree on $[0,T]$ for some given $T\in(0,1)$.  Let $\ppi^1,\ppi^2\in\Prob(\TGeo(\M))$ be liftings of $(\mu^1_t),(\mu^2_t)$ respectively as in \Cref{Cor:Lifting geos} (that we can apply thanks to the assumptions on $\M$). Produce a new plan $\ppi\in\Prob(\TGeo(\M))$ by gluing $(\restr_0^T)_*\ppi^1$ with $(\restr_T^1)_*\ppi^2$ along their common marginal $\mu^1_T$ --- obtaining a measure on the 
space of paths on $[0,2]$ --- and then reparametrize the curves in the obvious (piecewise affine) way to obtain a measure on curves on $[0,1]$. It is clear that $\ppi$ is still a lifting of $(\mu^2_t)$, so by  \Cref{Cor:Lifting geos} above is concentrated on $\TGeo(\M)$. Also, the construction ensures that $(\restr_0^T)_*\ppi=(\restr_0^T)_*\ppi^1$ and since our non-branching assumption can equivalently be stated by saying that $\restr_0^T:\TGeo(\M)\to\TGeo(\M)$ is injective, this proves that $\ppi^1=\ppi$, and thus that $\mu^1_t=(\e_t)_*\ppi^1=(\e_t)_*\ppi=\mu^2_t$ for every $t\in[0,1]$, as desired.
 \end{proof}

 When dealing with \emph{measured} spacetimes, it is natural to tailor the above concept in order to allow a `negligible set of exceptions', so to say.  

We shall use the following measure-theoretic variant of the above non-branching condition, akin to that  introduced by Braun \cite{Braun:2023Renyi} for $\QQ\in (0,1)$ (inspired by Rajala--Sturm \cite{RajalaSturm:2014} in positive signature). Here and below a family of measures in $\pem$ is said to be of bounded compression if for some $C>0$ each of the measures is $\leq C\mm$.
\begin{definition}[Timelike $q$-essential non-branching]\label{Def:TL ess nb} Let  $(\M,\uptau,\ell,\meas)$  be a forward mm  spacetime and $0\neq q<1$. We say that $\M$ is \emph{forward (resp.\ backward) timelike $\QQ$-essentially non-branching}, provided for any strongly timelike $\ell_q$-geodesic $(\mu_t)\subset\pem$ with bounded compression and for any lifting $\ppi\in\Prob(\CC([0,1];\M))$ we have that   $\ppi$  is concentrated on a forward (resp.\ backward)  timelike non-branching subset of $\TGeo(\M)$.
\end{definition}

 The analogue of \Cref{prop:MnonbrPMnonbr} for this last definition is the following result:

\begin{proposition}[When bounded compression strongly timelike $\ell_q$-geodesics are non-branching] 
Let  $(\M,\ell,\meas)$  be a forward mm spacetime and $0\neq q<1$. 

Assume that the collection of bounded compression strongly timelike $\ell_q$-geodesics in $\pem$ is forward (resp.\ backward) non-branching.  Then $(\M,\uptau,\ell,\mm)$ is forward (resp.\ backward) timelike $q$-essentially non-branching.

Conversely, assume that $(\M,\uptau,\ell,\mm)$ is forward (resp.\ backward)  timelike $q$-essentially non-branching,  that $\ell$ is upper semicontinuous and does not take the value $+\infty$ and that and that either $(A)$ or $(B)$ of \Cref{Th:Lifting} hold. Then  the collection of bounded compression strongly timelike $\ell_q$-geodesics in $\pem$ is forward (resp.\ backward) non-branching.
\end{proposition}

\begin{proof} Suppose that $\M$ is not forward  $q$-essentially non-branching. Then there is a bounded compression strongly timelike $\ell_q$-geodesic $(\mu_t)$   and a lifting $\ppi\in\Prob(\TGeo(\M))$ of it that is not concentrated on a timelike non-branching subset of $\TGeo(\M)$. With a bit of work  (see e.g.\ \cite{Gigli12a} for analogous arguments in positive signature) it is not hard to see that there are $t\in(0,1)$ and two Borel sets $\Gamma_1,\Gamma_2\subset\TGeo(\M)$  so that $(\restr_0^t)_*(\ppi\mres\Gamma_1)=(\restr_0^t)_*(\ppi\mres\Gamma_2)$ and $(\e_s)_*(\ppi\mres\Gamma_1)\neq (\ppi\mres\Gamma_2)$ for some $s>t$ (in particular $\ppi(\Gamma_1)=\ppi(\Gamma_2)>0$). Then define the plans $\ppi^i:=\ppi(\Gamma_i)^{-1}\ppi\mres\Gamma_i$ and the measures $\mu^i_t:=(\e_t)_*\ppi^i$ for $i=1,2$, $t\in[0,1]$. The construction ensures that $(\mu^1_t),(\mu^2_t)$ are strongly timelike $\ell_q$-geodesics with bounded compression, that they agree on $[0,t]$ and with $\mu^1_s\neq\mu^2_s$, as desired.

The converse implication follows along the very same lines used to prove \Cref{prop:MnonbrPMnonbr}.
\end{proof}

We turn to a sort of `infinitesimal variant' of the timelike non-branching condition. Notice that this concept is trivial in positive signature, due to the continuity of geodesics:
\begin{definition}[Non-branching at 0 or 1]\label{def:regularity}
Let $(\M,\ell)$ be a metric spacetime and $G$ a collection of rough timelike geodesics. We say that $G$ is timelike non-branching at 0 (resp.\ at 1) if  any two geodesics in $G$  that agree on $(0,1)$ also agree at 0 (resp.\ at 1).

If this holds with $G$ being the collection of all rough timelike geodesics, then we say that $(\M,\ell)$ is   is timelike non-branching at 0 (resp.\ at 1).
\end{definition}

It is worth noticing that under suitable compactness assumptions, the uniqueness encoded in this condition precisely detects continuity of geodesics:

\begin{proposition}[Geodesics are not rough in a forward spacetime non-branching at 1]
\label{R:regularity} Let $(\M,\uptau,\ell)$ be  a  Polish metric spacetime   that is timelike non-branching at 0 (resp.\ 1). Also, let $\gamma:[0,1]\to \M$ be a rough timelike geodesic such that for any $t\in[0,1]$ and $s_n\downarrow t$ (resp.\ $s_n\uparrow t$) there is $n_k\uparrow\infty$ such that $k\mapsto \gamma_{s_{n_k}}$ admits a limit.

 Then $\gamma$ is right- (resp.\ left-) continuous.
\end{proposition}

\begin{proof}We  imitate  the proof of \cite[Lem.~5]{McC:23}. We prove right-continuity of the rough  timelike geodesic $\gamma$ at $t=0$, then right-continuity at other times follows by scaling, and the arguments for left continuity  are analogous. Call $y$ a limit of a sequence $k\mapsto\gamma_{s_{n_k}}$ for $s_{n_k}\downarrow 0$ as in the statement. Let $(\eta_t)$ be equal to $(\gamma_t)$ on $(0,1]$ and defined as $\eta_0:=y$ for $t=0$. We clearly have $|\dot\eta_t|=|\dot\gamma_t|$ for a.e.\ $t$ and
\[
A_1(\eta)\geq\lims_{s\downarrow0}\int_s^1|\dot\eta_t|\,\d t=\lims_{s\downarrow 0}\int_s^1|\dot\gamma_t|\,\d t=A_1(\gamma)\geq\ell(\gamma_0,\gamma_1), 
\]
having used that $\gamma$ is an $\ell$-geodesic. Since $y\geq \gamma_0$ (here we use that $\{\ell\geq0\}$ is closed), by item $(iii)$ in \Cref{Pr:l-geodesics maximize} this suffices to establish that $\eta$ is also a rough timelike $\ell$-geodesic, so that non-branchingness at $0$ forces $y=\eta_0=\gamma_0$. This suffices to conclude right-continuity of $\gamma$ at $0$.
\end{proof}

Clearly, a sufficient condition for left-limits to exist is the forward completeness that we typically assume. More relevant for us however, will be the notion of timelike non-branching at $0$, 
which is linked to continuity at $t=0$ of geodesics in $\Prob(\M)$.

The following is a variant of \Cref{prop:MnonbrPMnonbr}.
Recall that timelike (and strongly timelike) $\ell_q$-geodesics $(\mu_t)$ satisfy $\ell_q(\mu_0,\mu_1) \in (0,\infty)$ as part of their definition.

\begin{proposition}[Heredity of non-branchingness at $0$ and/or $1$]
Let $(\M,\uptau,\ell)$ be a forward metric  spacetime. Let $0\neq q<1$ and assume that the collection $\cone_q$ of strongly timelike $\ell_q$-geodesics in $\pem$   has the following property: if $(\mu_t) \in \cone_q$  and $\nu\in\pem$   satisfies $\nu\preceq\mu_t$ for {all} $t\in(0,1)$ and $\ell_q(\nu,\mu_1)=\ell_q(\mu_0,\mu_1)$, then $\nu=\mu_0$ (resp.\  $\mu_t\preceq\nu$ for every $t\in(0,1)$ and $\ell_q(\mu_0,\nu)=\ell_q(\mu_0,\mu_1)$, then $\nu=\mu_1$). Then $(\M,\uptau,\ell)$ is timelike non-branching at 0 (resp.\ 1).

 Conversely, assume that  $(\M,\uptau,\ell)$ is timelike non-branching at 0 (resp.\ 1),  that $\ell$ is upper semicontinuous and does not take the value $+\infty$ and that either $(A)$ or $(B)$ of \Cref{Th:Lifting} hold. Then for every $0\neq q<1$,  
 the collection 
 $\cone_q$ has the property described above (respectively).
\end{proposition}



\begin{proof} The proof is similar to that of \Cref{prop:MnonbrPMnonbr}. 
We deal with the case of  non-branching at 0, that of non-branching at 1 being analogous. 

Assume that the collection $\cone_q$ of strongly timelike $\ell_q$-geodesics in $\pem$   has the indicated property. Since $\ell(x,y)=\ell_q(\delta_x,\delta_y)$ for any $x,y\in \M$ and the inclusion $\M\ni x\mapsto\delta_x\in\pem$ preserves the time separations, we easily deduce that $\TGeo(\M)$ has the following property: for $\gamma\in\TGeo(\M)$ and $x\in\M$ with $x\leq\gamma_t$ for {all} $t\in(0,1)$ and $\ell(x,\gamma_1)=\ell(\gamma_0,\gamma_1)$ we must have $x=\gamma_0$. It is easy to see that this is equivalent to the claim of  $(\M,\uptau,\ell)$ being timelike non-branching at 0, so in this case the conclusion follows.

Conversely, assume that $(\M,\ell)$ is timelike non-branching at 0, let $(\mu_t)\subset\pem$ be  a strongly timelike $\ell_q$-geodesic and  $\nu\in\pem$   with $\nu\preceq\mu_t$ for any $t\in(0,1)$ and $\ell_q(\nu,\mu_1)=\ell_q(\mu_0,\mu_1)$. Let $\ppi\in\Prob(\TGeo(\M))$ be a lifting of $(\mu_t)$ as in \Cref{Cor:Lifting geos} and, for every $t\in(0,1)$, let $\pi_t\in\Pi_\leq(\nu,\mu_t)$ be $\ell_q$-optimal. Use a gluing argument to find $\alpha_t\in\Prob(\M\times\TGeo(\M))$ such that 
\[
\begin{split}
({\rm Pr}_2)_*\alpha_t=\ppi,\qquad\text{and}\qquad({\rm Pr}_1,\e_t)_*\alpha_t&=\pi_t.
\end{split}
\]
Notice that the same arguments used in the proof of \Cref{Th:Lifting} easily implies that for any  $t_n\downarrow0$ the sequence $(\alpha_{t_n})$ is tight, thus up to passing to a non-relabelled subsequence it admits a narrow limit $\alpha$. Since $\alpha_t$ is concentrated on couples $(x,\gamma)$ with $\gamma\in\TGeo(\M)$ and $x\leq\gamma_t$, the closure of $\{\ell\geq0\}$ easily implies that $\alpha$ is concentrated on couples  $(x,\gamma)$ with $\gamma\in\TGeo(\M)$ and $x\leq\gamma_t$ for any $t\in(0,1]$. The  construction also yields
\[
\begin{split}
({\rm Pr}_2)_*\alpha=\ppi,\qquad\text{and}\qquad({\rm Pr}_1)_*\alpha&=\nu,
\end{split}
\]
thus from the assumption that $(\M,\ell)$ is  timelike non-branching at 0, to conclude it suffices to prove that for $\alpha$-a.e.\ $(x,\gamma)$ we have $\ell(x,\gamma_1)=\ell(\gamma_0,\gamma_1)$ (as this forces $x=\gamma_0$ and thus $\nu=\mu_0$, as required). To this aim, notice that 
\[
\begin{split}
u_q\circ\ell_q({\nu},\mu_1)&\geq\int u_q\circ\ell(x,\gamma_1)\,\d\alpha(x,\gamma)\\
\text{(as $x\leq\gamma_t$ for $\alpha$-a.e.\ $(x,\gamma)$)}\qquad& \geq\int \sup_{t\in(0,1)} u_q\circ\ell(\gamma_t,\gamma_1)\,\d\alpha(x,\gamma)\\
\text{(by \eqref{eq:uqinf} with $P=\{0,1\}$)}\qquad&\geq \iint_0^1u_q(|\dot\gamma_t|)\,\d t\,\d\ppi(\gamma)=u_q\circ\ell_q(\mu_0,\mu_1),
\end{split}
\]
having used \Cref{Cor:Lifting geos} in the last step. Since by assumption we have $\ell_q(\nu,\mu_1)=\ell_q(\mu_0,\mu_1)\in(0,+\infty)$,   all the inequalities are in fact equalities and  all the values are real: this can only be possible if $\ell(x,\gamma_1)=\sup_{t\in(0,1)}\ell(\gamma_t,\gamma_1)=\ell(\gamma_0,\gamma_1)$, forcing $x=\gamma_0$, for $\alpha$-a.e.\ $(x,\gamma)$. This implies $\nu=\mu_0$, as  desired.
\end{proof}

We are interested in forcing curves of measures with bounded compression to be non-branching at $0$.
For this, it is sufficient to introduce the following weakening of \Cref{def:regularity} above (analogous to \Cref{Def:TL ess nb}):
\begin{definition}[$q$-essential timelike non-branching at 0 or 1]
\label{D:$p$-regularity}
Fix $0\neq q<1$. 

We say that a forward mm  spacetime $(\M,\uptau,\ell,\mm)$ is   $\QQ$-essentially timelike non-branching at 0 (resp.\ at 1) if  the following holds. For each bounded compression strongly timelike $\ell_q$-geodesic $(\mu_t)\subset\pem$ and $\nu\in\pem$   with bounded compression such that   $\nu\preceq\mu_t$ for all $t\in(0,1)$ and $\ell_q(\nu,\mu_1)=\ell_q(\mu_0,\mu_1)$ 
we must have   $\nu=\mu_0$ (resp.\  $\mu_t\preceq\nu$ for every $t\in(0,1)$ and $\ell_q(\mu_0,\nu)=\ell_q(\mu_0,\mu_1)$ 
we must have $\nu=\mu_1$). 
\end{definition}
\Cref{R:regularity} gives the following  result, which is our reason for introducing this last concept:

\begin{corollary}[Continuity for bounded compression $\ell_q$-geodesics on a fixed emerald]
\label{cor:contfromreg}
Let $(\M,\uptau,\ell,\mm)$ be a forward mm  spacetime with $\mm$ giving finite mass to each emerald, $0\neq q<1$ and $(\mu_t)$ a rough strongly timelike $\ell_q$-geodesic of bounded compression 
such that 
$\supp\mu_t\subset E$ for some fixed emerald $E$ and {all} $t\in[0,1]$.

If $\M$ is $\QQ$-essentially  timelike non-branching at 0 (resp.\ at 1) then $(\mu_t)$ is right (resp.\ left) continuous.
\end{corollary}
\begin{proof} We know that $\mu_t\leq C\mm\mres E$ for some $C>0$ and every $t\in[0,1]$, which readily implies that the collection $\{\mu_t:t\in[0,1]\}$ is tight, hence relatively narrowly compact. The conclusion follows from \Cref{R:regularity}.
\end{proof}

\section{Lorentzian differential and Sobolev calculus}
\label{Ch:differential}

\subsection{Some basic properties of causal functions}\label{ss:causal functions}

\begin{definition}[Causal function]\label{Def:TF} Let $(\M,\ell)$ be a metric spacetime. A function $f\colon \M\to \bar{\R}$ is called a \emph{causal function} if it is causally monotone, i.e.\ $x \leq y$ implies $f(x)\leq f(y)$.

The \emph{domain} $\dom(f)$ of a causal  function $\smash{f\colon \M\to \bar{\R}}$ is the set of all $x\in \M$ such that $f(x)$ is a real number.
\end{definition}
Notice that the domain of a causal function is always causally convex.

If  $f$ and $g$ are causal functions and  $\lambda_1,\lambda_2 \geq 0$, then $\lambda_1\, f + \lambda_2\,g$ is again a causal function. If $f$ and $g$ are non-negative, then $f\,g$ is a causal function. If $\varphi\colon \bar{\R} \to \bar{\R}$ is non-decreasing on the image of $f$, then $\varphi\circ f$ is again a causal function. Lastly, arbitrary (pointwise) infima and suprema of causal functions are again causal functions.

\medskip

The following lemma is an analogue for causal functions of the fact that a non-decreasing map of the line to itself can have only countably many discontinuities. This notably includes the case where $f$ attains the values $\pm\infty$.

Recall a subset of $\M$ is \emph{achronal} if none of its elements are related by $\ll$.

\begin{lemma}[Chronological almost continuity]\label{L:chronological continuity}
Let $(\M,\ell)$ be a metric spacetime and $f\colon \M \to \bar{\R}$  causal. Define $\smash{f^\pm\colon \M\to\bar{\R}}$ by
\begin{align}\label{f+}
\begin{split}
f^+(y) = \inf_{z \in I^+(y) } f(z),\qquad\text{and}\qquad f^-(y) = \sup_{x \in I^-(y)} f(x),
\end{split}
\end{align}
where we employ the usual conventions $\inf\emptyset := +\infty$ and $\sup\emptyset := -\infty$. Then:
\begin{enumerate}[label=\textnormal{(\roman*)}]
    \item The functions $f^+$ and $f^-$ are causal functions. 
    \item We have $f^- \le f \le f^+$ everywhere on $\M$.
    \item Outside a countable union of achronal sets, $f^+=f^-$. 
\end{enumerate}
Let  $\M$ be equipped with a topology containing the chronological one. Then we also have:
\begin{enumerate}[label=\textnormal{(\roman*)}]
\setcounter{enumi}{3}    
\item If a sequence $(y_n)\subset I^\pm(y)$ converges to $y$,  then $\smash{\lim_n f^\pm(y_n) = f^\pm(y)}$.
    \item\label{(v)} Both $f^+$ and $-f^-$ are upper semicontinuous.  Hence $f$ is continuous at every point of the complement of the countable union provided by  \textnormal{(iii)}.
\end{enumerate}
\end{lemma}

\begin{proof}\ \\
(i) For $x\leq y\ll z$ we have $x\ll z $ by the push-up property (\Cref{R:push-up1}), whence $f^+(x) \le f^+(y)$ as desired.  An analogous argument shows $f^-$ is a causal function. \\
(ii) This is immediate from the causal property of $f$ and the definitions \eqref{f+}. \\
(iii)  Given a number $q \in \Q$ we define
\begin{align*}
S_{q}^- &:= \{ y \in \M : f^-(y)< q< f(y) \}\qquad\text{and}\qquad S_{q}^+ := \{ y \in \M : f(y)< q< f^+(y) \}.
\end{align*}
We claim that the sets $\smash{S_{q}^\pm}$ are achronal.  Indeed,  if   $y_1,y_2\in S_{q}^+$ were points satisfying $y_1\ll y_2$, we would have $f^+(y_1)\leq f(y_2)$ by \eqref{f+}. This would produce the contradiction $q<f^+(y_1)\leq f(y_2)<q$. An analogous argument applies to $\smash{S_{q}^-}$.

Since $f^+(y) > f^-(y)$   implies $y \in S_{q}^+\cup S_{q}^-$ for some $q \in \Q$, (iii) is proven.\\
(iv) We only show the statement for the function $\smash{f^+}$, the argument for $f^-$ follows analogous lines. By the transitivity of $\ll$, $I^+(y_n)$ is contained in $I^+(y)$ for every $n\in\N$. Therefore $f^+(y) \leq f^+(y_n)$ by \eqref{f+}, and consequently
\begin{align*}
f^+(y) \leq \liminf_{n\to+\infty} f^+(y_n).
\end{align*}
On the other hand, let $z\in I^+(y)$ be arbitrary. (Our assumption on the given sequence implies such a point does exist.) Since $(y_n)$ converges to $y$ and the set $I^-(z)$ is open, we eventually have $y_n \in I^-(z)$. Choosing $z$ as a competitor in the definition \eqref{f+} of $\smash{f^+(y_n)}$ implies $\limsup_{n\to+\infty} f^+(y_n) \leq f(z)$
and therefore
\begin{equation*}
\limsup_{n\to+\infty} f^+(y_n) \leq f^+(y)
\end{equation*}
by the arbitrariness of $z$.\\
(v)  
Since the chronological topology makes $I^-(z)$ open for each $z\in \M$, the formula
$$
f^+(y) = \inf_{z \in \M} f(z) + \infty 1_{\M\setminus I^-(z)}(y)
$$
shows $f^+$ is an $\inf$ of upper semicontinuous functions,
hence itself  upper semicontinuous.  Similarly $f^-$ is lower semicontinuous and both $f^\pm$ are continuous at points where they agree, as is $f$. 
\end{proof}
As a direct consequence of the above we have:
\begin{corollary}
[Measurability of causal functions]\label{Cor:Meas tf} Let $(\M,\ell)$ be a metric spacetime equipped with a ($\sigma$-algebra and a) measure $\mm$ that   assigns zero outer measure to every achronal subset. 

Then for every   $f\colon \M \to \bar{\R}$   causal we have $f^+=f^-$ $\meas$-a.e., where   $\smash{f^\pm}$ are from \eqref{f+}. 

If, in addition, open sets in the chronological topology are $\mm$-measurable, then so is $f$.
\end{corollary}

\begin{proof}  The first claim follows from our hypotheses on $\meas$ and  \Cref{L:chronological continuity}.  For the second it suffices to notice that for any $c\in\R$ we have
\[
\{f<c\}= \Big(\bigcup_{x\in\{f<c\}} I^-(x)\Big)\ \cup\ \Big(\{f<c\}\setminus \bigcup_{x\in\{f<c\}} I^-(x)\Big),
\]
that the first set on the right is chronologically open and the second achronal.
\end{proof}

A reinforcement of \Cref{Def:TF}   is the subsequent analog of Lipschitz continuity from metric geometry. Recall our infinity conventions from the beginning of \Cref{Ch:Curves}. 

\begin{definition}[Steepness] Let $(\M,\ell)$ be a metric spacetime, $L\in[0,\infty]$ and   $f\colon \M\to \bar{\R}$. We say that $f$ is \emph{$L$-steep} on $W\subset \M$ if 
\begin{align*}
f(y) - f(x) \geq L\,\ell(x,y)\qquad\forall x,y\in W.
\end{align*}
If, in addition, $W=\M$, we simply say $f$ is $L$-steep.
\end{definition}
Notice that a function is $0$-steep if and only if it is a causal function. Typical examples of 1-steep functions are $\ell(x,\cdot)$ and $-\ell(\cdot,x)$ for given $x\in \M$: their 1-steepness is a direct consequence of the reverse triangle inequality \eqref{eq:revtr} (and the infinity conventions from the beginning of \Cref{Ch:Curves}).

A steep function can always be extended to all of $\M$ by the following analog of the well-known McShane extension lemma from metric geometry, e.g. \cite[Thm.~1.33]{Weaver:2018}.  Note that the conventions \eqref{eq:conventions1} about differences do not have a role in the definitions of $f_\vee$ and $f_\wedge$ below.
\begin{lemma}[McShane-type extension]\label{Le:McShane} Let $(\M,\ell)$ be a metric spacetime. Let $f\colon W \to \bar{\R}$ be a given $L$-steep function, where $W\subset \M$ is any subset and $L\geq 0$. Define the functions $f_\wedge, f^\vee\colon \M\to\bar{\R}$ by
\begin{align*}
f_\wedge(y) &:= \sup\!\big\lbrace f(x)+L\,\ell(x,y): x\in W \textnormal{ \textit{with} } x\leq y \textnormal{ \textit{and} }  f(x)>-\infty \big\rbrace,\\
f^\vee(x) &:= \inf\!\big\lbrace f(y) - L\,\ell(x,y): y\in W\textnormal{ \textit{with} }y\geq x\textnormal{ \textit{and} }f(y)<+\infty\big\rbrace.
\end{align*}
Then $f_\wedge$ and $f^\vee$ are $L$-steep functions which coincide with $f$ on $W$. Furthermore, these extensions are extremal, in the sense that every $L$-steep extension $g\colon \M\to \bar{\R}$ of $f$ satisfies $f_\wedge \leq g\leq f^\vee$ everywhere on $\M$.
\end{lemma}
\begin{proof} We show all the claims for $f_\wedge$, those for $f^\vee$ being similar.

Let us first show that $\smash{f_\wedge}$ coincides with $f$ on $W$.  Let $x\in W$. Choosing $x$ as a competitor in the supremum defining $\smash{f_\wedge}$ and using $\ell(x,x)=0$ yields $\smash{f_\wedge(x) \geq f(x)}$. On the other hand, by $L$-steepness of $f$ on $W$, $\smash{L \ell(x,y) + f(x) \leq f(y)}$ for any $x,y \in W$ with $f(x)>-\infty$, taking the supremum over $x$ implies $\smash{f_\wedge(y) \leq f(y)}$, which is the claim.

Next, we show $L$-steepness of $\smash{f_\wedge}$ on $M$.  Let $y_1,y_2 \in \M$. We claim
\begin{equation}
\label{eq:permcshane}
f_\wedge(y_2)-f_\wedge(y_1)\geq L\,\ell(y_1,y_2).
\end{equation}
We may assume that $y_1 \leq y_2$, otherwise the claim is trivial. Let  $x\in W$ be a competitor in the definition of $f_\wedge(y_1)$ (we can assume such $x$ exists, otherwise $f_\wedge(y_1)=-\infty$ and according to the convention \eqref{eq:conventions1} the inequality \eqref{eq:permcshane} holds). Notice that $x$ is also a competitor in the definition of  $f_\wedge(y_2)$ and  that  in this case the reverse triangle inequality can be written as $\ell(y_1,y_2) + \ell(x,y_1) \leq \ell(x,y_2)$ and thus trivially
\[
f(x)+L\,\ell(x,y_2)  \geq f(x)+ L\,\ell(x,y_1)+ L\,\ell(y_1,y_2).
\]
Taking the sup in $x$ we obtain that 
\begin{equation}
\label{eq:almostmcshane}
f_\wedge(y_2)\geq f_\wedge(y_1)+L\,\ell(y_1,y_2)
\end{equation}
and thus in particular that  $f_\wedge$ is causal. Then, since by the conventions \eqref{eq:conventions1} the claim \eqref{eq:permcshane} holds provided either $f_\wedge(y_2)=+\infty$ or $f_\wedge(y_1)=-\infty$, we can assume that $f_\wedge(y_1),f_\wedge(y_2)\in\R$. In this case, though, the claim \eqref{eq:permcshane} is a direct consequence of \eqref{eq:almostmcshane}.

Finally, let $g$ be any $L$-steep extension of $f$ to $\M$ and $y\in \M$. We need to prove that 
\begin{equation}
    \label{eq:gfwedge}
    g (y)\geq f_\wedge (y).
\end{equation}
To this end, let $x \in W$ be a competitor in the definition of $f_\wedge(y)$ (as before: if such $x$ does not exist we have $f_\wedge(y)=-\infty$ and \eqref{eq:gfwedge} follows). By $L$-steepness of $g$ we have 
\begin{equation}
\label{eq:pergf}
g(y)-f(x)=g(y)-g(x) \geq L\,\ell(x,y).    
\end{equation}
By assumption we know that $f(x)>-\infty$ while we see from the above that if $f(x)=+\infty$ for some competitor $x$, then $g(y)=+\infty$ and \eqref{eq:gfwedge} follows. Thus we can assume that $f(x)\in\R$ for every $x$ competitor in the definition of $f_\wedge(y)$. In turn, the relation \eqref{eq:pergf} tells us that $g(y)\geq f(x)+L\,\ell(x,y)$ for any such $x$, thus taking the supremum in $x$ we conclude the statement.
\end{proof}
A key object of study of this manuscript is differential calculus with causal functions. To develop this a first step is to leverage the concept of causal speed introduced above to define a notion of `modulus of differential' by duality. In Polish metric spacetimes, relevant proxies of what will turn out to be `the correct' definition are  the \emph{forward} and \emph{backward} slopes $|\partial^{+}f|$ and $|\partial^{-}f|$, respectively, defined as: 
\begin{equation}\label{eq:defsubslopes}
\begin{split}
|\partial^{+}f|(x)=:\varliminf_{y\downarrow x}\frac{f^+(y)-f^-(x)}{\ell(x,y)}:=\sup \inf_{y\in U\cap J^+(x)}\frac{f^+(y)-f^-(x)}{\ell(x,y)},
\\
|\partial^{-}f|(x):=\varliminf_{z\uparrow x}\frac{f^+(x)-f^-(z)}{\ell(z,x)}:=\sup \inf_{z\in U\cap J^-(x)}\frac{f^{ +}(x)-f^-(z)}{\ell(z,x)},
\end{split}
\end{equation}
where both suprema are taken over all neighborhoods $U$ of $x$. In these formulas, $0/0$ is intended to be equal to $+\infty$,   thus, for instance, if $x$ has a neighborhood $U$ for which there is no $y\in U$ different from $x$ with $y\geq x$, then $|\partial^{+}f|(x)=+\infty$. Similarly for $|\partial^{-}f|$.  If $f$ is $L$-steep, then clearly we have
\begin{equation}
\label{eq:steepslopeloc}
|\partial^{+}f|(x),|\partial^{-}f|(x)\geq L\qquad\forall x\in \M.
\end{equation}
The functions $|\partial^{\pm}f|$  could be considered as Lorentzian counterparts of the ascending/descending slopes appearing in positive signature and they will have a similar role in the proof of the metric Brenier--McCann theorem (compare \cite[Eq.\ (2.6), Thm.\ 10.3]{AGS:14a} with \eqref{eq:defsubslopes} above and \Cref{Th:Brenier} respectively). In this direction, the analog of the local Lipschitz constant as e.g.~defined in \cite{AGS:14a}  would be the \emph{local steepness constant} ${\sf st}(f)\colon \M\to [0,+\infty]$ defined through the assignment
\[
{\sf st}(f)(x):=\min\{|\partial^{+}f|(x),|\partial^{-}f|(x)\}.
\]
In the smooth setting, the reverse Cauchy-Schwarz inequality yields:
\[
\partial_t(f\circ \gamma)_t=\d f_{\gamma_t}(\gamma'_t)\geq\|\d f\|_*\|\gamma'_t\|.
\]
The following result offers a first non-smooth counterpart to this.

\begin{proposition}[Pathwise growth exceeds integrated forward and backward slopes]
\label{prop:subslopes1} Let $(\M,\uptau,\ell)$ be a Polish metric spacetime, $f:\M\to\bar\R$ a causal function and $\gamma:[0,1]\to\M$ a causal curve. Then there is a Borel function $g:[0,1]\to[0,+\infty]$ with  $g_t\geq \max\{|\partial^{+}f|,|\partial^{-}f|\}(\gamma_t)$  for every $t\in[0,1]$ and such that
\begin{equation}
\label{eq:subslope}
f(\gamma_1)-f(\gamma_0)\geq\int_0^1 g_t|\dot\gamma_t|\,\d t.
\end{equation}
If the   topology $\uptau$ contains the chronological one  and  $\ell_+$ is lower semicontinuous, then $|\partial^{+}f|$ and $|\partial^{-}f|$ are Borel measurable.
\end{proposition} 
\begin{proof} 
If either $f(\gamma_1)=+\infty$ or $f(\gamma_0)=-\infty$ the  conventions \eqref{eq:conventions1} ensure that \eqref{eq:subslope} hold with $g\equiv+\infty$. Taking causality of $f,\gamma$ into account we can thus assume that $f$ is real valued.   

\Cref{L:chronological continuity} asserts $f^- \le f \le f^+$ with equality holding outside a countable union $\Sigma$ of achronal sets. 
Let $S:= \supp |\dot{\bm\gamma}| \subset[0,1]$ denote the smallest closed set outside of which the causal speed of $\gamma$ vanishes.  Its complement consists of countably many intervals $(a,b)$. Unless $a,b \in S$ are the endpoints of one of these intervals, if $a<b$ then \Cref{Pr:speed}(i) and the \Cref{Def:causal speed} of causal speed show $0< |\dot{\bm\gamma}|(a,b) \le \ell(\gamma_a,\gamma_b)$.
This chronological relation shows $S\cap\gamma^{-1}(\Sigma)$ is at most countable.  Thus the non-decreasing functions  $f^\pm \circ \gamma$ satisfy $f^- \circ \gamma \le f \circ \gamma \le f^+\circ \gamma$ on $[0,1]$,  with equality holding on $S$ outside of the countable set $S \cap \gamma^{-1}(\Sigma)$.

By Lebesgue's theorem, see e.g. \cite[Theorem 3.23]{Folland2013},
(or by \Cref{Pr:speed} applied to $T(s,t):=f(\gamma_t)-f(\gamma_s)$) 
a monotone real-valued function is differentiable Lebesgue almost everywhere and letting  $(f\circ \gamma)'$ be this pointwise derivative yields
\[
f(\gamma_1)-f(\gamma_0)\geq\int_0^1(f\circ \gamma)'(t)\,\d t. 
\]
Now let $s_t:=\liminf_{h\downarrow0}\min\{\frac{\ell(\gamma_t,\gamma_{t+h})}{h},\frac{\ell(\gamma_{t-h},\gamma_{t})}{h}\}$ and notice that $t\mapsto s_t$ is Lebesgue measurable  (as a consequence of the monotonicity of $\ell$ in its entries  --- see the already cited \cite[Thm.\ 4]{ChabCrou09}) and equal to $|\dot\gamma_t|$ for a.e.\ $t$. Let $D\subset S$ be the set of $t$'s where  $s_t>0$ for which the differentials of  $f\circ\gamma$ and $f^\pm\circ\gamma$ exist and agree. Set  $g_t:=\tfrac1{s_t}\lim_{h\to 0}\frac{f(\gamma_{t+h})-f(\gamma_t)}{h}$ on $D$ and  $g_t:=+\infty$ on $[0,1]\setminus D$. It is clear that with this choice \eqref{eq:subslope} holds, so we just have to prove that $g_t\geq \max\{|\partial^{+}f|,|\partial^{-}f|\}(\gamma_t)$ holds  for every $t\in D$.  To see this notice that
\[
\begin{split}
(f\circ\gamma)'(t)&=\lim_{h\downarrow0}\frac{f^+(\gamma_{t+h})-f^-(\gamma_t)}h
\geq s_t\,\liminf_{h\downarrow0}\frac{f^+(\gamma_{t+h})-f^-(\gamma_t)}{\ell(\gamma_t,\gamma_{t+h})}  \geq |\partial^{+}f|(\gamma_t)\,s_t.
\end{split}
\]
and similarly
\begin{align*}
(f\circ\gamma)'(t)&=\lim_{h\downarrow0}\frac{f^+(\gamma_{t})-f^-(\gamma_{t-h})}h\geq s_t \liminf_{h\downarrow0}\frac{f^+(\gamma_{t})-f^-(\gamma_{t-h})}{\ell(\gamma_{t-h},\gamma_{t})}
\geq |\partial^{-}f|(\gamma_t)\, s_t,
\end{align*}
proving the claim. The $g$ as defined is Lebesgue measurable, but modifying it in a suitable Lebesgue measurable set of measure 0 by setting  it to $+\infty$ therein, we retain all of the required properties and obtain Borel {measurability}
of $g$.

We turn to Borel {measurability}
of $|\p^{+} f|$ and $|\p^{-} f|$. Let $F:\M^2\to[0,+\infty]$ be defined as
\[
F(x,y):=\begin{cases}
    \displaystyle\frac{f^+(y) - f^-(x)}{\ell(x,y)} & \textnormal{if }\ell(x,y) >0,\\
    +\infty & \textnormal{otherwise}.
\end{cases}
\]
Recall that \Cref{L:chronological continuity}(v) yields upper semicontinuity of $\pm f^\pm$, and also $f^+(y)\ge f(y) \ge f(x) \ge f^-(x)$ whenever $\ell(x,y) \ge 0$.  As 
the set $\{x:\ell(x,y)>0\}$ is open for each $y\in \M$, 
the assumed lower semicontinuity of $\ell_+$ implies the assignment 
$(x,y) \mapsto F(x,y)$ is upper semicontinuous.

Now let $(U_n)$ be a countable base for the Polish topology of $\M$ and notice that from the monotonicity of the quantity $\inf_{y\in U\cap J^+(x)}(f^+(y)-f^-(x))/\ell(x,y)$ in $U$ with respect to inclusion it follows that
\[
|\partial^{+}f|(x)=\sup_{n\in\N}\inf_{y\in U_n}F(x,y)
\]
for every $x\in \M$, hence $\smash{\vert \partial^{+} f\vert}$ is Borel measurable.  Similarly, one shows that $|\partial^{-}f|$ is Borel measurable.
\end{proof}

\subsection{Test plans on metric measure spacetimes}\label{Sub:TPMMS}
The concept of `weak subslope' will be given by duality with test plans. The latter are defined by:
\begin{definition}[Test plans]
\label{Def:Test plans} Let $(\M,\uptau,\ell,\mm)$  be a mm spacetime. A measure $\bdpi\in \Prob(\CC([0,1];\M))$ is called a \emph{test plan} if it has bounded compression, i.e.\ there exists a constant $C>0$ such that $(\eval_t)_\push\bdpi \leq C\,\meas$ for every $t\in[0,1]$.

Also, we will call $\bdpi\in\Prob(\CC([0,1];\M))$ 
\begin{enumerate}
    \item[a)] an \emph{initial test plan} if for some $T\in(0,1)$ and $C>0$ we have  $(\eval_t)_\push\bdpi \leq C\,\meas$ for all $t\in[0,T]$ and moreover   $(\eval_t)_\push \bdpi \to (\eval_0)_\push\bdpi$  in the narrow topology as $t\downarrow  0$,  
    \item[b)] a \emph{final test plan} if for some $T\in(0,1)$ and $C>0$ we have  $(\eval_t)_\push\bdpi \leq C\,\meas$ for all $t\in[T,1]$ and moreover $(\eval_s)_\push\bdpi \to (\eval_1)_\push\bdpi$ in the narrow topology as $s\uparrow 1$.
\end{enumerate}
\end{definition}

\begin{remark}[Initial and final limits and time reversibility] 
\label{R:forward-backward}
In forward metric measure spacetimes, the narrow continuity at time one (stipulated in point $(b)$ above) holds automatically. 
We impose it anyways,  to ensure we
could work equally well with backward metric measure spacetimes,  at the cost of exchanging left-notions for right-notions throughout.
This ensures the theory is as symmetric as possible under time-reversal.\hfill$\blacksquare$
\end{remark}
We shall refer to the condition   ``$(\eval_t)_\push\bdpi \leq C\,\meas$ for every $t\in[0,1]$'' by saying that $\ppi$ has bounded compression.  Among other things, it ensures $\meas$-a.e.~defined functions to be well-defined along test plans as follows: for $t\in[0,1]$, it yields $f\circ\e_t = g\circ\e_t$ $\bdpi$-a.e.~if $f=g$ $\meas$-a.e.

\begin{example}[Statics]
\label{Ex:Const} Let $\mathrm{const}\colon \M \to \CC([0,1];\M)$ denote the continuous map sending a point $x\in \M$ to the curve constantly equal to $x$. 
If $\mu\in\Prob(\M)$ 
then $\bdpi := \mathrm{const}_\push\mu$
satisfies $(\eval_t)_\push \bdpi = \mu$
for all $t\in [0,1]$.
Thus $\bdpi$
is a test plan if and only if $\mu\leq C\mm$ for some $C>0$, i.e.\ if and only if $\mu$ itself has bounded compression. \hfill$\blacksquare$
\end{example}
We conclude this short section describing procedures to build new test plans from old. Let  $(\M,\uptau,\ell)$ be a Polish metric spacetime:
\begin{lemma}[Restrictions of test plans]
\label{Le:restr} Let $(\M,\uptau,\ell,\mm)$  be a mm spacetime and $\bdpi$  a test plan on it. Then for every $s,t\in [0,1]$ with $s<t$ and every Borel set $\Gamma\subset \CC([0,1];\M)$ with $\bdpi[\Gamma]>0$, the measures $(\restr_s^t)_\push\bdpi$ and $\bdpi[\Gamma]^{-1}\,\bdpi\mres \Gamma$ are again test plans.
\end{lemma}

\begin{proof} Both measures clearly belong to $\Prob(\CC([0,1];\M))$. Bounded compression in either case follows easily as well. Indeed, if $(\e_r)_*\ppi\leq C\mm$ for any $r$, then we have
\begin{align*}
    (\eval_r)_\push\big[(\restr_s^t)_\push \bdpi\big] = (\eval_r \circ \restr_s^t)_\push\bdpi = (\eval_{(1-r)s+rt})_\push\bdpi \leq C\,\meas
\end{align*}
and similarly
\begin{align*}
(\eval_t)_\push\big[\bdpi[\Gamma]^{-1}\,\bdpi\mres \Gamma\big]  \leq \bdpi[\Gamma]^{-1}\,(\eval_t)_\push\bdpi \leq \bdpi[\Gamma]^{-1}\,C\,\meas.\tag*{\qedhere}
\end{align*}
\end{proof}

\subsection{Weak subslopes}

From now on all causal  functions are assumed to be $\meas$-measurable unless
explicitly stated otherwise; (this hypothesis becomes redundant in those cases covered by \Cref{Cor:Meas tf} anyways).
{We reserve the term {\em rough causal} to refer to a causal function which is not necessarily $\mm$-measurable.}

We have already noticed that in the smooth category, given a smooth causal function $f$ and a smooth causal curve $\gamma$, the reverse Cauchy-Schwarz inequality yields:
\[
\partial_t(f\circ \gamma)_t=\d f_{\gamma_t}(\gamma'_t)\geq
\|\d f\|_*\|\gamma'_t\|\qquad\forall t.
\]
Also,   it is easy to see that $\|\d f\|_*$ is the maximal continuous function for which the above inequality holds for any curve and $t$. Starting from this consideration and inspired by analogous definitions in the metric setting \cite{AGS:14a}, we propose the following `weak'  notion:
\begin{definition}[Weak subslope]
\label{Def:Weak lower diff} Let $(\M,\uptau,\ell,\mm)$ be a mm  spacetime and $f\colon \M\to\bar{\R}$  a causal  function. A Borel function $G\colon \M \to [0,+\infty]$ is called \emph{weak subslope} of $f$ if
\begin{align}\label{eq:defwd}
\int \big[f(\gamma_1) - f(\gamma_0)\big]\,\d\bdpi(\gamma) \geq \iint_0^1 G(\gamma_t)\,\vert\dot\gamma_t\vert\,\d t\,\d\bdpi(\gamma)\qquad\text{ for every test plan }\ppi.
\end{align}
\end{definition}

Some elementary comments on this notion are in order. By non-negativity of the respective integrands, both integrals in \eqref{eq:defwd} are well-defined, possibly with value $+\infty$. Because test plans have bounded compression, \Cref{Def:Weak lower diff} is not influenced by changes of $f$ or $G$ on $\meas$-negligible sets.   The set of weak subslopes of a given causal function is never empty, as it always contains the function identically $0$ on $\M$. 

A direct and relevant consequence of working with test plans rather than with single curves is the following stability result:
\begin{proposition}[Stability of weak subslopes]
\label{Pr:closed} Let $(\M,\uptau,\ell,\mm)$ be a mm spacetime and  $(f_n)$  a sequence of causal  functions converging pointwise $\mm$-a.e.\ to a causal function $f$. Let $(G_n)$ be a sequence of weak subslopes $G_n$ of $f_n$ which converges pointwise $\meas$-a.e.~to a function $G$. 

Then   $G$ is a weak subslope of $f$.
\end{proposition}
\begin{proof} 
Let $\bdpi$ be a fixed test plan.  For the moment, we assume
\begin{align}\label{Eq:Supremum assumption}
\sup_{n\in\N} \big\Vert f_n\circ \eval_1 - f_n \circ \eval_0\big\Vert_{L^\infty(\CC([0,1];\M),\bdpi)} < +\infty.
\end{align}
Our hypothesis and bounded compression give $f_n \circ \eval_1 - f_n\circ \eval_0 \to f\circ\eval_1 - f\circ \eval_0$ $\bdpi$-a.e.~as $n\to+\infty$. Lebesgue's dominated convergence theorem and \eqref{Eq:Supremum assumption} imply
\begin{align*}
\int \big[f(\gamma_1) - f(\gamma_0)\big]\d\bdpi(\gamma) = \lim_{n\to+\infty} \int\big[f_n(\gamma_1) - f_n(\gamma_0)\big]\d\bdpi(\gamma).
\end{align*}
On the other hand, again by our hypothesis and bounded compression, the set of all pairs $(\gamma,t)\in \CC([0,1];\M) \times [0,1]$ with $G_n(\gamma_t)\to G(\gamma_t)$ as $n\to+\infty$ has full $\bdpi\otimes \mathscr{L}^1$-measure. Fubini's theorem and  Fatou's lemma then give
\begin{align*}
\liminf_{n\to+\infty} \iint_0^1 G_n(\gamma_t)\,\vert\dot\gamma_t\vert \d t\d\bdpi(\gamma) \geq \iint_0^1 G(\gamma_t)\,\vert\dot\gamma_t\vert\d t\d\bdpi(\gamma).
\end{align*}
This shows the desired property \eqref{eq:defwd} under the hypothesis \eqref{Eq:Supremum assumption}.

To get rid of the assumption \eqref{Eq:Supremum assumption}, we argue by approximation. Let $\bdpi$ be an arbitrary test plan. If $f\circ\eval_1 - f\circ \eval_0 = +\infty$ on a set of positive $\bdpi$-measure, the conclusion is trivial, so we can assume that $f\circ\eval_1 - f\circ\eval_0 < +\infty$ $\bdpi$-a.e. Then  since $f_n\circ \eval_1 - f_n \circ \eval_0 \to f\circ \eval_1 -f\circ\eval_0$ $\bdpi$-a.e.~as $n\to+\infty$, we obtain $\smash{\limsup_{n\in\N}\big[f_n(\gamma_1) - f_n(\gamma_0)\big] < +\infty}$ for $\bdpi$-a.e.~$\gamma\in\CC([0,1];\M)$. Thus, the  collection   $(\Gamma_{i,j})_{i,j\in\N}$ of $\ppi$-measurable sets defined as
\begin{equation*}
\Gamma_{i,j} := \big\lbrace \gamma \in \CC([0,1];\M) : \sup_{n\geq j} \big[f_n(\gamma_1)-f_n(\gamma_0)\big] \leq i\big\rbrace
\end{equation*}
satisfies $\ppi(\cup_{i,j}\Gamma_{i,j})=1$. As in   \Cref{Le:restr}, setting $\smash{\bdpi_{i,j} := \bdpi(\Gamma_{i,j})^{-1}\,\bdpi\mres\Gamma_{i,j}}$ gives a collection $(\bdpi_{i,j})_{i,j\in\N}$ of test plans such that, thanks to the previous part, \eqref{eq:defwd} holds for $\bdpi$ replaced by $\bdpi_{i,j}$ for every $i,j\in\N$. Since $\Gamma_{i,j}\subset\Gamma_{i',j'}$ for $i\leq i'$ and $j\leq j'$, letting $i,j\to +\infty$ in the resulting inequality and using Levi's monotone convergence theorem, the claim follows.
\end{proof}
Our next goal is to prove that the collection of weak subslopes of a given function has a maximal element, intended in the $\mm$-a.e.\ sense. To show this, it is convenient to have the following  pathwise characterization of weak subslopes:
\begin{lemma}[An essentially pathwise criterion for weak subslopes]
\label{Le:Distr der} Let $(\M,\uptau,\ell,\mm)$ be a mm spacetime,  $\smash{f\colon \M\to \bar{\R}}$ a causal function, and  $\smash{G\colon [0,+\infty]\to\bar{\R}}$  Borel. 

Then $G$ is a weak subslope of $f$ if and only if for any test plan $\bdpi$ we have:
\begin{itemize}
\item[i)]For $\bdpi$-a.e.~$\gamma\in\CC([0,1];\M)$ the function $t\mapsto  G(\gamma_t)\,\vert\dot\gamma_t\vert$ belongs to $\smash{L^1_{\mathrm{loc}}(\dom(f\circ\gamma);\mathscr{L}^1)}$; 
\item[ii)] the distributional derivative of the non-decreasing and real valued restriction of $f\circ\gamma$ to $\dom(f\circ\gamma)$ is bounded from below by $G\circ\gamma\,\vert\dot\gamma\vert\,\mathscr{L}^1$ on $\dom(f\circ\gamma)$. 
\end{itemize} 
Item $(ii)$ could also be replaced by
\begin{itemize}
\item[ii')]  the density $(f\circ\gamma)'$ of the absolutely continuous part of the   distributional derivative of  $f\circ\gamma$ in $\dom(f\circ\gamma)$ satisfies $(f\circ\gamma)'\geq G\circ\gamma\,\vert\dot\gamma\vert$  $\Leb^1$-a.e.\ in $  \dom(f\circ\gamma)$.
\end{itemize}
In particular, if these hold for a test plan $\ppi$   the function $t\mapsto G(\gamma_t)\,\vert\dot\gamma_t\vert$ belongs to $\smash{L^1([0,1];\mathscr{L}^1)}$ for $\bdpi$-a.e.~$\gamma\in \CC([0,1];\M)$ with $f(\gamma_1) - f(\gamma_0)<+\infty$.
\end{lemma}
\begin{proof} The equivalence of $(ii)$ and $(ii')$ is an obvious consequence of the fact that $f\circ\gamma$ is non-decreasing and the last claim is a trivial consequence of $(ii')$ and the bound $\int_0^1(f\circ\gamma)'_t\,\d t\leq f(\gamma_1)-f(\gamma_0)$, that in turn follows from the monotonicity of $f\circ\gamma$. 

  Now assume first that $G$ is a weak subslope of $f$. We claim that, given any $s,t\in [0,1]$ with $s<t$, $\bdpi$-a.e.~$\gamma\in\CC([0,1];\M)$ satisfies
\begin{equation}
\label{eq:diff1}
f(\gamma_t) - f(\gamma_s) \geq \int_s^t G(\gamma_r)\, |\dot \gamma_r|\, \d r.
\end{equation}
Indeed, if not there are $s,t\in[0,1]$ with $s<t$ as well as $\varepsilon > 0$ such that the set $\Gamma_\varepsilon$ of all  $\gamma\in\CC([0,1];\M)$ such that $[s,t] \subset \dom(f\circ\gamma)$ and
\begin{align}\label{Eq:TP}
f(\gamma_t) - f(\gamma_s) + \varepsilon \leq \int_s^t G(\gamma_r)\, |\dot \gamma_r| \,\d r
\end{align}
satisfies $\bdpi[\Gamma_\varepsilon]>0$. Up to further reducing $\varepsilon$ and slightly reducing $\Gamma_\varepsilon$ as well, as in the proof of  \Cref{Pr:closed}  we may and will assume $f\circ\eval_t - f\circ \eval_s$ to be uniformly bounded on $\Gamma_\varepsilon$. Then by \Cref{Le:restr}, $\smash{\bdpi' := (\restr_s^t)_\push\big[\bdpi[\Gamma_\varepsilon]^{-1}\,\bdpi\mres\Gamma_\varepsilon\big]}$ constitutes a test plan and is  admissible in \eqref{eq:defwd}, but the latter in conjunction with \eqref{Eq:TP} as well as the finiteness of $\smash{\iint_0^1 G(\gamma_\alpha)\,\vert\dot\gamma_\alpha\vert\,\d\alpha\,\d\bdpi'(\gamma)}$,  as implied by our hypothesis on $f\circ\eval_t - f\circ \eval_s$, directly leads to a contradiction.

We then deduce that for any $s,t\in[0,1]\cap \Q$ with $s<t$ the bound \eqref{eq:diff1} holds for $\ppi$-a.e.\ $\gamma$. Then approximating arbitrary reals number $t$ from below and $s$ from above we conclude that for $\ppi$-a.e.\ $\gamma$ the bound \eqref{eq:diff1} holds for any  $s,t\in[0,1]$ with $s<t$. From this, the validity of $(ii)$  easily follows.

Conversely, we assume $(ii)$. Let $\bdpi$ be a test plan. The hypothesis implies that for $\bdpi$-a.e.~$\gamma\in\CC([0,1];\M)$, the distributional derivative of the function $h\colon \dom(f\circ\gamma) \to \R$ given by
\begin{align*}
h(t) := f(\gamma_t) - \int_a^t G(\gamma_r)\,|\dot \gamma_r| \,\d r
\end{align*} 
is non-negative (here we assume $\dom(f\circ\gamma) = [a,b]$ for notational convenience). In turn, this gives $h(s) \leq h(t)$ for $\smash{\mathscr{L}^1}$-a.e.~$s,t\in[a,b]$ with $s< t$. Since $f$ is a causal function, for such $s$ and $t$ this implies
\begin{align}\label{Eq:Prec}
f(\gamma_1) - f(\gamma_0)\geq f(\gamma_t) - f(\gamma_s) \geq \int_s^t G(\gamma_r)\, |\dot \gamma_r|\,\d r.
\end{align}
The dependence of the right-hand side on $s$ and $t$ is continuous if the difference $f(\gamma_1)-f(\gamma_0)$ is finite; in this case, \eqref{Eq:Prec} holds for $s=0$ and $t=1$. If $f(\gamma_1)-f(\gamma_0) = +\infty$, the estimate  \eqref{Eq:Prec} holds vacuously for these choices of $s$ and $t$, and so does \eqref{Eq:Prec} for $\bdpi$-a.e.~$\gamma\in\CC([0,1];\M)$. Integration yields the claim.
\end{proof}

 This pathwise characterization immediately implies that when
\begin{equation}
\label{eq:maxsubslopes}
G_1,G_2\quad \text{are subslopes of $f$}\qquad\Rightarrow\qquad \max\{G_1,G_2\}\text{ is a subslope of $f$}.
\end{equation}
Indeed, if $\ppi$ is a test plan, by \Cref{Le:Distr der} and with the notation of item $(ii')$, for $\bdpi$-a.e.~$\gamma\in\CC([0,1];\M)$ we have $(f\circ\gamma)'\geq (G_1\circ\gamma)\,\vert\dot\gamma\vert$ and $(f\circ\gamma)'\geq  (G_2\circ\gamma)\,\vert\dot\gamma\vert$  $\Leb^1$-a.e.\ on  $\dom(f\circ\gamma)$.
It follows that  for $\bdpi$-a.e.~$\gamma$ we have  $(f\circ\gamma)'\geq \max\{G_1,G_2\}\circ\gamma \,\vert\dot\gamma\vert$ $\Leb^1$-a.e.\ on  $\dom(f\circ\gamma)$, which by  \Cref{Le:Distr der} suffices to deduce \eqref{eq:maxsubslopes}.

The main result of the section is:
\begin{theorem}[Unique existence of maximal weak subslopes]
\label{Th:MaxLoDiff} Let $(\M,\uptau,\ell,\mm)$ be a mm spacetime and  $f\colon \M\to\bar{\R}$  causal. Then there exists a Borel function $G\colon \M\to[0,+\infty]$ which is a weak subslope of $f$ and $\meas$-a.e.~no smaller than any other weak subslope of $f$.

In particular, $G$ is unique up to modifications on $\meas$-negligible sets.
\end{theorem}
\begin{proof} The conditions on  $G$ force it to be the $\mm$-essential supremum of the collection of all weak subslopes of $f$: what we need to prove is that this $G$ is still a weak subslope. From general properties of essential suprema we know that there is a sequence $(G_n)$ of weak subslopes such that $\mm$-a.e.\ we have $G=\sup_nG_n$. By \eqref{eq:maxsubslopes} the functions $H_n:=\max_{i\leq n}G_i$ are still weak subslopes and by \Cref{Pr:closed} (with $f_n=f$) we conclude that $G$ is also a weak subslope, as desired.
\end{proof}
The concept isolated by this last result deserves a definition:
\begin{definition}[Maximal weak subslope]
\label{Def:Maximal wld} Let $(\M,\uptau,\ell,\mm)$ be a mm  spacetime and   $f\colon \M\to\bar{\R}$ causal. The $\meas$-a.e.~unique element $G$ from  \Cref{Th:MaxLoDiff} is  called \emph{maximal weak subslope} of $f$ and denoted by $\vert\rmd f\vert$.
\end{definition}
Notice that if $f$ is $L$-steep, then we have
\begin{equation}
\label{eq:dflsteep}
|\d f|\geq L\qquad\mm-a.e.,
\end{equation}
as follows directly from the maximality of $|\d f|$, the definition of $L$-steepness and the bound \eqref{eq:boundspeed}.

A direct consequence of the definition and   \Cref{Pr:closed} is the following:
\begin{proposition}[`Upper semicontinuous' dependence]\label{Pr:USCdep} Let $(\M,\uptau,\ell,\mm)$ be a mm spacetime and  $(f_n)$  a sequence of causal  functions on it converging pointwise $\mm$-a.e.\ to a causal function $f$. Then
\begin{equation}
\label{eq:uscdf}
    \liminf_{n\to+\infty} \vert\rmd f_n\vert \leq \vert \rmd f\vert\quad\meas\textnormal{-a.e.}
\end{equation}
\end{proposition}
\begin{proof} The function $G_n:=\inf_{i\geq n}|\d f_i|\leq|\d f_n|$ is clearly a weak subslope of $f_n$. Since $G_n\to     \liminf_{n} \vert\rmd f_n\vert$ $\mm$-a.e.\ the conclusion follows from \Cref{Pr:closed}.
\end{proof}
\begin{remark}
The result above would be false if we replace $\liminf$ with $\limsup$, which is why we put the quotation marks in `upper semicontinuous'. Still, given that in many cases of interest, e.g.\ in studying calculus rules in the upcoming sections, the sequence $(|\d f_n|)$ admits a limit $\mm$-a.e., speaking of upper semicontinuity is not entirely inappropriate.

For a counterexample to \eqref{eq:uscdf} with the $\limsup$ consider $[0,1]$ with the Euclidean topology and $\ell(x,y):=y-x$ if $y\geq x$ and $-\infty$ otherwise. Find a sequence $(g_n)\subset L^1([0,1])$ of non-negative functions with $\limsup_ng_n\equiv 1$ and $\|g_n\|_{L^1}\downarrow0$, then define $f_n(x):=\int_0^xg_n$. Quite clearly we have $|\d f_n|=g_n$ (see also \Cref{ss:compatibility}), $f_n\to 0$ but $|\d 0|=0\ngeq 1$. 
\hfill$\blacksquare$
\end{remark}
The maximality and \Cref{prop:subslopes1} directly imply:
\begin{proposition}[Forward and backward slopes are weak subslopes]
\label{prop:subslopes2} Let $(\M,\uptau,\ell,\mm)$ be a mm spacetime whose topology contains the chronological one and so that $\ell_+$ is lower semicontinuous. Let  $f:\M\to\bar\R$ be a causal function.

Then $|\partial^{+}f|$ and $|\partial^{-}f|$ are weak subslopes. In particular, letting $L$ being the steepness constant of $f$ we have
\begin{equation}
\label{eq:slopesteepness}
|\d f|\geq |\partial^{+}f|,|\partial^{-}f|\geq L\qquad\mm-a.e.
\end{equation}
\end{proposition}
\begin{proof}
By \Cref{prop:subslopes1}  we see that $|\partial^{+}f|$ and $|\partial^{-}f|$ are Borel and thus also that the bound  $f(\gamma_1)-f(\gamma_0)\geq\int_0^1\max\{ |\partial^{+}f|,|\partial^{-}f|\}|\dot\gamma_t|\,\d t$ holds for any causal curve $\gamma$. Integrating w.r.t.\ an arbitrary test plan we deduce that $|\partial^{+}f|$ and $|\partial^{-}f|$ are weak subslopes. Then \eqref{eq:slopesteepness} follows from the definition of $|\d f|$ and \eqref{eq:steepslopeloc}.
\end{proof}

\subsection{Calculus rules}
This part establishes several calculus rules for the maximal weak subslope of causal functions, notably locality, the chain rule, and the Leibniz rule. We start with some elementary observations about regions where $|\d f|$ is equal to $+\infty$. A first simple comment comes from our conventions from \Cref{Sub:InfConv} and the maximality proclaimed in \Cref{Def:Maximal wld}: they readily imply that for any causal  function $f$ we have
\[
\text{$\vert\rmd f\vert = +\infty$ $\meas$-a.e.~on the complement of $\dom(f)$.}
\]
A related result involves the parts of our spacetime that are `visible' via test plans. We introduce the concept. Let $\bdpi$ be a given test plan and call  $\rho_{\bdpi}$  the $\meas$-density of the push-forward of the measure $\vert \dot\gamma_t\vert\d\bdpi(\gamma)\d t$ on $\CC([0,1];\M)\times[0,1]$ under the ``full'' evaluation map $\mathsf{e}(\gamma,t) := \gamma_t$. Let  ${\rm Vis}_{\bdpi}:=\{\rho_{\bdpi}>0\}\subset M$. We think of the Borel set ${\rm Vis}_{\bdpi}$ --- defined up to $\mm$-negligible sets --- as the set that is `visible' from $\ppi$, whence the notation. Then we define ${\rm Vis}(\M)$ as the $\mm$-essential union of all the ${\rm Vis}_{\bdpi}$ as $\ppi$ varies among test plans. Recall (see for instance  \cite[Prop.~3.1.9]{GP:20} for this basic result in measure theory) that ${\rm Vis}(\M)$  is characterized up to $\mm$-negligible sets as the smallest Borel set that contains all the ${\rm Vis}_{\bdpi}$'s up to $\mm$-negligible sets (this notion is closely related to that of $\mm$-essential supremum of a family of Borel functions that we have already seen in Theorem \ref{Th:MaxLoDiff}). We shall frequently use the fact that a property holds $\mm$-a.e.\ on ${\rm Vis}(\M)$ if and only if it holds $\mm$-a.e.\ on ${\rm Vis}_{\bdpi}$ for every test plan $\ppi$.

We then define the `invisible' set ${\rm Invis}(\M):=\M\setminus {\rm Vis}(\M)$ and notice that
\begin{equation}
\label{eq:invisible}
\text{$|\d f|=+\infty\quad\mm$-a.e.\  on ${\rm Invis}(\M)$ for any causal function $f$.}
\end{equation}
To see this, by pointwise maximality of $|\d f|$ it suffices to show that $G:=\infty 1_{{\rm Invis}(\M)}$ is a weak subslope of   $f$ and to this aim it is enough to prove that $\iint_0^1  G(\gamma_t)|\dot\gamma_t|\,\d t\,\d\ppi(\gamma)$ is equal to 0 for any test plan $\ppi$. This, however, is obvious from
\[
\iint_0^1G(\gamma_t)|\dot\gamma_t|\,\d t\,\d\ppi(\gamma)=\int G\,\d\eval_*\big(\vert \dot\gamma_t\vert\bdpi\times\Leb^1\mres{[0,1]}\big)=\int  G\rho_{\ppi}\,\d\mm=0.
\]
Because of \eqref{eq:invisible}, our calculus rules are sometimes phrased on the complement ${\rm Vis}(\M)$ of ${\rm Invis}(\M)$ as this will simplify our
formulas. Also, we will often restrict our attention to regions where $|\d f|<+\infty$, so that automatically from \eqref{eq:invisible} we will not care about what happens in the invisible set ${\rm Invis}(\M)$.

A simple property very much in this  vein 
is:
\begin{lemma}[Achronal sets are invisible]
Let  $(\M,\uptau,\ell,\mm)$ be a mm spacetime and $A\subset\M$ achronal and Borel. Then $\mm(A\setminus {\rm Invis}(\M))=0$.

In particular, for any causal $f:\M\to\bar\R$ we have $|\d f|=+\infty$ $\mm$-a.e.\ on $A$.
\end{lemma}
\begin{proof}
It suffices to prove that for any test plan $\ppi$ we have $\iint_0^11_A(\gamma_t)|\dot\gamma_t|\,\d t\,\d\ppi(\gamma)=0$. For $\gamma\in\CC([0,1];\M)$ and  $t,s\in\gamma^{-1}(A)$  with $t\leq s$ the achronality of $A$ forces $\ell(\gamma_t,\gamma_s)=0$. Then the causality of $\gamma$ implies that the same holds for $t\leq s$  in the convex hull $I\subset[0,1]$ of $\gamma^{-1}(A)$. It follows that $|\dot\gamma_t|=0$ for a.e.\ $t\in I$ and in particular for a.e.\ $t\in \gamma^{-1}(A)$. Hence   $1_A(\gamma_t)|\dot\gamma_t|=0$ for a.e.\ $t\in[0,1]$ and the conclusion follows.
\end{proof}

\begin{example} In the one point space the whole space is invisible, because any curve must be constant and thus have 0 causal speed.\hfill$\blacksquare$
\end{example}
With this said, we can turn to the study of calculus rules. We start with the simple
\begin{equation}
\label{eq:dconst}
|\d f |=0\qquad\mm-a.e.\ on\ {\rm Vis}(\M)\qquad\text{ if $f$ is constant,}
\end{equation}
which can be shown noticing that for any weak subslope $G$ and any test plan $\ppi$ we must have $0\geq \iint_0^1G(\gamma_t)|\dot\gamma_t|\,\d t\,\d\ppi(\gamma)=\int  G\rho_{\ppi}\,\d\mm$. We also have the following concavity properties,
of which (i) is essential and (ii), though not required subsequently,  offers a concrete perspective which some readers may find helpful:

\begin{proposition}[Concavity and positive 1-homogeneity]
\label{Pr:Linear comb} Let  $(\M,\uptau,\ell,\mm)$ be a mm spacetime, $f,g\colon :\M\to\bar{\R}$  causal. {\rm (i)} Then all $\alpha,\beta \geq 0$ satisfy

\begin{equation}
\label{eq:dsuper}
\vert\rmd (\alpha\,f + \beta\,g)\vert \geq \alpha\,\vert\rmd f\vert + \beta\,\vert \rmd g\vert\quad\meas\textnormal{-a.e. on $\M$}
\end{equation}
and 
\begin{equation}
\label{eq:1-homo}
\vert\rmd(\alpha f) \vert = \alpha\,\vert\rmd f\vert\quad
\meas\textnormal{-a.e.\ on Vis($\M$) if $\alpha=0$ and 
$\meas$-a.e. on $\M$ if $\alpha>0$}.
\end{equation}
{\rm (ii)} A concave positively 1-homogeneous function $h_x:[0,\infty)^2 \to [0,\infty]$ can be associated to each $x \in \M$ so that: all $(0,0)\ne(\alpha,\beta) \in [0,\infty)^2$ satisfy $|\rmd(\alpha f +\beta g)|(y) = h_y(\alpha,\beta)$ for $\mm$-a.e. $y \in\M$.

\end{proposition}

\begin{proof} (i) Evidently, $\alpha\,\vert\rmd f\vert + \beta\,\vert\rmd g\vert$ is a weak subslope of the causal function $\alpha\,f + \beta\,g$. 

The first claim \eqref{eq:dsuper} follows from 
 Definition \ref{Def:Maximal wld}. The second claim is obvious if $\alpha=0$ (here having restricted to attention to ${\rm Vis}(\M)$ matters), while for $\alpha>0$ we use twice the first claim with $g=0$ to get $|\d f|\geq \alpha |\d(\alpha^{-1}f)|\geq|\d f|$, which is \eqref{eq:1-homo}.

(ii) For $0\le \alpha \le \alpha '$ and $0\le \beta \le \beta'$,
since $|\rmd(\alpha f + \beta g)|$ is a weak subslope for $\alpha' f + \beta ' g$ we find (*)
$|\rmd (\alpha f + \beta g)| \le |\rmd (\alpha'f + \beta' g)|$
holds $\mm$-a.e. 
 For rational coefficients and all $\Q_+:=\Q\cap [0,\infty)$ combinations of $f$ and $g$, the $\mm$-negligible set $\M\setminus X$ of \eqref{eq:dsuper}--\eqref{eq:1-homo} can be taken independent of $\alpha,\beta \in \Q_+$ and of the combinations of $f$ and $g$. 
Taking $X$ smaller yet $\mm(X)=1$, (*) 
yields
$|\rmd(\alpha f + \beta g)| \le |\rmd(\alpha' f + \beta' g)|$
throughout $X$ for $\alpha',\beta'$ also rational
with $\alpha \le \alpha'$ and $\beta \le \beta'$.
Fix $x \in X$ and assume $|\rmd(f+g)(x)|<\infty$; otherwise a similar but simpler argument will yield $|\rmd(\alpha f + \beta g)|(x) =\infty$ for all real positive $\alpha,\beta>0$,
(and positive $1$-homogeneity for $\alpha,\beta \ge 0$ real such that 
$\alpha\beta=0$ but $\alpha+\beta >0$).
For $\alpha,\beta \in [0, \gamma] \cap \Q_+$ and $\gamma$
rational, 
$$|\rmd(\alpha f + \beta g)|(x) \le |\rmd(\gamma f + \gamma g)|(x)
= \gamma |\rmd(f+g)|(x) < \infty.
$$ 
Fixing $\alpha \in \Q_+$,  it follows that 
$\beta \in \Q_+ \mapsto |\rmd (\alpha f + \beta g)|(x)$
is concave and locally bounded --- hence
locally Lipschitz except possibly at $\beta=0$.
Fixing $\beta \in \Q_+$ instead,  the same conclusions apply to $\alpha \in \Q_+ \mapsto |\rmd (\alpha f + \beta g)|(x)$.
Thus $h_x(\alpha,\beta) := |\rmd(\alpha f +\beta g)|(x)$
extends from $\Q_+^2$ to a concave positively 1-homogeneous function on $[0,\infty)^2$ (which is locally Lipschitz on the interior of this quadrant and on its boundary, but may be merely lower semicontinuous on their union).
  Approximating real $\alpha \in (\alpha'',\alpha')$ and $\beta \in (\beta'',\beta')$ from above and below
  by rational $\alpha'',\alpha',\beta'',\beta'$,
the aforementioned continuity of $h_x(\cdot,\cdot)$ combines with (*) to yield $|\rmd(\alpha f + \beta g)|(y) = h_y(\alpha,\beta)$ for $\mm$-a.e. $y \in \M$ unless $\alpha=0=\beta$. 
\end{proof}


We turn to locality properties of the maximal weak subslope. They will closely interact with the chain rule and the Leibniz rule (for metric measure spaces these three properties of the minimal weak upper gradient \cite{AGS:08} are equivalent to each other \cite{AK:00,Gig:18}).

We  recall how calculus with BV functions on $\R$ works. Thus letting $I\subset\R$ be an interval, recall that  $u\in BV_{loc}(I)$ iff it is in $L^1_{loc}(I)$ and has distributional derivative represented by a (Radon) measure $Du$. We write 
$Du=u'\mathcal L^1\mres I+Du^\perp$ with $Du^\perp\perp\Leb^1$ for the Lebesgue decomposition of $Du$. Then:
\begin{enumerate}[label={BV\arabic*)}]
\item $u$ is $\Leb^1$-a.e.\ equal to the difference of two monotone functions (see e.g.\ \cite[Sect.~3.12]{AFP:00} and references therein).
\item\label{it:bvae}   If $u$ coincides with the representative given above, then it is $\Leb^1$-a.e.\ differentiable and the derivative  a.e.\ coincides with $u'$ (see    \cite[Thm.~3.107]{AFP:00}).
\item\label{it:bvloc} For $N\subset\R$ Borel and negligible we have $u'=0$ $\Leb^1$-a.e.\ on $u^{-1}(N)$ (see   \cite[Prop.~3.92]{AFP:00}).
\item\label{it:bvchain}  For $\varphi:\R\to\R$ Lipschitz we have $\varphi\circ u\in BV_{loc}(I)$ and $(\varphi\circ u)'=\varphi'\circ u\,u'$ $\Leb^1$-a.e.\ (this is the absolutely continuous part of the Vol'pert chain rule, see e.g.\   \cite[Thm.~3.99]{AFP:00}).  For $\varphi$ non-decreasing but not necessarily Lipschitz the same conclusions can be verified based on the a.e.\ differentiability of monotone functions and \ref{it:bvae}--\ref{it:bvchain}.
\end{enumerate} 


Lemma \ref{Le:Distr der} and the locality in item \ref{it:bvloc} easily yields the following:
\begin{proposition}[Locality]\label{Th:Locality} Let $(\M,\uptau,\ell,\mm)$ be a mm spacetime,   $f_1,f_2\colon \M\to\bar{\R}$ be causal functions and $E\subset\R$ Borel and $\Leb^1$-negligible. Then
\begin{equation}
\label{eq:strongloc}
\vert \rmd f_1\vert = \vert \rmd f_2\vert \quad \meas\textnormal{-a.e.~on } (f_1-f_2)^{-1}(E).
\end{equation}
In particular, we have
\begin{equation}
\label{eq:locpunt}
\vert \rmd f_1\vert = \vert \rmd f_2\vert \quad \meas\textnormal{-a.e.~on }\{f_1=f_2\}
\end{equation}
and 
\begin{equation}
\label{eq:loczero}
    \vert\rmd f_1\vert = 0\quad\meas\textnormal{{-a.e.~on} }f_1^{-1}(E) \cap {\rm Vis}(\M).
\end{equation}
\end{proposition}
\begin{proof} Put  $F:=(f_1-f_2)^{-1}(E)$. We shall prove that $|\d f_1|1_F+|\d f_2|1_{\M\setminus F} \le |\d f_2|$. 
Reversing the roles of $f_1,f_2$ this suffices to obtain \eqref{eq:strongloc}. Notice that in order for $f_1(x)-f_2(x)$ to belong to $E\subset \R$, we must have $x\in\dom(f_1)\cap \dom(f_2)$, thus $F\subset\dom(f_1)\cap\dom(f_2)$.

Now let  $\gamma\in \CC([0,1];\M)$ and notice that on the interval $I_\gamma:=\gamma^{-1}(\dom(f_1)\cap\dom(f_2))$ the function $t\mapsto f_1(\gamma_t)-f_2(\gamma_t)$ is in $BV_{loc}$ and thus, by item \ref{it:bvloc} and with the same notation we see that
\begin{equation}
\label{eq:perlocalita}
 (f_1\circ\gamma)'_t=(f_2\circ\gamma)'_t\qquad\text{ for $\Leb^1$-a.e.\ $t\in \gamma^{-1}(F)$}.
\end{equation}
Now let $\bdpi$ be a test plan. \Cref{Le:Distr der} implies that for $\ppi$-a.e.\ $\gamma$ we have
\[
(f_2\circ\gamma)'_t\geq |\d f_2|(\gamma_t)|\dot\gamma_t|\qquad \text{ for $\Leb^1$-a.e.\ }t\in \gamma^{-1}(\dom(f_2))\, \supset\, \big( \gamma^{-1}(\dom(f_2))\setminus \gamma^{-1}(F)\big)
\]
and taking into account also \eqref{eq:perlocalita} we see that it also holds that
\[
(f_2\circ\gamma)'_t=(f_1\circ\gamma)'_t\geq |\d f_1|(\gamma_t)|\dot\gamma_t|\qquad \text{ for $\Leb^1$-a.e.\ }t\in \gamma^{-1}(F).
\]
These last two bounds imply  that
\[
(f_2\circ\gamma)'_t\geq \big(|\d f_1|1_F+|\d f_2|1_{\M\setminus F}\big)(\gamma_t)|\dot\gamma_t|\qquad  \text{ for $\Leb^1$-a.e.\ }t\in \gamma^{-1}(\dom(f_2))
\]
holds for $\ppi$-a.e.\ $\gamma$, so \Cref{Le:Distr der} yields 
that $|\d f_1|1_F+|\d f_2|1_{\M\setminus F}$ 
is a weak subslope of $f_2$, as desired.
Now \eqref{eq:locpunt} follows by picking $E:=\{0\}$ in  \eqref{eq:strongloc} while  \eqref{eq:loczero} by taking $f_2\equiv 0$ in  \eqref{eq:strongloc} and recalling \eqref{eq:dconst}.
\end{proof}

Similarly,  Lemma \ref{Le:Distr der} and the chain rule in item \ref{it:bvchain} easily imply:



\begin{proposition}[Chain rule]\label{Th:Chain rule II} Let $(\M,\uptau,\ell,\mm)$ be a mm  spacetime, $f :\M\to\bar{\R}$ causal  and  $\varphi\colon\R\to \R$  non-decreasing. Then $\varphi\circ f$ is causal and 
\begin{equation}
\label{eq:chaingeneral}
    \vert\rmd(\varphi\circ f)\vert = (\varphi'\circ f)\,\vert\rmd f\vert\quad\meas\textnormal{\textit{-a.e. on ${\rm Vis}(\M)\cap\big(\{|\d f|>0\})\cup\{\varphi'\circ f<+\infty\}\big)$} },
\end{equation}
where $\varphi'$ is the density of the distributional derivative $D\varphi $ of $\varphi$ on its domain and it is intended that $\varphi(\pm\infty)=\pm\infty$ and $\varphi'(\pm\infty)=+\infty$.
\end{proposition}
\begin{proof}  The fact that  $\varphi\circ f$ is causal is trivial. Then notice that even though $\varphi'$ is defined up to equality $\Leb^1$-a.e., an application of \eqref{eq:loczero} shows that identity \eqref{eq:chaingeneral} is well-posed, i.e.\ its validity is unaffected by the particular representative of $\varphi'$ chosen. Also, both sides of \eqref{eq:chaingeneral} are $\mm$-a.e.\ equal to $+\infty$ on $ (\varphi\circ f)^{-1}(\{\pm\infty\})\cap\{|\d f|>0\}\supset f^{-1}(\{\pm\infty\})$, thus by locality \eqref{eq:strongloc} and truncation we can assume that $\varphi:\R\to\R$ is bounded. 

Now let $\gamma\in\CC([0,1];\M)$ and use item \ref{it:bvchain} and the notation therein to deduce that 
\begin{equation}
\label{eq:keychain}
(\varphi\circ f\circ\gamma)'_t=\varphi'(f(\gamma_t))(f\circ\gamma)'_t\qquad \textrm{for a.e.\ }t\in\dom(f\circ\gamma).
\end{equation}
By Lemma  \ref{Le:Distr der}  this suffices both to establish that $\geq$ holds in \eqref{eq:chaingeneral} (because $(f\circ\gamma)'_t\geq|\d f|(\gamma_t)|\dot\gamma_t|$ for a.e.\ $t $ and $\ppi$-a.e.\ $\gamma$ for $\ppi$ test) and that equality holds  on $f^{-1}(\{\varphi'=0\})$. It thus remains to prove that $\leq$ holds on $ f^{-1}(\{\varphi'>0\})$. For this we use   again  Lemma  \ref{Le:Distr der} to get that  $(\varphi\circ f\circ\gamma)'_t\geq |\d(\varphi\circ f)|(\gamma_t)|\dot\gamma_t|$ holds for a.e.\ $t$ and $\ppi$-a.e.\ $\gamma$ for any test $\ppi$, then the identity \eqref{eq:keychain} yields $(f\circ\gamma)'_t\geq \tfrac{|\d(\varphi\circ f)|(\gamma_t)}{\varphi'(f(\gamma_t))}|\dot\gamma_t|$ for a.e.\ $t$ and $\ppi$-a.e.\ $\gamma$, where we interpret the ratio as $0$ if the denominator is 0. Then  Lemma  \ref{Le:Distr der} once more allows to conclude.
\end{proof}

\begin{proposition}[Leibniz rule]\label{Cor:Leibniz} Let $(\M,\uptau,\ell,\mm)$ be a mm  spacetime and  $f,g: \M\to\bar{\R}$   non-negative and causal. Then $fg$ is causal and 
\begin{equation}
\label{eq:leibineq}
\vert\rmd(fg)\vert \geq f\,\vert \rmd g\vert + g\,\vert\rmd f\vert\quad\meas\textnormal{\textit{-a.e. on ${\rm Vis}(\M)$} }.
\end{equation}
\end{proposition}
\begin{proof} The fact that $fg$ is causal is obvious. Also, $\mm$-a.e.\ on the set ${\rm Vis}(\M)\cap\{f=+\infty\}\cap\{g>0\}$ both sides of \eqref{eq:leibineq} are equal to $+\infty$, whereas $\mm$-a.e.\ on ${\rm Vis}(\M)\cap\{g=0\}$ both sides are 0. Thus, also switching the roles of $f,g$, we are left to prove the conclusion on  ${\rm Vis}(\M)\cap\{f\in(0,+\infty)\} \cap\{g\in(0,+\infty)\}$. Replacing $f$ with $c\vee f \wedge c^{-1}$ for $c>0$ arbitrarily small and similarly for $g$ and using \Cref{Th:Locality}, we can --- and will --- assume that $f,g$ are $(0,+\infty)$-valued.

In this case \Cref{Th:Chain rule II} yields
\[
\begin{split}
\tfrac1{fg}|\d(fg)|&=|\d(\log(fg))|=|\d(\log(f)+\log(g))|\stackrel{\eqref{eq:dsuper}}\geq |\d(\log(f))|+|\d(\log(g))|=\tfrac1f|\d f|+\tfrac1g|\d g|,
\end{split}
\]
which is the conclusion.
\end{proof}

\section[Horizontal (inner) and vertical (outer) derivatives]{Horizontal (inner) and vertical (outer) derivatives}
\label{ch:horizontal and vertical derivatives}

\subsection{Horizontal differentiation}

Here  we want to propose a notion  of ``derivative'' of an arbitrary causal function in the ``direction'' of LCC curves.   Intuitively, we shall refer to the horizontal derivative of a causal function $f$ along an initial test plan $\ppi$ as at the limit
    \begin{align*}
\lim_{t\to 0} \int \frac{f(\gamma_t) - f(\gamma_0)}{t}\d\bdpi(\gamma)
\end{align*}
and we are interested in knowing when such a limit exists and/or how we can estimate it. A bound from below can be quite easily given, as we are going to show in a moment. To motivate the result, notice that in the smooth category, the reverse Cauchy-Schwarz and the Fenchel--Young inequalities (for $p,q<1$) give that
\[
\d f(v)\geq \|\d f\|_*\|v\|\geq\tfrac1p\|\d f\|_*^p+\tfrac1q\|v\|^q
\]
for any causal function $f$ and future directed vector $v$ (here and below for clarity we shall often write $\tfrac1pz^p$ rather than the more precise $u_p(z)$ as in \Cref{Sub:InfConv}). We are going to prove a nonsmooth analogue of this.  In the sequel, we say that a measure $\bdpi\in \Prob(\CC([0,1];\M))$ \emph{stays initially in a given Borel  set $E\subset M$} if there is $T\in(0,1]$ such that $(\eval_t)_\push\bdpi$ is concentrated on $E$ for every $t\in [0,T]$.

\begin{proposition}[Nonsmooth Fenchel--Young inequality]
\label{Pr:Lower bound}  Let $(\M,\uptau,\ell,\meas)$ be a forward mm spacetime,  $E\subset \M$ be a Borel  set. Let $\bdpi$ be an initial test plan which stays initially in $E$. Let $\PP^{-1}+\QQ^{-1}=1$ with $0 \neq \QQ < 1$. Let $f\colon \M\to\bar{\R}$ be a causal function and assume that 
\begin{enumerate}[label=\textnormal{(\roman*)}]
    \item if $p<0$ we have $\vert \rmd f\vert^p\in L^1(E,\meas\mres E)$  (note: this  implies that  $\vert\rmd f\vert^p\in L^1(E,(\eval_0)_\push\bdpi)$ since $\bdpi$ is an initial test plan, but the converse implication does not hold),
    \item if $q<0$ we have $ \liminf_{t\to 0} \frac{1}{t\QQ}\iint_0^t \vert\dot\gamma_r\vert^\QQ\d r\d\bdpi(\gamma)>-\infty$.
\end{enumerate}
Then
\begin{equation}
\label{eq:FYnonsmooth}
\liminf_{t\downarrow 0} \int \frac{f(\gamma_t) - f(\gamma_0)}{t}\,\d\bdpi(\gamma)  \geq \frac{1}{\PP}\int \vert\rmd f\vert^\PP(\gamma_0)\,\d\bdpi(\gamma) + \liminf_{t\downarrow 0} \frac{1}{t\QQ}\iint_0^t \vert\dot\gamma_r\vert^\QQ\,\d r\,\d\bdpi(\gamma).
\end{equation}
\end{proposition}
\begin{proof} Let $T>0$ be such that $\smash{(\restr_0^t)_\push\bdpi}$ is a test plan $t\in (0,T)$. Applying the Definitions  \ref{Def:Weak lower diff} and \ref{Def:Maximal wld} 
of (maximal) weak subslope as provided by \Cref{Th:MaxLoDiff} to the  restricted test plan $\smash{(\restr_0^t)_\push\bdpi}$ we obtain
\begin{align*}
\int f(\gamma_t) - f(\gamma_0)\,\d\bdpi(\gamma)  &\geq \iint_0^t \vert\rmd f\vert(\gamma_r)\,\vert\dot\gamma_r\vert\,\d r\,\d\bdpi(\gamma)\\
&\geq \iint_0^t \frac{1}{\PP}\vert \rmd f\vert^\PP(\gamma_r)   + \frac{1}{\QQ}\vert\dot\gamma_r\vert^\QQ\,\d r\,\d\bdpi(\gamma)\\
&= \frac{1}{\PP}\iint_0^t \vert \rmd f\vert^\PP(\gamma_r)\d r\d\bdpi(\gamma)  + \frac{1}{\QQ}\iint_0^t\vert\dot\gamma_r\vert^\QQ\,\d r\,\d\bdpi(\gamma)\qquad\text{for $t$ small enough,}
\end{align*}
where the scalar Fenchel--Young inequality has been used in the second inequality. We comment on the last equality. Since the two integrands both have a sign, and thus their integrals are well defined --- possibly with values $\pm\infty$ --- equality holds provided at least one of the two integrands is actually integrable. This follows from our assumptions:
\begin{itemize}
\item[-] if $p<0$, then $(i)$ and the fact that $\ppi$ is an initial test plan that stays initially in $E$ ensures that for $t>0$ sufficiently small the first integral is finite (see also formula \eqref{Eq:Fubb} below), 
\item[-]  if $q<0$, then $(ii)$ ensures that for $t>0$ sufficiently small the second integral is $>-\infty$. 
\end{itemize}
To conclude it then  suffices to prove
\begin{align}\label{Eq:IINT}
{\liminf_{t\to 0} \frac{1}{pt}\iint_0^t \vert \rmd f\vert^\PP(\gamma_r)\,\d r\, \d\bdpi(\gamma) \geq \frac1p\int \vert\rmd f\vert^\PP(\gamma_0)\,\d\bdpi(\gamma).}
\end{align}
We start with the case $p<0$, so  that $\vert\rmd f\vert^p\in L^1(E,\meas\mres E)$ by assumption $(i)$. In this case, we claim in fact that equality holds in \eqref{Eq:IINT}. Let $\rho_r$ be the density of $(\eval_r)_\push\bdpi$ with respect to $\mathfrak{m}$, where $r\in [0,T]$. By Fubini's theorem  we have
\begin{align}\label{Eq:Fubb}
\frac{1}{t}\iint_0^t \vert \rmd f\vert^\PP(\gamma_r)\,\d r\,\d\bdpi(\gamma) =  \int_E \vert \rmd f\vert^\PP\, \Big[\frac{1}{t}\int_0^t \rho_r\,\d r\Big]\,\d\mathfrak{m}.
\end{align}
Since $\smash{(\restr_0^T)_\push\bdpi}$ is a test plan, we have $\smash{\sup_{r\in [0,T]} \Vert \rho_r\Vert_{L^\infty(\mathfrak{m})}  < +\infty}$. Since $\bdpi$ is an initial test plan, we have the narrow convergence $(\eval_r)_\push\bdpi\to(\eval_0)_\push\bdpi$ as $r\to 0$, which together with the uniform $L^\infty$ bound yields  $\rho_r \to \rho_0$ as $r\to 0$ in duality with $L^1(\M,\mathfrak{m})$. Since we assumed  $\vert \rmd f\vert^\PP\in L^1(E,\mathfrak{m}\mres E)$, this suffices to prove that equality holds --- with a full $\lim$ rather than a $\liminf$  --- in \eqref{Eq:IINT}.

In the case $p>0$ we instead argue as follows. Since $\mm$ is a Radon measure and the topology on $\M$ is Polish (and in particular has the Lindel\"of property), $\mm$ is $\sigma$-finite. Considering an increasing sequence of sets with finite measure covering $\M$ we can thus find an increasing sequence $(g_n)$ of $L^1(\M,\mm)$ functions such that $g_n(x)\uparrow+\infty$ for every $x\in \M$. Thus $(\tfrac1p|\d f|^p)\wedge g_n$ is in $L^1(\M,\mm)$ for every $n\in\N$ and the same arguments used above give
\begin{align*}
    \liminf_{t\to 0}\frac{1}{t}\iint_0^t \tfrac{\vert\rmd f\vert^p}p(\gamma_r)\d r\,\d\bdpi(\gamma)&\geq    \liminf_{t\to 0}\frac{1}{t}\iint_0^t g_n\wedge\big(\tfrac{\vert\rmd f\vert^p}p\big)(\gamma_r)\,\d r\,\d\bdpi(\gamma)= \int  g_n\wedge\big(\tfrac{\vert\rmd f\vert^p}p\big) \,\rho_0\,\d\meas.
\end{align*}
The conclusion follows letting $n\uparrow+\infty$ in the above and using Levi's monotone convergence theorem noticing that $ g_n\wedge\big(\tfrac{\vert\rmd f\vert^p}p\big)\uparrow  \tfrac{\vert\rmd f\vert^p}p$.
\end{proof}

Inspired by a similar definition for metric measure spaces \cite{Gigli:2015}, which is itself motivated by De Giorgi's approach to gradient flows in metric spaces (cf.\ e.g.\ \cite{AGS:08}), the following definition is naturally based on requiring the saturation of the inequality from \Cref{Pr:Lower bound}. In the smooth setting, the $\PP$-gradient of $f$ refers to the Hamiltonian gradient of $f$ given by the nonlinear identification of the cotangent and tangent spaces induced by the Hamiltonian 
\[
H(\omega):=\left\{\begin{array}{ll}
\tfrac1p\|\omega\|_*^p,&\qquad\text{ if $\omega$ is future directed,}\\
+\infty,&\qquad\text{ otherwise}.
\end{array}
\right.
\]

\begin{definition}[Representing the \fff $p$-gradient of a causal function]\label{Def:represent gradient} Let  $(\M,\uptau,\ell,\meas)$ be a  mm  spacetime,  $f\colon \M\to \bar{\R}$ be a causal function and $\PP^{-1}+\QQ^{-1}=1$ with $0 \neq p, \QQ < 1$.  We say that an initial test plan $\bdpi$ \emph{represents the \fff $\PP$-gradient of $f$} if:
\begin{itemize}
\item[i)] There is a Borel set $E\subset\M$ so that  $\smash{\bdpi}$  stays initially in $E$;
\item[ii)] for $E$ as above we also have $\vert \rmd f\vert^\PP\in L^1(E,\mathfrak{m} \mres E)$;
\item[iii)] $\smash{\limsup_{t\downarrow 0} \tfrac1t\iint_0^t\vert \dot\gamma_r\vert^\QQ\d r\d\bdpi(\gamma)  < +\infty}$; and
\item[iv)] 
\begin{align}\label{Eq:Subtract}
 \limsup_{t\downarrow 0}\int \frac{f(\gamma_t) - f(\gamma_0)}{t}\d\bdpi(\gamma)  \leq \frac{1}{\PP}\int \vert \rmd f\vert^\PP(\gamma_0)\d\bdpi(\gamma)
+ \liminf_{t\downarrow 0} \frac{1}{t\QQ}\iint_0^t\vert \dot\gamma_r\vert^\QQ\,\d r\,\d\bdpi(\gamma).
 \end{align}
\end{itemize}
In this case we also say that $\ppi$ represents the \fff $\PP$-gradient of $f$ on $E$.
\end{definition}
We collect some comments about the above:
\begin{enumerate}
\item We are asking $(ii),(iii)$ irrespective of the signs of $p$ and $q$ 
\item A consequence of this definition and of \Cref{Pr:Lower bound} is that $f\circ\eval_t-f\circ\eval_0\in L^1(\ppi)$ for $t>0$ sufficiently small. 
\item If $p<0$, the assumptions above --- in particular $|\d f|^p\in L^1(E)$ ---  do not exclude the possibility that  $|\d f|=+\infty$ on a subset of $E$ of positive measure. For instance, in the one point space any function is causal with $|\d f|=+\infty$ and the (only) test plan  is concentrated in the (only) constant curve and represents the \fff $\PP$-gradient of $f$.
\item If the \fff $\PP$-gradient of a given causal function $f$ is represented by an initial test plan $\bdpi$ on $E$, then
\[
 \lim_{t\to 0}\int \frac{f(\gamma_t) - f(\gamma_0)}{t}\d\bdpi(\gamma)  = \frac{1}{\PP}\int \vert \rmd f\vert^\PP(\gamma_0)\d\bdpi(\gamma)
+ \liminf_{t\to 0} \frac{1}{t\QQ}\iint_0^t\vert \dot\gamma_r\vert^\QQ\,\d r\,\d\bdpi(\gamma),
\]
as follows by combining \eqref{eq:FYnonsmooth} with  \eqref{Eq:Subtract}. In this case, imitating the proof of \cite[Prop.\ 3.11]{Gigli:2015} it is not hard to see that the limit $\lim_{t\to 0} \frac{1}{t}\iint_0^t\vert \dot\gamma_r\vert^\QQ\,\d r\,\d\bdpi(\gamma)$ exists
in \eqref{Eq:Subtract}.

\item From \Cref{Pr:Lower bound} we  see   that if $\ppi$ represents the \fff $\PP$-gradient of $f$ in $E$ and $\Gamma\subset \CC([0,1];\M)$ is Borel with $\ppi(\Gamma)>0$, then the ``restricted'' plan $\ppi_\Gamma:=\ppi(\Gamma)^{-1}\ppi\mres\Gamma$ also represents the \fff $\PP$-gradient of $f$ in $E$.
\end{enumerate}

\subsection{Perturbation of causal functions} 
This section serves as preparation for the subsequent one, where we will start from a causal function $f$, perturb it as $f+\eps g$ for some $g$ and be interested in those cases for which the new function is still $\mm$-measurable and causal. This motivates the next definition, which is the main object of study here:
\begin{definition}[Perturbation of causal  functions]\label{Def:Pert} Let  $(\M,\uptau,\ell,\meas)$ be a  mm spacetime and   $f:\M\to \bar \R$  a causal  function. We define the class $\Pert(f)$ of \emph{perturbations} of $f$ as the collection of all $\meas$-measurable 
functions $g\colon \M\to \bar{\R}$ such that $f+\varepsilon g$ is a 
causal function for some $\varepsilon > 0$.
\end{definition}
We think of $\pert(f)$ as the tangent space at $f$  to the space of causal functions.  Notice that according to our infinity conventions in \Cref{Sub:InfConv}, we have $\pm\infty+\eps g(x)=\pm\infty$ for any $g:\M\to\bar\R$ and $x\in \M$. Thus the role of $g$ is only seen in the causally convex set $\dom(f)$.

The  properties of $\Pert(f)$  are better understood if we introduce the following order relation on (not necessarily causal) functions on $\M$. \red{By slight abuse of notation, we denote it by $\preceq$ just like the causal relation between probability measures, cf.\ \Cref{Sub:Lifting}, and hope no confusion ensues.}

\begin{definition}[A pre-order on the space of functions via finite differences]\label{Def:A partial}
Let $(\M,\ell)$ be a metric spacetime. For $g\colon \M\to\bar\R$  define $\delta g\colon \M\times \M\to\bar \R$ as
\[
\delta g(x,y):=g(y)-g(x).
\]
We  say that  $g_2$ is at least as steep as $g_1$ and  write  $g_1\preceq g_2$ provided 
\[
\delta g_1(x,y)\leq\delta g_2(x,y)\quad\text{whenever } x\leq y.
\]
\end{definition}
An analogous definition will also be used in the case $\M=\bar\R$. 
The following are direct consequences of the definition.
\begin{itemize}
\item A function $f\colon \M\to \bar{\R}$ is causal  if and only if $\delta f(x,y)\geq 0$ whenever $x\leq y$ and if and only if $c\preceq f$, where $c$ is any constant, real-valued function.
\item Given an interval $I\subset\bar\R$ consider the function $f_I\colon \bar\R\to\bar\R$ defined as 0 on $I$, as $-\infty$ on the left of $I$ and $+\infty$ on the right of $I$. Then $f_I\preceq f_J$ if and only if $I\supset J$.
\item $\preceq$ is a relation about ``first order derivatives'' and not about ``pointwise values'' of the functions. In particular, $g_1\preceq g_2$ does not say anything about the 
sign of $g_1-g_2$ outside $\pm\infty$ 
unless $g_1(x) = g_2(x) \in \R$ at some  $x\in \M$ (in which case $(g_2(y)- g_1(y))(y- x) \ge0$ for all $y\in\R$). It can happen that $g_1\preceq g_2$ yet $g_1(x) > g_2(x)$ for every $x\in \M$. This is for instance the case for the interval $[0,1]$ endowed with the usual order relation $\leq$ and the causal functions $g_1(x) := x$ and $g_2(x) := 2x-2$. Rather, $g_1 \preceq g_2$ should be thought of as a monotonicity condition on the \emph{derivatives} of $g_1$ and $g_2$, as illustrated in the next  remark. 
\end{itemize}

\begin{remark}[Smooth differential characterization of $\preceq$]\label{Re:Firsto} Assume $(M,g)$ is a smooth spacetime. Given $x\in \M$ we define an order $\leq_x$ on the cotangent space $T^*_xM$ by $\omega_1\leq_x\omega_2$ if $\omega_1(v)\leq \omega_2(v)$ for every future-directed causal vector $v\in T_xM$ (see also \Cref{ss:compatibility}). Then for $f_1,f_2\in C^1(M)$, we have $f_1\preceq f_2$ if and only if their differentials satisfy $\rmd f_1(x) \leq_x \rmd f_2(x)$ for every $x\in M$.

For the ``only if'' direction, fix $x\in M$ and a future-directed causal vector $v\in T_xM$. Since $x\leq \exp_x(\eps v)$  for every $\eps>0$, we have
\begin{align*}
\rmd f_1(x)(v) &=\lim_{\eps\downarrow0}\frac{f_1(\exp_x(\eps v))-f_1(x)}{\eps}\leq \lim_{\eps\downarrow0}\frac{f_2(\exp_x(\eps v))-f_2(x)}{\eps}=\d f_2(x)(v).
\end{align*}
For the ``if'' direction,  let $x\leq y$ and $\gamma$ be a smooth causal curve from $x$ to $y$. Then
\[
f_1(y)-f_1(x)=\int_0^1\rmd f_1(\gamma_t)(\gamma'_t)\d t\leq \int_0^1\rmd f_2(\gamma_t)(\gamma'_t)\d t= f_2(y)-f_2(x).
\] 
Notice that we shall be able to somehow formulate the implication ``$f_1\preceq f_2$ implies  $\rmd f_1 \leq \rmd f_2$ a.e.'' also in the non-smooth setting: see  the monotonicity stipulated in \Cref{Pr:Calc df}(v) below and the second part of \Cref{le:boundbasso}. 

To motivate the non-smooth statement \Cref{le:boundbasso} below, we point out that the smooth duality formula $\|\omega\|_*=\inf_{|v|=1}\omega(v)$ valid for all $\omega\geq_x0$ easily yields  the implication
\begin{equation}
\label{eq:compsmoothdiff}
\min\{\d  f_1(v),\d f_2(v)\}\leq \d f_3(v)\quad\forall v\text{ future directed}\quad\Rightarrow\quad \min\{\|\d f_1\|_*,\|\d f_2\|_*\}\leq \|\d f_3\|_*
\end{equation}
 for every $f_1,f_2,f_3$ causal functions. On the other hand, it is easy to construct causal functions $f_1,f_2,f_3$ such that $\max\{\d  f_1(v),\d f_2(v)\}\geq \d f_3(v)$ for every future directed $v$, but for which $ \max\{\|\d f_1\|_*,\|\d f_2\|_*\}\ngeq \|\d f_3\|_*$ (we can take suitable linear functions in flat Minkowski space).
 \hfill$\blacksquare$
\end{remark}

The rest of this part is devoted to basic properties of $\preceq$ that will become relevant in the subsequent discussion. We start with the nonsmooth analog of implication \eqref{eq:compsmoothdiff} above:
\begin{lemma}[
Finite difference and 
maximal weak subslope 
comparison among causal functions]
\label{le:boundbasso} Let  $(\M,\uptau,\ell,\meas)$ be a mm  spacetime,   $f_1,f_2,f_3:\M\to\bar\R$  causal and such  that $\min\{\delta f_1(x,y),\delta f_2(x,y)\}\leq\delta f_3(x,y)$ holds for any $x\leq y$. 

Then $\min\{|\d  f_1|,|\d f_2|\}\leq |\d f_3|$ $\mm$-a.e. In particular, if $f,\tilde f$ are causal functions with $f\preceq \tilde f$, then $|\d f|\leq|\d \tilde f|$ $\mm$-a.e.
\end{lemma}
\begin{proof} 

Notice that for $\gamma$ causal, the set $\dom(f_3\circ\gamma)$ is a subinterval of $[0,1]$ and if it is not a point (if it is, the conclusion is trivial)  our assumption easily implies $\dom(f_3\circ\gamma)\subset\dom(f_1\circ\gamma)\cup \dom(f_2\circ\gamma)$.

Now recall that Lebesgue's differentiation theorem of monotone functions on $\R$ implies that if both the curve $\gamma$ and the function $f$ are causal, then the density $(f\circ\gamma)'$ of the distributional derivative of $f\circ\gamma$ coincides $\Leb^1$-a.e.\ on $\dom(f)$ with the classically derivative $\lim_{h\downarrow0}\tfrac{f(\gamma_{t+h})-f(\gamma_t)}h$ of $f\circ \gamma$. 

 Therefore our assumption gives $(f_3\circ\gamma)'_t\geq\min\{(f_1\circ\gamma)'_t,(f_2\circ\gamma)'_t\}$ for a.e.\ $t\in[0,1]$, where  we interpret $(f_i\circ\gamma)'_t$ to be $+\infty$ if $t\notin\dom(f_i\circ\gamma)$. In particular, this holds for $\ppi$-a.e.\ $\gamma$ for any given test plan $\ppi$, so that the conclusion follows from \Cref{Le:Distr der}.

The second claim now follows by picking $f_1\equiv 0$ in the first part.
 \end{proof}

\begin{lemma}[
Ordering respects addition of causal  functions]
\label{eq:useful} Let  $(\M,\uptau,\ell,\meas)$ be a mm spacetime,  $f\colon \M\to\bar{\R}$ causal  function and $g_1,g_2\colon \M\to\bar{\R}$ are arbitrary functions with $g_1\preceq g_2$, then $f+  g_1 \preceq f+  g_2$.
\end{lemma}

\begin{proof} 
The proof is a simple case analysis. Let $x$ and $y$ be fixed points with $x\leq y$. We need to prove that
\begin{equation}
\label{eq:usef2}
\delta(f+  g_1)(x,y)\leq \delta(f+  g_2)(x,y).
\end{equation}

If $f(y)=+\infty$ then $(f+ g_i)(y)=+\infty$ by our infinity conventions and thus $\delta (f+ g_i)(x,y)=+\infty$ by \eqref{eq:conventions1}, thus in this case \eqref{eq:usef2} holds. Similarly, if $f(x)=-\infty$ we have $(f+ g_i)(x)=-\infty$ and then $\delta (f+ g_i)(x,y)=+\infty$, thus also in this case \eqref{eq:usef2} holds.

Since $f$ is causal it remains to check the case $f(x),f(y)\in\R$, but in this case it is easily verified that --- according to \eqref{eq:conventions1} --- we have  $\delta(f+ g_i)(x,y)=\delta f(x,y)+\delta g_i(x,y)$, so the conclusion follows from the assumption $g_1\preceq g_2$.
\end{proof}

We also have:

\begin{proposition}[Properties of $\Pert(f)$]\label{Pr:prec} Let  $(\M,\uptau,\ell,\meas)$ be a mm spacetime, and   $f\colon \M\to\bar{\R}$ causal. Then:
\begin{enumerate}[label=\textnormal{(\roman*)}]
\item\emph{Cone.} We have $\lambda g\in\pert(f)$ for every $g\in\pert(f)$ and every $\lambda\geq 0$.
\item\label{it:cone2}\emph{Convex cone.} We have $g_1+g_2\in\pert(f)$ for every $g_1,g_2\in\pert(f)$.
\item\emph{Monotonicity.} If $g_1\in\pert(f)$ and $g_1\preceq g_2$ then $g_2\in\pert(f)$.
\item \emph{Lattice.} If $g_1,g_2\in\pert(f)$ then $\min\{g_1,g_2\},\max\{g_1,g_2\}\in\pert(f)$.
\item\label{it:comp} \emph{Comparison.} Let $g\colon \M\to \bar{\R}$ and $\varphi_1,\varphi_2,\psi:\bar\R\to\bar \R$ be functions satisfying $\varphi_1\preceq\psi\preceq\varphi_2$ and $\varphi_1\circ g,\varphi_2\circ g\in\pert(f)$. \textnormal{(}We do not necessarily assume $g\in\pert(f)$ here.\textnormal{)}  Then $\psi\circ g\in\pert(f)$ and
\[
\begin{split}
\min\{\delta (f+\eps\,\varphi_1\circ g),\delta (f+\eps\,\varphi_2\circ g)\} (x,y)&\leq \delta (f+\eps\,\psi\circ g)(x,y)\\
&\leq\max\{\delta (f+\eps\,\varphi_1\circ g), \delta (f+\eps\,\varphi_2\circ g)\}(x,y).
\end{split}
\]
whenever $x\leq y$ and $\eps>0$. In particular,
\begin{equation}\label{eq:compd}
|\rmd (f+\eps\,\psi\circ g)|\geq \min\{|\rmd  (f+\eps\,\varphi_1\circ g)|,|\rmd (f+\eps\,\varphi_2\circ g)|\}\quad\meas\textnormal{-a.e.}
\end{equation}
\end{enumerate}
\end{proposition}

\begin{proof}\ 

(i) This is trivial.

(ii) This statement follows from the identity $
f+\eps(g_1+g_2)=(f+2\eps g_1)/2+(f+2\eps g_2)/2$ which is consistent with \eqref{eq:conventions1}, observing that $f+2\eps g_1$ and $f+2\eps g_2$ are both causal functions for sufficiently small $\varepsilon > 0$ thanks to $(i)$.

(iii) This is a consequence of \Cref{eq:useful} and the characterization of causal functions in terms of the finite difference operator $\delta$ mentioned after \Cref{Def:A partial}.

(iv) This follows from the identity $f+\eps \min\{g_1,g_2\}=\min\{f+\eps g_1,f+\eps g_2\}$ which is consistent with \eqref{eq:conventions1} and the fact that the minimum of causal functions is still a causal  function. Analogously, we show $\max\{g_1,g_2\}\in\pert(f)$.

(v) The inequality \eqref{eq:compd} follows from the first statement and \Cref{le:boundbasso}.

To show the first claim, we distinguish two cases.
\begin{itemize}
\item Assume $g(x)\leq g(y)$. Since $\varphi_1\preceq\psi\preceq\varphi_2$ we obtain 
\begin{align*}
\delta(\varphi_1\circ g)(x,y)\leq \delta(\psi\circ g)(x,y)\leq \delta(\varphi_2\circ g)(x,y).
\end{align*}
By applying \Cref{eq:useful} to the two-point space $\{x,y\}$ we deduce for $\eps>0$
\[
 \delta (f+\,\eps\varphi_1\circ g)(x,y)\leq \delta (f+\eps\psi\circ g)(x,y)\leq\delta (f+\eps\varphi_2\circ g)(x,y).
\]
\item Assume $g(x)\geq g(y)$. As above, this yields 
\begin{align*}
\delta(\varphi_2\circ g)(x,y)\leq \delta(\psi\circ g)(x,y)\leq \delta(\varphi_1\circ g)(x,y),
\end{align*}
and applying \Cref{eq:useful} to $\{x,y\}$ once more we deduce
\[
 \delta (f+\eps\varphi_2\circ g)(x,y)\leq \delta (f+\eps\psi\circ g)(x,y)\leq\delta (f+\eps\varphi_1\circ g)(x,y).\qedhere
\]
\end{itemize}
\end{proof}
In view of the chain and Leibniz rules stated in \Cref{Pr:Calc df} --- used especially in the proof of \Cref{Cor:Time dist II} as well as in \cite{braun+}, we  study the behavior of perturbations under composition and multiplication. The hypotheses are likely not optimal, but will suffice for our purposes.  

In the statement below we shall always extend the given function $\varphi:\R\to\R$ by putting  $\varphi(\pm\infty)=\pm\infty$.
\begin{lemma}[Composition and multiplication]\label{Pr:RelationsII} Let $(\M,\uptau,\ell,\mm)$ be a mm  spacetime and $f\colon \M\to\bar\R$ causal. Then:
\begin{enumerate}[label=\textnormal{\textcolor{black}{(}\roman*\textcolor{black}{)}}] 
\item \textnormal{Composition I.} Let  $\varphi\colon\R\to\R$ be Lipschitz continuous. Then $\smash{\varphi\circ f\in \Pert(f)}$. 
\item \textnormal{Composition II.} Let $\varphi\colon\R\to\R$ be $c$-steep (i.e.\ $\varphi(w)-\varphi(z)\geq (w-z)c$ for all $z < w$) for some $c>0$.  
Then  $\smash{\Pert(f)}\subset \smash{\Pert(\varphi\circ f)}$.  
\item \textnormal{Composition III.} Let $\smash{g\in \Pert(f)}$ and $\varphi\colon\R\to \R$ be non-decreasing and Lipschitz continuous. Then  $\smash{\varphi\circ g\in\Pert(f)}$.
\item \textnormal{Multiplication I.} Let $g,h\in \Pert(f)$ be non-negative and bounded. Then $g\,h\in\Pert(f)$.
\item \textnormal{Multiplication II.} Let $\smash{g\in\Pert(f)}$ be bounded and $\varphi\colon\R\to \R$ be non-negative, bounded and Lipschitz continuous. Then  $\smash{g\,\varphi\circ f \in \Pert(f)}$.
\end{enumerate}
\end{lemma}

\begin{proof} \

(i) If $\varphi$ is affine the statement is straightforward. For general Lipschitz continuous $\varphi$, we use the relation $-\textnormal{Lip}\, \varphi\,\textnormal{id} \preceq  \varphi \preceq \textnormal{Lip}\, \varphi\,\textnormal{id}$ directly implied from Lipschitz continuity and \Cref{Pr:prec}(v).

(ii) Note that $\varphi\circ f$ is a causal function, let $g\in\Pert(f)$ be such that  $f+g$ is causal (any $g \in \Pert(f)$ can be rescaled to satisfy this). We claim that $\varphi\circ f+cg$ is also causal, which clearly suffices to conclude. Pick  $x\leq y$, with $x,y\in \M$. We need to prove that $\varphi(f(y))+cg(y)\geq \varphi(f(x))+cg(x)$. If either $\varphi(f(y))=+\infty$ or $\varphi(f(x))=-\infty$ the claim holds by our infinity conventions. Since $\varphi(f(y))\geq\varphi(f(x))$, we can then assume $\varphi(f(y)),\varphi(f(x))\in\R$ and thus in particular $f(y),f(x)\in\R$.  If this is the case we have $\varphi(f(y))\geq \varphi(f(x))+cf(y)-cf(x)$ hence 
\[
\varphi(f(y))+cg(y)\geq \varphi(f(x))+cf(y)-cf(x)+cg(y)\geq \varphi(f(x))+cg(x),
\]
by our normalization of $g\in\Pert(f)$.

(iii) As in (i), the claim is clear if $\varphi$ is linear. The general case follows from the relation $0\preceq \varphi \preceq \textnormal{Lip}\,\varphi\,\textnormal{id}$, the causal property of $f$, and \Cref{Pr:prec}(v).

(iv)  Assume, by scaling, that $f+g,f+h$ are causal, let $x\leq y$ and $\varepsilon > 0$. As already noticed, to check that $(f+\eps gh)(y)\geq(f+\eps gh)(x)$ we can assume that $x,y\in\dom(f)$, otherwise the claim follows from our infinity conventions. Let $C$ be an upper bound on $g$ and $h$ on all of $\M$. Then
\begin{align*}
(f+\eps gh)(y)-(f+\eps gh)(x)&= f(y) - f(x) + \varepsilon\,g(y)\,\big[h(y) - h(x)\big] + \varepsilon\,h(x)\,\big[g(y) - g(x)\big]\\
&  \geq f(y) - f(x)+ C\varepsilon \max\{0, \textnormal{sgn}(h(x) - h(y))\}\,\big[h(y) - h(x)\big]\\
&\qquad + C\varepsilon \max\{0,\textnormal{sgn}(g(x)-g(y))\}\,\big[g(y)-g(x)\big],
\end{align*}
where the last inequality follows from a straightforward case distinction, using $g$ and $h$ are non-negative. It is now easy to check that for $\eps:=\tfrac1{2C}$  the right-hand side is non-negative, thus concluding the proof.

(v) It follows from (i) and (iv) that $(g - \inf g(\M))\,\varphi\circ f \in\Pert(f)$. Moreover, again (i) implies $\inf g(\M)\,\varphi\circ f\in\Pert(f)$. Since $\Pert(f)$ is closed under addition, the claim follows.
\end{proof}

\subsection{Vertical right-differentiation}\label{se:calcrules}
Recall that we put $u_p(z):=\tfrac1pz^p$ and that $u_p$ is suitably extended to all of $\bar\R$ in Section \ref{Sub:InfConv}.
\begin{definition}[Vertical right derivative]\label{Def:Vert} Let $(\M,\uptau,\ell,\mm)$ be a mm spacetime,  $f:\M\to\bar\R$ causal   and   $g\in\Pert(f)$. We  define the \emph{vertical right derivative} $\rmd^+g(\nabla f)\,|\rmd f|^{\PP-2}\colon \M\to\bar\R$ as
\begin{equation}\label{Eq:d+gdfdf}
\rmd^+g(\nabla f)\,\vert\rmd f\vert^{\PP-2} := \begin{cases}
     \displaystyle\lim_{\varepsilon \downarrow 0}\frac{u_\PP(\vert\rmd(f+\varepsilon g)\vert) - u_\PP(\vert\rmd f\vert)}{\varepsilon} & \textnormal{if }\vert\rmd f\vert <+\infty,\\
    0 & \textnormal{otherwise}.
\end{cases}
\end{equation}
\end{definition}
For the moment, the term $\smash{\rmd^+g(\nabla f)\,\vert\rmd f\vert^{\PP-2}}$ from \Cref{Def:Vert} should and will henceforth be considered as a single expression, not as the ``product'' of $\vert\rmd f\vert^{\PP-2}$ and (the not yet defined quantity) $\smash{\rmd^+ g(\nabla f)}$. For comments about the definition on $\{|\d f|=+\infty\}$ see \Cref{re:limitconventions}.

Let us  comment on the well-definedness of \eqref{Eq:d+gdfdf}. 
For $\eps>0$ sufficiently small   $f+\eps g$ is causal, thus the difference quotient is well defined. The concavity of the maximal weak subslope $\vert \rmd f\vert$ in $f$ (on the space of causal functions) from \Cref{Pr:Linear comb} and the fact that $u_\PP$ is concave and non-decreasing easily yield
\begin{align*}
\frac{u_p(\vert\rmd(f+\varepsilon_1 g)\vert) - u_p(\vert\rmd f\vert) }{\varepsilon_1}\geq \frac{u_p(\vert\rmd(f+\varepsilon_2 g)\vert) - u_p(\vert\rmd f\vert)}{\varepsilon_2}\quad\meas\textnormal{-a.e.}
\end{align*}
whenever $\varepsilon_2 > \varepsilon_1>0$ are sufficiently small. It follows that for every $\eps_n\downarrow0$ the functions $\frac{u_\PP(\vert\rmd(f+\varepsilon_n g)\vert) - u_\PP(\vert\rmd f\vert)}{\varepsilon_n}$ have a pointwise $\mm$-a.e.\ limit which is independent of the chosen sequence, and this limit $\d^+g(\nabla f)\,|\rmd f|^{\PP-2}$ satisfies
\begin{equation}
\label{eq:monotdgnf}
\d^+g(\nabla f)\,|\rmd f|^{\PP-2}=\meas\textnormal{-}\!\esssup_{\eps>0} \frac{u_\PP(\vert\rmd(f+\varepsilon g)\vert) - u_\PP(\vert\rmd f\vert)}{\varepsilon}\qquad \text{on }\{|\d f|<+\infty\}.
\end{equation}
From this it is easy to see  that  if $f,g$ are both causal, thus in particular $g\in\Pert(f)$, then the   following sort of reverse Cauchy--Schwarz inequality is in place:
\[
\rmd^+g(\nabla f)\,|\rmd f|^{\PP-2}\geq |\d g|\,|\d f|^{p-1}\quad\mm\textnormal{-a.e.}
\]
Indeed, the concave homogeneity in \eqref{eq:dsuper} gives $|\d(f+\eps g)|\geq|\d f|+\eps|\d g|$  and since $u_p$ is non-decreasing we have $u_p(|\d (f+\eps g)|)\geq u_p(|\d f|+\eps |\d g|)$. The conclusion follows from the above discussion  and the fact that $u_p'(z)=z^{p-1}$. Similar considerations give
\[
\d^+f(\nabla f)|\d f|^{p-2}=|\d f|^p\quad\meas\textnormal{-a.e.\ on }\{|\d f|<+\infty\},
\]
where here we need to restrict to the set $\{|\d f|<+\infty\}$ to avoid conflicts, in the case $p\in(0,1)$, with our choice $\rmd^+g(\nabla f)\,|\rmd f|^{\PP-2}=0$ on $\{|\d f|=+\infty\}$, see also \Cref{re:limitconventions}.

We also point out that according to our infinity conventions in \Cref{Sub:InfConv} we have
\begin{equation}
\label{eq:dfzero1}
\rmd^+g(\nabla f)\,|\rmd f|^{\PP-2}=+\infty,\quad\mm\textnormal{-a.e.\ on }\{|\d f|=0\}\quad\text{ if }p<0
\end{equation}
and, taking also into account the concavity in \Cref{Pr:Linear comb},
\begin{equation}
\label{eq:dfzero2}
\rmd^+g(\nabla f)\,|\rmd f|^{\PP-2}=
\begin{cases}
+\infty,&\quad\mm\textnormal{-a.e.\ on }\{|\d f|=0\}\cap E(f,g),\\ 
0,&\quad\mm\textnormal{-a.e.\ on }\{|\d f|=0\}\setminus E(f,g),
\end{cases}\qquad \text{ if $p\in(0,1)$,}
\end{equation}
where $E(f,g):=\cup_{\eps>0}\{|\d(f+\eps g)|>0\}=\cup_{n\in\N}\{|\d(f+\tfrac1n g)|>0\}$.

For later use it is worth demonstrating also the following. If $f:\M\to\bar\R$ is causal and $g\in\Pert(f)$, say $f+\bar\eps g$ is causal for $\bar\eps>0$, then  we claim the uniform bound from below
\begin{equation}
\label{eq:unifdiffquot}
\frac{u_\PP(\vert\rmd(f+\varepsilon g)\vert) - u_\PP(\vert\rmd f\vert)}{\varepsilon}\geq \begin{cases}
-\tfrac1{\bar\eps}\,u_p(|\d f|)&\qquad\text{ if }p\in(0,1),\\
\tfrac{2^{1-p}}{\bar\eps}\,u_p(|\d f|)&\qquad\text{ if }p<0,
\end{cases}\qquad\forall \eps\in(0,\tfrac12\bar\eps).
\end{equation}
Indeed, the bound for $p\in(0,1)$ follows from the monotonicity in $\eps$ of the left hand side together with the positivity of $u_p$. For $p<0$, the same monotonicity together with $u_p<0$ give $\frac{u_\PP(\vert\rmd(f+\varepsilon g)\vert) - u_\PP(\vert\rmd f\vert)}{\varepsilon}\geq \frac{u_\PP(\vert\rmd(f+{\bar\varepsilon}g/2)\vert)}{{\bar\varepsilon}/2}$. Then, since $f+\bar\eps g$ is causal, from \eqref{eq:dsuper} we get $\vert\rmd(f+{\bar\varepsilon}g/2)\vert\geq \tfrac12|\d f|$ and the claim follows from the monotonicity of $u_p$.

The following two propositions collect some basic calculus rules. Compare with those valid in positive signature obtained in \cite[Section 3.3]{Gigli:2015}.

\begin{proposition}[Calculus rules `for $g$']\label{Pr:Calc df} Let $(\M,\uptau,\ell,\mm)$ be a mm spacetime and   $f\colon \M\to \bar{\R}$ causal. Then
\begin{enumerate}[label=\textnormal{(\roman*)}]
\item\label{it:homg} \emph{Positive $1$-homogeneity}. For $g\in \pert(f)$ and every $\lambda\in(0,+\infty)$ we have
\begin{equation}
\label{eq:homg}
\d^+(\lambda g)(\nabla f)\,|\rmd f|^{\PP-2}=\lambda\, \d^+g(\nabla f)\,|\rmd f|^{\PP-2}\quad\meas\textnormal{-a.e.}
\end{equation}
For $\lambda=0$ the same holds $\mm$-a.e.\ on $\{|\d f|>0\}$.
\item\label{it:conc} \emph{Super-additivity.} For every $g_1,g_2\in\pert(f)$ we have
\begin{equation}
\label{eq:superaddg}
\rmd^+(g_1+g_2)(\nabla f)\,|\rmd f|^{\PP-2} \geq \rmd^+g_1(\nabla f)\,|\rmd f|^{\PP-2}+\rmd^+g_2(\nabla f)\,|\rmd f|^{\PP-2}\quad\mm\textnormal{-a.e.}
\end{equation}
\item\label{it:compleft}\emph{Comparison from the left.} Let $g:\M\to\bar\R$ be such that $\pm g\in\Pert(f)$. Then
\begin{equation}
\label{eq:dpdm}
\rmd^+g(\nabla f)\,|\rmd f|^{\PP-2}\leq -\,\rmd^+(-g)(\nabla f)\,|\rmd f|^{\PP-2}\quad\mm\textnormal{-a.e.\ on }\{|\d f|>0\}.
\end{equation}
\item\label{it:locg} \emph{Locality.} Let $g_1,g_2\in\Pert(f)$ and $E\subset\R$ Borel and negligible. Then  
\begin{equation}
\label{eq:locg}
\rmd^+g_1(\nabla f)\,|\rmd f|^{\PP-2}= \rmd^+g_2(\nabla f)\,|\rmd f|^{\PP-2}\quad\mm\textnormal{-a.e.\ on }(g_1-g_2)^{-1}(E).
\end{equation}
\item\label{it:mong} \emph{Monotonicity.} If $g_1,g_2\in\pert(f)$ satisfy   $g_1\preceq g_2$, we have
\[
\rmd^+g_1(\nabla f)\,|\rmd f|^{\PP-2}\leq \rmd^+g_2(\nabla f)\,|\rmd f|^{\PP-2}\quad\mm\textnormal{-a.e.}
\]
\item\label{it:chaing1} \emph{Chain rule I.} Let $g\in\pert(f)$ and let $\varphi\colon\R \to \R$ be a non-decreasing and Lipschitz function. Then 
\begin{equation}
\label{eq:chaing1}
\rmd^+(\varphi\circ g)(\nabla f)\,|\rmd f|^{\PP-2}=\varphi'\circ g\, \rmd^+g(\nabla f)\,|\rmd f|^{\PP-2}
\end{equation}
holds $\mm$-a.e.\ on $\{\varphi'\circ g>0\}\cup\{|\d f|>0\}$. Note that $\varphi\circ g\in\Pert(f)$ by \Cref{Pr:RelationsII}.
\item\label{it:chaing2} \emph{Chain rule II.} Let $\varphi\colon\R\to\R$ be Lipschitz continuous. Then
\begin{equation}
\label{eq:chaing2}
\rmd^+(\varphi\circ f)(\nabla f)\,\vert\rmd f\vert^{\PP-2} = \varphi'\circ f\,\vert\rmd f\vert^\PP
\end{equation}
holds $\mm$-a.e.\ on $\{|\d f|<+\infty\}\cap(\{\varphi'\circ f>0\}\cup\{|\d f|>0\})$. Note that $\varphi\circ f\in\Pert(f)$ by \Cref{Pr:RelationsII}.
\item\label{it:leib1} \emph{Leibniz rule I.} Let $g,h\in\Pert(f)$ be non-negative and bounded. Then 
\begin{equation}
\label{eq:leib1}
\rmd^+(g\,h)(\nabla f)\,\vert\rmd f\vert^{\PP-2} \geq g\,\rmd^+ h(\nabla f)\,\vert\rmd f\vert^{\PP-2} + h\,\rmd^+g(\nabla f)\,\vert\rmd f\vert^{\PP-2}\quad\meas\textnormal{-a.e.\ on }\{|\d f|>0 \}.
\end{equation}
Note that $gh\in\Pert(f)$ by \Cref{Pr:RelationsII}.
\item\label{it:leib2} 
\emph{Leibniz rule II.} Let $g\in\Pert(f)$ be bounded 
and $\varphi\colon\R\to\R$ be non-negative, bounded and Lipschitz continuous. Then
\begin{equation}
\label{eq:leib2}
\rmd^+(g\,\varphi\circ f)(\nabla f)\,\vert\rmd f\vert^{\PP-2} \geq g\,\varphi'\circ f\,\vert\rmd f\vert^\PP + \varphi\circ f\,\rmd^+g(\nabla f)\,\vert\rmd f\vert^{\PP-2}\quad\meas\textnormal{-a.e.\ on }\{|\d f|\in(0,+\infty)\}.
\end{equation}
Note that $g\varphi\circ f\in\Pert(f)$ by \Cref{Pr:RelationsII}.
\end{enumerate}
\end{proposition}
Note: the latter two Leibniz rules hold as equalities in the smooth case.
\begin{proof}\ 

\noindent\ref{it:homg} 
Follows directly from definition \eqref{Eq:d+gdfdf}.

\noindent\ref{it:conc} 
On $\{|\d f|=+\infty\}$ the conclusion is trivial, on $\{|\d f|=0\}$ it follows from \eqref{eq:dfzero1} and \eqref{eq:dfzero2}. On $\{|\d f|\in(0,+\infty)\}$ we argue as follows: for $\lambda\in(0,1)$ the concavity in \eqref{eq:dsuper} gives
\begin{align*}
\vert \rmd(f+ \varepsilon((1-\lambda) g_1 + \lambda g_2))\vert \geq (1-\lambda)\,\vert\rmd(f+\varepsilon g_1)\vert + \lambda\,\vert\rmd(f+\varepsilon g_2)\vert,
\end{align*}
thus using the concavity and non-decreasingness of $u_\PP$ we get
\[
u_p\big(\vert \rmd(f+ \varepsilon((1-\lambda) g_1 + \lambda g_2))\vert\big) \geq (1-\lambda)\,u_p(\vert\rmd(f+\varepsilon g_1)\vert) + \lambda\,u_p(\vert\rmd(f+\varepsilon g_2)\vert).
\]
The claim follows taking also $(i)$ into account.

\noindent\ref{it:compleft} On $\{|\d f|=+\infty\}$ both sides are 0 by definition, so we focus on $\{|\d f|\in(0,+\infty)\}$. For $\eps>0$ sufficiently small, the concavity in \eqref{eq:dsuper} gives $|\d f|\geq\tfrac12|\d(f+\eps g)|+\tfrac12|\d(f-\eps g)|$
thus using the concavity and non-decreasingness of $u_\PP$ we get
\[
u_p\big(|\d f|)\geq \tfrac12u_p(|\d(f+\eps g)|)+\tfrac12u_p(|\d(f-\eps g)|)
\]
and therefore (notice that  $u_p(|\d f|)\in\R$ on $\{|\d f|\in(0,+\infty)\}$  to justify cancellations) we have
\[
\begin{split}
\d^+g(\nabla f)|\d f|^{p-2}&=\meas\textnormal{-}\!\esssup_{\eps>0} \frac{u_\PP(\vert\rmd(f+\varepsilon g)\vert) - u_\PP(\vert\rmd f\vert)}{\varepsilon}\\
&\leq -\meas\textnormal{-}\!\esssup_{\eps>0} \frac{u_\PP(\vert\rmd(f-\varepsilon g)\vert) - u_\PP(\vert\rmd f\vert)}{\varepsilon}=-\d^+(-g)(\nabla f)|\d f|^{p-2}.
\end{split}
\]

\noindent\ref{it:locg} Direct consequence of  the locality in \Cref{Th:Locality} and definition \eqref{Eq:d+gdfdf}.

\noindent\ref{it:mong} On $\{|\d f|=+\infty\}$ the claim is obvious, so we focus on $\{|\d f|<+\infty\}$. The relation $g_1 \preceq g_2$ implies $f+\eps g_1\preceq f+\eps g_2$ for every $\eps>0$ by \Cref{eq:useful}. For $\eps>0$ sufficiently small both these functions are causal, thus the second claim in Lemma \ref{le:boundbasso} yields $|\d(f+\eps g_1)|\leq |\d(f+\eps g_2)|$. Then the conclusion follows from the fact that $u_p$ is non-decreasing.

\noindent\ref{it:chaing1} Let us assume for the moment that, with the stated assumptions, it holds
\begin{equation}
    \label{eq:comparisonchaing}
    \rmd^+(\varphi\circ g)(\nabla f)\,|\rmd f|^{\PP-2}\geq \varphi'\circ g\, \rmd^+g(\nabla f)\,|\rmd f|^{\PP-2}
\end{equation}
and let us see how to conclude from here.

Define $\psi:=({\rm Lip}(\varphi)+1){\rm id}-\varphi$, so that $\psi$ is also non-decreasing and Lipschitz with $\psi'>0$ a.e. Thus  from \eqref{eq:homg}, $\mm$-a.e.\ on $\{\varphi'\circ g>0\}\cup\{|\d f|>0\}$ we can write
\[
\begin{split}
({\rm Lip}(\varphi)+1)\,\rmd^+ g(\nabla f)\,|\rmd f|^{\PP-2}&=\rmd^+(\psi\circ g+\varphi\circ g)(\nabla f)\,|\rmd f|^{\PP-2}\\
\text{(by \eqref{eq:superaddg})}\qquad&\geq\rmd^+(\psi\circ g)(\nabla f)\,|\rmd f|^{\PP-2}+\rmd^+(\varphi\circ g)(\nabla f)\,|\rmd f|^{\PP-2}\\
\text{(by $\geq$ in \eqref{eq:chaing1})}\qquad&\geq \psi'\circ g\,\rmd^+ g(\nabla f)\,|\rmd f|^{\PP-2}+\varphi'\circ g\,\rmd^+ g(\nabla f)\,|\rmd f|^{\PP-2}\\
&=({\rm Lip}(\varphi)+1)\,\rmd^+ g(\nabla f)\,|\rmd f|^{\PP-2}.
\end{split}
\]
Thus the inequalities are in fact equalities, and on the set $\{\rmd^+ g(\nabla f)\,|\rmd f|^{\PP-2}<+\infty\}$ this can only occur if each of the inequalities we used was in fact an equality, yielding the desired conclusion. On $\{\rmd^+ g(\nabla f)\,|\rmd f|^{\PP-2}=+\infty\}$ the conclusion is obvious on $\{\varphi'\circ g>0\}$ (because we already know that $\geq$ holds in \eqref{eq:chaing1}), so --- inspecting the claim --- to conclude it suffices to deal with the set $\{\varphi'\circ g=0\}\cap\{|\d f|\in(0,+\infty)\}$ and prove that  $\mm$-a.e.\ in there we have $|\d(f+\eps \varphi\circ g)|=|\d f|$. This, however, is clear from the characterization in \Cref{Le:Distr der}, as for $\gamma\in\CC([0,1];\M)$ and $t\in[0,1]$ with $s\mapsto (\varphi(g(\gamma_s)),f(\gamma_s))$ differentiable at $t$ (this holds for $\ppi$-a.e.\ curve and $\Leb^1$-a.e.\ $t$ for any test plan $\ppi$) and $\varphi$ differentiable at $g(\gamma_t)$ with null derivative, we have $((f+\eps \varphi\circ g)(\gamma))_t'=(f\circ\gamma)_t'+\eps \varphi'(g(\gamma_t))(g\circ\gamma)_t'=(f\circ\gamma)_t'$.

We are thus left to proving \eqref{eq:comparisonchaing}. Fix  an interval $I=[a,b]\subset(0,+\infty)$  and then a compact set $K\subset(\varphi')^{-1}(I)$. We are going to prove that
\begin{equation}
    \label{eq:confrontoKeps}
     \rmd^+(\varphi\circ g)(\nabla f)\,|\rmd f|^{\PP-2}\geq \min\{a\, \rmd^+g(\nabla f)\,|\rmd f|^{\PP-2},b\, \rmd^+g(\nabla f)\,|\rmd f|^{\PP-2}\}\qquad\mm-a.e.\ on\ f^{-1}(K),
\end{equation}
which by the arbitrariness of $[a,b]$ and $K$ suffices to conclude. We can assume that $K$ is not empty, or else the claim is obvious. Thus fix $\bar z\in K$ and then consider the auxiliary function $\tilde\varphi:\R\to\R$ defined by $\tilde\varphi(\bar z)=\varphi(\bar z)$ and $\tilde\varphi'=1_K\varphi'+a1_{\R\setminus K}$. Since $\R\setminus K$ is open, and thus a countable union of disjoint intervals $(I_n)$, we see that the set $E:=\{\varphi(z)-\tilde\varphi(z)\ :\ z\in K\}$ is countable, as for any $z\in K$  the difference $\varphi(z)-\tilde\varphi(z)$ is equal to the sum of $\int_{I_n}\varphi'-a\,\d\mathcal L^1$ over those intervals $I_n$ that are entirely contained in either $[z,\bar z]$ or $[\bar z,z]$ (depending on whether $z>\bar z$ or $z<\bar z$). It follows from \eqref{eq:strongloc} that 
\begin{equation}
    \label{eq:varphitildevarphi}
|\d(f+\eps\varphi\circ g)|=|\d(f+\eps\tilde\varphi\circ g)|\qquad\mm-a.e.\ on \ K.
\end{equation}
To conclude, let $\varphi_1(z):=az$ and $\varphi_2(z):=bz$ for any $z\in\R$, so that $\varphi_1\preceq\varphi\preceq\varphi_2$. Then \eqref{eq:compd} ensures that
\[
|\d(f+\eps\tilde\varphi\circ g)|\geq\min\{|\d(f+\eps a g)|,|\d(f+\eps b g)|\} 
\]
and the monotonicity of $u_p$ yields
\[
     \rmd^+(\tilde\varphi\circ g)(\nabla f)\,|\rmd f|^{\PP-2}\geq \min\{a\, \rmd^+g(\nabla f)\,|\rmd f|^{\PP-2},b\, \rmd^+g(\nabla f)\,|\rmd f|^{\PP-2}\}\qquad\mm-a.e.,
\]
that together with \eqref{eq:varphitildevarphi} gives \eqref{eq:confrontoKeps} and the claim.
\noindent\ref{it:chaing2} For $\eps<\tfrac1{{\rm Lip}(\varphi)}$, by \Cref{Th:Chain rule II} we have $|\d(f+\eps\varphi\circ f)|=|\d(({\rm id}+\eps\varphi)\circ f)|=(1+\eps \varphi')\circ f|\d f|$. The conclusion follows by direct computation also noticing that $\{|\d f|<+\infty\}\subset  {\rm Vis}(\M)$.

\noindent\ref{it:leib1} On $\{g=0\}\cup\{h=0\}$ the conclusion follows from \eqref{eq:homg}, hence by a locality argument based on \eqref{eq:locg} to conclude it suffices to prove \eqref{eq:leib1} in the set $\{g,h\in[c,c^{-1}]\}$ for any given $c>0$. Fix such $c$ and notice that by \Cref{Pr:RelationsII} the functions  $\tilde g:=c\vee g\wedge (c^{-1})$ and $\tilde h:=c\vee h\wedge (c^{-1})$ still belong to $\Pert(f)$. Since $z\mapsto\log(z)$ is Lipschitz on $[c^2,+\infty)$  we can apply  \eqref{eq:chaing1} and get
\[
\begin{split}
\tfrac1{\tilde g\tilde h}\d^+(\tilde g\tilde h)(\cdots)=\d^+(\log(\tilde g\tilde h))(\cdots)&=\d^+(\log(\tilde g)+\log (\tilde h))(\cdots)\\
\text{(by \eqref{eq:superaddg})}\qquad&\geq \d^+(\log(\tilde g))(\cdots)+\d^+(\log (\tilde h))(\cdots)=\tfrac1{\tilde g}\d^+\tilde g(\cdots)+\tfrac1{\tilde h}\d^+\tilde h(\cdots),
\end{split}
\]
where $(\cdots)$ is shorthand for $(\nabla f)|\d f|^{p-2}$. Using again the locality property \eqref{eq:locg} the claim follows.

\noindent\ref{it:leib2} Let $c\geq 0$ be such that $g+c\geq 0$. Then we have
\[
\begin{split}
\d^+(g\varphi\circ f)(\cdots)&=\d^+((g+c)\varphi\circ f+(-c\varphi\circ f))(\cdots)\\
\text{(by \eqref{eq:superaddg})}\qquad&\geq \d^+((g+c)\varphi\circ f)(\cdots)+\d^+(-c\varphi\circ f)(\cdots)\\
\text{(by \eqref{eq:leib1} and \eqref{eq:chaing2})}\qquad&\geq (g+c)\d^+(\varphi\circ f)(\cdots)+ \varphi\circ f\d^+(g+c)(\cdots)-c\varphi'\circ f |\d f|^p\\
\text{(by \eqref{eq:locg} and \eqref{eq:chaing2})}\qquad&= (g+c)\varphi'\circ f|\d f|^p+ \varphi\circ f\d^+g(\cdots)-c\varphi'\circ f |\d f|^p
\end{split}
\]
and the conclusion follows because the term $c\varphi'\circ f |\d f|^p$ is finite and can be canceled out.
\end{proof}

\begin{proposition}[Calculus rules `for $f$'] 
\label{P:f calculus}
Let $(\M,\uptau,\ell,\mm)$ be a mm spacetime and   $f\colon \M\to \bar{\R}$ causal. Then:
 \begin{enumerate}[label=\textnormal{(\roman*)}]
 \item\label{it:homf}\emph{Positive $p-1$ homogeneity.} Let   $g\in\Pert(f)$ and $\lambda\in(0,+\infty)$. Then
 \begin{equation}
\label{eq:homf}
 \d^+g(\nabla (\lambda f))|\d (\lambda f)|^{p-2}=\lambda^{p-1}\d^+g(\nabla f)|\d f|^{p-2}\quad\mm\textnormal{-a.e.}
 \end{equation}
 Notice that $g\in\Pert(\lambda f)$ (trivially).
 \item\label{it:locf}\emph{Locality.} Let  $\tilde f:\M\to\bar\R$ be a causal function, $g\in\Pert(f)\cap\Pert(\tilde f)$  and $E\subset\R$ Borel and negligible. Then
 \begin{equation}
\label{eq:locf}
\d^+g(\nabla f)|\d f|^{p-2}=\d^+g(\nabla \tilde f)|\d \tilde f|^{p-2} \quad\mm\textnormal{-a.e.\ on }(f-\tilde f)^{-1}(E).
\end{equation}
\item\label{it:chainf} \emph{Chain rule.} Let $\varphi\colon\R\to\R$ be $c$-steep (i.e.\ $\varphi(w)-\varphi(z)\geq (w-z)c$ for every $z\leq w$) for some $c>0$ and let  $g\in\Pert(f)$. Then
\begin{equation}
\label{eq:chainf}
\rmd^+g(\nabla(\varphi\circ f))\,\vert\rmd(\varphi\circ f)\vert^{\PP-2} = (\varphi')^{\PP-1}\circ f\,\rmd^+g(\nabla f)\,\vert\rmd f\vert^{\PP-2}\quad\mm\textnormal{-a.e.}.
\end{equation}
Note that $g$ belongs to $\Pert(\varphi\circ f)$ by \Cref{Pr:RelationsII}.
 \end{enumerate}
\end{proposition}

\begin{proof}\ 

\ref{it:homf} For $\lambda>0$ the claim follows $\mm$-a.e. from 
the positive $1$-homogeneity in \Cref{Pr:Linear comb}:
\begin{align*}
 \d^+g(\nabla (\lambda f))|\d (\lambda f)|^{p-2}
 &=
 \lambda^{p-1} \lim_{\eps\downarrow 0} \frac{|\d(f+g\eps/\lambda|^p - |\d f|^p}{\eps/\lambda}
 =\lambda^{p-1}\d^+g(\nabla f)|\d f|^{p-2}.
\end{align*}
 

\ref{it:locf} Direct consequence of the definition and \Cref{Th:Locality}.

\ref{it:chainf} It is readily proved that for $\psi:\R\to\R$ with $\varphi\preceq\psi$ we have $\delta(\varphi\circ f+\eps g)\leq \delta(\psi\circ f+\eps g)$ and therefore, by \Cref{le:boundbasso}, that $|\d(\varphi\circ f+\eps g)|\leq |\d (\psi\circ f+\eps g)|$. Taking into account the  monotonicity of $u_p$ we then deduce that
\begin{equation}
    \label{eq:perchainf}
\d^+g(\nabla(\varphi\circ f))|\d(\varphi\circ f)|^{p-2}\leq\d^+g(\nabla(\psi\circ f))|\d(\psi\circ f)|^{p-2}\qquad\mm-a.e.
\end{equation}
The conclusion now follows as for the proof of \eqref{eq:chaing1}: we first use a locality argument to replace the original $\varphi$ with a $\tilde\varphi$ so that $\varphi'\in[a,b]\subset(0,+\infty)$, then we use 
\eqref{eq:perchainf} to compare the desired quantity with that appearing with the choices $\varphi_a(z):=az$ and $\varphi_b(z):=bz$. To conclude we then observe that for linear $\varphi$'s the conclusion is given by \eqref{eq:homf}.
\end{proof}

\begin{remark}[Limits and usefulness of conventions]\label{re:limitconventions} The --- arbitrary --- choice of defining $\d^+g(\nabla f)|\d f|^{p-2}$ to be 0 on $\{|\d f|=+\infty\}$ is only motivated by simplicity, as in any case we won't ever really care about what happens to such an expression where $|\d f|=+\infty$; setting to zero won't create integrability problems later on.  A side effect of this choice is that some formulas, such as \eqref{eq:chaing2} only hold on $\{|\d f|<+\infty\}$: we are aware of this fact, that causes no troubles in what comes next.

In a conceptually similar direction, 
we have defined the homogeneity formulas \eqref{eq:homg}  and \eqref{eq:homf} --- and thus with the chain rules, that ultimately depend on these --- with care to avoid conflicts
with \eqref{eq:dfzero1} or \eqref{eq:dfzero2}
on the set $\{|\d f|=0\}$.

In contrast with the situation up to the previous chapter, we haven't found a consistent set of conventions capable of handling all possible scenarios, and in any case having these would be irrelevant for the point we want to make, which is that in all interesting cases the relevant formulas hold `as is', without the need to resort to artificial choices.
\hfill$\blacksquare$
\end{remark}

\begin{remark}[Sign choices]\label{re:sign}  In line with the terminology in \cite{Gigli:2015}, the quantity $-\d^+(-\varphi)(\nabla f)|\d f|^{p-2}$ appearing in \eqref{eq:dpdm} might also be called $\d^-\varphi(\nabla f)|\d f|^{p-2}$. We haven't done so both to avoid introducing a new notation and because in our signature this would lead to the bizarre looking inequality $\d^+\varphi(\cdots)\leq \d^-\varphi(\cdots)$ where the `minus' term is bigger than the `plus' term. 

Speaking of this, we emphasize that  working with $\d^+\varphi(\cdots)$ rather than with the above-mentioned $\d^-\varphi(\cdots)$ is related to the fact that we shall work with the $\TMCP^\h_+$ condition rather than with the $\TMCP^\h_-$ one. In this direction it might be worth to keep in mind the following table
\[
\begin{array}{lcr}
\TMCP^\h_+(K,N)\qquad&vs&\qquad \TMCP^\h_-(K,N)\\
\d^+\varphi(\cdots)\qquad&vs &\qquad \d^-\varphi(\cdots)\\
\varphi\in\Pert(f)\qquad&vs&\qquad-\varphi\in\Pert(f),\\
\text{initial test plans}\qquad&vs& \qquad\text{final test plans},\\
\text{functions of the form }\psi_c\qquad&vs& \qquad\text{functions of the form }\psi^c
\end{array}
\]
(see \eqref{eq:defcconc} for the definition of $\psi_c$ and $\psi^c$) and notice that the first choice dictates all the others below it. This means that if we work on $\TMCP^\h_+(K,N)$ spaces, as we shall do most of the time, then the d'Alembertian comparison will be obtained for functions of the form $\psi_c$ by differentiating  functions $\varphi\in\Pert(f)$ along initial test plans and studying the quantity $\d^+\varphi(\nabla f)|\d f|^{p-2}$. Symmetrically if we work on $\TMCP^\h_-(K,N)$ spaces, see Section \ref{Sub:Modiback} for comments.

\medskip

Of totally different kind and  \emph{not} related to the above, is the choice  of dealing with forward spacetimes rather than backward ones, i.e.\ spacetimes that are forward complete rather than backward complete (recall  \Cref{D:spacetimes}). We shall always stick to this choice, so that in particular the results concerning $\TMCP^\h_-$ spaces  --- even though will be proved along similar lines of thought --- cannot be derived by a simple time-reversal.
\hfill$\blacksquare$
\end{remark}

\subsection{Relation of horizontal and vertical derivatives}

Here we discuss the connection between these two different notions of differentiation for arbitrary causal functions.

Our motivation is drawn once more from the case of a smooth spacetime $(M,g)$. Given a smooth causal function $f\colon M \to \R$ and any smooth and compactly supported function $g\colon M\to\R$, the quantity $\smash{\rmd g(\nabla f)\,\| \rmd f\|_*^{\PP-2}}$ can be computed in two ways. The first, related to our horizontal approach, is to consider a smooth curve $\gamma$ with initial speed $\smash{ {\gamma}'_0 = \|\rmd f\|_*^{\PP-2}\,\nabla f}$ and differentiate at zero:
\begin{equation*}
\frac{\rmd}{\rmd t}g(\gamma_t) \Big\vert_0 = \rmd g(\nabla f)\,\big\|\rmd f\big\|_*^{\PP-2}.
\end{equation*}
The second, related to our vertical approach, is to differentiate the $\varepsilon$-dependent quantity $\smash{| \rmd(f+\varepsilon g)|_*^\PP/\PP}$ (in other words, the underlying Hamiltonian applied  to the covector field $\d (f+\varepsilon g)$) at zero:
\begin{align*}
\,\frac{\rmd}{\rmd \varepsilon} \frac{\big\|\rmd(f+\varepsilon g)\big\|_*^\PP}{\PP}\Big\vert_0  = \rmd g(\nabla f)\,\big\|\rmd f\big\|_*^{\PP-2}.
\end{align*}
This last identity can be seen recalling that $\nabla f$ is --- inspecting the definition of the musical isomorphisms --- equal to the differential of $\tfrac12\|\cdot\|_*^2$ applied at $\d f$, see also \Cref{ss:compatibility}.  Since the right-hand sides of the last two formulas agree, so do the left ones. 

Remarkably, this consideration has a counterpart   in our context.  Recall the definition of `plan $\ppi$ representing the initial $p$-gradient of a causal function $f$' from \Cref{Def:represent gradient}.

\begin{theorem}[First-order differentiation formula]\label{thm:horver} Let $(\M,\uptau,\ell,\mm)$ be a mm spacetime,  $f\colon \M\to\bar\R$ causal, $0\neq \PP<1$  and let $\bdpi$ represent the \fff $\PP$-gradient of $f$.

 Then for every   $g\in\Pert(f)$ we have:
 \begin{enumerate}[label=\textnormal{(\roman*)}]
 \item\label{horver-i} 
 For any $t>0$ sufficiently small   the negative part of the function $g\circ\eval_t-g\circ\eval_0$ is in $L^1(\ppi)$,
 \item\label{horver-ii} 
 and
\begin {equation}
\label{eq:horverder}
\limi_{t\downarrow0}\int\frac{g(\gamma_t)-g(\gamma_0)}{t}\,\d\bdpi(\gamma)
 \geq \int \rmd^+g(\nabla f)\,|\rmd f|^{\PP-2}(\gamma_0)\,\d\bdpi(\gamma).
\end{equation}
 \item\label{horver-iii} 
 If $E\subset \M$  Borel so that $\ppi$ represents the initial $p$-gradient of $f$ on $E$, then the negative part of the function $\rmd^+g(\nabla f)\,|\rmd f|^{\PP-2}$ is in $L^1(\mm\mres E)$.
 \end{enumerate} 
 \end{theorem}
\begin{proof} Definition \ref{Def:represent gradient} implies, as already noticed right after it, that for $t>0$ sufficiently small we have $f\circ\eval_t-f\circ\eval_0\in L^1(\ppi)$ (and in particular the function is a.e.\ finite). On the other hand, for $\bar\eps>0$ we have that $f+\eps g$ is causal for any $\eps\in[0,\bar\eps]$, so that in particular $(f+\bar\eps g)\circ\eval_t-(f+\bar\eps g)\circ\eval_0$ is non-negative on $\CC([0,1];\M)$. Writing $g=\tfrac1{\bar\eps}((f+\bar\eps g)-f)$ the claim \ref{horver-i} easily follows. The claim \ref{horver-iii} follows from the uniform bound \eqref{eq:unifdiffquot} and the assumption $u_p(|\d f|)\in L^1(E,\mm\mres E)$ that comes with \Cref{Def:represent gradient}.

\ref{horver-ii} We now want to apply \Cref{Pr:Lower bound} to the function $f+\eps g$ for   $\eps\in(0,\tfrac12\bar\eps)$ and the plan $\ppi$; to this aim let us check the assumptions. If $q<0$ the requirement for $\ppi$ is satisfied, as the plan satisfies \Cref{Def:represent gradient}. If instead $p<0$ we need to check that $u_p(|\d(f+\eps g)|)\in L^1(E,\mm\mres E)$. As in the proof of the bound \eqref{eq:unifdiffquot}, this follows from the bound $0\geq u_p(|\d (f+\eps g)|)\geq u_p(\tfrac12|\d f|)$ and again the assumption $u_p(|\d f|)\in L^1(E,\mm\mres E)$.

We can therefore apply  \Cref{Pr:Lower bound} to the function $f+\eps g$, thus writing the corresponding inequality \eqref{eq:FYnonsmooth} for $f+\eps g$ and subtracting \eqref{Eq:Subtract} for $f$, after the --- justified --- cancellations we get
\[
\liminf_{t\to 0} \int \frac{g(\gamma_t) - g(\gamma_0)}{t}\,\d\bdpi(\gamma) \geq \int \frac{\vert\rmd(f+\varepsilon\,g)\vert^\PP(\gamma_0) - \vert\rmd f\vert^\PP(\gamma_0)}{\PP\eps}\,\d\bdpi(\gamma)\qquad\forall \eps\in(0,\tfrac12\bar\eps).
\]
The uniform bound \eqref{eq:unifdiffquot} and the assumption $u_p(|\d f|)\in L^1(E,\mm\mres E)$ ensure that we can use  Levi's monotone convergence theorem  to pass to the limit as $\eps\downarrow0$. The conclusion follows.
\end{proof}

We now analyze how the calculus developed above improves on infinitesimally Minkowskian spaces (recall \Cref{Def:Inf Minkow} from the introduction).  We  claim in this setting, the bound \eqref{eq:horverder} upgrades to 
a genuine first order differentiation formula equating the
horizontal and vertical derivatives,
\eqref{eq:firstorderdiff} below.

\begin{theorem}[Elements of calculus 
on infinitesimally Minkowskian spaces]
\label{T:VD=HD} Let $(\M,\uptau,\ell,\meas)$ be a forward metric measure spacetime that is infinitesimally Minkowskian and $0\ne p<1$.  Then:
\begin{enumerate}[label=\textnormal{(\roman*)}]
\item\label{it:dpdm} Let $f:\M\to\bar\R$ be causal and $\pm g\in\Pert (f)$. Then 
\begin{equation}
\label{eq:dpdm2}
-\rmd^+(-g)(\nabla f)\,|\rmd f|^{\PP-2} = \rmd^+g(\nabla f)\,|\rmd f|^{\PP-2}\qquad\mm-a.e.\ on\ \{|\d f|>0\};
\end{equation}
(compare with  inequality \eqref{eq:dpdm}).

\item\label{it:horder} 
If, in addition, $\ppi$ represents  the \fff $\PP$-gradient of $f$ on the Borel set $E$ then
\begin{equation}
\label{eq:firstorderdiff}
\lim_{t\downarrow0}\int\frac{g(\gamma_t)-g(\gamma_0)}{t}\,\d\bdpi(\gamma) = \int \rmd^+g(\nabla f)\,|\rmd f|^{\PP-2}(\gamma_0)\,\d\bdpi(\gamma).
\end{equation}
\item\label{it:scalprod} For $f,g,h:\M\to\bar\R$ causal define $\d^+ g(\nabla f)$ on the set $\{|\d f|\in(0,+\infty)\}$  as the product of $|\d f|^{2-p}$ and $\d^+g(\nabla f)|\d f|^{p-2}$. Then this quantity does not depend on $p$;
moreover, for $\alpha,\beta \in[0,\infty)$, both
\begin{equation}
\label{eq:dgnfscal}
\begin{split}
\d^+g(\nabla f)&=\d^+f(\nabla g) 
\\ {\rm and} \quad \d^+(\alpha g+\beta h)(\nabla f)&=\alpha\d^+g(\nabla f)+\beta \d^+h(\nabla f)
\end{split}
\end{equation}
hold $\mm$-a.e.\ on the set $\{x \in \M : 0 < |\d f| \wedge |\d g| \wedge |\d h| {\rm\ and}\ |\d(f+g+h)| <\infty\}$.

\end{enumerate}
\end{theorem}
\begin{proof}\

\noindent\ref{it:dpdm}  Let $\bar\eps>0$ be such that $f+\eps g$ is causal for every $\eps\in[-\bar\eps,\bar\eps]$. On $\{|\d f|=+\infty\}$ there is nothing to prove, as both sides are 0 by convention. On the set $\{|\d f|\in(0,+\infty)\}$, by the smoothness of $u_p$ on $(0,+\infty)$ it suffices to prove that
\begin{equation}
\label{eq:claimsimm}
\lim_{\eps\downarrow0}\frac{|\d(f+\eps g)|^2-|\d f|^2}{2\eps}=\lim_{\eps\downarrow0}\frac{|\d f|^2-|\d(f-\eps g)|^2}{2\eps},
\end{equation}
where the limits are intended as in \eqref{Eq:d+gdfdf} and their existence follows from the existence of the limit in \eqref{Eq:d+gdfdf}. To see the above write the defining identity \eqref{eq:definfmink} with $f-\eps g,f+\eps g$ in place of $f,g$ respectively, for $\eps\in(-\bar\eps,\bar\eps)$, to get
\[
2|\d(f-\eps g)|^2+2|\d (2f)|^2=|\d(f+\eps g)|^2+|\d(3f-\eps g)|^2\qquad\mm-a.e.
\]
Rearranging and recalling the positive homogeneity in \Cref{Pr:Linear comb} we get
\[
|\d(f+\eps g)|^2-|\d f|^2=2\big(|\d (f-\eps g)|^2-|\d f|^2\big)-9\big(|\d (f-\tfrac\eps3 g)|^2-|\d f|^2\big)
\] 
and --- since both  limits in  \eqref{eq:claimsimm} exist --- the identity \eqref{eq:claimsimm} easily follows.

\noindent\ref{it:horder} Applying   \Cref{thm:horver} to  $-g$ in place of $g$ we obtain
 \begin{align*}
\limsup_{t\to 0} \int \frac{g(\gamma_t) - g(\gamma_0)}{t}\d\bdpi(\gamma) \leq -\int \rmd^+(-g)(\nabla f)\,|\rmd f|^{\PP-2}(\gamma_0)\d\bdpi(\gamma).
\end{align*}
The conclusion follows combining  \Cref{thm:horver} as stated with item \ref{it:dpdm} above.





\noindent\ref{it:scalprod} For $f,g$ causal define the auxiliary function $B(f,g)$ on $\{0<|\d f|\wedge|\d g| \}\cap\{
|\d(f+ g)|<\infty \}$ as

\[
2B(f,g):=|\d (f+g)|^2-|\d f|^2-|\d g|^2.
\]
Then $B(f,g)=B(g,f)$ and is finite $\mm$-a.e.\ by \eqref{eq:dsuper}. 
We claim that for any $f,g,h$ causal
\begin{equation}
\label{eq:Blin}
B(f+h,g)=B(f,g)+B(h,g)
\end{equation}
holds where they are all defined. From the parallelogram law \eqref{eq:definfmink} defining infinitesimal Minkowskianity we see that
\[
\begin{split}
4|\d(f+h+g)|^2+4|\d h|^2&=2|\d(f+2h+g)|^2+2|\d(f+g)|^2,\\
2|\d(f+2h+g)|^2+2|\d f|^2&=|\d(2f+2h+g)|^2+|\d(2h+g)|^2,\\
2|\d(h+g)|^2+2|\d h|^2&=|\d(2h+g)|^2+|\d g|^2,\\
2|\d(f+h+g)|^2+2|\d(f+ h)|^2&=|\d(2f+2h+g)|^2+|\d g|^2.
\end{split}
\]
The quantities on the left of each line are finite,  hence the same is true of the quantities on the right.
Adding the first two identities and subtracting the last two we get \eqref{eq:Blin}. It is then clear that 
\begin{equation}
\label{eq:Bhom}
B(\alpha f,g)=\alpha B(f,g)
\end{equation}
first for $\alpha\in\N$, then for $\alpha\in \Q_{+}$ and finally, using that  $|\d (\alpha f+g)|\leq|\d (\beta f+g)|\leq|\d (\gamma f+g)|$ holds $\mm$-a.e.\ whenever $0\leq \alpha\leq\beta\leq\gamma$ as a consequence of \eqref{eq:dsuper}, we conclude that \eqref{eq:Bhom} holds for $\alpha\in[0,+\infty)$.

To conclude, let $f,g$ be causal, notice that   on the set $\{|\d f|\in(0,+\infty)\}$ the limits in \eqref{eq:claimsimm} are equal to $\d^+g(\nabla f)$ (from the discussions in the previous item) and therefore
\[
\d^+g(\nabla f)=\lim_{\eps\downarrow0}\frac{|\d(f+\eps g)|^2-|\d f|^2}{2\eps}=\lim_{\eps\downarrow0}\frac{B(f,\eps g)}{\eps}\stackrel{\eqref{eq:Bhom}}=B(f,g)
\]
and the conclusion follows.
\end{proof}

\begin{remark}[Lorentzian a.e.~inner product from infinitesimal Minkowskianity] 
In an infinitesimally Minkowskian forward spacetime,  we can regard \eqref{eq:dgnfscal}
as defining a symmetric and positively bilinear a.e.\ inner product on the convex cone of causal cotangent fields $\d f$ and $\d g$, at least where $0< |\rmd f|\wedge|\rmd g|$ and $|\rmd(f+g)|<\infty$.\hfill$\blacksquare$
\end{remark}

\section{Effects of timelike Ricci curvature assumptions}\label{Section: Effects of curvature assumption}\label{se:ricci}

\subsection{Synthetic timelike curvature-dimension bounds}
We introduce here the  curvature conditions we are going to work with in the rest of the paper. They involve (slight modifications of) the Timelike Measure Contraction Property  proposed  by Braun~\cite{Braun:2023Renyi}, 
and the entropic variant of Cavalletti--Mondino \cite{CM:20} upon which it is based. Although these conditions are expected to become equivalent
under suitable non-branching hypotheses, this is not yet established in Lorentzian signature.  Braun's formulation provides sharp constants, while Cavalletti and Mondino's yields narrow compactness in merely forward spacetimes via Lemma \ref{L:narrow coercivity} below.  In the absence of global hyperbolicity we are going to need both.  (When $(\M,\uptau,\ell,\mm)$ is globally hyperbolic on the other hand, all our results remain valid even if we replace  Boltzmann's entropy \eqref{BS2} by the trivial entropy $\scrS_\infty(\mu):=0$;  Lemma \ref{L:narrow coercivity} would remain true since the compactness of emeralds $E\subset \M$ which holds in this case implies tightness --- hence narrow compactness --- of 
the entire set 
$\Prob(E)$.%
)

For measures $\mu \in\Prob(E)$ supported on an emerald $E \subset \M$, the 
Boltzmann-Shannon entropy $\scrS_\infty \colon \Prob_\emr(\M)\to [-\log \mm(E),\infty]$ is 
defined by
\begin{align}\label{BS2}
    \scrS_\infty(\mu) :=
    \begin{cases}
\int_M \rho \log \rho \d\meas &\text{if}\ \mu=\rho\,\meas 
    \\ +\infty & \text{otherwise}. 
    \end{cases}
\end{align}
It is well-defined if $\mm(E)<\infty$ and its values lie in the indicated range by Jensen's inequality.  Cavalletti and Mondino's $\TMCP^e(K,N)$ condition asks for suitable growth bounds on $\scrS_\infty(\mu_t)$ along strongly timelike $\ell_q$-geodesics starting or ending at a point mass.  Braun's $\TMCP(K,N)$ 
asks for analogous decay estimates on the $N$-Rényi entropy $\scrS_N \colon \Prob(\M)\to [-\infty,0]$ defined by
\begin{align*}
    \scrS_N(\mu) := -\int_M \rho^{(N-1)/N}\d\meas\qquad\text{ for }\quad \mu=\rho\,\meas + \mu^\perp,\quad \mu^\perp\perp\mm.
\end{align*}
Both are detailed in Definition \ref{Def:TMCP} and \Cref{re:vardef} below.

For $\kappa\in\R$ define the generalized sine function $\sin_\kappa:\R\to\R$ as the only smooth $f$ solving 
\[
\begin{split}
f''+\kappa f=0,\qquad f(0)=0,\ f'(0)=1,
\end{split}
\]
so that
\begin{align*}
    \sin_\kappa(\theta) := \begin{cases}
        \kappa^{-1/2} \sin(\kappa^{1/2}\,\theta)
        & \textnormal{if } \kappa > 0,\\
        \theta & \textnormal{if }\kappa=0,\\
        |\kappa|^{-1/2} \sinh(|\kappa|^{1/2}\,\theta)
        & \textnormal{otherwise}.
    \end{cases}
\end{align*}
Then for $K\in\R$ and $N\in (1,+\infty)$  define two distortion coefficients as follows: for $\theta\geq 0$ and $t\in[0,1]$, set
\begin{align}\label{sigma}
    \sigma_{K,N}^{(t)}(\theta) := \begin{cases}
        \displaystyle\frac{\sin_{K/N}(t\theta)}{\sin_{K/N}(\theta)} & \textnormal{if }K\theta^2 < N\pi^2,\\
        +\infty & \textnormal{otherwise}
    \end{cases}
\end{align}
and, by geometrically averaging,
\begin{align}\label{tau}
    \tau_{K,N}^{(t)}(\theta) := t^{1/N}\,\sigma_{K,N-1}^{(t)}(\theta)^{(N-1)/N}.
\end{align}
In particular, when $K=0$ the distortion coefficients simply become $\sigma_{0,N}^{(t)}(\theta) = \tau_{0,N}^{(t)}(\theta) = t$.

The interpretation of the average \eqref{tau} is that in the smooth setting, the term $t^{1/N}$ measures the volume distortion in all ``tangential'' directions (which does not see curvature), while ``orthogonal'' directions are influenced by curvature and quantified by $\smash{\sigma_{K,N-1}^{(t)}(\theta)^{(N-1)/N}}$.

We now come to main definition of this section. Recall from \Cref{prop:geodindq}  that under the stated assumption the notion of $\ell_q$-geodesic to a Dirac mass does not depend on the chosen $0\neq q<1$.

\begin{definition}[Hybrid timelike measure contraction property]\label{Def:TMCP} 


Let  $(\M,\uptau,\ell,\meas)$ be a mm  spacetime  in which $\ell$ is upper semicontinuous and does not take the value $+\infty$. We say that it satisfies the   \emph{hybrid future timelike measure contraction property} $\smash{\TMCP^\h_+(K,N)}$ 
if for every   $\mu = \rho\, \meas \in \pem$  and   $x_1\in\supp\meas$  with $\log(\ell(\cdot,x_1))\in L^\infty(\mu)$
and $\scrS_\infty(\mu)<\infty$, there is an $\ell_q$-geodesic $(\mu_t)$ from $\mu$ to $\delta_{x_1}$ such that
\begin{align}\label{TMCP}
\scrS_{N}(\mu_t ) &\leq -\int \tau_{K,N}^{(1-t)}\circ\ell(\cdot,x_1)\,\rho^{1-1/N}\d\meas ,\qquad\forall t\in[0,1] \qquad\text{and}
\\ \label{TMCPe}
\scrS_{\infty}(\mu_t ) &\le \scrS_\infty(\mu) -N \log \sigma_{K,N}^{(1-t)}\big( \|\ell(\cdot, x_1)\|_{L^2(\d\mu)}\big),\qquad\forall t\in[0,1].
\end{align}
Similarly, $(\M,\ell,\meas)$ satisfies the \emph{hybrid past timelike measure contraction property} $\TMCP_-^\h(K,N)$ 
if for every   $\mu = \rho\, \meas \in \dom(\scrS_\infty)\subset \pem$  and   $x_0\in\supp\meas$  with $\log(\ell(x_0,\cdot))\in L^\infty(\mu)$, there is an $\ell_q$-geodesic $(\mu_t)$ from $\delta_{x_0}$ to $\mu$ such that
\begin{align}
\label{TMCP-}
\scrS_{N}(\mu_t ) &\leq -\int \tau_{K,N}^{(t)}\circ\ell(x_0,\cdot)\,\rho^{1-1/N}\d\meas ,\qquad\forall t\in[0,1] \qquad\text{and}
\\ \scrS_{\infty}(\mu_t )&\le \scrS_\infty(\mu) 
-N \log \sigma_{K,N}^{(t)}\big( \|\ell(x_0,\cdot)\|_{L^2(\d\mu)}\big),\qquad\forall t\in[0,1].
\label{TMCPe-}
\end{align}
\end{definition}

\begin{remark}[Related timelike measure contraction properties]\label{re:vardef}
Requiring \eqref{TMCP} but not \eqref{TMCPe}
--- or more precisely, replacing \eqref{BS2} with the trivial entropy $\scrS_\infty:=0$ ---
would yield
a future 
version $\TMCP_+(K,N)$ 
of Braun's timelike measure contraction property.  
In this case, replacing the $\tau$-distortion coefficients by the $\sigma$-distortion coefficients in \eqref{TMCP} yields the \emph{reduced future timelike measure contraction property} $\TMCP_+^*(K,N)$.
Similarly a future version $\TMCP_+^e(K,N)$ 
of Cavalletti and Mondino's timelike measure contraction property 
is obtained by requiring \eqref{TMCPe} but not \eqref{TMCP}
or equivalently setting $\scrS_N:=-\infty$.
Past versions of these various conditions are defined analogously.
In all that follows,  the hybrid assumption $\TMCP^\h_\pm(K,N)$ can be relaxed to $\TMCP_\pm(K,N)$ whenever $(\M,\uptau,\ell,\mm)$ is globally hyperbolic.

$\TMCP^\h_+(K,N)$ is stronger than $\smash{\TMCP_+^*(K,N)}$ \cite[Prop.~3.6]{Braun:2023Renyi} and allows to prove  all quantitative properties we care about (especially our comparison theorems) in their sharp form. However, with appropriate modifications that we occasionally specify, our results have evident adaptations to the more general setting $\smash{\TMCP_+^{\h,*}(K,N)}$ where the distortion coefficients $\tau$ are replaced by $\sigma$ in \eqref{TMCP}.

Notice also that --- in line with the definition in positive signature \cite{Ohta:2007}, \cite{Sturm:2006b} ---  the $\TMCP^\h_+(K,N)$ as formulated is not `dimensionally consistent', meaning that a priori a space could be $\TMCP^\h_+(K,N)$ but not $\TMCP^\h_+(K,N')$ for some $N'> N$. To enforce this property one should --- tautologically --- work with spaces that are $\TMCP^\h_+(K,N')$ for all $N'\geq N$ (and similarly for 
$\TMCP^{\h,*}_+(K,N)$ and the past versions of both).  
In timelike $\QQ$-essentially non-branching spaces, however, dimensional consistency does hold, as can be shown by localizing the inequalities defining $\TMCP_+^h(K,N)$ and $\TMCP_+^{h,*}(K,N)$ to a pathwise inequality following the lines of Braun \cite[Thms.~4.20, 4.21]{Braun:2023Renyi} and taking appropriate powers. In particular,  $\smash{\TMCP_+^{h,*}(K,N)}$ then recovers the entropic or reduced TMCP conditions of Cavalletti--Mondino \cite{CM:20} and Braun \cite{Braun:2023Renyi}. Moreover, in this situation we expect the  two hybrid conditions $\smash{\TMCP^h(K,N)}$ and $\smash{\TMCP^{h,*}(K,N)}$  to be equivalent, as suggested by corresponding results of Cavalletti--Sturm \cite{cavalletti-sturm} in positive signature.
\hfill$\blacksquare$
\end{remark}

\begin{remark}[Failure to reach endpoints]
Typically one assumes that the geodesic $(\mu_t)$ in the above definition satisfies $\mu_0=\mu$ and $\mu_1=\delta_{x_1}$, however in the current setting it seems us more natural  to work with this slightly weaker variant \eqref{eq:defgeo}. One of the reasons is that it seems more likely that this notion is stable by convergence of mm spacetimes (compare with the closure of $\CGeo(\M)$ established in \Cref{C:CGeo is closed}). 
\hfill$\blacksquare$\end{remark}

Notice that by nature of the definition, the $\TMCP^\h_+(K,N)$ (resp.\ $\TMCP^\h_-(K,N)$) condition  is only relevant for those points having non-empty chronological future (resp.\ past). In other words, introduce the `timelike final/initial sets' $\M_{\sf fin}$ and ${\sf \M_{in}}$   of $\M$  as
 \[
\begin{split}
 \M_{\sf {fin}}:=\{x\in \M\ :\ I^+(x)=\emptyset\},\qquad\text{and}\qquad \M_{\sf {in}}:=\{x\in \M\ :\ I^-(x)=\emptyset\}
\end{split}
 \]
 and notice that these are evidently achronal but possibly empty. If the topology of $\M$ contains the chronological one, then $\M_{\sf fin}$ and $\M_{\sf in}$ are automatically closed; moreover, the closure of any achronal set is still achronal. It is also clear that 
\begin{equation}
\label{eq:changem}
\left.\begin{array}{ll}
\text{$(\M,\uptau,\ell,\mm)$ is $\TMCP^\h_+(K,N)$} \\
\text{$\M_{\sf {fin}}$ is $\mm$-measurable}
\end{array}\right\}\qquad
\Rightarrow\qquad \text{$(\M,\uptau,\ell,\mm\mres{(\M_{\sf {fin}})^c})$ is $\TMCP^\h_+(K,N)$}
\end{equation}
and similarly for $\TMCP^\h_-(K,N)$ spaces. As we are going to see in the next section (\Cref{cor:fmeastmcp}), on $\TMCP^\h_+(K,N)$ spaces achronal sets are $\mm$-negligible (with the possible exception of $\M_{\sf {fin}}$), thus \eqref{eq:changem}  is relevant in connection to $\mm$-measurability of causal functions (recall \Cref{L:chronological continuity}).

\subsection{Good geodesics}

Our Sobolev calculus is built on the concept of test plans, that in turn asks for relevant measures to have bounded compression. On the other hand, the curvature condition we are asking only imposes entropy bounds. There is therefore a regularity gap that needs to be filled 
to link the geometry encoded in the curvature assumption with the analysis on the space as we are developing it. In positive signature, such a link was provided by  Rajala in \cite{Rajala2012,Rajala2013} in the setting of $\CD$ spaces: one of its first effects was to simplify the axiomatization of the $\RCD$ condition (making it depending only on the $\CD$ assumption plus infinitesimal Hilbertianity) along the lines suggested in \cite[Remark 4.20]{Gigli:2015} and put forward in \cite{AmbrosioGigliMondinoRajala12}. Rajala's construction was extended to $\MCP(K,N)$ spaces by Cavalletti-Mondino in \cite{cavalletti-mondino2017}. This latter result was later `Lorentzified' in \cite{Braun:2023Good}. In this section we shall show how it can be adapted to our current setting. The aim is to show existence of $\ell_\QQ$-geodesics satisfying not only the entropy bounds encoded in the $\TMCP^\h$ assumption, but also suitable $L^\infty$-density bounds. As is well known to experts, the construction involves 3 steps:
\begin{itemize}
\item[1)] `{\sc One step estimates}' Prove  existence of an intermediate point obeying the desired bounds.
\item[2)] `{\sc Discrete iteration}' Construct a  discretization of the desired geodesic.
\item[3)] `{\sc Passage to the limit}' Pass to the limit in the discretization to get the desired geodesic.
\end{itemize}
To implement this in our setting, we notice two differences w.r.t.\ \cite{Braun:2023Good}: one concerns the version of $\TMCP$ condition adopted, which has only a minor impact on the proof, and a conceptually deeper one related to our choice to give up global hyperbolicity in favor of forward completeness.

 Because of this passage from `compactness' to `completeness', the execution of the above plan, that would be more or less standard in a globally hyperbolic setting, becomes more involved and for this reason we give some detail below on how to proceed. This will involve step (3) in particular, where we will make crucial use of  the forward-narrow  completeness of the space of measures established in \Cref{prop:pmforward}. 
An additional ingredient is instead needed to adapt Rajala's construction to the forward context and get step (1) done.
For this we rely on the following lemma, already known to experts:

\begin{lemma}[Narrow coercivity of Boltzmann entropy]\label{L:narrow coercivity}
Fix an emerald $E=J(C_0',C_1')$ with $\mm(E)<\infty$ in a metric measure spacetime $(\M,\uptau,\ell,\meas)$. 
For each $c \in \R$, the following set is narrowly compact:
\begin{equation}\label{entropy sublevel}
S_c := \{ \mu \in \Prob(E) : \scrS_\infty(\mu)\leq c\}.
\end{equation}
\end{lemma}

\begin{proof}
Since $\meas_E := \meas \mres E$ assigns finite mass to the Polish space $\M$, there is a sequence of compact sets $C_i \subset C_{i+i} \subset \M$ which exhausts the mass of $\meas_E$.
As in the discussion below \cite[Eq. (3.4)]{GigliDGG:23}, since the function $x\mapsto x\log x$ is convex and greater than $-1$, Jensen's inequality can be used to estimate Boltzmann's entropy $\mathscr{S}_\infty(\mu)$ from below by $-\mm_E(E\setminus B)+\mu(B)\log\tfrac{\mu(B)}{\meas_E(B)}$ for any Borel set $B$. Hence, using \eqref{entropy sublevel}, the mass of $\mu \in S_c$ outside $C_i$ (i.e. setting $B=M\setminus C_i)$ can be estimated by
$$
\mu(\M \setminus C_i)\le \frac
{1+ c+\mm_E(C_i)}{-\log{\meas_E(M\setminus C_i)}}\to 0
$$
as $i \to \infty$ at a rate independent of $\mu$.   This shows $S_c$ is tight,  hence Prokhorov's theorem yields the desired narrow compactness.
\end{proof}


\medskip

In the discussion below we shall need the notion of $\lambda$-intermediate measure between a measure $\mu\in\Prob(\M)$ and a Dirac mass $\delta_{\bar x}$ at a point   $\bar x\in \M$ such that $\log(\ell(\cdot,\bar x))\in L^\infty(\mu)$. Given such $\mu$ and $\bar x$ and $\lambda\in(0,1)$ the set $\intm_\lambda(\mu,\delta_{\bar x})$ is defined to consist of those $\nu\in\Prob(\M)$ with $\mu\preceq\nu\preceq\delta_{\bar x}$ for which there is $\pi\in\Pi_\leq(\mu,\nu)$ such that
\begin{equation}
\label{eq:defintm}
\ell(x,y)=\lambda\ell(x,\bar x)\qquad\text{and}\qquad\ell(y,\bar x)=(1-\lambda)\ell(x,\bar x)\qquad\pi-a.e.\ (x,y).
\end{equation}
Equivalently $\nu \in \intm_\lambda(\mu,\delta_{\bar x})$
if and only if 
$$
\ell_q(\mu,\nu)=\lambda\ell_q(\mu,\delta_{\bar x})\qquad\text{and}\qquad\ell_q(\nu,\delta_{\bar x})=(1-\lambda)\ell_q(\mu,\delta_{\bar x}).
$$
The definition of $\ell_q$ and reverse triangle inequality make it clear that $\intm_\lambda(\mu,\delta_{\bar x})$ is convex.
(Although not needed here, in a geodesic spacetime, Propositions \ref{prop:geodindq} and \ref{P:heredity of geodesy} and their proofs can also be used to show a measure $\nu \in \intm_\lambda(\mu,\delta_{\bar x})$ is --- for any $0\neq q<1$ --- the value at $\lambda$ of a  
$\ell_q$-geodesic from $\mu$ to $\delta_{\bar x}$; c.f. \cite[Lemma 2.8]{McCann:2020} or \cite[Theorem 2.11]{CM:20}.)

The definition clearly implies that  $\log(\ell(\cdot,\bar x))\in L^\infty(\nu)$ as well, so  the concept can be iterated. Then   notice that a simple gluing argument together with the reverse triangle inequality yield
\begin{equation}
\label{eq:iterintm}
\nu\in\intm_{\lambda_1}(\mu,\delta_{\bar x}),\qquad \sigma\in\intm_{\lambda_2}(\nu,\delta_{\bar x})\qquad\Rightarrow\qquad \sigma\in\intm_{\lambda_3}(\mu,\delta_{\bar x}),
\end{equation}
for $(1-\lambda_3)=(1-\lambda_1)(1-\lambda_2)$.

\begin{proposition}[Step 1 --- one step estimates]\label{Pr:Int pts} 
 Let $(\M,\uptau,\ell,\meas)$ be  a   $\smash{\TMCP^\h_+(K,N)}$ mm  spacetime,   $\mu_0 = \rho_0\,\meas\in \dom (\scrS_\infty) \subset \pem$ and $x_1\in \supp\meas$ satisfy $\log(\ell(\cdot,x_1))\in L^\infty(\mu_0)$. Assume $\lambda\in(0,1)$ and $\mm(E)<\infty$ for some emerald $E$ containing $\{x_1\} \cup \spt \mu_0$.

 Then there is   $\nu_\lambda \in \intm_\lambda(\mu_0,\delta_{x_1})$ so that putting $D:=\sup\ell(\supp\mu_0\times\{x_1\})$ we have
 \begin{subequations}
\begin{align}
\label{eq:1stepbound}
\nu_\lambda&\leq  \tfrac{1}{(1-\lambda)^N}\,\mathrm{e}^{Dt\sqrt{K_-(N-1)}}\,\Vert \rho_0\Vert_{L^\infty(\meas)}\mm;
\\ \label{eq:1stepentbound}
\scrS_{N}(\nu_\lambda)&\leq  -\int \tau_{K,N}^{(1-\lambda)}\circ\ell(\cdot,x_1)\,\rho_0^{1-1/N}\d\meas;
\\ \label{eq:1stepBS}
\scrS_\infty(\nu_\lambda) &\le \scrS_\infty(\mu_0) 
- N \log \sigma^{(1-t)}_{K,N}\big(\|\ell(\cdot,x_1)\|_{L^2(\mu_0)}\big).
\end{align}
\end{subequations}
\end{proposition}
\begin{proof} 
The proof is a variant  of the clever manipulations in \cite{Rajala2012,Rajala2013}: we shall sketch it in the simplified case $K=0$, where computations are more transparent. We start noticing  that for $\mu_0,x_1$ as in the assumption, the set 
$$
\intm^{K,N}_\lambda(\mu_0,\delta_{x_1}):= 
\{ \nu_\lambda \in\intm_\lambda(\mu_0,\delta_{x_1}) \ \text{satisfying}\  
\eqref{eq:1stepentbound}-\eqref{eq:1stepBS}\}
$$
is not empty, as it contains the measure $\mu_\lambda$ from \Cref{Def:TMCP}; 
it is narrowly compact by 
Lemma~\ref{L:narrow coercivity},
since each $\nu \in \intm_\lambda(\mu_0,\delta_{x_1})$ vanishes outside of $E$.
We shall show that \eqref{eq:1stepbound} can be satisfied by studying the minimum on $\IKN_\lambda(\mu,\delta_{{x_1}})$ of the modified excess functional $\mathscr F_c$ defined below.  If the minimizer fails to satisfy \eqref{eq:1stepbound}, we derive a contradiction by constructing an admissible perturbation which lowers its excess.  

We claim that there is $\nu\in \IKN_\lambda(\mu,\delta_{{x_1}})$ for which \eqref{eq:1stepbound} holds. To this aim, let
$c:=\tfrac{\Vert \rho_0\Vert_\infty}{(1-\lambda)^N}$ and 
$\mathscr F_c:\mathscr \dom (\scrS_\infty) \to[0,1]$ be defined as $\mathscr F_c(\mu):=\int [(1+(\rho-c)^2_+)^{1/2}-1]\,\d\mm$
for   $\mu=\rho\mm$. 

Let $\mu=\rho\mm$ 
be an arbitrary measure in $ \IKN_\lambda(\mu_0,\delta_{{x_1}})$. If $\mathscr F_c(\mu)=0$ we are done, otherwise $A := \{ x \in \M : \rho(x)>c\}$  
is Borel and satisfies 
$\mu(A)>0$. 
For $\hat A:=\M\times A$,  let $\pi\in \mathscr P(\M^2)$ be an optimal coupling between $\mu_0$ and $\mu$, and denote the marginals of $\pi':=\mu(A)^{-1}\pi\mres{\hat A}$
by $\mu_0':=({\rm Pr}_1)_*\pi'\leq \mu(A)^{-1}\mu_0$ 
and $\mu' :=({\rm Pr}_2)_*\pi'\leq \mu(A)^{-1}\mu$.

Since 
$\scrS_\infty(\mu_0') \le \frac1{\mu(A)} \scrS_\infty(\mu_0) - \log {\mu(A)}$ is finite and $\log(\ell(\cdot,x_1))\in L^\infty(\mu_0')$,
the $\TMCP^\h_+(K,N)$ assumption 
yields the existence of $\mu''\in\IKN_\lambda( \mu_0',\delta_{{x_1}})$ with 
$$\scrS_{N}(\mu'')\leq (1-\lambda)\scrS_{N}(\mu_0')\leq -(1-\lambda)(\tfrac{\|\rho_0\|_{\infty}}{\mu(A)})^{-\frac1N}=-(\tfrac c{\mu(A)})^{-\frac1N}.$$
If $B\subset E$ is a Borel set where $\mu''$ is concentrated, Jensen's inequality and the bound just proved yield $\mm(B)\geq \tfrac{\mu(A)}{c}>\mm(A)$. It follows that $\mu''$ cannot concentrate entirely on $A$.  Set $A'=B\setminus A$
and $\hat A'=\M \times A'$. 
Let $\pi''\in \mathscr P(\M^2)$ be an optimal coupling between $\mu_0'$ and $\mu''$, and set $\pi''':=\mu(A')^{-1}\pi''\mres{\hat A'}$, so that 
$\mu''' :=({\rm Pr}_2)_*\pi''' \le \mu(A')^{-1} \mu''$
lies in 
$\intm_\lambda(\mu_0''',\delta_{x_1})$ 
where $\mu_0''':=({\rm Pr}_1)_*\pi''' \le \mu(A')^{-1}\mu_0'$.  We have now identified a portion $\mu_0'''\ne 0$
of $\mu_0$ whose $\lambda$-midpoint under $\pi$ (and $\pi'$) lies in $A$, but whose $\lambda$-midpoint $\mu'''$ under $\pi''$ (and $\pi'''$) lies outside $A$.
For $\eps>0$ small enough, setting $\mu''''= ({\rm Pr}_2)_* (\frac{\d \mu_0'''}{\d \mu_0}\pi)$
makes
$\mu^\eps = \mu + \eps(\mu'''-\mu'''') \in I_\lambda(\mu_0,\delta_{x_1})$ a perturbation
of $\mu$ which effectively `moves a bit of the mass of $\mu$ from above the threshold $c$ to below it'.  Note $\frac{\d \mu_0'''}{\d \mu_0} \le \mu(A)\mu''(A')$ is bounded.
For $\eps>0$
smaller still we claim $\mu^\eps \in \IKN_\lambda(\mu_0,\delta_{x_1})$
but $\mathscr F_c(\mu^\eps)< \mathscr F_c(\mu)$.  Once this
claim is established, choosing $\mu$ to minimize the (narrowly)
lower semicontinuous excess $\mathscr F_c$ on the compact set 
$\IKN_\lambda(\mu_0,\delta_{x_1})$ yields the desired contradiction to conclude the proposition.

The outstanding claim can be argued as follows.  All three functionals of interest --- $\scrS_N$, $\scrS_\infty$ 
and $\mathscr F_c$ --- are given by nonlinearities
$s(z) \in \{-z^{1-1/N}, z\log z, (1+(z-c)^2_+)^{1/2}-1\}$ which are convex 
and differentiable on the positive reals.
Formally
\begin{equation}
\label{eq:percontraddizione}
\lim_{\eps\downarrow0}
\int_E \frac{s(\frac{\d\mu^\eps}{\d\mm}) - s(\rho)}\eps d\mm 
= \int s'(\rho) \d(\mu'''-\mu'''') < 0,
\end{equation}
by the monotonicity of $s'$,
since $\mu'''$ is concentrated on $\{\rho \le c\}$ while
$\mu''''$ is concentrated on $\{\rho > c\}$.
Since all three functionals of $\mu^\eps$ are real-valued
and their integrands are convex functions of $\epsilon$, 
the dominated and monotone convergence theorems can be used to justify \eqref{eq:percontraddizione}, %
arguing separately on the intersections of
$\{\rho \le c\}$ and $\{\rho > c\}$ with $\{\rho \le 1/e\}$ and $\{\rho > 1/e\}$.
For $\eps>0$ sufficiently small, 
we now have $\scrS_N(\mu^\eps)<\scrS_N(\mu)$, $\scrS_\infty(\mu^\eps)<\scrS_\infty(\mu)$ and $\mathscr F_c(\mu^\eps) <\mathscr F_c(\mu)$ which establishes the desired claim and concludes the proof.
\end{proof}

\begin{proposition}[Step 2 --- discrete iteration]\label{Le:Densent} With the same assumptions and notation of \Cref{Pr:Int pts}  the following holds. Let $\lambda\in(0,1)$, put $\nu_{\lambda,0}:=\mu_0$  and then let $\nu_{\lambda,k+1}\in \intm_\lambda(\nu_{\lambda,k},\delta_{x_1})$  be given by \Cref{Pr:Int pts} above with $\nu_{\lambda,k}$ in place of $\mu_0$.

Then putting $t_{\lambda,k}:=1-(1-\lambda)^k$, for every $ k\in\N$ we have
\begin{subequations}
\begin{align}
\label{eq:discbound}
\nu_{\lambda,k}&\leq\tfrac1{(1-t_{\lambda,k})^N}\mathrm{e}^{Dt_{\lambda,k}\sqrt{K_-(N-1)}}\,\Vert\rho_0\Vert_{L^\infty(\M,\meas)}\mm,\\
\label{eq:discentbound}    
\scrS_{N}(\nu_{\lambda,k}) &\leq -\int \tau_{K,N}^{(1-t_{\lambda,k})}\circ\ell(\cdot,x_1)\,\rho_0^{1-1/N}\d\meas;
\\ \label{eq:discBS}
\scrS_\infty(\nu_{\lambda,k}) &\le \scrS_\infty(\mu_0) 
- N \log \sigma^{(1-t_{\lambda,k})}_{K,N}\big(\|\ell(\cdot,x_1)\|_{L^2(\mu_0)}\big);
\\
\label{eq:discgeo}
\nu_{\lambda,k'}&\in \intm_{s}(\nu_{\lambda,k},\delta_{x_1})\quad \text{for }s=\tfrac{t_{\lambda,k'}-t_{\lambda,k}}{1-t_{\lambda,k}}=1-(1-\lambda)^{k'-k}\quad\forall k'\geq k,
\end{align}
\end{subequations}
where  $D := \sup\ell(\supp\mu_0 \times\{x_1\})$.
\end{proposition}


\begin{proof}  The short discussion after the definition \eqref{eq:defintm} ensures that the measures $\nu_{\lambda,k}$ are well defined. Let us then introduce 
$\bar {\sf D}_{k} :=\sup \ell(\supp\nu_{\lambda,k} \times \{x_1\})$ and  $\underbar{\sf D}_{k}:=\| \ell(\cdot, x_1)^{-1}\|_{L^\infty(\nu_{\lambda,k})}^{-1} > 0$  respectively. Then   the definition \eqref{eq:defintm}   and an induction argument show that 
\begin{align}\label{Eq:supl}
\underbar{\sf D}_{0}(1-\lambda)^{k}\leq \underbar{\sf D}_{k} \leq \bar{\sf D}_{k}\leq \bar{\sf D}_{0}(1-\lambda)^{k},\qquad\text{and}\qquad \bar{\sf D}_{0}=D.
\end{align}
Now, letting $\nu_{\lambda,k}=\eta_{\lambda,k}\mm$, from the bound \eqref{eq:1stepbound}  we see that
\[
\begin{split}
    \Vert\eta_{\lambda,k}\Vert_{L^\infty(\M,\meas)}&\leq \frac{1}{(1-\lambda)^N}\,\mathrm{e}^{\bar{\sf D}_{k-1}\,\lambda\sqrt{K_-(N-1)}}\,\Vert\eta_{\lambda,k-1} \Vert_{L^\infty(\M,\meas)}\\
    &\leq\ldots\leq \frac{1}{(1-\lambda)^{kN}}\,\mathrm{e}^{\lambda\sqrt{K_-(N-1)}\sum_{i=0}^{k-1}\bar{\sf D}_{i}}\,\Vert\rho_0\Vert_{L^\infty(\M,\meas)}.
\end{split}
\]
Together with \eqref{Eq:supl}, this gives  \eqref{eq:discbound}. For  \eqref{eq:discentbound}--\eqref{eq:discBS} we notice that   from the monotonicity of 
$\theta\mapsto\tau^{(t)}_{K,N}(\theta)$ and
$\theta\mapsto\sigma^{(t)}_{K,N}(\theta)$
(they are increasing if $K>0$ and decreasing if $K<0$), the definition of $\nu_{\lambda,k}$ and 
\eqref{eq:1stepentbound}--\eqref{eq:1stepBS}, we get
\begin{align*}
    \scrS_{N}(\nu_{\lambda,k}) &\leq-\int_M \tau_{K,N}^{(1-\lambda)}\circ\ell(\cdot,x_1)\,(\eta_{\lambda,k-1})^{1-1/N}\d\meas\leq \tau_{K,N}^{(1-\lambda)}( {\sf D}^*_{k-1} )\,\scrS_{N}(\nu_{\lambda,k-1}),
\\ \scrS_{\infty}(\nu_{\lambda,k}) - 
\scrS_\infty(\nu_{\lambda,k-1}) 
&\le -\log \sigma_{K,N}^{(1-\lambda)} \big(\|\ell(\cdot,x_1)\|_{L^2(\nu_{\lambda,k-1})}\big)
\leq -\log \sigma_{K,N}^{(1-\lambda)}( {\sf D}^*_{k-1} ),
    \end{align*}
where $ {\sf D}^*_{k-1}$ is equal to $ \underbar{\sf D}_{k-1} $ if $K>0$ and to $ \bar{\sf D}_{k-1} $ if $K\leq 0$. With this convention, from \eqref{Eq:supl} we see that $-\tau_{K,N}^{(1-\lambda)}( {\sf D}^*_{k} )\leq -\tau_{K,N}^{(1-\lambda)}( {\sf D}^*_{0}(1-\lambda)^k )$ and similarly for $-\sigma^{(1-\lambda)}_{K,N}$. Thus taking also into account the product formulas $\prod_{i=0}^{k-1} \tau_{K,N}^{(1-\lambda)}(\theta(1-\lambda)^{i})= \tau_{K,N}^{((1-\lambda)^k)}(\theta)$ and similarly for $\sigma^{(1-\lambda)}_{K,N}$
(these follow from the definitions   \eqref{sigma} and \eqref{tau}  noticing that the products are telescopic), by induction from the above we conclude that
\begin{subequations}
\label{Crude}
\begin{align}\label{Eq:Crude}
    \scrS_{N}(\nu_{\lambda,k}) &\leq \tau_{K,N}^{({1-}t_{\lambda,k})}({\sf D}^*_{0})\,\scrS_{N}(\mu_0)    \qquad\forall k\in\N,
\\     \scrS_{\infty}(\nu_{\lambda,k}) &\leq \scrS_{\infty}(\mu_0)-\log \sigma_{K,N}^{(1-t_{\lambda,k})}({\sf D}^*_{0})\,    \qquad\forall k\in\N.
\label{eq:CrudeBS}
\end{align}\end{subequations}

Now we adapt the estimates \eqref{Crude} to establish the desired entropy inequalities. The idea is to decompose the transport into separate parts where the quantity $\ell(\cdot,x_1)$ is ``approximately constant'' and combine the continuity of the distortion coefficients with a similar argument as for \eqref{Crude}. Let $\varepsilon> 0$ and set $A_l := \{l\varepsilon < \ell(\cdot, x_1)\leq (l+1)\varepsilon\}$, where $l\in\N_0$. Clearly, these sets are mutually disjoint (and here we benefit from transporting to a Dirac mass); in particular, $\smash{\scrS_{N}}$ behaves additively under the decomposition of $\nu_{\lambda,k}$ into multiples of its restrictions to the $ A_l$'s. Since any transport toward $\delta_{x_1}$ does not mix the masses in different $A_l$'s, we see that the measure $\nu_{\lambda,k+1}$ obtained applying \Cref{Pr:Int pts} to $\nu_{\lambda,k}$ is equal to the suitable  linear combinations of the mutually singular measures obtained applying the same statement to the restrictions of $\nu_{\lambda,k}$ to the $A_l$'s.  The claimed estimate thus follows from repeating the argument from the above paragraph separately for every $l$ and finally sending $\varepsilon\to 0$.  
In the case of \eqref{eq:discBS}, to obtain the desired bound we use the convexity of $r \in (0,\frac{N}{K_-}\pi^2) \mapsto g(r):=\log \sigma_{K,N}^{(t)} (r^{1/2})$ for each $t \in (0,1)$ that we now verify.
When $K<0$,
setting $K=-N$ without loss of generality, $g(r)
=\log \frac{\sinh(tr^{1/2})}{\sinh{(r^{1/2})}}$ yields
$$
4r^{1/2} g''(r)=[(e^{4r^{1/2}}-1)-t(e^{4tr^{1/2}}-1)] +4r^{1/2}[e^{2r^{1/2}}- t^2 e^{2t r^{1/2}}]
$$
with both terms in square brackets being positive.
Similarly, $f(r):=\sigma^{(t)}_{N,N}(r^{1/2})$ yields
$$
4r^2 f''(r) = \Big(\frac{r^{1/2}}{\sin r^{1/2}}\Big)^2 + r^{1/2} \cot(r^{1/2}) 
-  tr^{1/2}\cot(tr^{1/2}) -  \Big(\frac{tr^{1/2}}{\sin tr^{1/2}}\Big)^2
$$
whose positivity follows from the monotonicity of $\theta \in (0,\pi)\mapsto  \theta\cot(\theta) + (\frac{\theta}{\sin\theta})^2$.

Finally, \eqref{eq:discgeo} holds by induction from \eqref{eq:iterintm}. 
\end{proof}
We are now ready to prove the main result of the section. In the statement below we need a few additional assumptions on the given mm spacetime: 
that it is forward, as in \Cref{D:spacetimes} (to ensure existence of suitable limits) and that $\ell_+$ is real valued (to be sure that it is bounded from above on emeralds so that suitable rigidity can be extracted from equality in the reverse triangle inequality).


\begin{theorem}[Step 3  --- passage to the limit]\label{Le:Existence test plans} With the same assumptions and notation of \Cref{Pr:Int pts} and assuming furthermore that $(\M,\uptau,\ell,\mm)$ is forward,  
that $\mm(E)<\infty$ for any emerald $E$, 
and that  $\ell$ is upper semicontinuous and does not take the value $+\infty$. 

Then there is a curve  $(\mu_t)\subset\pem$ that is a strongly timelike $\ell_q$-geodesic from $\mu$ to $\delta_{x_1}$ for any $0\neq q<1$ such that 
\begin{subequations}
\label{eq:insieme}\begin{align}
\label{eq:quantdensbound}
\mu_t  &\leq \tfrac{1}{(1-t)^N}\,\mathrm{e}^{Dt\sqrt{K_-(N-1)}}\,\Vert \rho\Vert_{L^\infty(\M,\meas)},\\
\label{eq:quantentbound}
\scrS_{N}(\mu_t)&\leq  -\int \tau_{K,N}^{(1-t)}\circ\ell(\cdot,x_1)\,\rho^{1-1/N}\d\meas,
\\ \label{eq:quantBS}
\scrS_\infty(\mu_t) &\le \scrS_\infty(\mu_0) 
- N \log \sigma^{(1-t)}_{K,N}\big(\|\ell(\cdot,x_1)\|_{L^2(\mu_0)}\big),
\end{align}
\end{subequations}
for any $t\in[0,1)$, where $D := \sup\ell(\supp\mu_0 \times \{x_1\})$.
\end{theorem}
\begin{proof} We shall build upon \Cref{Le:Densent} and borrow notation from it. Since $\mu_0\in\pem$, for suitable compact sets $K_0,K_1$ we have $\supp\mu_0\subset J(K_0,K_1)$. It is then clear that $\supp\nu_{\lambda,k}\subset E$ for $E:=J(K_0,\{x_1\})$ for any $\lambda\in(0,1)$, $k\in \N$. 
For each $T<1$,
the assumption $\mm(E)<\infty$ and  the estimate \eqref{eq:discbound} 
yield tightness of
the collection $\{\nu_{\lambda,k}:t_{\lambda,k}\leq T\}$. 

Now observe that for any $t\in[0,1]$ and   $\lambda_n\downarrow0$ there is  $(k_n)\subset\N$ so that $t_{\lambda_n,k_n}\to t$ as $n\to\infty$. This, the tightness just mentioned and a diagonalization argument imply that there is $\mathcal D\subset(0,1)$ dense and $\lambda_n\downarrow0$ such that for any $t\in\mathcal D$ there is $(k_n)\subset\N$ such that $t_{\lambda_n,k_n}\to t$  and $(\nu_{\lambda_n,k_n})$ narrowly converges to some $\eta_t$ as $n\to\infty$. We also include $0,1$ in $\mathcal D$ and define $\eta_0:=\mu_0$ and $\eta_1:=\delta_{x_1}$.  Since $\{\ell\geq 0\}$ is closed, we have $\eta_t\preceq \eta_{t'}$ for all  $t,t'\in\mathcal D$ with $t\leq t'$, thus by \Cref{prop:distord} (recall \Cref{prop:pmforward})   we see that there is a unique $(\mu_t)\in\CC([0,1];\Prob(\M))$ with $\eta_r\preceq\mu_t\preceq \eta_s$ for any $r<t\leq s$, $r,s\in\mathcal D$.

By the stability of the bounds \eqref{eq:insieme} w.r.t.\ narrow convergence it is easy --- starting from \eqref{eq:discbound}--\eqref{eq:discBS} ---  to deduce first that  \eqref{eq:insieme}  hold for $\eta_t$ with $t\in\mathcal D$ and then that these hold for $\mu_t$ for any $t\in[0,1)$.

Thus it only remains to prove that $(\mu_t)$ is a strongly timelike $\ell_q$-geodesic for any $0\neq q<1$. Fix $q$ and notice that from \eqref{eq:discgeo} and the very definition of $\intm_\lambda(\mu,\delta_x)$ it easily follows that $\ell_q(\nu_{\lambda,k},\nu_{\lambda,k'})\geq (t_{\lambda,k'}-t_{\lambda,k})\ell_q(\mu_0,\delta_{x_1})\in\R$. Then the upper semicontinuity in \Cref{le:USCRE} yields $\ell_q(\eta_t,\eta_s)\geq (s-t)\ell_q(\mu_0,\delta_{x_1})$  for any $t<s$, $t,s\in\mathcal D$. Now let $t,s\in[0,1]$ be with $t<s$ and find $t',s'\in\mathcal D$ with $t<t'<s'<s$. The construction of $(\mu_r)$ ensures that $\mu_t\preceq\eta_{t'}\preceq\eta_{s'}\preceq\mu_s$ which in turn implies $\ell_q(\mu_t,\mu_s)\geq\ell_q(\eta_{t'},\eta_{s'})\geq (s'-t')\ell(\mu_0,\delta_{x_1})$. Letting $t'\downarrow t$ and $s'\uparrow s$ we conclude from \Cref{Pr:l-geodesics maximize} that the curve $(\mu_r)$  is a causal  $\ell_q$-geodesic with $\ell_q(\mu_0,\mu_1)\in\R$. 

Also, $\mu_0\times\delta_{x_1}$ is the only admissible coupling of $(\mu_0,\delta_{x_1})$ and by assumption it is concentrated on $\{\ell\in(0,+\infty)\}$. It follows from \Cref{prop:ellqgeo} that any $\ell_q$-optimal coupling of $(\mu_0,\mu_1)$ is also concentrated on $\{\ell\in(0,+\infty)\}$, i.e.\ that $(\mu_t)$ is a strongly timelike $\ell_q$-geodesic, as desired.
\end{proof}

\begin{remark}[Variants] In relation to \Cref{re:vardef}, we notice that the statement of the previous \Cref{Le:Existence test plans} holds almost  unchanged by replacing every occurrence of $\smash{\TMCP^\h_+(K,N)}$ by $\smash{\TMCP_+^{\h,*}(K,N)}$. In this case only the slightly worse density bound
\begin{align*}
\mu_t  \leq \tfrac{1}{(1-t)^N}\,\mathrm{e}^{Dt\sqrt{K_-(N)}}\,\Vert \rho_0\Vert_{L^\infty(\M,\meas)},
\end{align*}
can be guaranteed. Similarly, if our space is $\TMCP^\h_+(K,N')$ for every $N'\geq N$ and an $\ell_q$-geodesic which satisfies each of the defining inequalities can be chosen independently of $N'$, then the same construction produces measures satisfying \eqref{eq:quantdensbound}--\eqref{eq:quantBS} with 
\[
\scrS_{N'}(\mu_t)\leq  -\int \tau_{K,N'}^{(t)}\circ\ell(\cdot,x_1)\,\rho_0^{1-1/N'}\d\meas,\qquad\forall N'\geq N,
\]
in place of \eqref{eq:quantentbound}. Similar comments apply to \Cref{Le:Existence test plans2} below.\hfill$\blacksquare$
\end{remark}
We shall typically use \Cref{Le:Existence test plans} above in conjunction with the lifting of $\ell_q$-geodesics given in  \Cref{Cor:Lifting geos}  to obtain:

\begin{corollary}[On existence of initial test plans with Dirac targets]
\label{cor:intestplan}
Let  $(\M,\uptau,\ell,\mm)$ be a $\TMCP^\h_+(K,N)$  forward mm spacetime so that $\mm(E)<\infty$ for any emerald $E$ and   $\ell$ is upper semicontinuous and does not take the value $+\infty$. Let $x_1\in \M$ and $\mu_0\in\pem$ with bounded compression satisfy $\log(\ell(\cdot,x_1))\in L^\infty(\mu_0)$. Assume also that $\M$ is timelike $q$-essentially  non-branching at 0 for some $0\neq q<1$ and that either $(A)$ or $(B)$ of \Cref{Th:Lifting} hold.

Then there exists an initial test plan $\ppi$ with $(\e_0)_*\ppi=\mu_0$ concentrated on timelike geodesics $\gamma$ from $\gamma_0$ to $x_1$.
\end{corollary}
\begin{proof}
Our assumptions allow to apply first \Cref{Le:Existence test plans} to find $(\mu_t)$ and then \Cref{Cor:Lifting geos} to obtain a lifting $\ppi\in\Prob(\CC([0,1];\M))$ of $(\mu_t)$ concentrated on timelike $\ell$-geodesics that lifts  $(\mu_t)$ and so that $(\e_0,\e_1)_*\ppi$ is $\ell_q$-optimal (for every $0\neq q<1$). Since $(\mu_t)$ is a $\ell_q$-geodesic from $\mu_0$ to $\delta_{x_1}$,  \Cref{Cor:Lifting geos} implies that $\ppi$ is  concentrated on  geodesics $\gamma$ from $\gamma_0$ to $x_1$. Also, from the bound \eqref{eq:quantdensbound} we see that $(\eval_t)_\push\bdpi \leq C\,\meas$ for every $t\in[0,1/2]$  and some $C>0$, so that to conclude it suffices to prove that $(\mu_t)$ narrowly converges to $\mu_0$ as $t\downarrow0$. In turn, this follows from \Cref{cor:contfromreg}, the $q$-essential timelike non-branching at 0 and the properties of $(\mu_t)$.
\end{proof}

A useful consequence of \Cref{Le:Existence test plans} is the following (compare with   \cite[Rem.~3.10]{CM:20}).

\begin{corollary}[Essential geodesy of $\supp\mm$]\label{Cor:Essgeo}  Let $(\M,\uptau,\ell,\meas)$ be  a  forward $\smash{\TMCP^\h_+(K,N)}$ mm  spacetime   in which $\ell$ is upper semicontinuous and does not take the value $+\infty$ and $\mm(E)<+\infty$ for every emerald $E \subset \M$.  Assume also that either $(A)$ or $(B)$ of \Cref{Th:Lifting} hold.

Then for every $y\in \supp\meas$ and $\meas$-a.e.~$x\in I^-(y)$ there exists a geodesic  $\gamma\in\TGeo(\M)$ from $x$ to $y$ with $\gamma_t\in\supp\mm$ for every $t\in[0,1]$.
\end{corollary}

\begin{proof}  Using that the background metric is separable and $\meas$ is locally finite, by the Lindel\"of property  we cover $I^-(y)$ with a countable number of finite $\meas$-measure open sets. Each of these sets is exhausted by a sequence of compact sets of positive $\meas$-measure where $\ell(\cdot,y)$ is bounded away from 0 and $+\infty$ (up to an $\meas$-negligible set). Let $\mu$ be the uniform distribution of any such compact set and  apply \Cref{Cor:Lifting geos} to the geodesic $(\mu_t)$ given by \Cref{Le:Existence test plans} to find a lifting $\ppi$ of $(\mu_t)$.  The $L^\infty$-bounds \eqref{eq:quantdensbound} ensure that $\supp\mu_t\subset\supp\mm$ for every $t\in[0,1)$ and thus that for $\ppi$-a.e.\ $\gamma$ we have $\gamma_t\in\supp\mm$ for every $t\in\Q\cap[0,1)$. Since $\ppi$ is concentrated on $\TGeo(\M)$, we just proved that there is $\gamma\in\TGeo(\M)$ with  $\gamma_t\in\supp\mm$ for every $t\in\Q\cap[0,1)$. By left continuity it follows that  $\gamma_t\in\supp\mm$ for every $t\in[0,1]$, as desired.
\end{proof}

{Another interesting consequence of  \Cref{Le:Existence test plans} is the following result. The argument for the proof is inspired from \cite{Gigli12a}. Surprisingly however, the non-branching assumption is only required in its weakest form, namely `essentially non-branching at 0', whereas the branching at later times that must be avoided in \cite{Gigli12a} is permitted in the present context.}


\begin{corollary}[Measurability of causal functions on $\TMCP^\h_+(K,N)$ spacetimes]\label{cor:fmeastmcp}
 Let $(\M,\uptau,\ell,\meas)$ be  a $\smash{\TMCP^\h_+(K,N)}$ forward mm  spacetime with $\uptau$ containing the chronological topology, $\mm(E)<\infty$ for every emerald $E$ and in which $\ell$ is upper semicontinuous and never $+\infty$. Assume also that it is $q$-essentially non-branching at 0 for some $0\neq q<1$.
 
 Let  $A\subset \M$ be achronal 
 with closure $\bar A$ disjoint from $\M_{\sf {fin}}$. Then  $\mm(\bar A)=0$.
 In particular, if $f:\M\to\bar\R$ is a rough (i.e.~possibly non-measurable) 
 causal function, then its restriction to  $\M\setminus \M_{\sf {fin}}$ is $\mm$-measurable.
 \end{corollary}
 
\begin{proof} The second statement is a consequence of the first and of \Cref{L:chronological continuity}, so we focus on the first. As $\uptau$ contains  the chronological topology, the closure of an achronal set is still achronal. Thus we can assume $A=\bar A$.

We argue by contradiction and assume  $\mm(A)>0$. From $ A\cap  \M_{\sf {fin}}=\emptyset$ we see that $A\subset \cup_{x\in\M} I^-(x)$, so using the Lindel\"of property of (the Polish space) $A$ we can find a countable collection $(x_n)$ so that $A\subset\cup_nI^-(x_n)$. Then by interior approximation we can find $ x_1\in\M$ and $C\subset A$ compact with $\mm(C)>0$ and  $\log(\ell(\cdot, x_1))\in L^\infty(C,\mm)$. Apply  \Cref{Le:Existence test plans}  with $\mu_0:=\mm(C)^{-1}\mm\mres C$ to find $(\mu_t)$ as in the statement. 

Then the uniform density bounds and the assumption of $q$-essentially non-branching at 0  (applied to $t\mapsto\mu_{ts}$ for some $s\in(0,1)$) implies that $(\mu_t)$ is narrowly continuous at $t=0$ (by \Cref{cor:contfromreg}).

Let $U\supset C$ be open with $\mm(U)<\tfrac32\mm(C)$. The narrow convergence implies $\lim_{t\downarrow0}\mu_t(U)= 1$ and letting $\mu_t=\rho_t\mm$,  the bound \eqref{eq:quantdensbound} gives $\limi_{t\downarrow0}\mm(\{\rho_t>0\})\geq \mm(C)$. It follows that for some $t>0$ we have $\mm(C\cap \{\rho_t>0\})>0$ and letting $\pi\in\Pi_\leq(\mu_0,\mu_t)$ (that exists, as $\mu_0\preceq\mu_t$), this means that  $\pi(\{(x,y): x,y\in C\})>0$. Since by \eqref{eq:tuttiinfila} we see that $\ell(x,y)=t\ell(x,x_1)$ for $\pi$-a.e.\ $(x,y)$ the assumption $\log(\ell(\cdot, x_1))\in L^\infty(C,\mm)$ implies that there are $x,y\in C$ with $\ell(x,y)>0$, contradicting the achronality of $C\subset A$.
\end{proof}

 \Cref{Le:Existence test plans} works with the target measure being a Dirac mass. An extension of this result holds for more general target measures, provided we add a suitable non-branching assumption to our spacetime and some further  assumptions on $\uptau$. Specifically, we will ask that either $(A)$ or $(B)$ of \Cref{Th:Lifting} hold: this is necessary as we will need to lift suitable geodesics appearing the statement and proof. Concerning non-branching: as shown by Gigli \cite{Gigli12a} in positive signature, a non-branching assumption together with a  lower Ricci bound in the form of a $\CD$ condition implies existence of optimal transport maps. This  has been extended to non-branching $\MCP$ spaces by Cavalletti--Mondino \cite{cavalletti-mondino2017} and then to timelike non-branching $\TMCP$ spacetimes in  \cite{CM:20}. This latter fact will be used in the proof of the next result, that closely follows the construction in positive signature given in  \cite[Prop.~4.3]{cavalletti-mondino2017}.
 
\begin{theorem}[On uniqueness of $\ell_q$-geodesics and existence of optimal maps]
\label{Le:Existence test plans2} Let $(\M,\uptau,\ell,\meas)$ be  a  forward $\smash{\TMCP^\h_+(K,N)}$ mm  spacetime  with  $\mm(E)<\infty$ for any emerald $E$ and in which $\ell_+$ is continuous and real valued.  Also, let $\mu_0,\bar\mu_1\in \pem$ with $\mu_0=\rho_0\mm$ having bounded compression and $0\neq q<1$.
Assume furthermore that:
\begin{itemize}
\item[-] either $(A)$ or $(B)$ of \Cref{Th:Lifting} hold,
\item[-] $\M$ is timelike $\QQ$-essentially forward non-branching,
\item[-] $\M$ is timelike $\QQ$-essentially non-branching at 0,
\item[-]  $\ell_q(\mu_0,\bar\mu_1)\in(0,+\infty)$ and  there is an 
 $\ell_\QQ$-optimal coupling ${\bar \pi}\in\Pi_\leq(\mu_0,{\bar \mu_1})$   concentrated on $\{\ell\in(c,\tfrac1c)\}$ for some $c>0$.
\end{itemize} 
Then:
\begin{itemize}
\item[i)] There is a unique $\ell_q$-geodesic $(\mu_t)$ from $\mu_0$ to $\bar \mu_1$ and it is strongly $\ell_q$-timelike;
\item[ii)] $(\mu_t)$ admits a unique lifting $\ppi\in\Prob(\CC([0,1];\M))$ and is induced by a map, i.e.\ there is $F:\M\to\CC([0,1];\M)$ with $\e_0\circ F$ being the identity $\mu_0$-a.e.\ and such that $\ppi=F_*\mu_0$;
\item[iii)] we have
\begin{equation}
\label{eq:sameell}
\ell(x,F_1(x))\in [c,\tfrac1c]\qquad \mu_0-a.e.\ x;
\end{equation}
\item[iv)] for  every $t\in[0,1)$ we have
\begin{subequations}
\label{eq:insieme2}
\begin{align}
\label{eq:densnonbr}\mu_t  &\leq \tfrac{1}{(1-t)^N}\,\mathrm{e}^{Dt\sqrt{K_-(N-1)}}\,\Vert \rho_0\Vert_{L^\infty(\M,\meas)}\mm,\\
\label{eq:entnonbr}
    \scrS_{N}(\mu_t) &\leq -\int \tau_{K,N}^{(1-t)}\circ\ell(\gamma_0,\gamma_1)\,\rho_0(\gamma_0)^{-1/N}\d\bdpi(\gamma),
\\ \label{eq:BSnonbr} \scrS_\infty(\mu_t) &\le \scrS_\infty(\mu_0) 
- N \log \min\{\sigma^{(1-t)}_{K,N}(c),\sigma^{(1-t)}_{K,N}(1/c)\}.
\end{align}
\end{subequations}
\end{itemize}
\end{theorem}

\begin{proof} We will   first consider the situation of $\bar\mu_1$ being a convex combination of Dirac masses, i.e. $\smash{\bar\mu_1 := \sum_i\lambda_i\,\delta_{y_i}}$ for given $\lambda_1,\dots,\lambda_n\in (0,1)$ and mutually distinct points $y_1,\dots,y_n\in\supp\meas$. For $i=1,\dots,n$, let $\pi_i:=\lambda_i^{-1}\pi\mres{\M\times\{y_i\}}$ and $\mu_{0,i}:=({\rm Pr}_1)_*\pi_i$. Then let $(\mu_{t,i})\subset \pem$ be given by \Cref{Le:Existence test plans} applied to $\mu_{0,i}$ and $\delta_{y_i}$ and then $\ppi_i\in \Prob(\CC([0,1];\M))$ be a lifting of it as in \Cref{Cor:Lifting geos}.  Also, let $\alpha_i:=(\e_1)_*\ppi_i\times\delta_{y_i}$ be the only coupling between $(\e_1)_*\ppi_i$ and $\delta_{y_i}$  (hence $\ell_q$-optimal). Then we have
\[
\begin{split}
\ell_q(\mu_{1,i},\delta_{y_i})&\geq\Big(\int \ell^q(\gamma_0,\gamma_1)\,\d\ppi_i(\gamma)\Big)^{\frac1q}+\Big(\int \ell^q(x,y)\,\d\alpha_i(x,y)\Big)^{\frac1q}\\
&\geq\ell_q(\mu_{1,i},\delta_{y_i})+\Big(\int \ell^q(x,y)\,\d\alpha_i(x,y)\Big)^{\frac1q}.
\end{split}
\]
The assumption  on $\pi$ ensures that $\ell_q(\mu_{1,i},\delta_{y_i})\in(0,+\infty)$, thus inspecting the equality case in \eqref{eq:revlq} we deduce that $\alpha_i$ is concentrated on $\{(x,y):\ell(x,y)=0\}$.

We claim that $\mu_i\perp\mu_j$ for $i\neq j$. In positive signature, this follows from the entropy estimates coming from lower Ricci bounds, a non-branching assumption and the narrow continuity of $W_2$-geodesics (see \cite{Gigli12a}). These arguments carry over also in our setting (see also  \cite[Thm.~3.20]{CM:20} and \cite[Thm.~4.16]{Braun:2023Renyi}), where we remark that the assumption of  timelike $\QQ$-essentially non-branching at 0 is used --- via \Cref{cor:contfromreg}  ---  to get the desired narrow continuity at $0$. Notice that with respect to the above references,  working with $q<0$ causes no additional difficulties and that the compactness of emeralds used in the proof of this result in \cite{CM:20, Braun:2023Renyi} is bypassed by our requirement that they have finite $\meas$-measure; c.f. Remark \ref{R:q<0,forward}.

Thus    $\mu_i\perp\mu_j$ for $i\neq j$ and since $y_i\neq y_j$ as well for $i\neq j$, a further use of the non-branching assumption tells that $(\eval_t)_*\ppi_i\perp (\eval_t)_*\ppi_j$ for every $t\in[0,1]$ and $i\neq j$. Putting $\ppi:=\sum_i\lambda_i\ppi_i$ and $\alpha:=\sum_i\lambda_i\alpha_i$, we thus have that $\alpha\in\Pi_\leq((\e_1)_*\ppi,\bar\mu_1)$ and $(\eval_t)_*\ppi=\rho_t\mm=\mu_t$ with
\[
\begin{split}
\|\rho_t\|_{L^\infty(\M,\mm)}&=\max_i\lambda_i\|\rho_{i,t}\|_{L^\infty(\M,\mm)},\qquad\text{ where }\rho_{i,t}:=\tfrac{\d(\eval_t)_*\bdpi_i}{\d\mm},
\\ \scrS_{N}(\mu_t)&=\sum_i \lambda_i^{1-1/N}\scrS_{N}(\mu_{i,t}),
\\ \scrS_{\infty}(\mu_t)&\le \sum_i \lambda_i \scrS_{\infty}(\mu_{i,t}),
 \end{split}
\]
so that $\ppi$ satisfies \eqref{eq:insieme2} as a consequence of the bounds \eqref{eq:insieme} for  the $\ppi_{i,t}$'s.  The fact that $(\mu_t)$ is a $\ell_q$-geodesic from $\mu_0$ to $\bar\mu_1$ follows from the construction (notice  that $\alpha$ is concentrated on $\{\ell=0\}$) and 
 \eqref{eq:sameell} follows from \Cref{prop:ellqgeo}.


Now assume that   $\supp\mu_0\times \supp\bar\mu_1\subset \{\ell\in(c,\tfrac1c)\}$. Since $\{\ell\in(c,\tfrac1c)\}$ is open, we can find a sequence   $(\bar\mu^n_1)$      of convex combinations of Dirac masses narrowly converging to  $\bar\mu_1$ and so that $\supp\mu_0\times \supp\bar\mu^n_1\subset \{\ell\in(c,\tfrac1c)\}$ for every $n\in\N$ so that there are $\ell_q$-optimal couplings $\bar \pi^n$ for $(\mu_0,\mu^n_1)$ narrowly converging to some $\bar \pi'$ that, by the continuity of $\ell$ in $\{\ell\in(c,\tfrac1c)\}$, is $\ell_q$-optimal. It is not hard to see that we can choose the $\bar\mu^n_1$'s so that they all have support in a fixed emerald $E$, that we can assume to contain also $\supp\mu_0$.  The construction and the continuity of $\ell_+$ ensure that $\ell_q(\mu,\bar\mu^n_1)\to\ell_q(\mu,\bar\mu_1)$ as $n\to\infty$. 
As before, there are $\ppi^n\in\Prob(\CC([0,1];\M))$  as above  for $\mu_0,\bar\mu^n_1$ and notice that the construction yields  $(\e_1)_*\ppi^n\preceq \bar\mu^n_1$.  We claim that these are tight. If $(A)$ holds, this is a quite direct consequence of the tightness of $(\bar\mu^n_1)$ and the same arguments in the proof of \Cref{Th:Lifting}. If $(B)$ holds we use  the uniform $L^\infty$ estimates \eqref{eq:densnonbr} and the fact that $\mm\mres E$ is tight to conclude that $((\e_t)_*\ppi^n)$ is tight for every $t<1$. This suffices, as in  the proof of \Cref{Th:Lifting}, to obtain tightness.

Hence a non-relabelled subsequence narrowly converges to a limit $\ppi$.  Since the relevant couplings are concentrated on $\{\ell\in(c,\tfrac1c)\}$ and $\ell_+$ is continuous,  using also  \Cref{prop-p-act-usc} we have
\[
\begin{split}
u_q\big(\ell_q(\mu_0,\bar\mu_1)\big)=\lim_{n\to\infty}u_q\big(\ell_q(\mu_0,\bar\mu^n_1)\big)=\lim_{n\to\infty}\int \KE_\QQ(\gamma)\,\d\ppi^n(\gamma)\leq\int \KE_\QQ(\gamma)\,\d\ppi(\gamma).
\end{split}
\]
Arguing as in the proof of \Cref{Th:Lifting} to pass to the limit in the marginals we see that  $(\e_1)_*\ppi\preceq\bar\mu_1$, thus the last claim in \Cref{Cor:Lifting geos} tells that  $t\mapsto\mu_t:=(\e_t)_*\ppi$ is a $\ell_q$-geodesic from $\mu_0$ to $\bar\mu_1$. 

Then again using the same techniques adopted in proving \Cref{Th:Lifting} when passing to the limit at the level of plans in $\Prob(\CC([0,1];\M))$, we can pass to the limit in \eqref{eq:insieme2}    stated for the $\ppi^n$'s and prove that the same bounds hold for $\ppi$: for \eqref{eq:densnonbr} this is trivial, passing to the limit in \eqref{eq:entnonbr} is a bit more technical, but nevertheless doable, see also the arguments in \cite[Lemma 3.3]{Sturm:2006b}. The fact that $\ppi$ is induced by a map will be (briefly) mentioned in a moment; from that the bound   \eqref{eq:sameell} is trivial from the construction
%

Removing the assumption that  $\supp\mu_0\times \supp\bar\mu_1\subset \{\ell\in(c,\tfrac1c)\}$  is done by decomposing the support of the given $\pi$  into countably many suitable  rectangles, see e.g.~the proof of \cite[Prop.~3.38]{Braun:2023Renyi}.

It thus remains to show the uniqueness properties in $(i),(ii)$. These, however, follow via the same arguments used to prove $\mu_i\perp\mu_j$ above, see for instance \cite{giglirajalasturm} for the proof in positive signature.
\end{proof}

\subsection{A converse Hawking--King--McCarthy theorem}
\label{S:Hawking-King-McCarthy}

Building on the results proved in the previous section, in this one   we study the interplay of our Sobolev calculus and  the ``metric'' features of the given metric measure spacetime $(\M,\uptau,\ell,\meas)$. Our motivation comes from a classical result of Hawking--King--McCarthy \cite{HKMcC:76}: if $(M_1,g_1)$ and $(M_2,g_2)$ are two spacetimes of the same dimension, the first being strongly causal, then every map $F\colon M_1\to M_2$ which preserves the respective time separations is in fact a smooth isometry, in particular $F_*g_{1} = g_{2}$. (The converse  is obvious.) 
Our analog of this result is stated in \Cref{Th:HawkingKingMcCarthy} below. In contradistinction to the smooth setting of Hawking--King--McCarthy,
 we use $\ell$ to define the (maximal weak sub)slope $|\rmd f|$
 instead of the other way around,  so it is the obvious direction from that setting which becomes challenging for us to prove. 
 
 The positive-signature predecessor of this result from metric measure geometry that inspires us is due to Gigli \cite[Lem.~4.19, Prop.~4.20]{gigli2013} in the context of the proof of the splitting theorem for $\RCD$ spaces. The so-called Sobolev-to-Lipschitz property introduced there \cite[§4.1.3]{gigli2013} was pioneered to characterize metric measure and Sobolev isometries. 
 
\medskip

We turn to the technical content. In positive signature, 1-Lipschitz functions fully characterize the underlying distance. Analogously, in our setting  1-steep functions fully characterize the underlying time separations, via the duality formula
\begin{equation}
\label{eq:dualsteep}
\ell(x,y)=\inf\{f(y)-f(x)\, :\, f:\M\to\bar\R\quad\text{ is 1-steep}\},\qquad\forall x,y\in \M,
\end{equation}
valid on any metric spacetime $(\M,\ell)$. Indeed, $\leq$ comes from the definition of 1-steepness, while for $\geq$ we pick $f(z):=\ell(x,z)$ (in the case of smooth spacetimes, the duality formula characterizes stability of the spacetime \cite[Thm.\ 4.6]{Minguzzi2019}).  If we are on a metric measure spacetime, then we know from \eqref{eq:dflsteep} that every 1-steep function $f$ satisfies $|\d f|\geq 1$ $\mm$-a.e., therefore in this case we have
\begin{equation}
\label{eq:dualsteep2}
\ell(x,y) \geq  \inf\{f(y) - f(x) \,:\quad 
f:\supp\meas\to\bar{\R}\textnormal{ causal, }\vert\rmd f\vert\geq 1\ \meas\textnormal{-a.e.}\},\qquad\forall x,y\in \M.
\end{equation}
Understanding whether equality holds here is relevant if one wants to use the Sobolev-like calculus we just developed in order to derive precise metric information on the underlying spacetime. Still, without further assumptions, we cannot expect equality to hold in \eqref{eq:dualsteep2}, as there could be a gap between the class of 1-steep functions and that of functions with $\vert\rmd f\vert\geq 1$ $\mm$-a.e. It is then natural to propose the following:
\begin{definition}[Sobolev-to-steepness property] Let $(\M,\uptau,\ell,\mm)$ be a mm spacetime.  We  say that it has   the \emph{Sobolev-to-steepness property} if every causal  function $f\colon \M\to \bar{\R}$ with $\vert\rmd f\vert\geq 1$ $\meas$-a.e.\ is  $1$-steep on 
$\supp\meas$.
\end{definition}
Notice that unlike the positive signature case, here  there is be no need to pass to $\meas$-representatives, as the property of being causal already depends on the value of the function at any point.

From the previous consideration it is now easy to see that the following holds:
\begin{proposition}[Duality formula]\label{Cor:Intrinsic} A metric measure spacetime $(\M,\uptau,\ell,\meas)$ satisfies  the Sobolev-to-steepness property if and only if for every $x,y\in\supp\meas$ we have
\begin{align}\label{essentially intrinsic}
\ell(x,y) = \inf\{f(y) - f(x) \,:\, 
f:\supp\meas\to\bar{\R}\textnormal{ causal, }\vert\rmd f\vert\geq 1\ \meas\textnormal{-a.e.}\}.
\end{align}
\end{proposition}

\begin{proof}  ``If''. Any causal function $f\colon\supp\meas\to\bar\R$ with $\vert\rmd f\vert\geq 1$ $\meas$-a.e.~is admissible in \eqref{essentially intrinsic}, hence all points $x,y\in \supp\meas$ satisfy the inequality $f(y) - f(x)\geq \ell(x,y)$. This is to say $f$ is $1$-steep on $\supp \mm$.

 ``Only if''. Inequality $\geq$ holds by \eqref{eq:dualsteep2}. For $\leq$ we notice that every function $f$ in the infimum is $1$-steep on $\supp\meas$ by the assumed Sobolev-to-steepness property, thus we conclude by applying \eqref{eq:dualsteep} on $\spt \mm$.
 \end{proof}
In the following, given a monotone map $T\colon \supp \meas_1 
\to \supp \meas_2$ between two given forward mm spacetimes $(\M_1,\uptau_1,\ell_1,\meas_1)$ and $(\M_2,\uptau_2,\ell_2,\meas_2)$ we define the map $\mathscr{T}:\CC([0,1];\supp\meas_1)\to \CC([0,1];\supp\meas_2)$ as the one sending $\gamma$ to the element of $\CC([0,1];\supp\meas_2)$ associated to the monotone map $T\circ\gamma$ via \Cref{prop:distord} (in other words, $\mathscr T(\gamma)_0:=T(\gamma_0)$ and $\mathscr T(\gamma)_t:=\lim_{s\uparrow t}T(\gamma_s)$ for $t>0$). Notice that by \Cref{L:countable discontinuities2} we see that $T(\gamma_t)=\mathscr T(\gamma)_t$ for $t=0$ and every $t\in(0,1]$ except at most a countable number, thus the very definition of the (distance ${\sfD}$ inducing the) topology of $\CC([0,1];\supp\meas_2)$ shows that if $T$ is Borel, then so is $\mathscr T$.

\begin{lemma}[One-sided nonsmooth Hawking--King--McCarthy theorem]\label{Le:One sid} Let $(\M_1,\uptau_1,\ell_1,\meas_1)$ and $(\M_2,\uptau_2,\ell_2,\meas_2)$ be two forward mm spacetimes and  $T\colon \supp\meas_1\to\supp\meas_2$ be a bijective Borel map. Assume that $T_\push\meas_1\leq C\,\meas_2$ for some constant $C>0$ and let $\mathscr{T}:\CC([0,1];\supp\meas_1)\to \CC([0,1];\supp\meas_2)$ be defined as above.

Consider the two statements:
\begin{enumerate}[label=\textnormal{\roman*)}]
\item $T:\supp\meas_1\to M_2$ is \emph{1-steep}, i.e.~for every $x,y\in\supp\meas_1$ we have $\ell_2(T(x),T(y))\geq \ell_1(x,y).$
\item If $g:\M_2\to\bar\R$ is causal then so is $g\circ T:\M_1\to\bar \R$ and $\vert\rmd (g\circ T)\vert_1 \geq \vert \rmd g\vert_2\circ T$  $\meas_1$-a.e.
\end{enumerate}
Then $(i)\Rightarrow(ii)$ and, conversely, if   $(\M_1,\ell_1,\meas_1)$ satisfies the Sobolev-to-steepness property we also have $(ii)\Rightarrow(i)$.
\end{lemma}

\begin{proof} (i) $\Longrightarrow$ (ii). That $g\circ T$ is causal is obvious. To show the stated bound, let $\ppi$ be a test plan on $\M_1$. We claim that $\mathscr T_*\ppi$ is a test plan on $\M_2$ and since it is clearly a measure on $\CC([0,1];\M_2)$ all we need to  do is to show that $(\e_t)_*\mathscr T_*\ppi\leq C'\mm_2$ for every $t\in[0,1]$ and some $C'>0$. We have already noticed  for every $\gamma\in\CC([0,1];\M_1)$ that $\mathscr T(\gamma)_0=T(\gamma_0)$ and $\mathscr T(\gamma)_t=T(\gamma_t)$ holds
except at a countable number of $t \in (0,1]$. It follows that $(\e_0)_*\mathscr T_*\ppi= T_*(\e_0)_*\ppi\leq CC''\mm_2$, where  $C''>0$ is so that $(\e_t)_*\ppi\leq C''\mm_1$ for every $t\in[0,1]$. Also, for $s,t\in[0,1]$ with $s<t$ we have $\int_s^t(\e_r)_*\mathscr T_*\ppi=\int_s^t T_*(\e_r)_*\ppi\leq(s-t)CC''\mm_2$, so that dividing by $s-t$ and letting $s\uparrow t$ using that $\mathscr T_*\ppi$ is concentrated on left continuous paths we conclude that $(\e_t)_*\mathscr T_*\ppi\leq CC''\mm_2$, establishing that $\mathscr T_*\ppi$ is a test plan, as desired.

Now the claim follows by duality. Let $g:\M_2\to\bar\R$ be causal, note that $g\circ T$ is a causal  function on  $\supp\meas_1$ by (i). For a test plan $\bdpi$ on $\M_1$ we just proved that  $\mathscr{T}_\push\bdpi$ is a test plan on $\M_2$,  therefore,
\begin{align*}
\int g\circ T(\gamma_1) - g\circ T(\gamma_0)\,\d\bdpi(\gamma) &= \int g(\sigma_1) - g(\sigma_0)\,\d\mathscr{T}_\push\bdpi(\sigma)\\
&\geq \iint_0^1 \vert\rmd g\vert_2(\sigma_r)\,\vert\dot{\sigma}_r\vert \,\d r\,\d\mathscr{T}_\push\bdpi(\sigma)\geq \iint_0^1 \vert \rmd g\vert_2( T(\gamma_r))\,\vert \dot\gamma_r\vert\,\d r\,\d\bdpi(\gamma).
\end{align*}
In the last step, we have employed the hypothesized noncontractivity (i) and \Cref{Pr:speed}. This shows $\vert \rmd  g\vert_2\circ T$ is a weak subslope of $g\circ T$. The desired inequality thus follows from the maximality asserted by \Cref{Th:MaxLoDiff}.

(ii) $\Longrightarrow$ (i). Let $x,y\in \M_1$ and then $g:\M_2\to\bar\R$ be defined as $g(z):=\ell_2(T(x),z)$. Then $g$ is 1-steep, hence (by \eqref{eq:dflsteep}) $|\d g|_2\geq1$ $\mm_2$-a.e.\ and thus our assumptions ensure that $g\circ T$ is causal with $|\d (g\circ T)|_1\geq 1$ $\mm_1$-a.e. Since $\M_1$ has the Sobolev-to-steepness property we deduce that $g\circ T$ is 1-steep on $\spt \mm_1$, thus in particular $g(T(y))-g(T(x))\geq \ell_1(x,y)$. Since the left hand side of this identity is equal to $\ell_2(T(x),T(y))$, the proof is complete.
\end{proof}

\begin{theorem}[Nonsmooth Hawking--King--McCarthy theorem]\label{Th:HawkingKingMcCarthy} Assume that $(\M_1,\uptau_1,\ell_1,\meas_1)$ and $(\M_2,\uptau_2,\ell_2,\meas_2)$ are forward mm spacetimes with  the Sobolev-to-steepness property. Let $T\colon \supp\meas_1\to\supp\meas_2$ be a surjective measure-preserving Borel map, i.e.~$T_\push\meas_1 = \meas_2$. 

Then the following are equivalent.
\begin{enumerate}[label=\textnormal{(\roman*)}]
\item $T$ is an isometry of metric measure spacetimes, i.e.~for every $x,y\in\supp\meas$,
\begin{align}\label{eq:isoommee}
\ell_2(T(x),T(y)) = \ell_1(x,y).
\end{align}
\item $f:\M_2\to\bar \R$ is causal if and only if so is $f\circ T:\M_1\to\bar\R$ and in this case $\vert\rmd (f\circ T)\vert_1 = \vert\rmd f\vert_2\circ T$ holds $\meas_1$-a.e.
\end{enumerate}
\end{theorem}
\begin{proof} We shall apply \Cref{Le:One sid} first to $T$ and then to its inverse. To do so we must first verify that $T$ is invertible. If $(i)$ holds, the condition $T(x) =T(y)$ and \eqref{eq:isoommee} imply $\ell_1(x,y) = \ell_1(y,x) = 0$, which forces $x=y$. If $(ii)$ holds a similar argument works from the causality preservation. Thus in either case $T$ is injective, thus bijective (as we assumed surjectivity) and the inverse is also Borel (see e.g.\ \cite[Cor.~15.2]{kechris1995}).  
\end{proof}

It remains to find sufficient conditions ensuring that   the Sobolev-to-steepness property holds. An answer is given by  the next result:

\begin{theorem}[Sobolev-to-steepness property from curvature conditions]\label{Th:SobtoSteep}  Let $(\M,\uptau,\ell,\meas)$ be  a $q$-essentially timelike non-branching at 0, forward $\smash{\TMCP^\h_+(K,N)}$ mm  spacetime with $\uptau$ containing the chronological topology,  $\mm(E)<\infty$ for each emerald $E$, and with $\ell$ upper semicontinuous and never $+\infty$.  Assume also that either $(A)$ or $(B)$ of \Cref{Th:Lifting} hold.

Let $f:\M\to\bar\R$ be causal with $|\d f|\geq 1$ $\mm$-a.e. Then for every $y\in\M$ we have
\begin{equation}
\label{eq:almoststeep}
f(y)-f(x)\geq\ell(x,y)\qquad\mm-a.e.\ x.
\end{equation}
In particular, if $\ell_+$ is continuous, {real valued and for all} $x\ll y$ with $x,y\in\supp\mm$ we have $\mm(U\cap J(x,y))>0$ for every neighbourhood $U$ of $x$, then $\M$ has the Sobolev-to-steepness property. 
\end{theorem}
\begin{proof} The second claim is a trivial consequence of the first, so we focus on this one. The claimed inequality $f(y) - f(x) \geq \ell(x,y)$ is clear whenever $\ell(x,y)\in\{ -\infty,0\}$, thus we can assume $x\ll y$. In this case, by our infinity conventions (\Cref{Sub:InfConv}) the conclusion holds if either $x$ or $y$ do not belong to $\dom(f)$. Thus, we only have to treat the case $x,y\in\dom(f)$ with $x\ll y$.

By \Cref{Cor:Essgeo} for $\mm$-a.e.\ $x\in I^-(y)$ there is  a timelike geodesic $\gamma$ from $x$ to $y$ with image  contained in $\supp\meas$. Fix such $x$ and $\gamma$: we shall prove that \eqref{eq:almoststeep} holds for such $x$ and this suffices to conclude.

By \Cref{L:chronological continuity}, $\gamma$ crosses the set of discontinuity points of $f$ only countably many times. Hence, fix $s,t\in[0,1]$ with $s<t$ such that $f$ is continuous at $\gamma_s$ and $\gamma_t$ and let  $\varepsilon> 0$ be with $s+\varepsilon < t$. Let the metric $\met$ metrize our Polish topology of $\M$. For sufficiently small $r>0$, the induced ball $B_r(\gamma_s)$ has finite and positive $\meas$-measure (since $\gamma_s\in\supp\meas$) and is  contained in  $I^+(\gamma_0)\cap I^-(\gamma_t)$, as this set is open. In particular, $f$ is bounded on  $B_r(\gamma_s)$.


Since $\gamma$ is timelike, interior approximation yields a compact set $C\subset B_r(\gamma_s)$ with $\mm(C)>0$ and $\log(\ell(\cdot,\gamma_t))\in L^\infty(C,\mm\mres C)$.   Apply  \Cref{cor:intestplan} with   $\mu_0$ being the uniform distribution  of $C$ and $x_1$ being $\gamma_t$ to get  the existence of a   plan $\smash{\bdpi}^r$ concentrated on timelike  geodesics $\eta$ from $\eta_0$ to $\gamma_t$. The $L^\infty$ bounds in \Cref{Le:Existence test plans}  ensure that for any $\xi\in(0,1)$ the plan $\smash{\bdpi_\xi := (\restr_0^{1-\xi})_\push\bdpi^r}$ is a test plan. Since   $f$ is causal we get
\begin{equation}
\label{eq:persobtosteep}
\begin{split}
    f(\gamma_t) - \int f\, \d(\eval_0)_\push\bdpi^r = f(\gamma_t) - \int f \,\d(\eval_0)_\push\bdpi_\xi&\geq  \int \big[f(\sigma_1) - f(\sigma_0)\big]\,\d\bdpi_\xi(\sigma)\\ 
 \text{($\bdpi_\xi$ is test)}\qquad   &\geq \iint_0^1 \vert\rmd f\vert(\sigma_u)\,\vert\dot\sigma_u\vert\,\d u\,\d\bdpi_\xi(\sigma)\\
 \text{($|\d f|\geq 1$ $\mm$-a.e.\ and $\bdpi_\xi[\TGeo(\M)]=1$ 
 )}\qquad     &\geq \int \ell(\sigma_0,\sigma_1)\,\d\bdpi_\xi(\sigma)\\
    &= (1-\xi) \int \ell(\cdot,\gamma_t) \,\d(\eval_0)_\push\bdpi^r.
\end{split}
\end{equation}
Now notice that for $r>0$ sufficiently small  the ball $\smash{\bar{B}_r(\gamma_s)}$ is contained in the chronological past of $\gamma_{s+\varepsilon}$, hence the  reverse triangle inequality tells that $\int \ell(\cdot,\gamma_t) \d(\eval_0)_\push\bdpi^r \geq \ell(\gamma_{s+\varepsilon},\gamma_t) = (t-s-\varepsilon)\,\ell(\gamma_0,\gamma_1)$ for $r$ small. Thus letting first $r\downarrow0$ then $\xi\downarrow0$ in \eqref{eq:persobtosteep} using also that  $\gamma_s$ is a continuity point of $f$ we get
\[
f(y) - f(x) \geq f(\gamma_t)-f(\gamma_s)\geq  (t-s-\varepsilon)\,\ell(\gamma_0,\gamma_1).
\]
First letting $\varepsilon \to 0$, then  $s\downarrow 0$ and $t\uparrow 1$  gives the claim.
\end{proof}

\subsection{A metric Brenier--McCann theorem}

The purpose of this part is to link optimal transport to our Sobolev calculus. More concretely, in \Cref{Th:Brenier} we prove that optimal geodesic plans represent the \fff gradients of their  Kantorovich potentials after \Cref{Def:represent gradient}. Ultimately, this will be used to prove d'Alembert comparison theorems for such potentials in the next section. \Cref{Th:Brenier} itself is a variant of Brenier's famous polar factorization theorem \cite{Brenier:1991} generalized by McCann to Riemannian manifolds~\cite{McCann:01}. Its metric counterpart is due to Ambrosio--Gigli--Savaré \cite{AGS:14a} and Gigli \cite{Gigli:2015}. 
Versions of related results appear implicitly in the recent literature about 
optimal transport in Lorentzian spacetimes: see Suhr \cite{Suh:18}, McCann \cite{McCann:2020}, and \Cref{Re:Comp smooth} below.

We first recapitulate some basic notions about Kantorovich duality in the Lorentzian setting. We refer to McCann \cite{McCann:2020}, Mondino--Suhr \cite{MS:22}, and Cavalletti--Mondino \cite{CM:20} for details (and to Villani \cite{Villani:2009} for classical Kantorovich duality), even though our presentation differs slightly from the one in these references, as we are going to allow extended real valued cost functions and make use of the conventions in Section \ref{Sub:InfConv} in formulas such as \eqref{eq:defcconc} below.

\medskip


Let us thus consider a  cost function $c:\M\times \M\to\bar\R$ on a given Polish space $\M$. For $\psi:\M\to\bar  \R$ we define the two functions $\psi_c,\psi^c:\M\to\bar \R$ as
\begin{equation}
\label{eq:defcconc}
\psi_c(x):=\inf_{y\in \M}(-c(x,y)+\psi(y))\qquad\text{ and }\qquad\psi^c(y):=\sup_{x\in \M}(c(x,y)+\psi(x)).
\end{equation}
A function $f:=\M\to\bar\R$ is called $c$-concave  if it is of the form $\psi_c$   for some $\psi$. If the cost function $c$ never attains the value $+\infty$, then it is easy to see that
\begin{equation}
\label{eq:psicc}
(\psi_c)^c\leq \psi\qquad \text{ and }\qquad (\psi^c)_c\geq \psi 
\end{equation}
 for any $\psi$ (because $a+(-a+b)\leq b$ and $(-a)+(a+b)\geq b$ for any $a\in \R\cup\{-\infty\}$, $b\in\bar\R$). It follows that $f$ is $c$-concave iff $f=(f^c)_c$. In particular, we have
\begin{equation}
\label{eq:csuperdiff}
f^c(y)\geq  c(x,y)+f(x)\qquad\forall x,y\in \M.
\end{equation}
The set of couples for which both sides are real numbers and equal is called the $c$-superdifferential of $f$ and denoted $\partial_cf\subset \M\times \M$. Equivalently, we have
\begin{equation}
\label{eq:cdiff}
(x,y)\in \partial_cf\quad\Leftrightarrow\quad f(x),-f^c(y)>-\infty,\ c(x,y)\in\R\quad\text{ and }\quad f(x)+c(x,y)\geq f^c(y).
\end{equation}

We are interested in Polish metric spacetimes $(\M,\uptau,\ell)$ and in the  cost $c_q:=u_q\circ\ell$ for $0\neq q<1$. It is clear that in this case functions of the form $\psi_{c_q}$ or $\psi^{c_q}$ are rough causal --- meaning causal functions except possibly without any measurability. Indeed, if $x \leq y$, the reverse triangle inequality implies $\ell(x,z)\geq \ell(y,z)$ for every $z\in \M$, thus $-u_q(\ell(x,z))\leq -u_q(\ell(y,z))$ and therefore
\begin{align*}
\psi_{c_q}(x) &= \inf_{z\in \M} \Big[ \big(-u_q(\ell(x,z))\big)+\psi (z)\Big]\leq   \inf_{z\in \M} \Big[ \big(-u_q(\ell(y,z))\big)+\psi (z)\Big]  = \psi_{c_q}(y),
\end{align*}
showing that $\psi_{c_q}$ is rough causal. One can argue similarly for $\psi^{c_q}$. Other than this, and the regularity this implies according to \Cref{L:chronological continuity} and \Cref{cor:fmeastmcp}, it seems hard to obtain any regularity at all, the problem being that  for  $q\in(0,1)$ the cost function is not lower semicontinuous (regardless of any assumed continuity of $\ell_+$).

However, for $q<0$ the rough causality of $c_q$-concave functions is upgraded to causality by:


\begin{proposition}[On upper semicontinuity of $c_q$-concave functions for $\QQ<0$]\label{P:c-concave usc}
Let $(\M,\uptau,\ell)$ be a Polish metric spacetime with $\ell_+$ continuous and  $q<0$. Also, let $f:\M\to\bar\R$ be $c_q$-concave. 

Then $f$ and $-f^{c_q}$ are upper semicontinuous and $\partial_{c_q}f\cap\{\ell\in(0,+\infty)\}$ is $\sigma$-closed in $\{\ell\in(0,+\infty)\}$.
\end{proposition}
\begin{proof}
For $q<0$ the function $-u_q:\bar\R\to\bar\R$ is continuous, non-increasing and equal to $+\infty$ in 0. Hence the continuity of $\ell_+$ ensures that  $(x,y)\mapsto   -u_q(\ell(x,y))$ is   upper semicontinuous. Since the infimum of an arbitrary family of upper semicontinuous functions is still upper semicontinuous, the claims about  $f$ and $-f^{c_q}$ follow. In particular, for every $n\in\N$ the sets $\{f\geq -n\}$ and $\{f^{c_q}\leq n\}$ are closed and since $u_q\circ\ell$ is continuous  in $\{\ell\in[\tfrac1n,n]\}$, the last claim also follows.
\end{proof}

The next theorem is the main result of the section.  Recall the definition of the backward slope $\smash{\vert\partial^{-}f\vert}$ from \eqref{eq:defsubslopes}  and the notion of a plan representing the initial $p$-gradient of a causal function from \Cref{Def:represent gradient}.
\begin{theorem}[A metric Brenier--McCann theorem]\label{Th:Brenier}
Let $(\M,\uptau,\ell,\mm)$ be a forward mm spacetime satisfying either $(A)$ or $(B)$ in \Cref{Th:Lifting} and  so that $\ell_+$ is 
continuous and real valued. Also, let  $\tfrac1p+\tfrac1q=1$ be with $0\neq p,q<1$ and   $f\colon \M\to \bar\R$ 
a $c_q$-concave and $\mm$-measurable function. 

Let $\ppi$ be an initial test plan concentrated on $\TGeo(\M)$ such that $(\e_0,\e_1)_*\ppi$ is concentrated on  $\partial_{c_q}f$ and assume that:
\begin{itemize}
\item[i)] for some Borel  set $E\subset \M$ the plan $\ppi$ stays initially in $E$ and $|\rmd f|^\PP\in L^1(E,\mm\mres E)$,
\item[ii)] $\int\ell(\gamma_0,\gamma_1)^\QQ\,\d\ppi(\gamma)<+\infty.$ 
\end{itemize}  
Then $\ppi$ represents the \fff $\PP$-gradient of $f$  and 
\begin{equation}
\label{eq:brmcc}
|\partial^{-}f |(\gamma_0)= |\rmd f|(\gamma_0)=\ell(\gamma_0,\gamma_1)^{\QQ-1}\qquad \ppi-a.e.\ \gamma.
\end{equation}
\end{theorem}
\begin{proof}
Let $(x,y)\in \p_{c_q}f$ with $\ell(x,y)>0$. As already mentioned, this implies $f(x)\in\R$
by \eqref{eq:csuperdiff}--\eqref{eq:cdiff}.
Since $f^+ \ge f \ge f^-$ from 
\Cref{L:chronological continuity},
for every $z\leq x$ we have $\ell(z,y)\geq\ell(x,y)>0$ and thus
\begin{align*}
f^+(x)-f^-(z)\ge f(x)- f(z)\stackrel{\eqref{eq:cdiff}}\ge u_q(\ell(z,y)) -u_q(\ell(x,y))&\geq u_q\big((\ell(x,y)+\ell(z,x)\big)-u_q(\ell(x,y)).
\end{align*}
Notice that for every $\eps>0$ the set $\{z:\ell(z,x)<\eps\}$ is a neighbourhood of $x$ by the assumed continuity   of $\ell_+$, hence in taking the limit in the definition \eqref{eq:defsubslopes} of  $|\partial^{-}f|$ we have $\ell(z,x)\to 0$. 
Thus, dividing the above by $\ell(z,x)$ and letting $z\uparrow x$ we have  $|\partial^{-}f|(x)\geq \ell(x,y)^{\QQ-1}$ and thus
\begin{equation}
\label{eq:boundsubslope}
\frac{1}{\PP}|\partial^{-}f|^\PP(x)\geq \frac{1}{\PP}\ell(x,y)^\QQ.
\end{equation}
Let now $\gamma$ be a geodesic such that $\ell(\gamma_0,\gamma_1)>0$ and $(\gamma_0,\gamma_1)\in \p_{c_q}f$,
which again implies $f(\gamma_0)\in\R$.  
By the same argument as above, we obtain the estimate
\[
f(\gamma_t)-f(\gamma_0)\leq u_q(\ell(\gamma_0,\gamma_1))-u_q(\ell(\gamma_t,\gamma_1))=\tfrac{1}{\QQ}\ell(\gamma_0,\gamma_1)^\QQ(1-(1-t)^\QQ).
\]
Integrating and dividing by $t$  we get
\[
\begin{split}
\lims_{t\downarrow0}\int\frac{f(\gamma_t)-f(\gamma_0)}t\d\ppi(\gamma)\leq \int\ell(\gamma_0,\gamma_1)^\QQ\d\ppi(\gamma)&\stackrel{\eqref{eq:boundsubslope}}\leq\int \tfrac{|\partial^{-}f|^\PP}p(\gamma_0)\d\ppi(\gamma_0)+\int\tfrac{\ell(\gamma_0,\gamma_1)^\QQ}{\QQ}\d\ppi(\gamma)\\
\text{(by \Cref{prop:subslopes2})}\qquad\qquad&\stackrel{\phantom{\eqref{eq:boundsubslope}}}\leq \int \tfrac{|\rmd f|^\PP}p(\gamma_0)\d\ppi(\gamma_0)+\int\tfrac{\ell(\gamma_0,\gamma_1)^\QQ}q\d\ppi(\gamma),
\end{split}
\]
where the use of \Cref{prop:subslopes2} is justified by the assumed continuity of $\ell_+$. This shows $\ppi$ represents the \fff $\PP$-gradient of $f$. Recalling the bound in \Cref{Pr:Lower bound} (which is applicable, as both the integrability requirements are satisfied), we see that the last three inequalities must be equalities. Since the integrals are finite, this easily yield the claim \eqref{eq:brmcc}. 
\end{proof}

\begin{remark}[Comparison with the smooth case]\label{Re:Comp smooth} \Cref{Th:Brenier} should be compared to the shape of optimal maps and $\ell_\QQ$-geodesics in smooth spacetimes, cf.~McCann \cite[Thm.~5.8, Cor.~5.9]{McCann:2020} and Mondino--Suhr \cite[Lem.~3.2]{MS:22}. Indeed, for simplicity assume $\mu$ is concentrated on a compact set $C \subset I^-(o)$, where $o$ is a given point in a globally hyperbolic smooth spacetime $M$. Setting $\nu := \delta_o$ and 
\begin{align*}
f^o := -u_q(\ell(\cdot,o))\qquad(=-\tfrac{1}{\QQ}\ell(\cdot,o)^\QQ\text{ on }C)
\end{align*}
it is clear that the unique $\smash{\ell_\QQ}$-geodesic from $\mu$ to $\nu$ is given by
\begin{align*}
\mu_t := \exp_\cdot(t\,\|\rmd f^o\|_*^{\PP-2}\,\nabla f^o)_\push\mu,
\end{align*}
where $\PP$ is the conjugate exponent to $\QQ$. This $\ell_\QQ$-geodesic is clearly induced by the push-forward $\bdpi$ of $\mu$ via  the map $x\mapsto \exp_x(\cdot\,\|\rmd f^o\|_*^{\QQ-2}\,\nabla f^o)$. Therefore, $\bdpi$-a.e.~$\gamma\in \TGeo(\M)$ satisfies $\ell(\gamma_0,\gamma_1) = \|\rmd f^o\|_*^{\PP-1}(\gamma_0)$, which is \eqref{eq:brmcc}.

An analogous argument applies to more general situations covered by the results of \cite{McCann:2020,MS:22} quoted above.\hfill$\blacksquare$
\end{remark}
In order to apply this last result we need to be sure that $(\e_0,\e_1)_*\ppi$ is concentrated on $\partial_{c_q}f$. To this aim, the following two simple lemmas are useful.

\begin{lemma}[A Kantorovich/Karush--Kuhn--Tucker characterization of optimal couplings]
\label{le:basebase}
Let $(\M,\uptau,\ell)$ be a Polish metric spacetime with $\ell$ upper semicontinuous and never $+\infty$. Let $0\neq q<1$, $f:\M\to\bar\R$ be a $c_q$-concave function and $\mu,\nu\in\pem$ be so that  $f,f^{c_q}$ are $\mu$-measurable and $\nu$-measurable, respectively, with $f\in L^1(\mu)$ and $f^{c_q}\in L^1(\nu)$.

Assume that there is an admissible coupling $\pi$ of $(\mu,\nu)$ that is concentrated on $\partial_{c_q}f$. Then $\pi$ is optimal and every other    $c_q$-optimal coupling of $(\mu,\nu)$ is also concentrated on $\partial_{c_q}f$.
\end{lemma}
\begin{proof} Let $\pi'\in\Pi_\leq(\mu,\nu)$ be arbitrary. The marginal condition and the  integrability assumption imply that $(x,y)\mapsto f^{c_q}(y)-f(x)$ is in $L^1(\pi')$ and  
\begin{equation}
\label{eq:piprimo}
\int f^{c_q}\,\d\nu-\int f\,\d\mu=\int  f^{c_q}(y)-f(x)\,\d\pi'(x,y)\geq \int c_q(x,y)\,\d\pi'(x,y),
\end{equation}
the inequality being a consequence of  \eqref{eq:csuperdiff}. For $\pi$ admissible and  concentrated on $\partial_{c_q}f$ the above becomes
\[
\int f^{c_q}\,\d\nu-\int f\,\d\mu=\int  f^{c_q}(y)-f(x)\,\d\pi(x,y)= \int c_q(x,y)\,\d\pi(x,y),
\]
showing that $\int c_q \,\d\pi' \leq \int c_q \,\d\pi$, i.e.\ the optimality of $\pi$. Also, if $\pi'$ is optimal equality must hold in \eqref{eq:piprimo} and by the integrability assumptions --- that in particular  ensure that $c_q(x,y),f(x),f^{c_q}(y)\in\R$ for $\pi'$-a.e.\ $(x,y)$ --- this can only occur if equality holds in  \eqref{eq:csuperdiff}  for $\pi'$-a.e.\ $(x,y)$, proving that $\pi'$ is concentrated on $\partial_{c_q}f$, as desired.
\end{proof}

\begin{lemma}[Heredity of a $c_q$-concave dual potential by geodesic endpoints]
\label{le:altrocdiff}
Let $(\M,\uptau,\ell)$ be a Polish metric spacetime with $\ell$ upper semicontinuous and not attaining $+\infty$. Let $0\neq q<1$, $f:\M\to\bar\R$ be a $c_q$-concave function and $\nu_0,\nu_1\in\pem$ be such that $\ell_q(\nu_0,\nu_1)\in(0,+\infty)$ and any $c_q$-optimal plan from $\nu_0$ to $\nu_1$ is concentrated on $\partial_{c_q}f$.

Then for every geodesic $(\mu_t)$ from $\nu_0$ to $\nu_1$, any $c_q$-optimal coupling of $(\mu_0,\mu_1)$ is concentrated on $\partial_{c_q}f$ as well.
\end{lemma}

\begin{proof}
The assumptions on $\ell$ and the measures ensure that $c_q$-optimal couplings $\pi^1,\pi^2,\pi^3\in\Prob(\M^2)$ of $(\nu_0,\mu_0)$, $(\mu_0,\mu_1)$, $(\mu_1,\nu_1)$ exist. Let $\pi\in\Prob(\M^4)$ be a gluing of them along the common marginals, so that for $\pi$-a.e.\ $(x_0,x_1,x_2,x_3)$ we have $x_0\leq x_1\leq x_2\leq x_3$ and thus
\[
\begin{split}
\ell_q(\nu_0,\nu_1)\geq \Big(\int\ell^q(x_0,x_3)\,\d\pi\Big)^{\frac1q}\geq\Big(\int\ell^q(x_1,x_2)\,\d\pi\Big)^{\frac1q}=\ell_q(\mu_0,\mu_1)\stackrel{\eqref{eq:geo}}=\ell_q(\nu_0,\nu_1).
\end{split}
\]
By our assumption on $\nu_0,\nu_1$ and $f$ we deduce that for $\pi$-a.e.\ $(x_0,x_1,x_2,x_3)$ we have $(x_0,x_3)\in\partial_{c_q}f$ and $\ell(x_0,x_3)=\ell(x_1,x_2)$ (by \Cref{prop:geodindq} for this having assumed $\ell_q(\nu_0,\nu_1)\in(0,+\infty)$ matters). Also, from \eqref{eq:cdiff} we thus see that $f(x_0)+c_q(x_0,x_3)\geq f^{c_q}(x_3)$ holds for $\pi$-a.e.\ quadruple, thus from the causality of $f,f^{c_q}$, and  the causal relation $x_0\leq x_1\leq x_2\leq x_3$  and identity $\ell(x_0,x_3)=\ell(x_1,x_2)$  valid for $\pi$-a.e.\ quadruple, we conclude that  $f(x_1)+c_q(x_1,x_2)\geq f^{c_q}(x_2)$ holds for $\pi$-a.e.\ quadruple, meaning that $\pi^2$ is concentrated on $\partial_{c_q}f$. Since $\pi^2$ was an arbitrary $c_q$-optimal coupling for $(\mu_0,\mu_1)$, the proof is complete.
\end{proof}

We conclude the section with a result closely related to the metric Brenier--McCann theorem and to the Sobolev-to-steepness property. It offers yet another viewpoint on the role that existence of good geodesics have in relation to Sobolev calculus: the `rigidity' encoded  in equality \eqref{eq:dell} as opposed to inequality \eqref{eq:semprevalido} should be compared to that in \Cref{Cor:Intrinsic} as opposed to inequality \eqref{eq:dualsteep2}. Recall that \eqref{eq:dflsteep} and the 1-steepness of $\ell(o,\cdot)$ and $-\ell(\cdot, o)$ give that
 \begin{equation}
\label{eq:semprevalido}
|\d \ell(o,\cdot)|,\,|\d (-\ell(\cdot,o))|\geq 1,\qquad\mm-a.e.,
\end{equation}
on any metric measure spacetime. From the $\TMCP^\h_+(K,N)$ condition we can --- via   \Cref{Le:Existence test plans} ---  get equality  for the function $-\ell(\cdot, o)$; (equality for $\ell(o,\cdot)$ follows from the $\TMCP^\h_-(K,N)$ condition, see Section \ref{Sub:Modiback}).

\begin{corollary}[Unit maximal weak subslope of time separation to a point] \label{Th:dftau_is_1_under_tcd} Let  $0\neq \QQ<1$, $K\in\R$,  $N>1$ and  a $\smash{\TMCP^\h_+(K,N)}$ forward mm spacetime $(\M,\uptau,\ell,\meas)$ with
$\ell$ upper semicontinuous and never $+\infty$ be timelike $q$-essentially non-branching at 0 and satisfy $\mm(E)<\infty$ for every emerald $E$.   
Assume also that either $(A)$ or $(B)$ of \Cref{Th:Lifting} holds.

Let  
$o\in \supp\mm$ and consider the functions 
\[
f^o:=-u_q\circ \ell(\cdot,0),\qquad\text{and}\qquad g^o := -\ell(\cdot,o).    
\]
Then
\begin{equation}
\label{eq:dell}
|\d f^o|=\ell(\cdot,o)^{q-1}\qquad\text{and}\qquad \vert\rmd g^o\vert = 1\quad\meas\textnormal{-a.e. on }I^-(o).
\end{equation}
\end{corollary}
\begin{proof} 
  Let $\smash{C \subset I^-(o)}$ be a compact set  with $\meas(C)>0$. Let $\mu_0\in\Prob(\M)$  denote the uniform distribution of $C$. By \Cref{cor:intestplan}  there exists an initial   test plan $\bdpi $ lifting a $\ell_q$-geodesic $(\mu_t)$ from $\mu_0$ to  $\delta_o$  concentrated on timelike geodesics $\gamma$ from $\gamma_0$ to $o$, thus  by  \Cref{Le:Distr der} we know  that given any $\varepsilon \in (0,1)$ sufficiently small, $\bdpi$-a.e.~$\gamma\in\CC([0,1];\M)$ satisfies
\begin{align*}
\varepsilon\,\ell(\gamma_0,o) &= g^o(\gamma_{\varepsilon}) - g^o(\gamma_0)\geq \int_0^{\varepsilon}  \vert\rmd g^o\vert(\gamma_r)\,\vert\dot\gamma_r\vert\d r= \ell(\gamma_0,o)\int_0^{\varepsilon}\vert\rmd g^o\vert(\gamma_r)\d r.
\end{align*}
It follows that $\tfrac1\eps\int_0^{\varepsilon}\vert\rmd g^o\vert(\gamma_r)\d r\leq1$, thus an integration and the  same arguments as for the proof of \eqref{Eq:IINT} give
\[
 \int\vert \rmd g^o\vert\,\d\mu_0= \int \vert \rmd g^o\vert(\gamma_0)\,\d\bdpi(\gamma) \leq\lim_{t\to 0} \frac{1}{t}\iint_0^t \vert \rmd g^o\vert(\gamma_r)\,\d r\,\d\bdpi(\gamma)\leq 1.
\]
By \eqref{eq:semprevalido}, this forces $\vert \rmd g^o\vert =1$ $\meas\mres C$-a.e. The arbitrariness of $C$ and inner regularity of $\mm$ conclude the proof of the claim for $g^o$. The conclusion for $f^o$ now follows from the chain rule in \Cref{Th:Chain rule II}.
\end{proof}

\subsection{Nonlinear $\PP$-d'Alembert comparison}\label{subsec-q-dal-comp}
We turn to our main result: the nonlinear $\PP$-d'Alembert comparison theorem in weak form.

%

\medskip

To state the result it is worth to introduce the function  $\smash{\tilde{\tau}_{K,N}(\theta)}$ as
\begin{align*}
\text{ for $\smash{\theta \in [0, \pi\sqrt{\tfrac{N-1}{ K_+}}\,)}$ we put}\quad  
  \tilde{\tau}_{K,N}(\theta) = \begin{cases}
        \displaystyle \frac1N+\frac{\theta}{N}\sqrt{K(N-1)}\cot\!\Big(\theta \sqrt{\frac{K}{N-1}}\Big) & \textnormal{if }K>0,\\
        1 & \textnormal{if }K=0,\\
        \displaystyle \frac1N+\frac{\theta}{N}\sqrt{-K(N-1)}\coth\!\Big(\theta\sqrt{\frac{-K}{N-1}}\Big)  & \textnormal{if }K<0,
    \end{cases}
\end{align*}
with $\tilde{\tau}_{K,N}\big(\pi\sqrt{\tfrac{N-1}{ K_+}}\big):=-\infty$ and $K_+ := K\vee 0$. 

Notice that   $\smash{\tilde{\tau}_{K,N}(\theta)}$ is the derivative at 1 of the function $\smash{r\mapsto \tau_{K,N}^{(r)}(\theta)}$ from \eqref{tau}. We  have:
\begin{theorem}[$\PP$-d'Alembert comparison for $-\tfrac1q\ell^q(\cdot,o)$]\label{Cor:Time I} Fix  $0\neq \QQ<1$, $K\in\R$,  $N>1$ and let  
a $\smash{\TMCP^\h_+(K,N)}$ forward mm spacetime $(\M,\uptau,\ell,\meas)$ satisfy:
\begin{itemize}
\item[-] $\mm(E)<\infty$ for every emerald $E$;
\item[-]   $\ell_+$ is  continuous and real valued;
\item[-] timelike $q$-essential non-branchingness at 0;
\item[-] either $(A)$ or $(B)$ of \Cref{Th:Lifting}, which we recall here for clarity:
\begin{itemize}
\item[A)] $\M$ is globally hyperbolic, i.e. emeralds are compact, or
\item[B)] the topology on $\M$ is locally causally convex.
\end{itemize}
\end{itemize}
Let $o$ be a given point in $\supp\meas$ and put $f := -u_q\big(\ell(\cdot,o)\big)$. Then for every  $\varphi\in\Pert(f)$   non-negative, bounded and with  support in an emerald contained in $I^-(o)$ we have
\begin{equation}
\label{eq:dalcomp}
\int \rmd^+\varphi(\nabla f)\,\vert\rmd f\vert^{\PP-2}\,\d\meas \leq N\int \tilde{\tau}_{K,N}\circ\ell(\cdot,o)\,\varphi\,\d\meas,
\end{equation}
where  $\PP^{-1}+\QQ^{-1}=1$. 
\end{theorem}

\begin{proof} Notice that the integral in the left hand side is well-defined (a priori possibly $+\infty$) by the bound \eqref{eq:unifdiffquot} and the identity $|\d f|^p=\ell(\cdot,o)^q$ in $I^-(o)$ that follows from \eqref{eq:dell} because on $\supp\varphi$ the function $\log\ell(\cdot,o)$ is bounded. The integral on the right is well-defined (a priori possibly $-\infty$)  because on $\supp\varphi$ the function $ \tilde{\tau}_{K,N}\circ\ell(\cdot,o)$ is bounded from above (if $K>0$  we use the Bonnet-Myers-type estimate that ensures that the timelike diameter of a $\smash{\TMCP_+(K,N)}$ spacetime is  $\leq \pi\sqrt{\tfrac{N-1}{K}}$ --- see \cite[Proposition 5.10]{CM:20}\cite[Cor.~3.14]{Braun:2023Renyi} and Remark \ref{R:q<0,forward}).


We only discuss the details of the proof in the situation $K>0$, as the degeneracy of the involved distortion coefficients close to the Bonnet--Myers $\ell$-diameter bound $\pi\sqrt{(N-1)/K}$ requires a small extra argument. Given $K'\in (0,K)$ fixed until the end, $(\M,\ell,\meas)$ obeys the slightly weaker property $\smash{\TMCP_+(K',N)}$ \cite[Prop.~4.6]{Braun:2023Renyi}.

As $\meas$ is  finite on emeralds,  the assumptions on $\varphi$ imply its $\meas$-integrability. The conclusion of the theorem immediately follows from the locality asserted by \Cref{Pr:Calc df}(iv) if $\meas(\supp \varphi) = 0$, hence we assume $\supp \varphi$ is not $\meas$-negligible. 

We start with some preparations. First, we notice that $f$ is $c_q$-concave (just pick $\psi$ to be equal to 0 on $o$ and to $+\infty$ everywhere else in formula \eqref{eq:defcconc}) and that $\supp\varphi\times\{o\}\subset \p_{c_q}f$ (directly from \eqref{eq:cdiff}, the bound $f^{c_q}\leq\psi$ that comes from \eqref{eq:psicc} and the fact that $c_q(\cdot,o)$ is finite on  $I^-(o)$).

Then, we ``raise'' $\varphi$ to make it uniformly bounded away from zero. To this aim, let $A:=J^+(\supp\varphi)\cap J^-(o)$ denote the emerald spanned by $\supp \varphi$ and $\{o\}$, so $\mm(A)<+\infty$ by assumption,  and notice that for every $\alpha>0$ the function $\varphi_\alpha := \varphi + \alpha\,1_A$ still belongs to $\Pert({f})$ (because $f=+\infty$ outside $J^-(o)$).  Now define 
 the Borel probability measure $\mu_0 := \rho_0\,\meas$, where
\[
\rho_0 := (c_{\alpha}\,\varphi_\alpha)^{\tfrac{N}{N-1}}\qquad\text{$c_\alpha$ being the normalization constant.}
\]
Let $\mu_1$ and $\bdpi\in\Prob(\CC([0,1];\M))$  be given by \Cref{cor:intestplan} (with $K'$ in place of $K$)  and let $(\mu_t) =((\eval_t)_* \bdpi))$ be the $\ell_\QQ$-geodesic from $\mu_0$ to $\delta_o$ induced by $\bdpi$. An application of \Cref{le:altrocdiff}, together with the fact that  $\mu_0\times\delta_{o}$ is the only admissible plan for $(\mu_0,\delta_{o})$ show that $(\e_0,\e_1)_*\ppi$ is concentrated on $\partial_{c_q}f$. Let also $\delta:=\inf_{\supp\varphi}\ell(\cdot,o)$, notice that $\delta>0$  by the continuity assumption on $\ell_+$ and put $E:=A\cap\{\ell(\cdot,o)>\tfrac\delta2\}$. Then $\ppi$ stays initially in $E$ and $|\d f|^p\in L^1(E,\mm\mres E)$ (as $\mm(E)<\infty$ and recalling \eqref{eq:dell}). We can thus apply \Cref{Th:Brenier} and conclude that $\ppi$ represents the initial $p$-gradient of $f$.

With this said, we know that \eqref{TMCP} holds with $K'$ in place of   $K$, thus subtracting $\scrS_N(\mu_0)$ from both sides and dividing by $t$ leads to
\begin{align*}
\frac{\scrS_N(\mu_t) - \scrS_N(\mu_0)}{t} \leq \int \frac{1-\tau_{K',N}^{(1-t)}\circ\ell(\gamma_0,\gamma_1)}{t}\,\rho_0(\gamma_0)^{-1/N}\d\bdpi(\gamma),
\end{align*}
and since --- by discussion made at the beginning --- the functions  $ \frac{1-\tau_{K',N}^{(1-t)}\circ\ell(\gamma_0,\gamma_1)}{t}$ are uniformly bounded from above on $\supp\ppi$, by Fatou's lemma for the $\limsup$ and the fact that $\ell(\gamma_0,\gamma_1)=\ell(\gamma_0,o)$ for $\ppi$-a.e.\ $\gamma$ (by \Cref{cor:intestplan})   we get
\begin{equation}
\label{eq:upperbounddalcomp}
\begin{split}
\limsup_{t\downarrow 0} \frac{\scrS_N(\mu_t) - \scrS_N(\mu_0)}{t} &\leq \int \tilde{\tau}_{K',N}\big( \ell(\gamma_0,\gamma_1)\big)\,\rho_0(\gamma_0)^{-\tfrac1N}\d\bdpi(\gamma)=c_{\alpha} \int \varphi_\alpha\,\tilde{\tau}_{K',N}\big(\ell(\cdot,o)\big)\d\meas.
\end{split}
\end{equation}
On the other hand, given any $t\in (0,1)$ the convexity of $z\mapsto \s_N(z):=-z^{1-\tfrac1N}$ gives
\begin{align*}
\frac{\scrS_N(\mu_t) -  \scrS_N(\mu_0)}{t} &= \int \frac{\s_N\circ \rho_t - \s_N\circ \rho_0}{t}\d\meas\\
&\geq \int \s_N'\circ \rho_0\,\frac{\rho_t - \rho_0}{t}\d\meas= \int \frac{\s_N'\circ \rho_0(\gamma_t) - \s_N'\circ\rho_0(\gamma_0)}{t}\d\bdpi(\gamma)
\end{align*}
so that  combining the metric Brenier--McCann \Cref{Th:Brenier}  with the first order differentiation formula in \Cref{thm:horver} (notice that  $\s'_N\circ\rho_0\in\Pert(f)$ by item $(iii)$ in \Cref{Pr:RelationsII} --- here and below  it matters having picked  $\alpha>0$, so that the relevant  functions are Lipschitz in the range of $\rho_0,\varphi_\alpha$ and thus an argument based on locality justifies the computations) we get
\begin{equation}
\label{eq:anchedopo3}
\liminf_{t\downarrow0}\frac{\scrS_N(\mu_t) -  \scrS_N(\mu_0)}{t}\geq \int \d^+(\s'_N\circ\rho_0)(\nabla f)\,|\d f|^{p-2}\,\rho_0\,\d\mm.
\end{equation}
It is now convenient to introduce the `pressure' ${\sf p}_N(z):=z\,\s_N'(z)-\s_N(z)=\tfrac1Nz^{1-\frac1N}$, so that we have ${\sf p}_N'(z)=z\s_N''(z)$ and ${\sf p}_N(\rho_0)=\tfrac1Nc_{\alpha}\,\varphi_\alpha$. Then  the chain rule \eqref{eq:chaing1} gives
\[
\begin{split}
\int \d^+(\s'_N\circ\rho_0)(\nabla f)\,|\d f|^{p-2}\,\rho_0\,\d\mm&=\int \rho_0\,\s_N''(\rho_0)\d^+\rho_0(\nabla f)\,|\d f|^{p-2}\,\d\mm\\
&=\int \d^+({\sf p}_N\circ\rho_0)(\nabla f)\,|\d f|^{p-2}\,\d\mm=\frac{c_{\alpha}}{N}\int  \rmd^+\varphi(\nabla f)\,\vert\rmd f\vert^{\PP-2}\d\meas,
\end{split}
\]
having used also \eqref{eq:locg} in the last step to replace $\varphi_\alpha$ with $\varphi$. Coupling all this with \eqref{eq:upperbounddalcomp} we get
\begin{align*}
\int  \rmd^+\varphi(\nabla f)\,\vert\rmd f\vert^{\PP-2}\d\meas \leq N\int \varphi_\alpha\, \tilde{\tau}_{K',N}\big(\ell(\cdot,o)\big)\,\d\meas.
\end{align*}
and letting first   $\alpha\downarrow 0$ and then $K'\uparrow K$, using  the monotone  convergence theorem (recall the comments at the beginning of the proof to see that the integrand in the right hand side is uniformly bounded from above in $\alpha\in(0,1)$, $K'\in(K/2,K)$) we conclude.
\end{proof}
From the above result and the chain rules we established in Section \ref{se:calcrules} we can derive similar comparison results for other functions of the time separation. The following one is particularly relevant:
\begin{corollary}[$\PP$-d'Alembert comparison for Lorentz distance]\label{Cor:Time dist II} Let $p,q$ and $(\M,\uptau,\ell,\meas)$ be as in \Cref{Cor:Time I} above,   $o$ be a point in $\supp\meas$ and set $g := -\ell(\cdot,o)$.

Then for every non-negative and bounded $\varphi\in \Pert(g)$ supported in an emerald inside $I^-(o)$ we have
\begin{equation}
\label{eq:daldist}
\int\rmd^+\varphi(\nabla g)\,\vert\rmd g\vert^{\PP-2}\,\d\meas \leq \int \frac{N\,\tilde{\tau}_{K,N}\circ \ell(\cdot,o) - 1}{\ell(\cdot,o)} \,\varphi\,\d\meas.
\end{equation}
\end{corollary}
\begin{proof} The fact that the integrals are well defined follows as in the proof of \Cref{Cor:Time I} above. Let $f$ be as in  \Cref{Cor:Time I}, notice that  $\varphi$ is emerald supported  in $I^-(o)$ so that  the continuity of $\ell_+$ ensures that there is a $c$-steep function $\psi_q:\R\to\R$ for some $c>0$ such that $g=\psi_q\circ f$ holds on $\supp(\varphi)$.  By direct computation we see that $(\psi_q')^{p-1}\circ f=\ell(\cdot,o)^{-1}$, thus from the chain rule \eqref{eq:chainf} and the locality property \eqref{eq:locf} we see that
\[
\rmd^+\varphi(\nabla g)\,\vert\rmd g\vert^{\PP-2}=\ell(\cdot,o)^{-1}\rmd^+\varphi(\nabla f)\,\vert\rmd f\vert^{\PP-2}\qquad\mm-a.e.\ on\ \supp(\varphi).
\]
Similar considerations justify the use of the Leibniz rule \eqref{eq:leib2} and the chain rules \eqref{eq:chaing1}, \eqref{eq:chainf} to get that
\[
\begin{split}
\ell(\cdot,o)^{-1}\rmd^+\varphi(\nabla f)\,\vert\rmd f\vert^{\PP-2}&\leq\rmd^+(\ell(\cdot,o)^{-1}\varphi)(\nabla f)\,\vert\rmd f\vert^{\PP-2}-\varphi\, \rmd^+(\ell(\cdot,o)^{-1})(\nabla f)\,\vert\rmd f\vert^{\PP-2}\\
\text{(using also \eqref{eq:dell})}\qquad&=\rmd^+(\ell(\cdot,o)^{-1}\varphi)(\nabla f)\,\vert\rmd f\vert^{\PP-2}-\varphi \,\ell(\cdot,o)^{-1}
\end{split}
\]
holds $\mm$-a.e.\ on $\supp(\varphi)$. Now we use \Cref{Pr:RelationsII} to see that the assumption $\varphi\in\Pert(g)$ implies that $\tfrac{\varphi}{\ell(\cdot,o)}\in\Pert(f)$, so that the conclusion follows from \Cref{Cor:Time I}.
\end{proof}

\begin{remark}[Eschenburg's d'Alembert comparison]\label{Re:Esch} 
The bound \eqref{eq:daldist} is a weak form of the heuristic inequality 
\begin{align*}
    \Box_\PP(-\ell(\cdot,o)) \leq \frac{N\,\tilde{\tau}_{K,N}\circ \ell(\cdot,o) - 1}{\ell(\cdot,o)}\quad\textnormal{on }I^-(o),
\end{align*}
where $\Box_\PP :=-\textnormal{div}(\vert\rmd \cdot\vert^{\PP-2}\,\nabla\,\cdot)$ is the \emph{$\PP$-d'Alembert operator}.  Notice also that  for $K=0$ the estimate \eqref{eq:daldist}   reads as
\begin{align*}
\int\rmd^+\varphi(\nabla g)\,\vert\rmd g\vert^{\PP-2}\,\d\meas \leq \int \frac{N-1}{\ell(\cdot,o)}\,\varphi\,\d\meas,
\end{align*}
which is already sharp \cite{Treude:2011,TreudeGrant:2013}. Since $\rmd g$ has unit magnitude $\meas$-a.e.~on $I^-(o)$ thanks to \Cref{Th:dftau_is_1_under_tcd}, this extends Eschenburg's  d'Alembert comparison theorem from \cite[§5]{Eschenburg:1988} across the past timelike cut locus of $o$.  {For smooth spacetimes, 
a 
simple direct proof of this extension using more classical techniques can be found in \cite{QuintetEllipticsplitting},  where the range of validity of the theorem is widened to include timelike geodesically complete spacetimes that need not be forward.} 
\hfill$\blacksquare$
\end{remark}

Theorem \ref{Cor:Time I} formally establishes  an upper bound for $\Box_\PP f^o$ and since $\Box_\PP$ is elliptic with $\Box_\PP f=\Box_\PP(f+c)$ for any $c\in\R$, we expect from the maximum principle to be able to establish upper bounds for $\Box_\PP f$ for any $c_q$-concave function $f$. This expectation turns out to be correct, but   rather than being based on the maximum principle, our argument is a variant of the one just used to prove \Cref{Cor:Time I}. To carry it out we shall need a non-branching assumption to apply \Cref{Le:Existence test plans2} and get suitable good geodesics which, as before, can be used to link  lower Ricci bounds and Sobolev calculus. Before stating the result we require a preliminary lemma.

\begin{lemma}[Timelike diameter bounds the steepness of $c_q$-concave functions]
\label{le:stimasteep}
Let $(\M,\ell)$ be a metric spacetime with $L:=\sup_{x\leq y}\ell(x,y)<+\infty$ and $0\neq p,q<1$ satisfy $\tfrac1p+\tfrac1q=1$.

Then any $c_q$-concave function is $L^{q-1}$-steep. In particular, if $(\M,\uptau,\ell,\mm)$ is a mm spacetime and $f$ is a  $\mm$-measurable $c_q$-concave  function, then
\[
|\d f|^{p-1}\leq  L\qquad\mm-a.e,\quad\text{where}\qquad L:=\sup_{x\leq y}\ell(x,y).
\]
\end{lemma}

\begin{proof}
The second claim follows from the first one and inequality \eqref{eq:dflsteep}. For the first we observe that, directly from the definition, we see that the infimum of an arbitrary family of $L^{q-1}$-steep functions is $L^{q-1}$-steep. Hence to conclude it suffices to show that for any $y\in\M$ and $c\in\R$ the function $x\mapsto -u_q(\ell(x,y))+c$ is $L^{q-1}$-steep. This, however, is obvious from the fact that $x\mapsto -\ell(x,y)$ is 1-steep and $z\mapsto -u_q(-z)$ has derivative bounded from below by $L^{q-1}$ (and thus is $L^{q-1}$-steep) on $[-L,0]$.
\end{proof}

For $\QQ<0$ the measurability hypotheses imposed on $f^{c_q}$ and $\bar F$
in \Cref{Th:Dalem comp}(i)-(ii) are easily satisfied by combining \Cref{P:c-concave usc} with a measurable selection theorem.
To focus on the new additional difficulties, we will ask in point $(iii)$ 
below an integrability assumption stronger than what is actually necessary; in particular, the result below does not extend Theorem \ref{Cor:Time I}  (compare with the definition of $A$ in the proof of Theorem \ref{Cor:Time I}).

\begin{theorem}[$\PP$-d'Alembert comparison for $c_q$-concave functions]\label{Th:Dalem comp}  Let $p,q$ and $(\M,\uptau,\ell,\meas)$ be as in \Cref{Cor:Time I} and assume also that $\M$ is  timelike $\QQ$-essentially non-branching.

Let $f:\M\to\bar\R$ be $c_q$-concave and  $\varphi\in\Pert(f)$ non-negative, bounded,  and emerald  supported. Assume  that:


\begin{itemize}
\item[i)] $f^{c_q}$ is  Borel and $f$ is bounded on $\supp\varphi$;
\item[ii)] there is a Borel $\bar F:\supp\varphi\to \M$ and an emerald $E$ with
$\bar F(\supp\varphi) \cup \supp\varphi \subset E$. Moreover, for some  $c>0$ we have  $(x,\bar F(x))\in\partial_{c_q}f$ and $\ell(x,\bar F(x))\in(c,\tfrac1c)$  for $\mm$-a.e.\ $x\in\supp\varphi$;
\item[iii)] 
we have  $\vert\rmd f\vert^p\in L^1(E, \meas\mres E)$.
\end{itemize}
Then
\begin{equation}
\label{eq:gendalcomp}
\int \rmd^+\varphi(\nabla f)\,\vert\rmd f\vert^{\PP-2}\,\d\meas \leq N\int \,\tilde{\tau}_{K,N}\circ\vert\rmd f\vert^{\PP-1}\,\varphi\,\d\meas.
\end{equation}
\end{theorem}
\begin{proof}
The  backbone of the proof is the same of that of \Cref{Cor:Time I}, thus we shall focus on the differences. 
Recall that $c_q$-concave functions are rough causal, as discussed right after \eqref{eq:cdiff}, and then 
notice the assumption $\ell(x,\bar F(x)) \ge c >0$ for $\mm$-a.e. $x \in \spt \varphi$ ensures $\varphi\mm$ is concentrated on  $\M\setminus \M_{\sf {fin}}$. \Cref{cor:fmeastmcp} therefore implies that $f$ and the integrands in \eqref{eq:gendalcomp} are well defined and $\mm$-measurable.  The integral on the left  is well defined (a priori with value $+\infty$) by the bound \eqref{eq:unifdiffquot} and the integrability assumption $\vert\rmd f\vert^p\in L^1(E)$, that in particular yields integrability of $|\d f|^p$ on $\supp\varphi \subset E$. The integral on the right hand side is also well-defined because $\tilde{\tau}_{K,N}\circ\vert\rmd f\vert^{\PP-1}$ is (essentially) uniformly bounded from above on $\{\varphi>0\}$. Indeed, in the course of the proof we will show that
\begin{equation}
\label{eq:dafarvedere}
|\d f|^{p-1}(x)\in[c,\tfrac1c]\qquad\mm-a.e.\ on\ \supp\varphi.
\end{equation}
If $K\leq0$ this suffices to conclude, as  $\tilde\tau_{K,N}:\R^+\to\R$ is continuous. If $K>0$ we notice on one side that $\tilde\tau_{K,N}$ is bounded from above on $[0,\pi\sqrt{\tfrac{N-1}{K}}]$ and on the other that $|\d f|^{p-1}\leq \pi\sqrt{\tfrac{N-1}{K}}$ $\mm$-a.e.\ on $\supp\varphi$, this latter bound being  a consequence of  the Bonnet-Myers-type estimate that ensures that the timelike diameter of a $\smash{\TMCP_+(K,N)}$ spacetime is  $\leq \pi\sqrt{\tfrac{N-1}{K}}$ (see \cite[Proposition 5.10]{CM:20} \cite[Cor.~3.14]{Braun:2023Renyi}) and Remark \ref{R:q<0,forward}) and \Cref{le:stimasteep} above.

With this said,  as in the previous case we deal with the technically slightly more involved case $K>0$ and fix $K'\in(0,K)$. Also as before, we know that  $\varphi\in L^1(\mm)$ and we can assume $\mm(\{\varphi>0\})>0$.

 Let $\smash{\tilde{f}\colon \M\to\bar{\R}}$ be equal to $f$ on $J^-( E  )$ and $+\infty$ otherwise and notice that since $f$ is causal   so is $\tilde{f}$. It is then clear that  for any $\alpha>0$  we have $\varphi_{\alpha} := 1_{E}(\varphi + \alpha)\in\Pert(\tilde{f})$.

We can  then define  the Borel probability measure $\mu_0 := \rho_0\,\meas$, where
\[
\rho_0 := (c_{\alpha}\,\varphi_{\alpha})^{\frac N{N-1}},\qquad\text{$c_{\alpha }$ being the normalization constant.}
\]
Let  $\bar\mu_1:=\bar F_*\mu_0$ and notice that it is emerald supported because  $\bar F(\supp\varphi)$ is contained in some emerald by assumption. The assumptions also assert $f\in L^\infty(\mu_0)\subset L^1(\mu_0)$, and that the function $x\mapsto f^{c_q}(\bar F(x))=c_q(x,\bar F(x))+f(x)$ is  bounded (here we used that $\ell(x,\bar F(x))\in[c,\tfrac1c]$ $\mu_0$-a.e.) and thus $f^{c_q}\in L^\infty(\bar\mu_1)\subset L^1(\bar\mu_1)$. 
Recalling the $\mu_0 \ll\mm\mres{(\M\setminus \M_{\sf {fin}})}$ measurability of $f$ from above,
the assumption that $f^{c_q}$ is Borel yields the $\bar\mu_1$-measurability of $f^{c_q}$, so
 \Cref{le:basebase} ensures
 that  the coupling  ${\bar \pi}:=({\rm id},\bar F)_*\mu_0$ is optimal. Also, it is clearly concentrated on $\{\ell>0\}$ and we have $\ell_q(\mu_0,\bar\mu_1)\in(0,+\infty)$. We can therefore  apply  \Cref{Le:Existence test plans2} (with $K'$ in place of $K$) and obtain an $\ell_q$-geodesic $(\mu_t)$ from $\mu_0$ to $\bar\mu_1$ and a lifting   $\bdpi\in\Prob(\TGeo(\M))$ of it satisfying  \eqref{eq:insieme2}. We claim that $\ppi$ represents the initial $p$-gradient of $f$ on the set $E$ given in the statement and to this aim we shall verify that it satisfies the assumptions on the metric Brenier--McCann \Cref{Th:Brenier}. We know from \Cref{le:basebase} that every $\ell_q$-optimal coupling of $(\mu_0,\bar\mu_1)$ is concentrated on $\partial_{c_q}f$ and from \Cref{le:altrocdiff} that the same holds for $(\e_0,\e_1)_*\ppi$ (which  is $\ell_q$-optimal for $(\mu_0,\mu_1)$ by construction). The identity \eqref{eq:sameell} together with the definition of $\bar \pi$ yield $\log(\ell(\gamma_0,\gamma_1))\in L^\infty(\ppi)$ and since $\ppi$ is concentrated on timelike geodesics, this suffices to prove that the integrability assumption $(ii)$ in \Cref{Th:Brenier} holds. Since assumption $(i)$ of \Cref{Th:Brenier}  holds by our hypotheses $(iii)$ and since $\ppi$ does not leave $E$ we can indeed apply \Cref{Th:Brenier} to conclude that $\ppi$  represents the initial $p$-gradient of $f$. We notice also that for $F:\M\to\CC([0,1];\M)$ inducing $\ppi$ as in \Cref{Le:Existence test plans2}  we have 
\[
|\d f|^{p-1}(x)\stackrel{\eqref{eq:brmcc}}=\ell(x,F_{1}(x))\stackrel{\eqref{eq:sameell}}\in[c,\tfrac1c] \qquad\mu_0-a.e.\ x\in\M,
\]
so that \eqref{eq:dafarvedere} holds. With this said, the same computations leading to \eqref{eq:upperbounddalcomp} yield  
\begin{align*}
\limsup_{t\downarrow 0} \frac{\scrS_N(\mu_t) - \scrS_N(\mu_0)}{t} \leq c_{\alpha} \int \varphi_{\alpha}\,\tilde{\tau}_{K',N}\big(\vert\rmd f\vert^{\PP-1}\big)\,\d\meas= c_{\alpha} \int_E(\varphi+{\alpha})\,\tilde{\tau}_{K',N}\big(\vert\rmd f\vert^{\PP-1}\big)\,\d\meas,
\end{align*}
where the use of the reverse Fatou's lemma to bring the $\limsup$ inside the integral is justified  as in the proof of \Cref{Cor:Time I}. 
  
Since $\ppi$ does not leave the set $E$ where we have $f=\tilde f$, by \Cref{Th:Brenier} we deduce also that $\ppi$ represents the initial $p$-gradient of $\tilde f$. Since, again as in the proof of \Cref{Cor:Time I},  we have    $\s_N'\circ\rho_0\in \Pert(\tilde f)$,   the same computations as in \eqref{eq:anchedopo3} and the identity below it give
\[
\liminf_{t\downarrow 0}\frac{\scrS_N(\mu_t) -  \scrS_N(\mu_0)}{t} \geq \frac{c_{\alpha}}{N}\int   \rmd^+\varphi_{\alpha}(\nabla \tilde f)\,\vert\rmd \tilde f\vert^{\PP-2}\d\meas= \frac{c_{\alpha}}{N}\int   \rmd^+\varphi(\nabla f)\,\vert\rmd f\vert^{\PP-2}\d\meas,
\]
where for the equality we used the locality properties \eqref{eq:locg} and \eqref{eq:locf}. Coupling these last two bounds we get
\[
\int   \rmd^+\varphi(\nabla f)\,\vert\rmd f\vert^{\PP-2}\d\meas\leq    N \int_E (\varphi+{\alpha})\,\tilde{\tau}_{K',N}\big(\vert\rmd f\vert^{\PP-1}\big)\,\d\meas .
\]
To conclude by letting $\alpha\downarrow0$ and then $K'\uparrow K$ it suffices to prove that  $\int_E \tilde{\tau}_{K',N}\big(\vert\rmd f\vert^{\PP-1}\big)\,\d\meas<+\infty$. For $K=0$ this follows from $\mm(E)<+\infty$ and for $K<0$ from the fact that $\theta\mapsto\tilde\tau_{K,N}(\theta)$ has linear growth together with the extra integrability assumption made in item $(iii)$. For $K>0$ we use the Bonnet-Myers-type of timelike diameter bound already recalled together with the fact that on $[0,\pi\sqrt{\tfrac{N-1}{K}}]$ the functions $\tilde\tau_{K',N}$ are uniformly bounded from above (in $K\in(0,K)$) and  \Cref{le:stimasteep} below.
\end{proof}

\begin{remark}[Nonsharp variants]\label{Re:Sharp impr} Assuming $\smash{\TMCP_+^{\h,*}(K,N)}$ instead of $\smash{\TMCP^\h_+(K,N)}$ in \Cref{Th:Dalem comp}, \Cref{Cor:Time I}, and \Cref{Cor:Time dist II}, the resulting estimates hold --- with the same proof --- with $\tilde \sigma_{K,N}$ in place of $\tilde \tau_{K,N}$, where 
\begin{align*}
\tilde{\sigma}_{K,N}(\theta)
= \begin{cases}\displaystyle\theta\sqrt{\frac{K}{N}}\cot\!\Big(\theta\sqrt{\frac{K}{N}}\Big) & \textnormal{if }K>0,\\
1 & \textnormal{if }K=0,\\
\displaystyle\theta\sqrt{\frac{-K}{N}}\cot\!\Big(\theta\sqrt{\frac{-K}{N}}\Big) & \textnormal{if } K<0,
\end{cases}
\end{align*}
is the derivative at 1 of the function $\smash{r\mapsto \sigma_{K,N}^{(r)}}$ from \eqref{sigma}.\hfill$\blacksquare$
\end{remark}

\subsection{Distributional $\PP$-d'Alembert operator}
\label{S:defining square}

The aim of this section is to propose a first answer to the question: what  is the $p$-d'Alembertian? By no means is the presentation here exhaustive: on the contrary, our aim is more to  show viability of this research direction in the nonsmooth setting. For some further results see \cite{braun+}.

 Taking inspiration from Gigli's theory in positive signature \cite{Gigli:2015}, we see that the calculus developed above hints at the possibility of defining the $p$-d'Alembert operator `distributionally'. Notice indeed that in the smooth category from the definition $\Box_pf=-\div(|\d f|^{p-2}\nabla f)$ we see that 
\[
\int \d\varphi(\nabla f)|\d f|^{p-2}\,\d\mm=\int \varphi \,\Box_pf\,\d\mm,
\]
so that $\Box_p f$ can be defined weakly as the operator sending suitable $\varphi$'s to $\int \d\varphi(\nabla f)|\d f|^{p-2}\,\d\mm$. Now, in our context a priori we do not have a definition for $ \d\varphi(\nabla f)|\d f|^{p-2}$, but only of its  `proxies' $ \d^+\varphi(\nabla f)|\d f|^{p-2}$ and $- \d^+(-\varphi)(\nabla f)|\d f|^{p-2}$.   Still, in infinitesimally Minkowskian spaces these two agree on $\{|\d f|>0\}$   (recall \Cref{T:VD=HD}) and denoting  their common value by $ \d\varphi(\nabla f)|\d f|^{p-2}$, from \eqref{eq:homg} and \eqref{eq:superaddg} we see that the linearity relation
\begin{equation}
\label{eq:lind}
\d(\alpha_1\varphi_1+\alpha_2\varphi_2)(\nabla f)|\d f|^{p-2}=\alpha_1\d\varphi_1(\nabla f)|\d f|^{p-2}+\alpha_2\d \varphi_2(\nabla f)|\d f|^{p-2}\quad\mm-a.e.\ on\ \{|\d f|>0\}
\end{equation}
holds for any $\alpha_1,\alpha_2\in\R$  whenever   $\varphi_1,\varphi_2:\M\to\R$ are so that  $\pm\varphi_1,\pm\varphi_2\in\Pert(f)$ (notice that in this case by \Cref{Pr:prec} we have   $\alpha_1\varphi_1+\alpha_2\varphi_2\in\Pert(f)$ and that we are insisting that the functions be real valued). Let us give a name to this space of functions: for $f:\M\to\bar\R$ causal we put
\begin{equation}\label{PertSym}
\Pert^\sym(f):=\big\{\varphi:\M\to \R\ :\ \varphi,-\varphi\in \Pert(f)\big\}.
\end{equation}
For $U\subset M$ open we also define
\[
{\Pert_\be^\sym(f, U)}:=\big\{\varphi\in\Pert^\sym(f)\ :\ \supp\varphi\subset U\text{ is contained in an emerald and } \varphi\text{ is bounded}\big\}.
\]
The following is now natural:
\begin{definition}[Distributional $p$-d'Alembertian]\label{def:distrdal}
Let $(\M,\uptau,\ell,\mm)$ be an infinitesimally Minkowskian mm  spacetime, $U\subset \M$  open, $0\neq p<1$ and  $f:\M\to\bar\R$ causal. We shall say that $f$ has a distributional $p$-d'Alembertian on $U$ provided $ \d^+ \varphi(\nabla f)|\d f|^{p-2}\in L^1(U,\mm\mres U)$ for every  $\varphi\in \Pert_\be^\sym(f, U)$. In this case  the map
\begin{equation}
\label{eq:distrdal}
\Pert_\be^\sym(f, U)\ni\varphi\qquad\mapsto\qquad \int \d^+ \varphi(\nabla f)|\d f|^{p-2}\,\d\mm\in\R
\end{equation}
 will be called the distributional $p$-d'Alembertian of $f$ in $U$.
\end{definition}
Notice that the integrability assumption  $ \d^+ (\pm\varphi)(\nabla f)|\d f|^{p-2}\in L^1(U,\mm\mres U)$ together with \eqref{eq:dfzero1} and \eqref{eq:dfzero2} imply that  \eqref{eq:dpdm2} for $g=\varphi$ extends to the whole of $\M$ (the common value of both sides still denoted by $\d \varphi(\nabla f)|\d f|^{p-2}$) and we conclude that the distributional $p$-d'Alembertian, when it exists, is a linear operator. Notice also  that if  $|\d f|^p\in L^1(E,\mm\mres E)$ for every emerald $E\subset U$, then the bound \eqref{eq:unifdiffquot} easily implies that $f$ has a distributional $p$-d'Alembertian in $U$. 

We can then interpret a one-sided bound like the one provided by  \Cref{Cor:Time I} as structural information on the distributional $p$-d'Alembertian, which can be used to begin a regularity theory for it. For instance:

\begin{proposition}[Existence and bounds for the distributional $p$-d'Alembert operator]
\label{Pr:A bounded} Using the same assumptions and notation of \Cref{Cor:Time I} and assuming the spacetime to also be infinitesimally Minkowskian, the following holds. Let $U\subset \M$ be the open set $I^-(o)$ if $K\leq 0$ and $I^-(o)\cap\{\ell(\cdot,o)<\pi\sqrt{(N-1)/K}\}$ if $K>0$.

Then $f$ has a distributional $p$-d'Alembertian on $U$ and for every  $\smash{\varphi\in\Pert_\be^\sym(f,U)}$ non-negative we have
\begin{equation}
\label{eq:distrdalbound}
\int \d \varphi(\nabla f)|\d f|^{p-2}\,\d\mm-\int \varphi\,\d\nu\leq 0, \qquad\text{where}\qquad     \nu := N\,\tilde\tau_{K,N}\circ\ell(o,\cdot)\,\meas\mres U
\end{equation}
(notice that $\nu$ is a Radon measure on $U$). Moreover, assume that for some emerald $E\subset U$ there exists a non-negative `cut-off' function $\eta\in \Pert_\be^\sym(f,U)$ that is $\geq 1$ on $E$. Then for every $\varphi\in\Pert_\be^\sym(f,U)$ with support in $E$ we have
\begin{equation}
\label{eq:bounddal}
\big|\int \d \varphi(\nabla f)|\d f|^{p-2}\,\d\mm\big|  \leq \sup_{x\in U}|\varphi(x)|\Big(2|\nu|(E)\sup_{x\in \M}\eta(x)+\int \d \eta(\nabla f)|\d f|^{p-2}\,\d\mm\Big).
\end{equation}
\end{proposition}
\begin{proof} From \eqref{eq:dell}  we see that $|\d f|^p=\ell (\cdot,o)^q$, so that $|\d f|^p\in L^1(E,\mm\mres E)$ for every emerald $E\subset U$ and the previous discussion ensures that $f$ has distributional $p$-d'Alembertian on $U$. The bound \eqref{eq:distrdalbound} is a restatement of \eqref{eq:dalcomp}, so we turn to the last claim.

Notice that for $\varphi\in \Pert_\be^\sym(f,U)$ item $(v)$ in \Cref{Pr:prec} ensures that its positive and negative parts $\varphi^+$ and $\varphi^-$, whose support is contained in that of $\varphi$, are also in $ \Pert_\be^\sym(f,U)$.  Thus writing  ${\sf s}:=\sup_{x \in U}|\varphi(x)|$ for brevity,  the functions ${\sf s}\eta-\varphi^\pm$ are non-negative and in $\Pert_\be^\sym(f,U)$. Thus putting, again for brevity, $L(\varphi):=\int \d \varphi(\nabla f)|\d f|^{p-2}\,\d\mm$, the bound \eqref{eq:distrdalbound} tells us that
\[
\begin{split}
L(\varphi^\pm)&\leq \int\varphi^\pm\,\d\nu\leq {\sf s}|\nu|(E),\\
L(\varphi^\pm)&=-L({\sf s}\eta-\varphi^\pm)+{\sf s}L(\eta)\geq -\int {\sf s}\eta-\varphi^\pm\,\d\nu +{\sf s}L(\eta)\geq -{\sf s}(|\nu|(E)\sup_{x\in \M}\eta(x)+L(\eta))
\end{split}
\]
and the claim \eqref{eq:bounddal} follows.
\end{proof}
We collect a few informal comments on this result:
\begin{itemize}
\item[a)] There is nothing special about the function $f$ in the above  beside the bound \eqref{eq:distrdalbound} established in  \Cref{Cor:Time I}:  analogous statements can be made for $f$'s as in \Cref{Cor:Time dist II} or \Cref{Th:Dalem comp}, in the latter case provided we also assume the spacetime to be non-branching as in that theorem. Similar results are also in place in `past' $\TMCP^\h_-(K,N)$  spacetimes as discussed in   the next section.
\item[b)] The definition we gave for the distributional $p$-d'Alembertian depends on the space  $ \Pert_\be^\sym(f,U)$, thus it is important to know that this space is rich. We do not have general results in this direction unless the spacetime is smooth, where they are provided by \Cref{ss:null distance}.
\item[c)] From bounds like \eqref{eq:distrdalbound} and \eqref{eq:bounddal} and results like that of Riesz  and Daniell one can surely hope that the distributional $p$-d'Alembertian can be represented by a unique measure. See \Cref{Cor:SignedR1} and \cite{braun+} for positive results in this direction (in the latter case obtained by more constructive arguments inspired by those from \cite{CavallettiMondino20} in positive signature) and 
Remark~\ref{re:measuresfunctionals} for a caveat. 
\item[d)] In order for the linearity in \eqref{eq:lind} to be in place, and thus for the map in \eqref{eq:distrdal} to be linear, it is not necessary to assume infinitesimal Minkowskianity. As the arguments leading to \eqref{eq:lind} show, it suffices that $\d^+\varphi(\nabla f)|\d f|^{p-2}$ coincides with $-\d^+(-\varphi)(\nabla f)|\d f|^{p-2}$ whenever $\pm\varphi\in\Pert(f)$. Spaces with this property are called \emph{infinitesimally strictly concave} in  \cite{braun+} (see also the concept of \emph{infinitesimal strict convexity} defined in \cite{Gigli:2015} for the positive signature case).
\item[e)] 
Without assuming infinitesimal strict concavity, after noticing that the concavity \eqref{eq:superaddg}  yields convexity of $\varphi\mapsto -\d^+(-\varphi)(\nabla f)|\d f|^{p-2}$, one might say that  a distributional $p$-d'Alembertian of $f$ in $U$ is any linear functional $L:{\Pert_\be^\sym(f, U)}\to\R$ such that
\[
\int \d^+ \varphi(\nabla f)|\d f|^{p-2}\,\d\mm\leq L(\varphi)\leq \int  -\d^+(-\varphi)(\nabla f)|\d f|^{p-2}\,\d\mm\qquad\forall \varphi\in \Pert_\be^\sym(f, U).
\]
While this is doable (see also  \cite{Gigli:2015}  for the positive signature case and \cite{braun+} for more in the spacetime setting), solely from a bound like this one cannot deduce any relation between $L$ and the right hand side in  \eqref{eq:dalcomp}. Thus if one looks for a link like that in \Cref{Pr:A bounded} above, some sort of selection procedure is necessary.
\end{itemize}

\subsection{Modifications for past $\TMCP$  spaces}\label{Sub:Modiback}
All results from the previous parts of \Cref{Section: Effects of curvature assumption} have counterparts for the past timelike measure contraction property $\TMCP^\h_-(K,N)$ from \Cref{Def:TMCP} and these lead to $p$-d'Alembertian estimates for functions of the kind $\psi^{c_q}$, rather than $\psi_{c_q}$, defined in \eqref{eq:psicc}.

Roughly speaking,  $\TMCP^\h_-(K,N)$  ``reverses'' the transport direction of $\TMCP^\h_+(K,N)$ by imposing convexity properties of the involved entropies along chronological transports \emph{from} Dirac masses \emph{to} $\meas$-absolutely continuous distributions. Because of this, most of the results follow via minor technical modifications --- see also Remark \ref{re:sign} --- so that below we only give a brief outline of the main arguments.

We stress that we will still work with  the original causal structure set up by $\ell$ (and not its time-reversal). In particular this means that we will still work with forward spacetimes and with left continuous causal curves, rather than on backward spacetimes and right continuous curves, hence strictly speaking the results collected here cannot be derived from those already proved by time-reversal. Still, the `sign choice' in our completeness  assumption is almost invisible through our constructions, and reduces to the (inessential) possible discontinuity at $t=0$ discussed in \Cref{re:indirdisc} and to the (also inessential) fact that test plans are automatically continuous at $t=1$, so that the continuity requirement in the definition of final test plan could be removed (see also \Cref{R:forward-backward}).

\medskip

The following calculus rules are valid in any forward metric measure spacetime $(\M,\uptau,\ell,\mm)$. Here $0\neq \PP,\QQ<1$ are conjugate, i.e.\ $\tfrac1\PP+\tfrac1\QQ=1$.
\begin{itemize}
\item[(A)] \emph{(Final version of nonsmooth Fenchel-Young inequality --- compare with  \Cref{Pr:Lower bound})} 

Let $f:\M\to\bar \R$ be causal, $\ppi$ a final test plan that stays finally in the Borel set $E\subset M$. Then 
\[
\liminf_{s\uparrow 1} \int \frac{f(\gamma_1) - f(\gamma_s)}{1-s}\d\bdpi(\gamma)  \geq \frac{1}{\PP}\int \vert\rmd f\vert^\PP(\gamma_1)\d\bdpi(\gamma) + \liminf_{s\uparrow 1} \frac{1}{(1-s)\QQ}\iint_s^1 \vert\dot\gamma_r\vert^\QQ\d r\d\bdpi(\gamma)
\]
holds whenever
\begin{enumerate}[label=\textnormal{(\roman*)}]
    \item $\vert \rmd f\vert^p\in L^1(E,\meas\mres E)$ in the case $p<0$,
    \item the latter limit inferior is not equal to $-\infty$ in the case $q<0$.
\end{enumerate}
\item[(B)]\emph{(Plans representing  final $p$-gradients  --- compare with \Cref{Def:represent gradient})} 

With  $f,\ppi,E$ as above, we say that $\ppi$ represents the final $p$-gradient of $f$ provided 
\begin{enumerate}[label=\textnormal{(\roman*)}]
\item $\vert \rmd f\vert^p\in L^1(E,\meas\mres E)$ 
\item$\limsup_{s\uparrow 1} \frac{1}{(1-s)\QQ}\iint_s^1 \vert\dot\gamma_r\vert^\QQ\d r\d\bdpi(\gamma)<+\infty$ 
\item we have
\[
\limsup_{s\uparrow 1} \int \frac{f(\gamma_1) - f(\gamma_s)}{1-s}\d\bdpi(\gamma)  \leq \tfrac{1}{\PP}\int \vert\rmd f\vert^\PP(\gamma_1)\d\bdpi(\gamma) + \liminf_{s\uparrow 1} \tfrac{1}{(1-s)\QQ}\iint_s^1 \vert\dot\gamma_r\vert^\QQ\d r\d\bdpi(\gamma).
\]
\end{enumerate}
\item[(C)]\emph{(First order differentiation formula  --- compare with \Cref{thm:horver})}

 Let $\ppi$ represent the final $p$-gradient of $f$ on $E$. Then for every  $-g\in\Pert(f)$ we have:
 \begin{itemize}
 \item[(i)] for any $s<1$ sufficiently big    the positive  part of the function $g\circ\eval_1-g\circ\eval_s$ is in $L^1(\ppi)$;
 \item[(ii)] the positive part of the function $-\rmd^+(-g)(\nabla f)\,|\rmd f|^{\PP-2}$ is in $L^1(\mm\mres E)$;
 \item[(iii)] we have
\begin{equation}
\label{eq:horverder2}
\lims_{s\uparrow1}\int\frac{g(\gamma_1)-g(\gamma_{s})}{1-s}\d\bdpi(\gamma)
 \leq \int- \rmd^+(-g)(\nabla f)\,|\rmd f|^{\PP-2}(\gamma_1)\d\bdpi(\gamma).
\end{equation}
 \end{itemize} 
 \item[(D)] \emph{(Metric Brenier--McCann theorem --- compare with \Cref{Th:Brenier})}

Assume also the spacetime is forward complete, that either $(A)$ or $(B)$ of \Cref{Th:Lifting} hold, that $\ell_+$ is continuous and real valued and let $f\colon \M\to \bar\R$ be $\mm$-measurable and of the form $f=\psi^c$ for some $\psi:\M\to\bar\R$ (recall \eqref{eq:psicc}). Let $\ppi$ be a final test plan concentrated on $\TGeo(\M)$ such that $(\e_0,\e_1)_*\ppi$ is concentrated on the set $\partial^{c_q}f$ of couples $(x,y)$ such that $f_c(x),c_q(x,y), f(y)\in\R$ and $f_c(x)+c_q(x,y)\geq f(y)$. 
Assume also that:
\begin{itemize}
\item[i)] for some $\mm$-measurable set $E\subset \M$ the plan $\ppi$ stays finally in $E$ and $|\rmd f|^\PP\in L^1(E,\mm\mres E)$,
\item[ii)] $\int\ell(\gamma_0,\gamma_1)^\QQ\d\ppi(\gamma)<+\infty.$ 
\end{itemize}  
Then $\ppi$ represents the final $\PP$-gradient of $f$ on $E$ and 
\[
|\partial^{+}f |(\gamma_1)= |\rmd f|(\gamma_1)=\ell(\gamma_0,\gamma_1)^{\QQ-1}\qquad \ppi-a.e.\ \gamma.
\]

\end{itemize}
The above calculus rules are now coupled with a curvature condition: in the results below we shall assume that $(\M,\uptau,\ell,\mm)$ is:
\begin{itemize}
\item[-] A forward  $\smash{\TMCP^\h_-(K,N)}$ mm spacetime,
\item[-] $q$-essentially non-branching at 1,
\item[-] such that $\mm(E)<+\infty$ for every emerald $E$,
\item[-] such that $\ell_+$ is continuous and does not take the value $+\infty$,
\item[-] such that either $(A)$ or $(B)$ of \Cref{Th:Lifting} hold.
\end{itemize} 
Then:
\begin{itemize}
\item[(E)] \emph{(Existence of good final test plans with Dirac source --- compare with Thm. \ref{Le:Existence test plans} and Cor. \ref{cor:intestplan})} 

Let $\mu_1 = \rho_1\,\meas\in \Prob(\M)$ be   supported  in an emerald and have bounded compression.  Let $x_0\in \supp\meas$ satisfy $\log(\ell(x_0,\cdot))\in L^\infty(\mu_1)$. Then there is a final test plan $\bdpi$ concentrated on timelike geodesics $\gamma$ from $x_0$ to $\gamma_1$ so that for  $\mu_t:=(\eval_t)_\push\ppi$ we have:
\[
\begin{split}
\mu_t  &\leq \tfrac{1}{t^N}\,\mathrm{e}^{Dt\sqrt{K_-(N-1)}}\,\Vert \rho_1\Vert_{L^\infty(\M,\meas)},\\
\scrS_{N}(\mu_t)&\leq  -\int \tau_{K,N}^{(t)}\circ\ell(\cdot,x_1)\,\rho_1^{1-1/N}\d\meas,
\\ \scrS_{\infty}(\mu_t )&\le \scrS_\infty(\mu) 
-N \log \sigma_{K,N}^{(t)}\big( \|\ell(x_0,\cdot)\|_{L^2(\d\mu)}\big),
\end{split}
\]
for any $t\in(0,1]$, where $D := \sup\ell(\{x_0\}\times\supp\mu_1)$.
\begin{remark}[The role of forward completeness]\label{re:indirdisc} In the proof of the above, and the analogous results \Cref{Le:Existence test plans} and \Cref{cor:intestplan}, forward completeness is used to perform suitable limiting procedures (e.g.\ in the proof of the lifting theorem) but is not at all related to curvature conditions nor to the fact that here we start from a Dirac mass and end up in a diffused measure (contrary to what happens in the case of $\TMCP^\h_+$ spaces). The fact that the plan $\ppi$ in the above has marginal at time 1 equal to given measure $\mu_1$ comes from the assumption of $q$-essentially non-branching at 1, the density bounds and finite mass of emeralds together with \Cref{cor:contfromreg}.

On the other hand, we notice that $\mu_0$ is precisely $\delta_{x_0}$, which is a non-essential difference with \Cref{Le:Existence test plans} and \Cref{cor:intestplan}, where we could have $\mu_1\neq\delta_{x_1}$. This happens  because the discrete procedure used to build the final geodesic keeps $\mu_0$ always fixed, and eventually re-defines $\mu_1$ as narrow forward limit of measures defined on a suitable dense set of times (and the previous considerations ensure that such new definition coincides with the initially given measure $\mu_1$). Notice that in all this it might be that $(\mu_t)$ does not weakly converge to $\mu_0$ as $t\downarrow0$ (not even if we further assume that a narrow limit exists).
\hfill$\blacksquare$
\end{remark}
 \item[(F)] \emph{($p$-d'Alembertian comparison for $\tfrac1q\ell^q(o,\cdot)$ --- compare with \Cref{Cor:Time I})}
 
Assume also that $\ell_+$ is  continuous, let $o\in \supp\meas$ and define $f := \frac{1}{\QQ}\ell(o,\cdot)^\QQ$.

Then for every $\varphi:\M\to\R^+$ bounded, with $\supp\varphi\subset I^+(o)$  and  $-\varphi\in \Pert(f)$ we have
\begin{equation}
\label{eq:dalcomppast}
\int -\rmd^+(-\varphi)(\nabla f)\,\vert\rmd f\vert^{\PP-2}\,\d\meas \geq -\int N\,\tilde{\tau}_{K,N}\circ\ell(o,\cdot)\,\varphi\,\d\meas.
\end{equation}
\end{itemize}
An analogous comparison result is in place for $\ell(o,\cdot)$ and, under appropriate non-branching assumptions, also for more general functions $f$ of the form $f=\psi^c$ in analogy with  \Cref{Cor:Time dist II} and \Cref{Th:Dalem comp} . All the above results can be proved following rather pedantically the analogous proofs in the $\TMCP^\h_+(K,N)$ case. Let us, for added clarity, quickly and formally discuss how the last statement $(F)$ follows from the previous ones.

Let $\rho_1:=(c\varphi)^{N/(N-1)}$, $c>0$ being the normalization constant, and let $\ppi$ be given by item $(E)$  above. Then the (time-reversal of the) bound \eqref{TMCP} and the identity $\tau_{K,N}^{(1)}(\theta)=1$ give
\[
\frac{\scrS_N(\mu_1)-\scrS_N(\mu_s)}{1-s}\geq\int-\frac{\tau^{1}_{K,N}-\tau^{(s)}_{K,N}}{1-s}\circ\ell(x_0,\cdot)\rho_1^{1-\frac1N}\,\d\mm\qquad\stackrel{s\uparrow1}\to\qquad-c\int \tilde{\tau}_{K,N}\circ\ell(o,\cdot)\,\varphi\,\d\meas.
\]
On the other hand the bound $\s_N(b)-\s_N(a)\leq \s_N'(b)(b-a)$ for the convex function $\s_N(z):=-z^{1-1/N}$ gives
\[
\begin{split}
\frac{\scrS_N(\mu_1)-\scrS_N(\mu_s)}{1-s}\leq\int\frac{\s_N'(\rho_1(\gamma_1))-\s_N'(\rho_1(\gamma_s))}{1-s}\,\d\ppi(\gamma),
\end{split}
\]
thus recalling \eqref{eq:horverder2} and then the chain rules established in \Cref{se:calcrules} we deduce
\[
\begin{split}
\limsup_{s\uparrow0}\frac{\scrS_N(\mu_1)-\scrS_N(\mu_s)}{1-s}&\leq \int -\d^+(-\s_N'(\rho_1))(\nabla f)|\d f|^{p-2}\rho_1\,\d\mm\\
&=\frac cN\int -\d^+(-\varphi )(\nabla f)|\d f|^{p-2} \,\d\mm.
\end{split}
\]
Coupling this  with the previously found lower bound gives the desired conclusion \eqref{eq:dalcomppast}.

We conclude pointing out that, again in analogy with the $\TMCP^\h_+(K,N)$ case, combining the metric Brenier--McCann theorem with the existence of good final test plans give the following two results, valid in any mm spacetime as in $(F)$ above (possibly relaxing the continuity assumption on $\ell_+$ to upper semicontinuity):
\begin{itemize}
\item[(G)] \emph{(Unit maximal weak subslope of time separation from a point --- compare with \Cref{Th:dftau_is_1_under_tcd})}

For every $o\in \supp\meas$, the  functions $f :=u_q\circ \ell(o,\cdot)$ and $g:=\ell(o,\cdot)$ satisfy
\begin{equation*}
    \vert\rmd f\vert = \ell(o,\cdot)^{q-1}\qquad\text{and}\qquad     \vert\rmd g\vert = 1\qquad\meas\textnormal{-a.e. on }I^+(o).
\end{equation*}
\item[(H)] \emph{(Sobolev-to-steepness property --- compare with \Cref{Th:SobtoSteep})}

Assume also that the spacetime is $q$-essentially timelike non-branching at 1, that $\ell_+$ is continuous and that for any $x\ll y$ with $x,y\in\supp\mm$ we have $\mm(U\cap J(x,y))>0$ for any $U$ neighbourhood of $y$. Then $(\M,\uptau,\ell,\mm)$ has the Sobolev-to-steepness property.
\end{itemize}

\appendix

\section{Appendices}

\subsection[Compatibility with the smooth setting]{Compatibility with the smooth setting}
\label{ss:compatibility}
In this section we show how our calculus notions from the non-smooth setting reduce to well known ones if the underlying space (but not the objects we are differentiating) is smooth. We shall work with strongly causal spacetimes, i.e.\ smooth spacetimes $M$ so that for every $p\in M$ and neighbourhood $U$ of $p$ there is another neighbourhood $V\subset U$ such that any causal curve with endpoints in $V$ remains in $U$.

 We start with differentiation of curves:
\begin{theorem}[Speed of causal curves: smooth vs non-smooth]\label{thm:curvediff}
Let $(M,g)$ be a smooth strongly causal spacetime and $\gamma:[0,1]\to M$ a causal curve. 

Then:
\begin{itemize}
\item[i)] $\gamma$ is differentiable at a.e.\ $t$; 
\item[ii)]  for a.e.\ $t$ the Lorentzian norm $\|\gamma'_t\|$ of its derivative coincides with the Lebesgue density $|\dot\gamma_t|$ of the causal speed from \Cref{Def:causal speed}.
\end{itemize} 
\end{theorem}
\begin{proof} Consider first Minkowski space $M$. Then in a basis $e_1,\ldots,e_n$ such that the cone $\{\sum\alpha_ie_i:\alpha_i\geq0\}$ contains the future cone, the  curve $\gamma$ is increasing in each component, hence a.e.\ differentiable.  In the general case we can cover $M$ with charts $(U_i,\varphi_i)$ such that $\d\varphi_i$ sends the future cone at any $x\in U$ inside the future cone in flat Minkowski space. Thus the post-composition $\varphi_i\circ\gamma$, where defined, is a causal curve in Minkowski and thus a.e.\ differentiable. At any differentiability (and thus also continuity) point $t\in[0,1]$ we have $\gamma_t'=\lim_{h\to 0}v_{t,h}$ where the future vectors $v_{t,h}$ are defined as $\frac{\exp_{\gamma_t}^{-1}(\gamma_{t+h})}{h}$ (because the differential of $\exp_x$ at $x$ is --- by construction --- the identity). Thus if $t\in[0,1]$ is such that also \eqref{eq:speedlimit} holds we have 
\[
|\dot\gamma_t|=\lim_{h\downarrow 0}\frac{\ell(\gamma_t,\gamma_{t+h})}h=\lim_{h\downarrow 0}\frac{\|\exp_{\gamma_t}^{-1}(\gamma_{t+h})\|_{g_{\gamma_t}}}h=\|\gamma_t'\|_{g_{\gamma_t}},
\]
concluding the proof.
\end{proof}
We now want to study the `dual' problem of differentiability of causal functions  on smooth spacetimes. To do so it is worth recalling few basic facts about Lorentzian scalar products.

Thus let $V$ be a finite dimensional real vector space and $g$ a scalar product on it  with signature $(+,-,\dots,-)$. We shall always think such $g$ as coming with a fixed half $F\subset V$ of the cone $\{v :g(v,v)\geq 0\}$ to be called the future; $g$ induces on $F$ a hyperbolic norm $\|\cdot\|$ (see also the terminology in \Cref{ss:hyperbolic norms}) via the formula 
\begin{equation}
\label{eq:normdag}
\|v\|:=\sqrt{g(v,v)}.
\end{equation}
It is well known that on $F$ we have the reverse Cauchy--Schwarz inequality:
\begin{equation}
\label{eq:revCS}
g(v,w)\geq \|v\|\,\|w\|\qquad\forall v,w\in F.
\end{equation}
A direct way of proving this is to identify $(V,F)$ with the standard Minkowski space $\R^{1,n}$, $n\geq 0 $, and the standard future cone, with coordinates $(t,x)$, so that $(t,x)$ is a future vector iff $t\geq \|x\|_{\R^n}$ and $g((t,x),(s,y))=ts-\langle x,y\rangle$, where  $\|\cdot\|_{\R^n},\langle\cdot,\cdot\rangle$ are the Euclidean norm and scalar product respectively. Then the standard Cauchy--Schwarz inequality yields $g((t,x),(s,y))\geq ts-\|x\|_{\R^n}\,\|y\|_{\R^n}$ and \eqref{eq:revCS} follows easily. This choice of coordinates also helps check that
\begin{equation}
\label{eq:charF}
v\in F\qquad\Leftrightarrow\qquad g(v,v')\geq 0\qquad\forall v'\in F.
\end{equation}
Indeed, $\Rightarrow$ follows from \eqref{eq:revCS}. For $\Leftarrow$ we notice that if $v=(t,x)$ is not in $F$, then $t<\|x\|_{\R^n}$, which  easily implies the existence of $s\geq \|x\|_{\R^n}$ such that $ts<\|x\|_{\R^n}^2$. Thus  the choice $v':=(s,x)\in F$ proves the claim.

As $g$ is non-singular, it induces a musical (or Riesz) isomorphism $V\ni v\mapsto v^\flat\in V^*$ via $v^\flat:=g(v,\cdot)$.  Letting $V^*\ni\omega\mapsto \omega^\sharp\in V$ be the inverse isomorphism, $g$ induces a bilinear form $g^*$ on $V^*$ by the formula
\[
g^*(\omega_1,\omega_2):=g(\omega_1^\sharp,\omega_2^\sharp)\qquad\forall  \omega_1,\omega_2\in V^*
\]
and it is  clear that $g^*$ still has signature $(+,-,\dots,-)$. The dual cone  $F^*\subset V^*$ of $F$ is the collection of $\omega$'s such that $\omega(v)\geq  0$ for any $v\in F$. Then from \eqref{eq:charF} it directly follows that
\[
\omega\in F^*\qquad\Leftrightarrow\qquad \omega^\sharp\in F\qquad\qquad\qquad\text{and}\qquad\qquad \qquad v\in F\qquad\Leftrightarrow\qquad v^\flat\in F^*.
\]
Calling $\|\cdot\|_*$ the norm induced by $g^*$ on $F^*$, we also point out the duality formulas
\begin{equation}
\label{eq:dualvar}
\|\omega\|_*=\inf_{v\in F\atop \|v\|\geq 1}\omega(v)\qquad\qquad\text{and}\qquad\qquad\|v\| =\inf_{\omega\in F^*\atop \|\omega\|_*\geq 1}\omega(v)
\end{equation}
valid for any $v\in F$ and $\omega\in F^*$. Indeed the reverse Cauchy--Schwarz inequality gives $\omega(v)=g(\omega^\sharp,v)\geq \|\omega^\sharp\|\|v\|=\|\omega\|_*\|v\|$, while the choice $v:=\tfrac{\omega^\sharp}{\|\omega\|_*}$ for $\|\omega\|_*>0$ provides equality in the first.
If $\|\omega\|_*=0$ we perturb $\omega^\sharp$ as $\omega^\sharp_\eps$ with $\|\omega^\sharp_\eps\|=1$ and $\omega(\omega^\sharp_\eps)\to 0$ as $\eps\searrow0$. The second claim follows in an analogous way. 

Now let $g_1,g_2$ be two Lorentzian scalar products on the same vector space $V$, with future cones $F_1,F_2$, respectively. Then the definition of dual cone immediately gives
\[
F_1\subset F_2\qquad\Leftrightarrow\qquad F_1^*\supset F_2^*
\]
and if this occurs, then \eqref{eq:dualvar} also gives
\begin{equation}
\label{eq:ggstar}
g_1\leq g_2\quad\text{ on } F_1\qquad\Leftrightarrow\qquad g_1^*\geq g_2^*\quad\text{ on }F_2^*.
\end{equation}
Let us say that $g_1\prec g_2$ if $F_1 \setminus\{0\}\subset {\rm int}(F_2)$ and $g_1(v,v)<g_2(v,v)$ for any $v\in F_1\setminus\{0\}$. With this definition it is easy to see, given $g_1$, that
\[
\{g_2\text{ Lorentzian scalar products on $V$}\ :\ g_1\prec g_2\}\quad\text{is open},
\]
w.r.t.\ the topology of uniform convergence on a compact subset of $V^2$ with non-empty interior (this is independent of the chosen subset and coincides with the topology of uniform convergence of the coefficients of the tensor in any given basis of $V$). Our nonstandard definition of $g_1 \prec g_2$ yields
\begin{equation}
\label{eq:g1inf}
g_1=\inf \{g_2\ :\ g_1\prec g_2\},\qquad \text{meaning that}\qquad  
\begin{cases}
 F_1=\cap_{g_2\succ g_1}F_2,\\
 g_1(v,v)=\inf_{g_1\prec g_2} g_2(v,v)\quad\forall v\in F_1.
\end{cases}
\end{equation}
This can be seen by identifying $(V,g_1)$ with the standard Minkowski space and then `slightly widening' the light cone.

\bigskip

We are ready to prove the second, and last,  main result of this section:
\begin{theorem}[Differential of causal functions: smooth vs non-smooth]\label{thm:diffcaus}
Let $(M,g)$ be a  smooth strongly causal   spacetime and $f:M\to\R$ a causal function.

Then $f$ is ${\rm vol}_g$-measurable and:
\begin{itemize}
\item[i)] $f$ is differentiable at ${\rm vol}_g$-a.e.\ point with  ${\rm D}f(x)$ in $F^*_xM\subset T^*_xM$ for ${\rm vol}_g$-a.e.\ $x\in M$,
\item[ii)]  the Lorentzian (dual) norm $\| {\rm D} f\|_*$     coincides with the maximal weak subslope  $|\d f|$ ${\rm vol}_g$-a.e.
\end{itemize} 
Furthermore, $f$ is locally a BV function and $ {\rm D} f$ is (a representative of) the density w.r.t.\ ${\rm vol}_g$ of the absolutely continuous part of its distributional derivative.
\end{theorem}
\begin{proof}\ \\
\noindent{\sc Step 1: smooth functions on Minkowski}. We claim that for $f:\R^{1,n}\to\R$ smooth and causal we have $|\d f|=\| {\rm D}f\|_*$ a.e. Indeed, for $\gamma:[0,1]\to \R^{1,n}$ causal, the function $f\circ\gamma$ is monotone, hence differentiable a.e.\ and denoting by $(f\circ\gamma)'$ its a.e.\ defined derivative we have $f(\gamma_1)-f(\gamma_0)\geq\int_0^1(f\circ\gamma)'_t\,\d t$. Since $f$ is smooth, using \Cref{thm:curvediff} we see that
\[
(f\circ\gamma)'_t= {\rm D}  f_{\gamma_t}(\gamma_t')\geq \| {\rm D} f\|_*(\gamma_t)\,\|\gamma_t'\|=\| {\rm D}  f\|_*(\gamma_t)\,|\dot \gamma_t|\qquad a.e.\ t\in[0,1].
\]
It follows from \Cref{Le:Distr der} that $\| {\rm D} f\|_*$ is a weak subslope of $f$ and thus by maximality that $|\d f|\geq \| {\rm D} f\|_*$ holds ${\rm vol}_g$-a.e. For the converse inequality let $\mu\in\mathscr P(\R^{1,n})$ be with ${\rm vol}_g\ll\mu\leq{\rm vol}_g$ and $v\in T\R^{1,n}\sim\R^{1,n}$ be a future vector with $\|v\|=1$. Let $T$ be the map sending $x\in\R^{1,n}$ to the causal curve $t\mapsto x+tv$ and notice that $\ppi:=T_*\mu$ is a test plan so that by \Cref{Le:Distr der}, since $t\mapsto f(x+tv)$ is smooth  we obtain
\begin{equation}
\label{eq:anchedopo2}
{\rm D}f(x+tv)(v)\geq  |\d f|(x+tv)\qquad \mathcal L^{n+1}-a.e.\ (x,t).
\end{equation}
By Fubini's theorem and the translation invariance of $\mathcal L^{n+1}$ this easily implies ${\rm D}f(v)\geq |\d f|$ $\mathcal L^{n+1}$-a.e.  Taking the inf on a countable dense set of $v$'s,  by \eqref{eq:dualvar} we conclude that $\|{\rm D}f\|_*\geq |\d f|$, as desired.

\noindent{\sc Step 2: causal functions on Minkowski space}. Let $f:\R^{1,n}\to\R$ be causal. Pick a basis $e_0,\ldots,e_n$ of $\R^{1,n}$ made of vectors in $F$, so that the cone generated by the $e_i$'s is contained in $F$. In this coordinate system $f$ is coordinate-wise monotone. By  \cite[Thm. 4]{ChabCrou09} this suffices to show measurability. Also, for any $a_0,\ldots,a_n\in[0,1]$ and $i=0,\ldots,n$, the distributional derivative of $t\mapsto f(\ldots,a_{i-1},t,a_{i+1},\ldots)$ is a non-negative measure of mass bounded by $f(1,\ldots,1)-f(0,\ldots,0)$. It follows from the characterization in \cite[Rem.\ 3.104]{AFP:00} that $f$ is locally a BV function, as claimed. In particular, it is approximately a.e.\ differentiable by  \cite[Thm.\ 3.83]{AFP:00}  and the approximate differentiability improves to actual differentiability thanks to the coordinate-wise monotonicity, see \cite[Thm.\ 14]{ChabCrou09} (see also the presentation in \cite[Thm.\ 1.19]{Minguzzi:19}). Thus calling $\mathcal Df$ the co-vector valued measure ($\sigma$-additive only on uniformly bounded subsets --- see also \Cref{re:measuresfunctionals}) representing the distributional derivative of $f$, the Lebesgue decomposition
\begin{equation}
\label{eq:lebesguedf}
\mathcal Df={\rm D}f\,\mathcal L^{1+n}+\mathcal Df^\perp\qquad\text{ with }\qquad \mathcal Df^\perp\perp\mathcal L^{1+n}
\end{equation}
follows again from \cite[Thm.\ 3.83]{AFP:00}.

We now claim that for any $E\subset \R^{1,n}$ Borel and bounded we have $\mathcal Df(E)\in F^*$. By definition of $F^*$ and an approximation argument this will follow if we show that for any $v\in F$ and $\varphi\in C^\infty_c(\R^{1,n})$ non-negative we have $\int \varphi\,\d\mathcal D f(v)\geq 0$ (here and below $\mathcal D f(v)$ is defined on bounded Borel sets as $\mathcal Df(v)(E):=\mathcal Df(E)(v)$ and ---  trivially --- coincides with the distributional derivative of $f$ in the direction $v$). We thus have
\[
\begin{split}
\int \varphi\,\d\mathcal D f(v)=\int- {\rm D} \varphi(v)\, f\,\d\mathcal L^{1+n}=\lim_{h\to 0}\int\tfrac{\varphi(\cdot)-\varphi(\cdot+hv)}{h}f\,\d\mathcal L^{1+n}=\lim_{h\to 0}\int\varphi \tfrac{f(\cdot+hv)-f(\cdot)}{h}\,\d\mathcal L^{1+n}
\end{split}
\]
and the latter quantity is non-negative as $f$ is causal, $v\in F$ and $\varphi\geq0$. Hence our claim is proved and from the decomposition \eqref{eq:lebesguedf} it follows that
\begin{equation}
\label{eq:partifstar}
{\rm D}f\in F^*\qquad \mathcal L^{1+n}-a.e.\ \qquad\text{ and }\qquad \mathcal Df^\perp(E)\in F^*\qquad\forall E\subset\R^{1,n}\text{ Borel and bounded}.
\end{equation}
It remains to prove that $\|{\rm D}f\|_*=|\d f|$.  Start by observing that  for any future $v\in F$, Fubini's theorem and the translation invariance of $ \mathcal L^{1+n}$ yield for a.e.\ $x\in \R^{1,n}$ and a.e.\ $t\in[0,1]$ that the point $x+tv$ is a differentiability point of $f$. Thus the arguments leading to \eqref{eq:anchedopo2} are still valid, hence so is \eqref{eq:anchedopo2},  and as above we conclude that $\|{\rm D}f\|_*\geq |\d f|$ holds $\mathcal L^{n+1}$-a.e.

For the converse inequality let $(\rho_\eps)$ be a family of mollifiers and notice that $f\ast\rho_\eps$ is still a causal function. Thus  from  the previous step and \Cref{Pr:USCdep} we have $|\d f|\geq\limi_{\eps\downarrow0}|\d(f\ast\rho_\eps)|=\limi_{\eps\downarrow0}\|{\rm D}(f\ast\rho_\eps)\|_*$ (here the $\limi$ are intended as $\mathcal L^{1+n}$-essential-$\limi$). Since `convolution' and `distributional differentiation' commute, recalling \eqref{eq:lebesguedf} we have
\[
\begin{split}
\|{\rm D}(f\ast\rho_\eps)\|_*= \|(\mathcal Df)\ast\rho_\eps\|_*=\|({\rm D}f)\ast\rho_\eps+(\mathcal Df^\perp)\ast\rho_\eps\|_*\qquad\text{on}\quad \R^{1,n}
\end{split}
\]
and since \eqref{eq:partifstar} implies that $({\rm D}f)\ast\rho_\eps,(\mathcal Df^\perp)\ast\rho_\eps\in F^*$, the above and the reverse triangle inequality give $\|{\rm D}(f\ast\rho_\eps)\|_*\geq \|({\rm D}f)\ast\rho_\eps\|_*$ on $\R^{1,n}$. Now we observe --- the computations being justified because ${\rm D}f$ is locally integrable --- that
\[
\begin{split}
\|({\rm D}f)\ast\rho_\eps\|_*=\Big\|\int ({\rm D}f)(\cdot-y)\rho_\eps(y)\,\d\mathcal L^{1+n}(y)\Big\|_*\geq \int \|{\rm D}f\|_*(\cdot-y)\rho_\eps(y)\,\d \mathcal L^{1+n}(y)=\|{\rm D}f\|_*\ast\rho_\eps,
\end{split}
\]
where the inequality follows from   the concavity and continuity of $\|\cdot\|_*$ on $F^*$. Since $\|{\rm D}f\|_*\ast\rho_\eps\to \|{\rm D}f\|_*$ $\mathcal L^{1+n}$-a.e., collecting what we proved we conclude that $|\d f|\geq \|{\rm D}f\|_*$ $\mathcal L^{1+n}$-a.e., as desired.

\noindent{\sc Step 3: general case}. We argue by comparison. Say that a coordinate chart $(U,\varphi)$ with $U\subset M$ open and $\varphi:U\to \R^{1,n}$ is `good' provided:
\begin{itemize}
\item[1)]  $U$ and $\varphi(U)$ are causally convex.
\item[2)] For any $x\in  U$ the pullback $\varphi^*g_{1,n}$ of the Minkowskian tensor at $\varphi(x)\in \R^{1,n}$ is such that  $(\varphi^*g_{1,n})^*_x\prec g^*_x$ (recall the discussion after \eqref{eq:ggstar}). 
\end{itemize}
Clearly good charts cover $M$, and recalling the dual relation \eqref{eq:ggstar} we see that if $(U,\varphi)$ is a good chart, then $f\circ\varphi^{-1}$ is a causal function on $\varphi(  U)\subset \R^{1,n}$, hence as already argued it is measurable and a.e.\ differentiable with differential in $F^*\subset(\R^{1,n})^*$. Thus $f$ is ${\rm vol}_g$-measurable, ${\rm vol}_g$-a.e.\ differentiable and for every good chart $(U,\varphi)$ we have ${\rm D}(f\circ\varphi^{-1})\in F^*\subset(\R^{1,n})^*$. Then an approximation argument based on \eqref{eq:g1inf} and the fact that $\prec$ is open shows that ${\rm D}f(x)\in F^*_xM\subset T^*_xM$, thus (i) is settled.

We now prove that $|\d f|\geq  \|{\rm D}f\|_*$ holds ${\rm vol}_g$-a.e. To see this, let $\ppi$ be a test plan on $M$ and notice that for $\ppi$-a.e.\ $\gamma$ the function $f\circ \gamma$ is monotone (trivially) and differentiable at $t$ with derivative given by ${\rm D}f_{\gamma_t}(\gamma_t')\geq \|{\rm D}f_{\gamma_t}\|_*\,\|\dot\gamma_t'\|$ (with an argument based on the a.e.\ differentiability of $f$,  Fubini's theorem and \eqref{eq:dualvar}). It follows from the maximality of $|\d f|$ and \Cref{Le:Distr der} that  $|\d f|\geq  \|{\rm D}f\|_*$ holds ${\rm vol}_g$-a.e., as desired.

For the opposite inequality notice that a simple  argument based  on \eqref{eq:g1inf} (applied for $g_1:=g^*$) and on the fact that the relation $\prec $ is open shows that we can find a countable collection $(U_i,\varphi_i)$ of good charts such that
\[
g^*_x=\inf_{i:x\in U_i}(\varphi^*g_{1,n})^*_x\qquad\text{ this being interpreted as in \eqref{eq:g1inf}.}
\]
Thus if we adopt the convention that $\|{\rm D}(f\circ\varphi^{-1})\|_*\circ\varphi$ is equal to $+\infty$ outside $U$, from the above we see that
\begin{equation}
\label{eq:perchiudere}
\|{\rm D}f\|_*=\inf_i \|{\rm D}(f\circ\varphi_i^{-1})\|_*\circ\varphi_i=\inf_i |\d(f\circ\varphi_i^{-1})|\circ\varphi_i\qquad\text{${\rm vol}_g$-a.e.,}
\end{equation}
having used --- in the second identity --- that we know the conclusion on flat Minkowski space. Now notice that since, for every good chart $(U,\varphi)$, the sets $U$ and $\varphi( U)$  are causally convex and the ambient spacetimes are globally hyperbolic, we have that both $U$ and $\varphi(U)$ are forward mm spacetimes when equipped with the structure induced by the ambient spaces.  Also, from the dual relation \eqref{eq:ggstar} we see that $\ell(\varphi_i^{-1}(p),\varphi_i^{-1}(q))\geq \|q-p\|$ 
for any $p,q\in \varphi_i(\bar U_i)\subset\R^{1,n}$, thus the implication $(i)\Rightarrow(ii)$ in   \Cref{Le:One sid} allows to deduce that  $|\d (f\circ\varphi^{-1})|\geq |\d f|\circ\varphi^{-1}$. Recalling \eqref{eq:perchiudere} we can conclude that $\|{\rm D}f\|_*\geq |\d f|$, as desired.
\end{proof}

An easy consequence of this last theorem, of  $\lim_{\eps\downarrow0}\frac{g^*(\omega+\eps\eta,\omega+\eps\eta)-g^*(\omega,\omega)}{2\eps}=g^*(\omega,\eta)=\omega^\sharp(\eta)$ valid for any cotangent vectors $\omega,\eta$, and of the chain rule is that the value of
\[
\d \varphi(\nabla f)|\d f|^{p-2}
\] 
on $\{|\d f|\in(0,+\infty)\}$ for  $f: M\to\R$ causal and $\varphi\in\Pert(f)$ is independent of the fact that we are considering it as in \Cref{Def:Vert} or directly as the coupling of the cotangent vector field $\d \varphi$ (well defined as $\varphi$ is the difference of two causal functions) with the vector field $\nabla f$ multiplied by $|\d f|^{p-2}$. This observation is relevant in interpreting identity \eqref{eq:mufnuf} in the next section.

\subsection{Null distance and perturbations of causal functions}
\label{ss:null distance}

This section deals with elementary properties of causal functions and the null distance. Our goal is to show that the latter can be used to produce `many' perturbations of the time separation and related functions. The technique is flexible and likely applicable to more general causal functions: we chose to focus on powers of time separation to a point to present a concrete case study, which can also be used to build a measure valued $p$-d'Alembertian  following the presentation in \Cref{S:defining square}.

The concept of  null distance  was originally introduced in the smooth setting by Sormani and Vega \cite{SV:16} as  an attempt at bridging the gap between topology and causality on Lorentzian manifolds. It has subsequently been studied in the synthetic setting of Lorentzian pre-length spaces in Kunzinger and Steinbauer \cite{KS:2022}. Let us tailor the discussion on this latter reference to our setting: 

\begin{definition}[Null distance]\label{Def:Null dstance}
Let $(\M,\ell)$ be  metric  spacetime and  $f\colon \M \to \bar\R$ a causal function. The null distance $\hat\sfd_f$ induced by $f$ is defined as
\[
\hat\sfd_f(x,y):=\inf \sum_{i=0}^{n-1}|f(x_i)-f(x_{i+1})|,
\]
where the $\inf$ is taken among all $n\in\N$ and sequences $x=x_0,\ldots,x_n=y$ such that for any $i=0,\ldots,n-1$ we either have $x_i\leq x_{i+1}$ or $x_{i+1}\leq x_i$.
\end{definition}
It is clear that $\hat\sfd_f:\M^2\to[0,+\infty]$ is symmetric, satisfies the triangle inequality and $\hat\sfd_f(x,x)=0$ for any $x$. In general, nothing more can be said, as it could both be that $\hat\sfd_f\equiv 0$ (e.g.\ if $f$ is constant) and that $\hat\sfd_f(x,y)=+\infty$ for any $x\neq y$ (if for no $x\neq y\in \M$ we have $x\leq y$).

We record the following simple general statement:
\begin{proposition}[$\hat\sfd_f$-Lipschitz functions are perturbations]\label{Pr:1.lip}
Let $(\M,\ell)$ be a metric  spacetime, $f\colon \M\to\R$ causal and $\hat\sfd_f$  the corresponding null distance. Let $g:\M\to\R$ be    $\hat\sfd_f$-Lipschitz, i.e.\ $|g(x)-g(y)|\leq L\hat\sfd_f(x,y)$ for some $L\geq 0$ and any $x,y\in \M$.

Then  $f+\tfrac1Lg:\M\to\R$ is causal. 
\end{proposition}
\begin{proof}  By scaling, we may and will assume that $g$ is $1$-Lipschitz with respect to $\smash{\hat{\met}_f}$. Then
\begin{align*}
-(g(y) - g(x)) \leq \vert g(y) -g(x)\vert\leq \hat{\met}_f(x,y) \leq f(y) - f(x),\qquad\forall  x\leq y,
\end{align*}
 proving  that $f+g$ is causal and thus giving the claim.
\end{proof}
Thanks to this result the question is now to find sufficient conditions that ensure that $\hat\sfd_f$-Lipschitz functions are  continuous and suitably dense in the space of continuous functions, so that the distributional $p$-d'Alembertian as studied in \Cref{Pr:A bounded}  can be represented by a measure via Riesz's theorem.

We do not have general conditions on non-smooth spacetimes, but the results in \cite{SV:16}  allow to quickly prove the following, showing that our theory is compatible with this basic case.
\begin{proposition}[Metrizing on manifolds by the null distance of time separation to a point]
\label{Pr:From point}
Let $(M,g)$ be a smooth, globally hyperbolic spacetime, let $o \in \M$ and consider the causal function $f:=\ell(o,\cdot)$. Then $\hat\sfd_{f}$ is a distance on $I^+(o)$ that induces the manifold topology. 

The same holds for the function $f:=-\ell(\cdot,o)$ on $I^-(o)$ and, more generally, for $\tfrac1q\ell(o,\cdot)^q $ and $-\tfrac1q\ell(\cdot,o)^q $ on $I^+(o)$ and $I^-(o)$ respectively for any $0\neq q\leq 1$.
\end{proposition}

\begin{proof} 
We shall consider the spacetime $\tilde M:=I^+(o)$. According to \cite[Proposition 3.15, Theorem 5.4]{SV:16} the conclusion for $f=\ell(o,\cdot)$ will follow if we show that $f$ is the cosmological time function of $\tilde M$ and that is regular in the terminology of \cite{AGH98} recalled below. 

 For $f$ to be the cosmological time function means that it satisfies  $f(x)=\sup_{y\in \tilde M, y\leq x}\ell(y,x)$: this follows from   the continuity of $f$ on $J^+(o)\supset I^+(o)$ (see \cite[Theorem 3.6]{McCann:2020}). 

For $f$ to be regular means that it is finite valued on $\tilde M$ (which it clearly is) and that if $\gamma:[0,1)\to \tilde M$ is any inextendible past curve, then $f(\gamma_t)\to 0$ as $t\to 1$. We verify this latter property and notice that for such $\gamma$ we have $\gamma_t\in J^-(\gamma_0)\cap J^+(o)\subset M$ for any $t\in[0,1)$. Thus using first  the global hyperbolicity of $M$ and then the causal character of $\gamma$ we see that for some $p\in \M$, $p\geq o$ we have $\gamma_t\to p$ as $t\to 1$. Since $\gamma$ is past inextendible on $\tilde M$, we must have $p\notin \tilde M$, that together with $p\geq o$ yields $f(p)=0$. It then follows from  the continuity of $f$ on $J^+(o)$ that $f(\gamma_t)\to 0$ as $t\to 1$.

The case of $f=-\ell(\cdot,o)$ is handled analogously, then the rest of the claim follows noticing that $u_q(z)=\tfrac1qz^q$ is smooth on $(0,\infty)$ and its inverse is smooth on $u_q((0,\infty))$.
\end{proof}
We illustrate how the above can be combined with the $p$-d'Alembertian estimates we proved in \Cref{subsec-q-dal-comp} and the concept of   distributional $p$-d'Alembertian discussed in \Cref{S:defining square}. Recall also the discussion at the end of \Cref{ss:compatibility} that ensure that the meaning of $\d\varphi(\nabla f)|\d f|^{p-2}$ appearing below is unambiguous. 
\begin{corollary}[Distributional $\PP$-d'Alembertian in the smooth setting]\label{Cor:SignedR1} Let $(M,g)$ be a smooth, globally hyperbolic spacetime of dimension $n$ whose timelike Ricci curvature is no smaller than $K\in\R$, endowed with its canonical volume measure $\meas$. 

Moreover, let $0\neq \QQ <1$ and $\PP^{-1} + \QQ^{-1}=1$. Given a point $o\in \M$, consider the open set $U\subset M$ defined as $I^-(o)$ if $K\leq 0$ and $I^-(o)\cap\{\ell(\cdot,o)<\pi\sqrt{(N-1)/K}\}$ if $K>0$. Also, consider the functions $f:=-\tfrac1q\ell(\cdot,o)^q$ and $h:=-\ell(\cdot,o)$ and the Radon measures $\nu_f,\nu_h$ on $U$ given by
\[
\nu_f:=N\,\tilde\tau_{K,N}\circ\ell(o,\cdot)\,\meas,\qquad\qquad\qquad \nu_h:=\frac{N\,\tilde{\tau}_{K,N}\circ \ell(\cdot,o) - 1}{\ell(\cdot,o)} \mm.
\]
Then $ f$ and $h$ both have distributional $p$-d'Alembertian in $U$ according to \Cref{def:distrdal} and there are unique non-positive Radon measures $\mu_f,\mu_h$ on $U$ such that 
\begin{equation}
\label{eq:mufnuf}
\int\varphi\,\d\mu_f=\int \d\varphi(\nabla f)|\d f|^{p-2}\,\d\mm-\int\varphi\,\d\nu_f\qquad\forall \varphi\in \Pert_\be^\sym(f, U)\cap C(U)
\end{equation}
and similarly for $h$ in place of $f$.
\end{corollary}
\begin{proof} The collection $ \Pert_\be^\sym(f, U)\cap C(U)$ defined after \eqref{PertSym} is a vector space and the map $L$ sending $\varphi\in  \Pert_\be^\sym(f, U)\cap C(U)$ to the right-hand side of \eqref{eq:mufnuf} is linear and sends non-negative functions to non-positive reals.

We claim that $L(\varphi_n)\uparrow0$ if $(\varphi_n)\subset  \Pert_\be^\sym(f, U)\cap C(U)$ is such that $\varphi_n(x)\downarrow 0$ as $n\to\infty$ for every $x$. By   Dini's theorem and the fact that $\supp\varphi_n\subset\supp\varphi_0$ (the latter being compact as the spacetime is globally hyperbolic) we get that $(\varphi_n)$ converges to 0 uniformly. Also, by \Cref{Pr:From point} for every $C\subset U$ compact,  there is $\eta\in \Pert_\be^\sym(f, U)\cap C(U)$   identically 1 on $C$ (pick $\eta:=(1-n\,\hat\sfd_f(\cdot,C))^+$ for $n\gg1$). Picking  $C:=\supp\varphi_0$  and using the bound \eqref{eq:bounddal} in conjunction with the uniform convergence just proved shows our claim.

We can thus apply \cite[Thm.\ 7.8.1]{Bog:07b} to $-L$ and deduce that $L(\varphi)=\int \varphi\,\d\mu_f$ for a unique non-positive measure $\mu_f$ on $U$ defined on the smallest $\sigma$-algebra $\mathcal A$ making all the functions in $ \Pert_\be^\sym(f, U)\cap C(U)$ measurable. Clearly $\mathcal A$ is contained in the Borel $\sigma$-algebra of $U$ and to show that they agree it suffices to prove that any function in $C_c(U)$ can be uniformly approximated by functions in $ \Pert_\be^\sym(f, U)\cap C(U)$. This, however, is obvious from the fact that the null distance induced by $f$ induces, in turn, the manifold topology.  The fact that $\mu_f$ is Radon now follows noticing that  --- by  \eqref{eq:bounddal} --- it is finite on compact sets.

The argument for $h$ is analogous.
\end{proof}

\begin{remark}[Differences of positive measures need not be measures]
\label{re:measuresfunctionals}Since $\mu_f\leq 0$, by \eqref{eq:mufnuf} it is surely tempting to say that the $p$-d'Alembertian of $f$ is represented by a measure $\leq\nu_f$, this  `measure' being $\mu_f+\nu_f$. The problem with this is that it might be that $\mu_f(U)=-\infty$ and $\nu_f(U)=+\infty$, so that their sum is not a measure on $U$ as classically intended. This happens because, more broadly speaking, the collection of Radon measures on a given topological space is not a vector space in general. 

To get around this technical issue, one possibility is to  keep in mind Riesz's representation theorem and to deal with the space of  so-called Radon functionals, as in \cite{CavallettiMondino20}. Another possibility is to `ignore' this issue and to notice that objects like $\mu_f+\nu_f$ are still well-defined and $\sigma$-additive as far as one works with uniformly relatively compact subsets of $U$. In particular, the `measure' of any relatively compact set is well defined and so is the integral of  compactly supported functions. This sort of consideration also allows to perform operations typically done on measures (e.g.\ push forward, Radon-Nikodym derivatives...) whose interpretation at the level of Radon functionals would be cumbersome. Because of this some authors (see e.g.\ \cite[Remark 2.12]{AB18}) adopt this point of view  and still call Radon measures objects such as $\mu_f+\nu_f$. Notice that this is what we also did in the course of the proof of \Cref{thm:diffcaus}, when we represented by the `measure' $\mathcal Df$ the distributional differential of the $BV_{loc}$ function $f$.
\hfill$\blacksquare$
\end{remark}

\subsection{{Hyperbolic norms and polarization on cones}}\label{ss:hyperbolic norms}

This section deals with a class of Lorentzian structures which arises on vector spaces from  hyperbolic norms, such as  Lorentz--Finsler norms.  We discuss polarization and weak curvature properties.

\begin{definition}[Hyperbolic norm]
\label{D:hyperbolic norm}
A \textit{hyperbolic norm} on a real vector space $V$ is a function $\n:V \to [0,+\infty) \cup \{-\infty\}$ such that
\begin{subequations}
\begin{align}
\label{eq:invtrn}
\n(\alpha v+\beta w)&\geq\alpha\n(v)+\beta\n(w)\qquad\forall v,w\in V,\ \alpha,\beta\geq0,\\
\label{eq:asimmn}
\n(v),\n(-v)\geq 0&\qquad\qquad \Leftrightarrow\qquad \qquad v=0.
\end{align}
\end{subequations}
\end{definition}

\begin{example}
\label{E:hyperbolic norms from scalar products} Let $p\in[1,+\infty)$, $d\in\N$, $d>0$ and define $\n_p$ on $\R^d$ as
\[
\n_p(v):=\Big(|v_0|^p - \sum_{i=1}^{d-1} |v_i|^p\Big)^{1/p}
\]
if $|v_0|^p \geq \sum_{i=1}^{d-1} |v_i|^p$ and $\n_p(v)=-\infty$ otherwise. Then this is a hyperbolic norm.

For $p=2$ this coincides with the norm induced by the Lorentzian scalar product $g(v,w)=v_0w_0-\sum_{i>0}v_iw_i$ as in formula \eqref{eq:normdag} and in this case (and only in this if $d>1$) the norm $\n_p$ is positively polarizable in the sense discussed below. For more about these  norms we refer to \cite{Gigli24+}.
\hfill$\blacksquare$
\end{example}
Hyperbolic norms naturally induce time separation functions via the formula
\begin{equation}
\label{eq:elln}
\ell(v,w):=\n(w-v)\qquad\forall v,w\in V,
\end{equation}
and it is clear from \eqref{eq:invtrn} that $\ell$ satisfies the reverse triangle inequality and from \eqref{eq:asimmn} that the causal order induced by $\ell$ is a partial order. Thus, much like in the non-smooth setting, $\n$ induces the causal and chronological  relations on $V$ and the latter one induces a topology. The following is worth noting:

\begin{proposition}[Hyperbolic normed vector spaces as globally hyperbolic metric spacetimes]
\label{Prop: structureofstrictlyconcavenorm}
Let $\n$ be an hyperbolic norm on $\smash{\R^N}$, where $N\in\N$. Assume that  there   exists $\smash{v \in \R^N}$ with $\n(v) > 0$ and that the positive part $\n_+$ of $\n$ is continuous (w.r.t.\ the Euclidean topology).

Then $(\R^N,\ell)$ is a globally hyperbolic metric spacetime, where $\ell$ is as in \eqref{eq:elln}, and the chronological topology coincides with the Euclidean one.
\end{proposition}

\begin{proof} The continuity of $\n_+$ implies that the chronological topology is weaker than the Euclidean one. For the converse inclusion it suffices to prove that there is a sufficient widening of the usual Minkowskian cones in some coordinate system that contain the future cone $\{v:\n(v)\geq 0\}$ (because the chronological relation induced by the Minkowskian norm coincides with the Euclidean topology). To see this, let $\smash{Y:= \{\n\geq 0\} \cap \bS^{N-1}}$. Then by \eqref{eq:invtrn}, $Y$ is a compact, convex subset of $\bS^{N-1}$ (considered as a Riemannian manifold) that, by \eqref{eq:asimmn}, does not contain antipodal points. Thus it is contained in a half-sphere and, by compactness, in an $\bS^{N-1}$-ball of radius $R < \pi/2$. This readily implies our claim.

For global hyperbolicity we now observe that for any $C_0,C_1\subset \R^N$ compact, what was already proved implies that $J(C_0,C_1) \subset \tilde{J}(C_0,C_1)$, where $\tilde{J}(C_0,C_1)$ is the emerald spanned by $C_0,C_1$ under this extra Minkowskian structure. Since the latter is compact and the former closed, the conclusion follows.
\end{proof}

We come to weak lower Ricci bounds, referring to \cite{Braun:2023Renyi}  for the concept of $\TCD_q(K,N)$ space. The proof below strongly imitates that of the analogous result in positive signature provided in \cite[p. 926]{Villani:2009}, thus we shall not give all the details.

\begin{proposition}[TCD condition for hyperbolic normed vector spaces]\label{Pr:TCD COND}
Let $\n$ be as in  \Cref{Prop: structureofstrictlyconcavenorm}. Then the induced globally hyperbolic metric measure spacetime structure on $\smash{\R^N}$ \textnormal{(}with $\meas$ being the Lebesgue measure $\smash{\Leb^N}$\textnormal{)} satisfies  $\smash{\TCD_\QQ(0,N)}$ for every $0 \ne \QQ < 1$. 
\end{proposition}

\begin{proof}We only sketch the argument freely using terminology and notation from \cite{BO:23}, to which the reader is referred for details. We first consider the case $\n|_{\{\n >0\}}$ is smooth and $\n$ is locally uniformly concave. Then the map $\smash{L\colon\R^{N}\times \R^N \to \R}$ defined by
\begin{align*}
    L(x,v):=\begin{cases}
        \n^2(v) & \textnormal{if }\n(v)>0,\\
        0 & \text{otherwise}
    \end{cases}
\end{align*}
defines a smooth Lorentz--Finsler structure on $\R^N$. Given any pair $(\mu,\nu)$ of compactly supported elements of $\Prob(\M)$ with $\mu$ being $\meas$-absolutely continuous which are $\QQ$-separated by $(\pi,u,v)$ \cite[Def.~4.8]{BO:23}, $\pi = (\textnormal{Id},\mathscr{F}_1)_\push\mu$ is the unique maximizer of $\ell_\QQ(\mu,\nu)$, where $\smash{\mathscr{F}_t\colon  \R^N \to \R^N}$ is the Lorentz--Finslerian displacement interpolation and is of simple linear form (both in its argument and in $t$; recall that we have unique geodesics given by straight lines) by \cite[Thm.~4.17]{BO:23}. Moreover, the unique $\smash{\ell_\QQ}$-geodesic from $\mu$ to $\nu$ is given by $\mu_t:=(\mathscr{F}_t)_\push \mu$ \cite[Cor. 4.18]{BO:23}. Keeping this in mind, the fact that this smooth Lorentz-Finsler spacetime satisfies $\TCD_\QQ(0,N)$ follows by the same (linear algebraic) arguments as the $\mathrm{CD}(0,N)$-condition for smooth, uniformly convex norms on $\smash{\R^N}$, cf.~\cite[p.~926]{Villani:2009} for details.

For general $\n$ we first `widen a bit the future cone' and then regularize the norm, thus obtaining approximating norms $\n_n$ whose future cones $F_n:=\{\n_n(v)\geq 0\}$ contain the limit one $F:=\{\n\geq 0\}$ in their interior. Given $\mu,\nu\in \Prob(\R^n)$ with compact support and causally ordered $\mu\preceq\nu$, we have $\mu\preceq_n\nu$ for every $n\in\N$, thus by the previous point we know there is a geodesic $(\mu_{n,t})$ (w.r.t.\ the $q$ cost induced by the $\ell_n$ coming from $\n_n$ via \eqref{eq:elln}) from $\mu$ to $\nu$ along which the $N$-Renyi entropy is convex. The measures $\{\mu_{n,t}:n\in\N,t\in[0,1]\}$ all have support in the same compact set, thus by diagonalization we can assume that $\mu_{n,t}\to\mu_t$ for every $t\in\Q\cap[0,1]$ and some measures $\mu_t$. The lower semicontinuity of the entropy grants that $\mu_t$ obeys the correct entropy bounds, while passing to the limit in
\[
\ell_{n,q}(\mu_{n,t},\mu_{n,s})=(s-t)\ell_{n,q}(\mu,\nu)\qquad\forall t,s\in[0,1],\ t<s
\]
we easily see that the $\mu_t$'s lie along an $\ell_q$-geodesic from $\mu$ to $\nu$. By narrow forward completeness we can extend the curve $(\mu_t)$ to all $t\in[0,1]$ resulting in a $\ell_q$-geodesic and then narrow lower semicontinuity of the entropy yields the desired conclusion.
\end{proof}

Let $\n\colon V \to \{-\infty\}\cup[0,+\infty)$ be  a hyperbolic norm on a vector space $V$ with domain (or future) $F:= \{x \in V \mid \n(x) \ge 0\}$.   
We say $\n(\cdot)$ is {\em positively polarizable} provided 
\begin{equation}\label{bilinear}
(x,y):=\frac12\big(\n(x+y)^2 - \n(x)^2 -\n(y)^2\big)
\end{equation}
acts positive bilinearly on $F$, or equivalently if 
\begin{equation}
\label{eq:poslin}
(\alpha_1x_1+\alpha_2x_2,y)=\alpha_1(x_1,y)+\alpha_2(x_2,y)\qquad\forall x_1,x_2,y\in F,\ \alpha_1,\alpha_2\geq0.
\end{equation}
Notice that by symmetry, if the above holds, a similar statement is in place for $y$. It is quite clear that if $\n:V\to\{-\infty\}\cup[0,+\infty)$ is positively polarizable, then the form  $(x,y)$ defined by \eqref{bilinear} for $x,y \in F$ extends uniquely to a indefinite inner product $\langle \cdot, \cdot\rangle$ on ${\rm span}(F)$. Indeed, uniqueness is clear as well as the fact that we must have
\[
\langle x_1-x_2,y\rangle=(x_1,y)-(x_2,y)\qquad\forall x_1,x_2,y\in F.
\]
The conclusion follows observing that this formula is well posed,  i.e.\ that if $x_1-x_2=z_1-z_2$, then the corresponding right hand sides in the above agree: this is a trivial consequence of  \eqref{eq:poslin}.

The next lemma  reinforces and explains our \Cref{Def:Inf Minkow} of infinitesimal Minkowskianity.
\begin{lemma}[Parallelogram law and polarization of hyperbolic norms]
\label{L:parallelogram law and polarization}
Let $\n \colon V\to\{-\infty\}\cup[0,+\infty)$ be a hyperbolic norm on a real vector space $V$. 

Then $\n $ is positively polarizable  if and only if
\begin{equation}\label{parallelogram}
    \n(x+2y)^2+\n(x)^2=2\n(x+y)^2+2\n(y)^2,\qquad\forall x,y\in V,\ \text{with }\n(x),\n(y)\geq 0.
\end{equation}
\end{lemma}
\begin{proof}
The \emph{only if} is a trivial consequence of the identity $\n(v)=(v,v)$ and of \eqref{eq:poslin}. For the  \emph{if} we argue as in the proof of \Cref{T:VD=HD}. Thus let $x,y,z\in V$ be with $\n(x),\n(y),\n(z)\geq 0$ and notice that \eqref{parallelogram} yields
\begin{align*}
    2\n(x+2y+z)^2+2\n(x+z)^2&=4\n(x+y+z)^2+4\n(y)^2,\\
    \n(2x+2y+z)^2+\n(2y+z)^2&=2\n(x+2y+z)^2+2\n(x)^2,\\
    \n(2y+z)^2+\n(z)^2&=2\n(y+z)^2+2\n(y)^2,\\
    \n(2x+2y+z)^2+\n(z)^2&=2\n(x+y+z)^2+2\n(x+y)^2.
\end{align*}
Adding the first two identities and then subtracting the second two we conclude that $(x+y,z)=(x,z)+(y,z)$. It then follows that $(\alpha x,y)=\alpha (x,y)$ holds for $\alpha\in\N$ and the for $\alpha\in\Q^+$. The monotonicity of $\R^+\ni \alpha\mapsto \n(\alpha x+y)$ that comes from \eqref{eq:invtrn} now allows to conclude that  $(\alpha x,y)=\alpha (x,y)$ holds for $\alpha\in\R^+$, as desired.
\end{proof}

Finally, let us demonstrate that if a symmetric bilinear form induces a ``norm" that satisfies either  the triangle inequality or the reverse triangle inequality, then the form is necessarily positive definite or of Lorentzian signature, respectively.

\begin{lemma}[Triangle inequalities indicate signature]
Let $V$ denote a finite-dimensional real vector space and $g\colon V \times V \to \R$ a non-degenerate symmetric bilinear form. We denote by $C$ a cone with non-empty interior contained in $\{v:g(v,v)>0\}$  and define $\| v \| :=\sqrt{g(v,v)}$ for $v \in C$. 

Then the following statements hold.
\begin{enumerate}[label=\textnormal{(\roman*)}]
    \item Assume that $\| \cdot \|$ satisfies the triangle inequality, i.e.\ $\|v + w \| \leq \|v\| + \|w\|$     for all $v,w \in C$. Then $g $ is positive definite.
    \item Assume that $\| \cdot \|$ satisfies the reverse triangle inequality, i.e.\ $\|v + w\| \geq \|v\| + \|w\|$ for all $v,w \in C$. Then $g$ has Lorentzian signature $(+,-,\dots,-)$.
\end{enumerate}
\end{lemma}

\begin{proof} Pick $v$ in the interior of $C$ and recall that, regardless of the signature of $g$, there is an orthonormal basis $e_1,\ldots,e_n$ of $V$   such that $g(e_i,e_j)=0$ for $i\neq j$ and $g(e_i,e_i)=\pm1$ and so that $e_1:=\tfrac{v}{\|v\|}$. This can be proved by induction by noticing that the orthogonal complement of $e_1$ is a subspace of dimension exactly $n-1$. Unless $n=1$, in such subspace there must be a vector $v'$ with $g(v',v')\neq 0$, so that an iteration proves the claim.

 The signature of $g$ is dictated by the signs of  $g(e_i,e_i)$ for $i=2,\ldots,n$.  Fix such $i$ and notice that for $|t|\ll1$ we have $e_1+te_i\in C$, thus $g(e_1+te_i,e_1+te_i)>0$, thus the function
    \begin{align*}
        f(t):=\|e_1 + t e_i\| = \sqrt{1 + t^2g(e_i,e_i)},
    \end{align*}
is well defined in a neighbourhood of 0. Then $f$ is smooth near $t=0$ with $f''(0) =g(e_i,e_i)$. Now, the triangle inequality for $\| \cdot \|$ implies convexity of $f$, hence $g(e_i,e_i)=1$. Likewise, the reverse triangle inequality implies concavity of $f$, so  $g(e_i,e_i)=-1$. This finishes the proof.
\end{proof}
\section*{Acknowledgments} The authors are grateful to Simone Vincini for comments on an early version of this work.

The authors are grateful for the support of the Fields Institute for Research in Mathematical Sciences 2022 Thematic Program on Nonsmooth Riemannian and Lorentzian Geometry, where this collaboration began. This research was also funded in part by the Austrian Science Fund (FWF) [Grants DOI \href{https://doi.org/10.55776/PAT1996423}{10.55776/PAT1996423}, \href{https://doi.org/10.55776/P33594}{10.55776/P33594}, \href{https://doi.org/10.55776/STA32}{10.55776/STA32}, \href{https://doi.org/10.55776/EFP6}{10.55776/EFP6} and \href{https://doi.org/10.55776/J4913}{10.55776/J4913}]. For open access purposes, the authors have applied a CC BY public copyright license to any author accepted manuscript version arising from this submission. 

MB was supported in part by the Fields Institute for the Mathematical Sciences. He gratefully acknowledges financial support by the EPFL through a Bernoulli Instructorship. A large part of this work was carried out while he held a Postdoctoral Fellowship at the University of Toronto. NG was supported in part by the MUR PRIN-2022JJ8KER grant ``Contemporary perspectives on geometry and gravity" (code 2022JJ8KER – CUP G53D23001810006)
and the Fields-McMaster Dean's Distinguished Visitorship.
RM was supported in part by the Canada Research Chairs program CRC-2020-00289, Natural Sciences and Engineering Research Council of Canada Discovery Grant RGPIN- 2020--04162,
and a grant from the Simons Foundation (923125, McCann).
AO was supported in part by the ÖAW scholarship of the Austrian Academy of Sciences. 
CS was also supported by the European Research Council (ERC), under the European’s Union Horizon 2020 research and innovation programme, via the ERC Starting Grant “CURVATURE”, grant agreement No.\ 802689.


\addcontentsline{toc}{section}{References}


\begin{thebibliography}{100}
\bibliographystyle{abbrv}

\bibitem{APS20}
L.~Ak\'e{}~Hau, A.~J. Cabrera~Pacheco, and D.~A. Solis.
\newblock On the causal hierarchy of {L}orentzian length spaces.
\newblock {\em Classical Quantum Gravity}, 37(21):215013, 22, 2020.

\bibitem{Aleksandrov51}
A.~D. Aleksandrov.
\newblock A theorem on triangles in a metric space and some of its applications.
\newblock {\em Trudy Mat. Inst. Steklov.}, 38:5--23, 1951.

\bibitem{AB08}
S.~B. Alexander and R.~L. Bishop.
\newblock Lorentz and semi-{R}iemannian spaces with {A}lexandrov curvature bounds.
\newblock {\em Comm. Anal. Geom.}, 16(2):251--282, 2008.

\bibitem{Ambrosio90}
L.~Ambrosio.
\newblock Metric space valued functions of bounded variation.
\newblock {\em Ann. Scuola Norm. Sup. Pisa Cl. Sci. (4)}, 17(3):439--478, 1990.

\bibitem{Ambrosio04}
L.~Ambrosio.
\newblock Transport equation and {C}auchy problem for {$BV$} vector fields.
\newblock {\em Invent. Math.}, 158(2):227--260, 2004.

\bibitem{AB18}
L.~Ambrosio and J.~Bertrand.
\newblock D{C} calculus.
\newblock {\em Math. Z.}, 288(3-4):1037--1080, 2018.

\bibitem{AFP:00}
L.~Ambrosio, N.~Fusco, and D.~Pallara.
\newblock {\em Functions of bounded variation and free discontinuity problems}.
\newblock Oxford Mathematical Monographs. The Clarendon Press, Oxford University Press, New York, 2000.

\bibitem{AmbrosioGigliMondinoRajala12}
L.~Ambrosio, N.~Gigli, A.~Mondino, and T.~Rajala.
\newblock Riemannian {R}icci curvature lower bounds in metric measure spaces with $\sigma$-finite measure.
\newblock {\em Trans. Amer. Math. Soc.}, 367(7):4661--4701, 2012.
\newblock arXiv:1207.4924.

\bibitem{AGS:08}
L.~Ambrosio, N.~Gigli, and G.~Savar\'{e}.
\newblock {\em Gradient flows in metric spaces and in the space of probability measures}.
\newblock Lectures in Mathematics ETH Z\"{u}rich. Birkh\"{a}user Verlag, Basel, second edition, 2008.

\bibitem{AGS:14a}
L.~Ambrosio, N.~Gigli, and G.~Savar\'{e}.
\newblock Calculus and heat flow in metric measure spaces and applications to spaces with {R}icci bounds from below.
\newblock {\em Invent. Math.}, 195(2):289--391, 2014.

\bibitem{AGS:14b}
L.~Ambrosio, N.~Gigli, and G.~Savar\'{e}.
\newblock Metric measure spaces with {R}iemannian {R}icci curvature bounded from below.
\newblock {\em Duke Math. J.}, 163(7):1405--1490, 2014.

\bibitem{AK:00}
L.~Ambrosio and B.~Kirchheim.
\newblock Currents in metric spaces.
\newblock {\em Acta Math.}, 185(1):1--80, 2000.

\bibitem{AGH98}
L.~Andersson, G.~J. Galloway, and R.~Howard.
\newblock The cosmological time function.
\newblock {\em Classical Quantum Gravity}, 15(2):309--322, 1998.

\bibitem{AH98}
L.~Andersson and R.~Howard.
\newblock Comparison and rigidity theorems in semi-{R}iemannian geometry.
\newblock {\em Comm. Anal. Geom.}, 6(4):819--877, 1998.

\bibitem{BK85}
D.~Bakry and M.~\'Emery.
\newblock Diffusions hypercontractives.
\newblock In {\em S\'eminaire de probabilit\'es, {XIX}, 1983/84}, volume 1123 of {\em Lecture Notes in Math.}, pages 177--206. Springer, Berlin, 1985.

\bibitem{BHNR23}
T.~Beran, J.~Harvey, L.~Napper, and F.~Rott.
\newblock A {T}oponogov globalisation result for {L}orentzian length spaces.
\newblock {\em arXiv:2309.12733}, 2023.

\bibitem{BeranOhanyanRottSolis23}
T.~Beran, A.~Ohanyan, F.~Rott, and D.~A. Solis.
\newblock The splitting theorem for globally hyperbolic {L}orentzian length spaces with non-negative timelike curvature.
\newblock {\em Lett. Math. Phys.}, 113(2):Paper No. 48, 47, 2023.

\bibitem{Bog:07a}
V.~I. Bogachev.
\newblock {\em Measure theory. {V}ol. {I}}.
\newblock Springer-Verlag, Berlin, 2007.

\bibitem{Bog:07b}
V.~I. Bogachev.
\newblock {\em Measure theory. {V}ol. {II}}.
\newblock Springer-Verlag, Berlin, 2007.

\bibitem{Braun:2023Good}
M.~Braun.
\newblock Good geodesics satisfying the timelike curvature-dimension condition.
\newblock {\em Nonlinear Anal.}, 229:Paper No. 113205, 30, 2023.

\bibitem{Braun:2023Renyi}
M.~Braun.
\newblock R\'{e}nyi's entropy on {L}orentzian spaces. {T}imelike curvature-dimension conditions.
\newblock {\em J. Math. Pures Appl. (9)}, 177:46--128, 2023.

\bibitem{braun+}
M.~Braun.
\newblock Nonsmooth d'{A}lembertian for {L}orentz distance functions.
\newblock {\em arXiv:2408.16525}, 2024.

\bibitem{QuintetEllipticsplitting}
M.~Braun, N.~Gigli, R.~J. McCann, A.~Ohanyan, and C.~S{\"a}mann.
\newblock An elliptic proof of the splitting theorems from {L}orentzian geometry.
\newblock arXiv:2410.12632.

\bibitem{QuintetNonsmooth25+}
M.~Braun, N.~Gigli, R.~J. McCann, A.~Ohanyan, and C.~S{\"a}mann.
\newblock Lorentzian splitting theorems for metrics and weights of regularity {$C^1$}.
\newblock In preparation.

\bibitem{BraunMcCann:2023}
M.~Braun and R.~J. McCann.
\newblock {C}ausal convergence conditions through variable timelike {R}icci curvature bounds.
\newblock {\em arXiv:2312.17158}, 2023.

\bibitem{BO:23}
M.~Braun and S.~Ohta.
\newblock Optimal transport and timelike lower {R}icci curvature bounds on {F}insler spacetimes.
\newblock {\em Trans. Amer. Math. Soc.}, 377(5):3529--3576, 2024.

\bibitem{Brenier:1991}
Y.~Brenier.
\newblock Polar factorization and monotone rearrangement of vector-valued functions.
\newblock {\em Comm. Pure Appl. Math.}, 44(4):375--417, 1991.

\bibitem{BG-H:2024}
A.~Burtscher and L.~Garc\'{\i}a-Heveling.
\newblock Time functions on {L}orentzian length spaces.
\newblock {\em Annales Henri Poincaré}, 2024.
\newblock to appear.

\bibitem{Busemann:1967}
H.~Busemann.
\newblock Timelike spaces.
\newblock {\em Dissertationes Math. (Rozprawy Mat.)}, 53:52, 1967.

\bibitem{BykovMinguzziSuhr:2024+}
A.~Bykov, E.~Minguzzi, and S.~Suhr.
\newblock Lorentzian metric spaces and {GH}-convergence: the unbounded case.
\newblock {\em arXiv:2412.04311}, 2024.

\bibitem{Calabi:1958}
E.~Calabi.
\newblock An extension of {E}. {H}opf's maximum principle with an application to {R}iemannian geometry.
\newblock {\em Duke Math. J.}, 25:45--56, 1958.

\bibitem{CalistiOhanyanSalamo}
M.~Calisti, A.~Ohanyan, and M.~Sálamo~Candal.
\newblock Optimal transport for general {Lorentz} costs.
\newblock {\em in preparation}.

\bibitem{CaponioOhanyanOhto24+}
E.~Caponio, A.~Ohanyan, and S.~Ohta.
\newblock {Splitting theorems for weighted Finsler spacetimes via the $p$-d'Alembertian: beyond the Berwald case}.
\newblock {\em arXiv:2412.20783}, 2024.

\bibitem{CavallettiManiniMondino24+}
F.~Cavalletti, D.~Manini, and A.~Mondino.
\newblock Optimal transport on null hypersurfaces and the null energy condition.
\newblock {\em arxiv:2408.08986}, 2024.

\bibitem{cavalletti-milman2021}
F.~Cavalletti and E.~Milman.
\newblock The globalization theorem for the curvature-dimension condition.
\newblock {\em Invent. Math.}, 226(1):1--137, 2021.

\bibitem{cavalletti-mondino2017}
F.~Cavalletti and A.~Mondino.
\newblock Optimal maps in essentially non-branching spaces.
\newblock {\em Commun. Contemp. Math.}, 19(6):1750007, 27, 2017.

\bibitem{CavallettiMondino20}
F.~Cavalletti and A.~Mondino.
\newblock New formulas for the {L}aplacian of distance functions and applications.
\newblock {\em Anal. PDE}, 13(7):2091--2147, 2020.

\bibitem{CavallettiMondino22}
F.~Cavalletti and A.~Mondino.
\newblock A review of {L}orentzian synthetic theory of timelike {R}icci curvature bounds.
\newblock {\em Gen. Relativity Gravitation}, 54(11):Paper No. 137, 39, 2022.

\bibitem{CM:20}
F.~Cavalletti and A.~Mondino.
\newblock Optimal transport in {L}orentzian synthetic spaces, synthetic timelike {R}icci curvature lower bounds and applications.
\newblock {\em Camb. J. Math.}, 12(2):417--534, 2024.

\bibitem{CM24}
F.~Cavalletti and A.~Mondino.
\newblock A sharp isoperimetric-type inequality for {L}orentzian spaces satisfying timelike {R}icci lower bounds.
\newblock {\em arXiv:2401.03949}, 2024.

\bibitem{cavalletti-sturm}
F.~Cavalletti and K.-T. Sturm.
\newblock Local curvature-dimension condition implies measure-contraction property.
\newblock {\em J. Funct. Anal.}, 262(12):5110--5127, 2012.

\bibitem{ChabCrou09}
Y.~Chabrillac and J.-P. Crouzeix.
\newblock Continuity and differentiability properties of monotone real functions of several real variables.
\newblock {\em Math. Programming Stud.}, 30(30):1--16, 07 1987.
\newblock Nonlinear analysis and optimization (Louvain-la-Neuve, 1983).

\bibitem{Cheeger99}
J.~Cheeger.
\newblock Differentiability of {L}ipschitz functions on metric measure spaces.
\newblock {\em Geom. Funct. Anal.}, 9(3):428--517, 1999.

\bibitem{CheegerColding96}
J.~Cheeger and T.~H. Colding.
\newblock Lower bounds on {R}icci curvature and the almost rigidity of warped products.
\newblock {\em Ann. of Math. (2)}, 144(1):189--237, 1996.

\bibitem{CC97}
J.~Cheeger and T.~H. Colding.
\newblock On the structure of spaces with {R}icci curvature bounded below. {I}.
\newblock {\em J. Differential Geom.}, 46(3):406--480, 1997.

\bibitem{CC20_1}
J.~Cheeger and T.~H. Colding.
\newblock On the structure of spaces with {R}icci curvature bounded below. {II}.
\newblock {\em J. Differential Geom.}, 54(1):13--35, 2000.

\bibitem{CC20_2}
J.~Cheeger and T.~H. Colding.
\newblock On the structure of spaces with {R}icci curvature bounded below. {III}.
\newblock {\em J. Differential Geom.}, 54(1):37--74, 2000.

\bibitem{CorderoMcCannSchmuckenschlager01}
D.~Cordero-Erausquin, R.~J. McCann, and M.~Schmuckenschl\"{a}ger.
\newblock A {R}iemannian interpolation inequality \`a la {B}orell, {B}rascamp and {L}ieb.
\newblock {\em Invent. Math.}, 146(2):219--257, 2001.

\bibitem{DeGiorgi92}
E.~De~Giorgi.
\newblock New problems on minimizing movements.
\newblock In {\em Boundary value problems for partial differential equations and applications}, volume~29 of {\em RMA Res. Notes Appl. Math.}, pages 81--98. Masson, Paris, 1993.

\bibitem{Den21}
Q.~Deng.
\newblock {\em H{\"o}lder {C}ontinuity of {T}angent {C}ones and {N}on-{B}ranching in {RCD}({K},{N}) {S}paces}.
\newblock ProQuest LLC, Ann Arbor, MI, 2021.
\newblock Thesis (Ph.D.)--University of Toronto (Canada).

\bibitem{DiPernaLions89}
R.~J. DiPerna and P.-L. Lions.
\newblock Ordinary differential equations, transport theory and {S}obolev spaces.
\newblock {\em Invent. Math.}, 98(3):511--547, 1989.

\bibitem{EM:17}
M.~Eckstein and T.~Miller.
\newblock Causality for nonlocal phenomena.
\newblock {\em Ann. Henri Poincar\'{e}}, 18(9):3049--3096, 2017.

\bibitem{ErbarKuwadaSturm:2015}
M.~Erbar, K.~Kuwada, and K.-T. Sturm.
\newblock On the equivalence of the entropic curvature-dimension condition and {B}ochner's inequality on metric measure spaces.
\newblock {\em Invent. Math.}, 201(3):993--1071, 2015.

\bibitem{Eschenburg:1988}
J.-H. Eschenburg.
\newblock The splitting theorem for space-times with strong energy condition.
\newblock {\em J. Differential Geom.}, 27(3):477--491, 1988.

\bibitem{Folland2013}
G.~Folland.
\newblock {\em Real Analysis: Modern Techniques and Their Applications}.
\newblock Pure and Applied Mathematics: A Wiley Series of Texts, Monographs and Tracts. Wiley, 1999.

\bibitem{Fuk87a}
K.~Fukaya.
\newblock Collapsing of {R}iemannian manifolds and eigenvalues of {L}aplace operator.
\newblock {\em Invent. Math.}, 87(3):517--547, 1987.

\bibitem{Galloway:89}
G.~J. Galloway.
\newblock The {L}orentzian splitting theorem without the completeness assumption.
\newblock {\em J. Differential Geom.}, 29(2):373--387, 1989.

\bibitem{Gigli24+}
N.~Gigli.
\newblock Hyperbolic {B}anach spaces.
\newblock In preparation.

\bibitem{Gigli12a}
N.~Gigli.
\newblock Optimal maps in non branching spaces with {R}icci curvature bounded from below.
\newblock {\em Geom. Funct. Anal.}, 22(4):990--999, 2012.

\bibitem{Gigli:2015}
N.~Gigli.
\newblock On the differential structure of metric measure spaces and applications.
\newblock {\em Mem. Amer. Math. Soc.}, 236(1113):vi+91, 2015.

\bibitem{Gig:18}
N.~Gigli.
\newblock Nonsmooth differential geometry---an approach tailored for spaces with {R}icci curvature bounded from below.
\newblock {\em Mem. Amer. Math. Soc.}, 251(1196):v+161, 2018.

\bibitem{GigliDGG:23}
N.~Gigli.
\newblock De {G}iorgi and {G}romov working together.
\newblock {\em arXiv:2306.14604}, 2023.

\bibitem{gigli2013}
N.~Gigli.
\newblock {T}he splitting theorem in non-smooth context, \textit{Mem. Amer. Math. Soc.}, to appear.

\bibitem{GP:20}
N.~Gigli and E.~Pasqualetto.
\newblock {\em Lectures on nonsmooth differential geometry}, volume~2 of {\em SISSA Springer Series}.
\newblock Springer, Cham, [2020] \copyright 2020.

\bibitem{giglirajalasturm}
N.~Gigli, T.~Rajala, and K.-T. Sturm.
\newblock Optimal maps and exponentiation on finite-dimensional spaces with {R}icci curvature bounded from below.
\newblock {\em J. Geom. Anal.}, 26(4):2914--2929, 2016.

\bibitem{Gromov1999MetricStructures}
M.~Gromov.
\newblock {\em Metric structures for {R}iemannian and non-{R}iemannian spaces}, volume 152 of {\em Progress in Mathematics}.
\newblock Birkh\"auser Boston, Inc., Boston, MA, 1999.
\newblock Based on the 1981 French original [MR0682063 (85e:53051)], With appendices by M. Katz, P. Pansu and S. Semmes, Translated from the French by Sean Michael Bates.

\bibitem{Haw66}
S.~W. Hawking.
\newblock The occurrence of singularities in cosmology. {I}.
\newblock {\em Proc. Roy. Soc. London Ser. A}, 294:511--521, 1966.

\bibitem{HKMcC:76}
S.~W. Hawking, A.~R. King, and P.~J. McCarthy.
\newblock A new topology for curved space-time which incorporates the causal, differential, and conformal structures.
\newblock {\em J. Math. Phys.}, 17(2):174--181, 1976.

\bibitem{kechris1995}
A.~S. Kechris.
\newblock {\em Classical descriptive set theory}, volume 156 of {\em Graduate Texts in Mathematics}.
\newblock Springer-Verlag, New York, 1995.

\bibitem{Ket15}
C.~Ketterer.
\newblock Cones over metric measure spaces and the maximal diameter theorem.
\newblock {\em J. Math. Pures Appl. (9)}, 103(5):1228--1275, 2015.

\bibitem{Ket24}
C.~Ketterer.
\newblock Characterization of the null energy condition via displacement convexity of entropy.
\newblock {\em J. Lond. Math. Soc. (2)}, 109(1):Paper No. e12846, 24, 2024.

\bibitem{KorevaarSchoen93}
N.~J. Korevaar and R.~M. Schoen.
\newblock Sobolev spaces and harmonic maps for metric space targets.
\newblock {\em Comm. Anal. Geom.}, 1(3-4):561--659, 1993.

\bibitem{KMcM:98}
P.~Koskela and P.~MacManus.
\newblock Quasiconformal mappings and {S}obolev spaces.
\newblock {\em Studia Math.}, 131(1):1--17, 1998.

\bibitem{KronheimerPenrose:1967}
E.~H. Kronheimer and R.~Penrose.
\newblock On the structure of causal spaces.
\newblock {\em Proc. Cambridge Philos. Soc.}, 63:481--501, 1967.

\bibitem{KS:18}
M.~Kunzinger and C.~S\"{a}mann.
\newblock Lorentzian length spaces.
\newblock {\em Ann. Global Anal. Geom.}, 54(3):399--447, 2018.

\bibitem{KS:2022}
M.~Kunzinger and R.~Steinbauer.
\newblock Null distance and convergence of {L}orentzian length spaces.
\newblock {\em Ann. Henri Poincar\'{e}}, 23(12):4319--4342, 2022.

\bibitem{Lisini:2007}
S.~Lisini.
\newblock Characterization of absolutely continuous curves in {W}asserstein spaces.
\newblock {\em Calc. Var. Partial Differential Equations}, 28(1):85--120, 2007.

\bibitem{LottVillani:2009}
J.~Lott and C.~Villani.
\newblock Ricci curvature for metric-measure spaces via optimal transport.
\newblock {\em Ann. of Math. (2)}, 169(3):903--991, 2009.

\bibitem{LuMinguzziOhtasplitting23}
Y.~Lu, E.~Minguzzi, and S.-i. Ohta.
\newblock Geometry of weighted {L}orentz--{F}insler manifolds ii: A splitting theorem.
\newblock {\em International Journal of Mathematics}, 34(01):2350002, 2023.

\bibitem{Malament77}
D.~B. Malament.
\newblock The class of continuous timelike curves determines the topology of spacetime.
\newblock {\em J. Math. Phys.}, 18(7):1399--1404, 1977.

\bibitem{McCann95}
R.~J. McCann.
\newblock Existence and uniqueness of monotone measure-preserving maps.
\newblock {\em Duke Math. J.}, 80(2):309--323, 1995.

\bibitem{McCann97}
R.~J. McCann.
\newblock A convexity principle for interacting gases.
\newblock {\em Adv. Math.}, 128(1):153--179, 1997.

\bibitem{McCann:01}
R.~J. McCann.
\newblock Polar factorization of maps on {R}iemannian manifolds.
\newblock {\em Geom. Funct. Anal.}, 11(3):589--608, 2001.

\bibitem{McCann:2020}
R.~J. McCann.
\newblock Displacement convexity of {B}oltzmann's entropy characterizes the strong energy condition from general relativity.
\newblock {\em Camb. J. Math.}, 8(3):609--681, 2020.

\bibitem{McC:23}
R.~J. McCann.
\newblock A synthetic null energy condition.
\newblock {\em Comm. Math. Phys.}, 405(2):Paper No. 38, 24, 2024.

\bibitem{Min08}
E.~Minguzzi.
\newblock Limit curve theorems in {L}orentzian geometry.
\newblock {\em J. Math. Phys.}, 49(9):092501, 18, 2008.

\bibitem{minguzzitopapp}
E.~Minguzzi.
\newblock Convexity and quasi-uniformizability of closed preordered spaces.
\newblock {\em Topology Appl.}, 160(8):965--978, 2013.

\bibitem{Minguzzi2019}
E.~Minguzzi.
\newblock Causality theory for closed cone structures with applications.
\newblock {\em Rev. Math. Phys.}, 31(5):1930001, 139, 2019.

\bibitem{Minguzzi:19}
E.~Minguzzi.
\newblock Lorentzian causality theory.
\newblock {\em Living reviews in relativity}, 22(1):3, 2019.

\bibitem{Min:23}
E.~Minguzzi.
\newblock Further observations on the definition of global hyperbolicity under low regularity.
\newblock {\em Classical Quantum Gravity}, 40(18):Paper No. 185001, 9, 2023.

\bibitem{MinguzziSuhr:2022}
E.~Minguzzi and S.~Suhr.
\newblock Lorentzian metric spaces and their {G}romov--{H}ausdorff convergence.
\newblock {\em Lett. Math. Phys.}, 114(3):Paper No. 73, 2024.

\bibitem{MN19}
A.~Mondino and A.~Naber.
\newblock Structure theory of metric measure spaces with lower {R}icci curvature bounds.
\newblock {\em J. Eur. Math. Soc. (JEMS)}, 21(6):1809--1854, 2019.

\bibitem{MS:22}
A.~Mondino and S.~Suhr.
\newblock An optimal transport formulation of the {E}instein equations of general relativity.
\newblock {\em J. Eur. Math. Soc. (JEMS)}, 25(3):933--994, 2023.

\bibitem{Mue22}
O.~M{\"u}ller.
\newblock Gromov-{H}ausdorff metrics and dimensions of {L}orentzian length spaces.
\newblock {\em arXiv:2209.12736}, 2022.

\bibitem{Mue24}
O.~M{\"u}ller.
\newblock Maximality and {C}auchy developments of {L}orentzian length spaces.
\newblock {\em arXiv:2404.06428}, 2024.

\bibitem{Newman:90}
R.~P. A.~C. Newman.
\newblock A proof of the splitting conjecture of {S}.-{T}. {Y}au.
\newblock {\em J. Differential Geom.}, 31(1):163--184, 1990.

\bibitem{Nol04}
J.~Noldus.
\newblock A {L}orentzian {G}romov-{H}ausdorff notion of distance.
\newblock {\em Classical Quantum Gravity}, 21(4):839--850, 2004.

\bibitem{Ohta:2007}
S.~Ohta.
\newblock On the measure contraction property of metric measure spaces.
\newblock {\em Comment. Math. Helv.}, 82(4):805--828, 2007.

\bibitem{Oht09}
S.~Ohta.
\newblock Finsler interpolation inequalities.
\newblock {\em Calc. Var. Partial Differential Equations}, 36(2):211--249, 2009.

\bibitem{Ohta:2014}
S.~Ohta.
\newblock On the curvature and heat flow on {H}amiltonian systems.
\newblock {\em Anal. Geom. Metr. Spaces}, 2(1):81--114, 2014.

\bibitem{OS:09}
S.~Ohta and K.-T. Sturm.
\newblock Heat flow on {F}insler manifolds.
\newblock {\em Comm. Pure Appl. Math.}, 62(10):1386--1433, 2009.

\bibitem{OttoVillani00}
F.~Otto and C.~Villani.
\newblock Generalization of an inequality by {T}alagrand and links with the logarithmic {S}obolev inequality.
\newblock {\em J. Funct. Anal.}, 173(2):361--400, 2000.

\bibitem{Pen65}
R.~Penrose.
\newblock Gravitational collapse and space-time singularities.
\newblock {\em Phys. Rev. Lett.}, 14:57--59, 1965.

\bibitem{Rajala2012}
T.~Rajala.
\newblock Interpolated measures with bounded density in metric spaces satisfying the curvature-dimension conditions of {S}turm.
\newblock {\em J. Funct. Anal.}, 263(4):896--924, 2012.

\bibitem{Rajala2013}
T.~Rajala.
\newblock Improved geodesics for the reduced curvature-dimension condition in branching metric spaces.
\newblock {\em Discrete Contin. Dyn. Syst.}, 33(7):3043--3056, 2013.

\bibitem{RajalaSturm:2014}
T.~Rajala and K.-T. Sturm.
\newblock Non-branching geodesics and optimal maps in strong {$CD(K,\infty)$}-spaces.
\newblock {\em Calc. Var. Partial Differential Equations}, 50(3-4):831--846, 2014.

\bibitem{Santambrogio15}
F.~Santambrogio.
\newblock {\em Optimal transport for applied mathematicians}, volume~87 of {\em Progress in Nonlinear Differential Equations and their Applications}.
\newblock Birkh\"auser/Springer, Cham, 2015.
\newblock Calculus of variations, PDEs, and modeling.

\bibitem{Sau89}
J.-L. Sauvageot.
\newblock Tangent bimodule and locality for dissipative operators on {$C^*$}-algebras.
\newblock In {\em Quantum probability and applications, {IV} ({R}ome, 1987)}, volume 1396 of {\em Lecture Notes in Math.}, pages 322--338. Springer, Berlin, 1989.

\bibitem{Sau90}
J.-L. Sauvageot.
\newblock Quantum {D}irichlet forms, differential calculus and semigroups.
\newblock In {\em Quantum probability and applications, {V} ({H}eidelberg, 1988)}, volume 1442 of {\em Lecture Notes in Math.}, pages 334--346. Springer, Berlin, 1990.

\bibitem{Shanmugalingam00}
N.~Shanmugalingam.
\newblock Newtonian spaces: an extension of {S}obolev spaces to metric measure spaces.
\newblock {\em Rev. Mat. Iberoamericana}, 16(2):243--279, 2000.

\bibitem{SW96}
R.~D. Sorkin and E.~Woolgar.
\newblock A causal order for spacetimes with {$C^0$} {L}orentzian metrics: proof of compactness of the space of causal curves.
\newblock {\em Classical Quantum Gravity}, 13(7):1971--1993, 1996.

\bibitem{SV:16}
C.~Sormani and C.~Vega.
\newblock Null distance on a spacetime.
\newblock {\em Classical Quantum Gravity}, 33(8):085001, 29, 2016.

\bibitem{Sturm:2006a}
K.-T. Sturm.
\newblock On the geometry of metric measure spaces. {I}.
\newblock {\em Acta Math.}, 196(1):65--131, 2006.

\bibitem{Sturm:2006b}
K.-T. Sturm.
\newblock On the geometry of metric measure spaces. {II}.
\newblock {\em Acta Math.}, 196(1):133--177, 2006.

\bibitem{Suh:18}
S.~Suhr.
\newblock Theory of optimal transport for {L}orentzian cost functions.
\newblock {\em M\"{u}nster J. Math.}, 11(1):13--47, 2018.

\bibitem{TaoPoincare}
T.~Tao.
\newblock 254a, lecture 8: The mean ergodic theorem.
\newblock https://terrytao.wordpress.com/2008/01/30/254a-lecture-8-the-mean-ergodic-theorem/, 2008.

\bibitem{Treude:2011}
J.-H. Treude.
\newblock Ricci curvature comparison in {Riemannian} and {Lorentzian} geometry.
\newblock Master's thesis, Universit{\"a}t Freiburg, 2011.

\bibitem{TreudeGrant:2013}
J.-H. Treude and J.~D.~E. Grant.
\newblock Volume comparison for hypersurfaces in {L}orentzian manifolds and singularity theorems.
\newblock {\em Ann. Global Anal. Geom.}, 43(3):233--251, 2013.

\bibitem{Villani:2009}
C.~Villani.
\newblock {\em Optimal transport}, volume 338 of {\em Grundlehren der mathematischen Wissenschaften [Fundamental Principles of Mathematical Sciences]}.
\newblock Springer-Verlag, Berlin, 2009.
\newblock Old and new.

\bibitem{vonRenesseSturm05}
M.-K. von Renesse and K.-T. Sturm.
\newblock Transport inequalities, gradient estimates, entropy, and {R}icci curvature.
\newblock {\em Comm. Pure Appl. Math.}, 58(7):923--940, 2005.

\bibitem{Wea00}
N.~Weaver.
\newblock Lipschitz algebras and derivations. {II}. {E}xterior differentiation.
\newblock {\em J. Funct. Anal.}, 178(1):64--112, 2000.

\bibitem{Weaver:2018}
N.~Weaver.
\newblock {\em Lipschitz algebras}.
\newblock World Scientific Publishing Co. Pte. Ltd., Hackensack, NJ, second edition, 2018.

\end{thebibliography}
\end{document}